%% file: bipolar-ldp-sdp-main.tex
\documentclass[11pt]{article}
\usepackage[T1]{fontenc}
\usepackage{lmodern}

\usepackage{amsmath,amsthm,amssymb}
\usepackage[usenames,dvipsnames]{xcolor}
\usepackage{enumerate}
\usepackage{graphicx}
\usepackage{cite}
\usepackage{comment}
\usepackage{oands}
\usepackage{tikz}
\usepackage{changepage}
\usepackage{bbm}
\usepackage{mathtools}
\usepackage[margin=1in]{geometry}
\usepackage{enumitem} % controls labels/refs for lists
\usepackage[colorinlistoftodos]{todonotes}

\newcommand{\jb}[1]{{\color{black}{#1}}}
\newcommand{\jbb}[1]{{\color{black}{#1}}}

\newcommand{\ewain}[1]{{\color{black}{#1}}}

\usepackage[pdftitle={Directed distances in bipolar-oriented triangulations: exact exponents and scaling limits},
  pdfauthor={Jacopo Borga and Ewain Gwynne},
colorlinks=true,linkcolor=NavyBlue,urlcolor=RoyalBlue,citecolor=PineGreen,bookmarks=true,bookmarksopen=true,bookmarksopenlevel=2,unicode=true,linktocpage]{hyperref}

\theoremstyle{plain}
\newtheorem{thm}{Theorem}[section]

\newtheorem{lem}[thm]{Lemma}
\newtheorem{prop}[thm]{Proposition}

\newtheorem{conj}[thm]{Conjecture}

\def\@rst #1 #2other{#1}
\newcommand\MR[1]{\relax\ifhmode\unskip\spacefactor3000 \space\fi
  \MRhref{\expandafter\@rst #1 other}{#1}}
\newcommand{\MRhref}[2]{\href{http://www.ams.org/mathscinet-getitem?mr=#1}{MR#2}}

\theoremstyle{definition}
\newtheorem{defn}[thm]{Definition}
\newtheorem{remark}[thm]{Remark}

\newtheorem{prob}[thm]{Problem}

\numberwithin{equation}{section}

\def\alb#1\ale{\begin{align*}#1\end{align*}}
\def\allb#1\alle{\begin{align}#1\end{align}}
\newcommand{\eqb}{\begin{equation}}
\newcommand{\eqe}{\end{equation}}
\newcommand{\eqbn}{\begin{equation*}}
\newcommand{\eqen}{\end{equation*}}

\newcommand{\BB}{\mathbbm}
\newcommand{\ol}{\overline}
\newcommand{\ul}{\underline}
\newcommand{\op}{\operatorname}

\newcommand{\im}{\operatorname{Im}}
\newcommand{\re}{\operatorname{Re}}
\newcommand{\frk}{\mathfrak}
\newcommand{\eqD}{\overset{d}{=}}
\newcommand{\ep}{\varepsilon}

\newcommand{\wt}{\widetilde}
\newcommand{\wh}{\widehat}
\newcommand{\mcl}{\mathcal}

\newcommand{\bdy}{\partial}
\newcommand{\rng}{\mathring}

\newcommand{\el}{l}

\newcommand{\eps}{\varepsilon}

\newcommand{\uibot}{M_{-\infty,\infty}}

\let\originalleft\left
\let\originalright\right
\renewcommand{\left}{\mathopen{}\mathclose\bgroup\originalleft}
\renewcommand{\right}{\aftergroup\egroup\originalright}

\let\originalsim\sim
\renewcommand{\sim}{{\originalsim}}

\usepackage[capitalise,noabbrev]{cleveref}

\title{Directed distances in bipolar-oriented triangulations:\\
exact exponents and scaling limits}
 \date{ }
 \author{
\begin{tabular}{c} Jacopo Borga\footnote{jborga@mit.edu} \\[-3pt]\small MIT \end{tabular} 
\begin{tabular}{c} Ewain Gwynne\footnote{ewain@uchicago.edu} \\[-3pt]\small University of Chicago \end{tabular} 
}

\begin{document}

\newcommand{\XDP}{{\operatorname{XDP}}}
\newcommand{\bc}{{\operatorname{b}}}

\maketitle

\begin{abstract}
We study longest and shortest directed paths in the following natural model of directed random planar maps: the uniform infinite bipolar-oriented triangulation (UIBOT), which is the local limit of uniform bipolar-oriented triangulations around a typical edge.
We construct the Busemann function which measures directed distance to $\infty$ along a natural interface in the UIBOT. We show that in the case of longest (resp.\ shortest) directed paths, this Busemann function converges in the scaling limit to a $2/3$-stable Lévy process (resp.\ a $4/3$-stable Lévy process). 

We also prove up-to-constants bounds for directed distances in finite bipolar-oriented triangulations sampled from a Boltzmann distribution, and for size-$n$ cells in the UIBOT. These bounds imply that in a typical subset of the UIBOT with $n$ edges, longest directed path lengths are of order $n^{3/4}$ and shortest directed path lengths are of order $n^{3/8}$. These results give the scaling dimensions for discretizations of the (hypothetical) $\sqrt{4/3}$-directed Liouville quantum gravity metrics.

The main external input in our proof is the bijection of Kenyon-Miller-Sheffield-Wilson (2015). We do not use any continuum theory. 
We expect that our techniques can also be applied to prove similar results for directed distances in other random planar map models and for longest increasing subsequences in pattern-avoiding permutations. 
\end{abstract}

\tableofcontents

\bigskip
\noindent\textbf{Acknowledgments.} We thank Morris Ang, Olivier Bernardi, Manan Bhatia, William Da Silva, Sayan Das, and Scott Sheffield for helpful discussions. We also thank Yuanzheng Wang for comments on an early draft of this paper. J.B.\ was partially supported by NSF grant DMS-2441646.
E.G.\ was partially supported by NSF grant DMS-2245832. 

\bigskip
\noindent\textbf{Lean formalization.} A Lean formalization of all 90 numbered definitions, theorems, propositions, and lemmas in this paper, including those in Appendix A, is available at \url{https://github.com/jacopoborga/DLQG1-bipolar-triangulations-Lean}. The repository includes the supporting results needed for their proofs and a detailed correspondence between the paper and the formal statements, documenting any differences in formulation.

\section{Introduction}
\label{sec-intro}

\input{tex/intro.tex}

\section{Related models and future outlook}
\label{sec-related}

\input{tex/related.tex}

\section{KMSW bijection and conditional independence property}
\label{sec-prelim}

\input{tex/prelim.tex}

\section{Cut vertices and geodesics in the UIQBOT}
\label{sec-uiqbot}

\input{tex/uiqbot.tex}

\section{Existence and basic properties of the Busemann function}
\label{sec-busemann}

\input{tex/busemann.tex}

\section{Recursive equations, tail asymptotics, and scaling exponents} 
\label{sec-recursive}

\input{tex/recursive.tex}

\section{The case of finite triangulations}
\label{sec-finite}

\input{tex/finite.tex}

\section{Open problems}
\label{sec-open-problems}

\input{tex/open-problems.tex}

\appendix

\section{Proof of the Tauberian result}\label{sect:proofTau}

\input{tex/tauberian.tex}

\bibliography{cibib,cibib2}
\bibliographystyle{hmralphaabbrv}

\end{document}

%% file: tex/intro.tex
\subsection{Overview} 
\label{sec-overview}

A planar map is a graph embedded in the complex plane $\BB C$, viewed modulo orientation-preserving homeomorphisms $\BB C\to\BB C$. %A \textbf{triangulation} is a planar map whose faces all have degree three, except possibly the external (unbounded) face. 
Random planar maps have attracted an enormous amount of attention in probability and mathematical physics in recent years. One reason for this is that random planar maps are discrete analogs of \textbf{Liouville quantum gravity (LQG)}, a one-parameter family of random fractal surfaces, indexed by $\gamma \in (0,2]$, which arise in various areas of physics. 
Although the results of this paper are partially motivated by LQG, we will not use this theory in this paper; see~\cite{bp-lqg-notes,sheffield-icm,gwynne-ams-survey} for introductory expository articles. 

There is a vast literature concerning graph distances in random planar maps which are sampled uniformly from some finite set of possibilities. 
A major achievement in this area is the independent proofs by Le Gall~\cite{legall-uniqueness} and Miermont~\cite{miermont-brownian-map} of the following statement. If we take a uniform triangulation (or uniform $p$-angulations for even $p\geq 4$) with $n$ edges and scale distances by $n^{-1/4}$, then the resulting metric space converges in law as $n\to\infty$, with respect to the Gromov-Hausdorff topology, to a random metric space called the \textbf{Brownian map}. There are numerous extensions of the Brownian map convergence result to other models, including different local constraints on the maps and different underlying topologies; see, e.g., the survey~\cite{legall-brownian-geometry}. Miller and Sheffield~\cite{lqg-tbm1,lqg-tbm2} showed that the Brownian map is isomorphic to LQG with parameter $\gamma=\sqrt{8/3}$. 
 
In contrast, much less is known about graph distances in \textbf{decorated random planar maps}, where instead of sampling a planar map uniformly from a set of possibilities, we sample a random pair $(M,S)$ consisting of a planar map $M$ and some sort of decoration $S$ on $M$. Some interesting choices of decoration include, e.g., a spanning tree, a statistical physics model, or a bipolar orientation. When we sample a decorated planar map, the marginal law of $M$ is not uniform: rather, it is sampled with probability proportional to the partition function of the decoration (which, roughly speaking, is the weighted number of possibilities for $S$ on a given choice of $M$). 

For most interesting choices of decoration, it is believed that $M$ should have very different large-scale behavior as compared to uniform random planar maps. In many cases, the scaling limit of $M$ is believed to be described by LQG with parameter $\gamma \in (0,2]\setminus \{\sqrt{8/3}\}$ depending on the particular type of decoration (see~\cite{ghs-mating-survey} or~\cite[Section 4]{bp-lqg-notes} for an overview of conjectures of this type). Moreover, for a map with $n$ edges, the proper scaling for distances is believed to be $n^{-1/d_\gamma}$, where $d_\gamma$ is the Hausdorff dimension of the $\gamma$-LQG metric space. In some cases (including bipolar-oriented triangulations, as defined just below), this can be proven up to an $n^{o(1)}$ multiplicative error, see~\cite{ghs-map-dist,dg-lqg-dim}. However, the value of $d_\gamma$ is unknown for $\gamma\not=\sqrt{8/3}$ (we have $d_{\sqrt{8/3}}=4$), and computing it is one of the most important open problems in LQG theory. 

%The study of distances in uniform planar maps relies on various combinatorial miracles, including the Schaeffer bijection~\cite{schaeffer-bijection} and the peeling procedure~\cite{angel-peeling,curien-peeling-notes}. Combinatorial miracles exist for some decorated random planar map models, often in the form of mating-of-trees bijections (see the survey~\cite{ghs-mating-survey}). However, these bijections do not encode graph distances in a tractable way.

In this paper, we study \textbf{directed distances} (lengths of shortest and longest directed paths) in a natural model of directed random planar maps, namely uniform bipolar-oriented triangulations (Definition~\ref{def-bipolar} just below). This is the random planar map analog of studying \emph{directed} first- or last-passage percolation (as opposed to \emph{undirected} first-passage percolation). See Section~\ref{sect:fpp-lpp} for further discussion of the analogy between this paper and results for directed first- and last-passage percolation.

\begin{defn} \label{def-bipolar}
A \textbf{directed planar map} is a planar map together with an orientation on its edges. The orientation is \textbf{acyclic} if there are no directed cycles. The orientation is \textbf{bipolar} if it is acyclic and it has at most one \textbf{source} (vertex with only outgoing edges) and  at most one \textbf{sink} (vertex with only incoming edges).\footnote{For a finite planar map, an acyclic orientation always has at least one source and at least one sink, but for an infinite planar map the source, the sink, or both can be ``at $\infty$''.}
A \textbf{bipolar-oriented triangulation} is a bipolar-oriented planar map whose faces all have degree three, except possibly the external (unbounded) face. \jb{We always assume that both the sink and the source of a bipolar-oriented triangulation are on the boundary of the external face.}
See the left-hand side of Figure~\ref{fig-bipolar} for an illustration. 
\end{defn}

Whenever we speak about \emph{left} and \emph{right}, we always refer to left and right with respect to the orientation of the map (and not with respect to how we draw the map on the plane).

Bipolar orientations on planar graphs, often also called $st$-planar orientations, have been intensively studied in combinatorics. Indeed, they are related to the study of upward planar graph drawing and more general graph drawing on surfaces~\cite{tamassia2013handbook}, the study of the distributive lattice on the set of $st$-orientations of a fixed plane graph under cycle-flip moves~\cite{propp2002lattice,felsner2004lattice}, and the enumerative and bijective properties of various combinatorial structures, including Baxter permutations and non-intersecting lattice paths~\cite{fps-counting-bipolar,bbf-baxter}. Moreover, they have been used as bridges to study other structures on planar graphs, including (grand-)Schnyder woods~\cite{ffno-baxter-family,lsw-schnyder-wood,bernardi2025grand}.
Bipolar orientations enjoy the following properties: (i) given a finite bi-connected planar map and a designated source and sink pair on the outer face, there is always a way to orient the edges so that the result is a bipolar orientation~\cite{lempel1967algorithm}; (ii) every bipolar-oriented map can be drawn in the plane with every edge oriented from south-west to north-east~\cite{di1988algorithms} (see also Figure~\ref{fig-bipolar}); (iii) a bipolar orientation on a map induces a bipolar orientation on its dual map~\cite{de1995bipolar}.
 
A key tool in the study of bipolar-oriented random planar maps is the KMSW bijection of Kenyon, Miller, Sheffield, and Wilson~\cite{kmsw-bipolar}, which, when restricted to triangulations, encodes a bipolar-oriented triangulation in terms of a walk in $\BB Z^2$ with increments in 
\[\{(1,-1), (-1,0) , (0,1)\}.\] 
We review this bijection in Section~\ref{sec-kmsw}. 
It can be seen from the KMSW bijection that uniform bipolar-oriented triangulations should be in the universality class of LQG with $\gamma=\sqrt{4/3}$~\cite[Section 4]{kmsw-bipolar}.  

The main results of this paper show that, in apparent contrast to undirected distances, directed distances in uniform bipolar-oriented triangulations are exactly solvable.  
In particular, we first establish the existence of the \textbf{Busemann function}, which measures ``directed distances to $\infty$'' along a natural interface in the uniform infinite bipolar-oriented triangulation (UIBOT); see Theorem~\ref{thm-busemann}. We then show that the Busemann function converges in the scaling limit to a $2/3$-stable Lévy process in the case of longest directed paths, and to a $4/3$-stable Lévy process in the case of shortest directed paths (Theorem~\ref{thm-busemann-conv}). 

We also obtain the exact exponents for the typical lengths of shortest and longest directed paths in certain finite bipolar-oriented triangulations, with no $o(1)$ error in the exponent (Theorems~\ref{thm-boltzmann-ldp}, \ref{thm-cell-ldp}, and~\ref{thm-cell-sdp}). In particular, these results imply that in typical size-$n$ submaps of the UIBOT, the lengths of longest directed paths grow like $n^{3/4}$ and the lengths of shortest directed paths grow like $n^{3/8}$. Prior to this work, almost\footnote{It is possible to extract from~\cite[Corollary 1.14]{bgs-meander} that the length of the longest directed path in a uniform bipolar-oriented map with $n$ edges is typically of order $o(n)$.} no non-trivial bounds were known for directed distances in any model of bipolar-oriented random planar maps. 

The only inputs in our proof are the KMSW bijection and some elementary probabilistic facts. We do not use any continuum theory. See Section~\ref{sec-outline} for an outline. 

Our results can be viewed as a substantial first step toward obtaining a full scaling limit of (shortest and longest) directed distances in uniform bipolar-oriented triangulations, analogous to the scaling limit of undirected distances in uniform planar maps to the Brownian map~\cite{legall-uniqueness,miermont-brownian-map}. The limiting objects should be \textbf{directed versions of the $\sqrt{4/3}$-LQG metric} (one version for longest paths and one for shortest paths), or equivalently $\sqrt{4/3}$-LQG analogs of the \textbf{directed landscape} of~\cite{dov-dl}. Our results suggest that the scaling exponents for the longest (resp.\ shortest) directed path version of the $\sqrt{4/3}$-LQG metric should be $3/4$ (resp.\ $3/8$). These scaling exponents are the directed analogs of the (reciprocal of) the dimension of the undirected $\sqrt{4/3}$-LQG metric space. We discuss this in more detail in Section~\ref{sec-lqg}.

Moreover, our proof strategy is fairly general. It can potentially be extended to obtain similar results for, e.g., (a) directed distances in other random planar map models, including bipolar-oriented maps with other face degree distributions and planar maps decorated by Schnyder woods; (b) shortest directed paths hit in order by various paths on random planar maps, such as the contour exploration of a uniform spanning tree or a percolation interface; and (c) longest increasing subsequences in various types of random permutations, such as Baxter permutations, semi-Baxter permutations, and strong Baxter permutations. We have several works in progress along these lines with various coauthors. See Section~\ref{sec-other} for further discussion.

\begin{figure}[ht!]
\begin{center}
\includegraphics[width=0.52\textwidth]{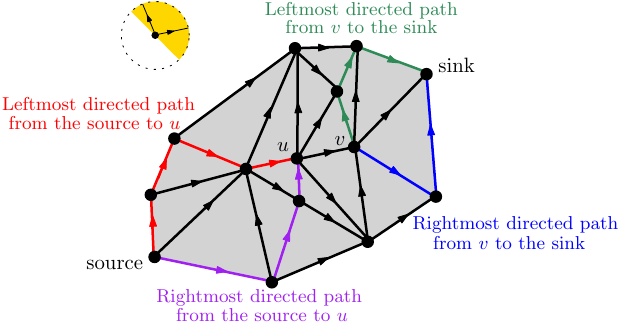}  
\includegraphics[width=0.47\textwidth]{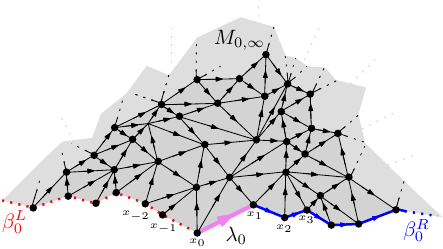}  
\caption{\label{fig-bipolar} 
\textbf{Left:}  A finite bipolar-oriented triangulation, drawn so that edges are directed from south-west to north-east (i.e.\ edges can point in the direction highlighted in yellow in the top-left circle).
The four colored directed path denote the leftmost/rightmost paths introduced just before Definition~\ref{defn-future-map}.
\textbf{Right:} A portion of the infinite rooted submap $(M_{0,\infty},\lambda_0)$ of the UIBOT introduced in Definition~\ref{defn-future-map}. Edges on the boundary of $M_{0,\infty}$ lying to the west of 
the root edge $\lambda_0$ are not considered to be part of $M_{0,\infty}$. See Definition~\ref{defn:planar-map-with-missing-edges} for a precise definition of planar map with missing edges.}
\end{center}
\end{figure}

\subsection{Busemann function on the UIBOT}
\label{sec-busemann-intro}

We will henceforth assume that all of the planar maps considered in this paper are equipped with an acyclic orientation (Definition~\ref{def-bipolar}). Submaps of a map $\frk m$ are always assumed to be equipped with the orientation they inherit from $\frk m$.

\begin{defn} \label{def-ldp} 
For a directed planar map $\frk m$ and vertices $x,y\in \frk m$, we write $\op{LDP}_{\frk m}(x,y)$ 
(resp.\ $\op{SDP}_{\frk m}(x,y)$) for the \textbf{length of the longest} (resp.\ \textbf{shortest}) \textbf{directed path} in $\frk m$ from $x$ to $y$, if at least one such path exists. 
If there is no directed path from $x$ to $y$, we instead write $\op{LDP}_{\frk m}(x,y) = -\infty$ and $\op{SDP}_{\frk m}(x,y) = \infty$. 
Many of our theorem statements and proofs are the same for both longest and shortest directed paths. To express this concisely, we write $\XDP \in \{\op{LDP} , \op{SDP}\}$ to refer to either shortest or longest directed paths.  
\end{defn}

The \textbf{uniform infinite bipolar-oriented triangulation} is an infinite bipolar-oriented triangulation\footnote{See Definition~\ref{defn-future-map} for the motivation for the notation $M_{-\infty,\infty}$.} $(M_{-\infty,\infty},\lambda_0 )$, equipped with a directed root edge $\lambda_0$, with no source or sink (the source is at $-\infty$ and the sink is at $\infty$). It arises as the (directed) Benjamini-Schramm~\cite{benjamini-schramm-topology} local limit of uniform bipolar-oriented triangulations with $n$ edges and a fixed number of left and right boundary edges, rooted at a uniformly chosen edge~\cite[Proposition 3.16]{ghs-map-dist}; see Definition~\ref{def-bs-topology}. 

As explained in~\cite[Section 1.2]{kmsw-bipolar}, for each vertex $v$ of $M_{-\infty,\infty}$, the edges of $M_{-\infty,\infty}$ incident to $v$ in cyclic order consist of a single group of incoming edges followed by a single group of outgoing edges. Hence, there is a unique outgoing edge from $v$ that is leftmost in the sense that the edge immediately to its left is directed toward $v$. 
By traversing the leftmost outgoing edge from $v$, then the leftmost outgoing edge from the endpoint of this edge, and so forth, we obtain the \textbf{leftmost directed path} $\beta$ from $v$ to $\infty$. The path $\beta$ has the property that each edge traversed by $\beta$ is the leftmost outgoing edge from its initial vertex. The leftmost directed path is called the ``northwest path'' in~\cite{kmsw-bipolar}. Since our convention is that edges go from south-west to north-east, instead of from south to north as in~\cite{kmsw-bipolar}, the leftmost directed path travels in the north direction in our setting. See the green path on the left-hand side of Figure~\ref{fig-bipolar} for an example (in the finite setting).

Similarly, we can define the \textbf{rightmost directed path} from $v$ to $\infty$, which travels in the east direction; see the blue path on the left-hand side of Figure~\ref{fig-bipolar} for an example. By performing the above procedure with the orientations of the edges reversed, then reversing the resulting path, we can also define the \textbf{leftmost} and \textbf{rightmost} directed paths from $-\infty$ to $v$. The time reversals of these paths travel in the west and south directions, respectively. See the red and purple paths on the left-hand side of Figure~\ref{fig-bipolar} for an example.

Let $\beta_0^R$ be the rightmost directed path from the terminal endpoint of the root edge $\lambda_0$ to $\infty$. 
Let $\beta_0^L$ be the leftmost directed path from $-\infty$ to the initial endpoint of $\lambda_0$. 

\begin{defn}\label{defn-future-map}
    Let $(M_{0,\infty},\lambda_0)$ be the rooted submap of $(M_{-\infty,\infty},\lambda_0)$ consisting of the vertices, edges, and faces of $M_{-\infty,\infty}$ that lie to the left of $\beta_0^R \cup \{\lambda_0\} \cup \beta_0^L$. We consider all of the vertices lying on $\beta_0^R \cup \{\lambda_0\} \cup \beta_0^L$, the edge $\lambda_0$, and the edges on $\beta_0^R$ to be part of $M_{0,\infty}$, but not the edges on $\beta_0^L$. 
    This means that $M_{0,\infty}$ is an infinite acyclically oriented triangulation with some missing\footnote{See Definition~\ref{defn:planar-map-with-missing-edges} for a precise definition of a planar map with missing edges.} boundary edges. 
\end{defn}

The reason why $M_{0,\infty}$ is natural is that $M_{0,\infty}$ is exactly the submap of $M_{-\infty,\infty}$ explored by the KMSW procedure after time 0; see Lemma~\ref{lem-uibot-bdy}. 
See the right-hand side of Figure~\ref{fig-bipolar} for an illustration of $M_{0,\infty}$. 
 
Let $\{x_k\}_{k\in\BB Z}$ be the vertices on the boundary of $M_{0,\infty}$, enumerated in directed order (i.e.\ from west to east) so that $\lambda_0 = (x_0 ,x_1)$. All vertices of $M_{0,\infty}$, including the boundary vertices, have an outgoing edge (Lemma~\ref{lem-uibot-bdy}). In particular, there is always a directed path starting from each boundary vertex.
Our first theorem establishes the existence of the Busemann function for longest (and shortest) directed paths from the boundary vertices of $M_{0,\infty}$ to $\infty$. 
Heuristically speaking, we want the Busemann function to be the function $\mcl X : \BB Z\to\BB Z$ defined by 
\eqb  \label{eqn-busemann-heuristic} 
\mcl X(k) = \XDP_{M_{0,\infty}} (x_k,\infty) -  \XDP_{M_{0,\infty}} (x_0 , \infty)   .
\eqe  
This definition does not make literal sense since $\XDP_{M_{0,\infty}} (x_k,\infty)=\infty$. But we can make sense of~\eqref{eqn-busemann-heuristic} by, roughly speaking, taking a limit of $ \XDP_{M_{0,\infty}} (x_k,w) - \XDP_{M_{0,\infty}} (x_0,w)$ as $w\to\infty$.

\begin{thm}\label{thm-busemann}
Let $\XDP \in \{\op{LDP} , \op{SDP}\}$. 
Let $\{x_k\}_{k\in\BB Z}$ be the boundary vertices of $M_{0,\infty}$, as above.
There exists a unique function 
\[\mcl X :  \BB Z \to \BB Z,\] 
called the \textbf{Busemann function}, such that $\mcl X(0) = 0$ and the following is true. Let $k,k' \in \BB Z$ and let\footnote{Our convention is that $\BB N = \{1,2,\dots\}$ and $\BB N_0 = \{0,1,2,\dots\}$.}
$\{w_j\}_{j\in\BB N}$ be a sequence\footnote{The existence of at least one such sequence is ensured by Lemma~\ref{lem-uibot-bdy}: the rightmost directed paths in $M_{0,\infty}$ started from $x_k$ and $x_{k'}$ coalesce into each other, so we can take $\{w_j\}$ to be a sequence of points hit by both of these rightmost directed paths.} of distinct vertices of $M_{0,\infty}$ such that, for each $j\in\BB N$, each of $x_k$ and $x_{k'}$ can be joined to $w_j$ by a directed path in $M_{0,\infty}$. 
Then
\eqb \label{eqn-busemann-lim}
\mcl X(k') -\mcl X(k) = \XDP_{M_{0,\infty}}(x_{k'} , w_j) - \XDP_{M_{0,\infty}}(x_{k } , w_j) ,\quad \text{for each sufficiently large $j\in\BB N$.}
\eqe 
\end{thm}

\begin{remark} \label{remark-busemann}
The term Busemann function originates from the work~\cite{busemann-geodesics}. 
Busemann functions are an important tool in the theory of first-passage percolation, in which setting they have been studied extensively; see~\cite[Chapter 5]{ahd-fpp-survey} for a survey. In the context of (exponential) last-passage percolation, the  existence of Busemann functions was established in~\cite{fp-second-class} and led to the construction of the stationary horizon~\cite{busani-busemann-lpp}. 
Recently, Busemann functions have also been important in the study of the directed landscape~\cite{dov-dl}, see~\cite{rv-infinite-geo, gz-difference-profile, bss-dl-semi} and the references therein. See Section~\ref{sect:fpp-lpp} for further discussion.
\end{remark}

To prove Theorem~\ref{thm-busemann}, we, roughly speaking, show that there are infinitely many vertices of $M_{0,\infty}$ that are each visited by every sufficiently long directed path started from $x_k$ or $x_{k'}$ (Lemma~\ref{lem-infty-cut}). 
In particular, this implies (Lemma~\ref{lem-geo-infty}) that there exists a (unique) leftmost infinite $\XDP$ geodesic started from each boundary vertex $x_k$ (Definition~\ref{def-infinite-geo}), and the leftmost infinite $\XDP$ geodesics from $x_k$ and $x_{k'}$ eventually coalesce into each other, i.e.\ we have coalescence of geodesics. But we currently do not have any quantitative bounds on the positions of the coalescence points. Lemma~\ref{lem-busemann-cls} expresses the Busemann function in terms of the coalescence times of the leftmost $\XDP$ geodesics.
%See Section~\ref{sec-outline} for a more detailed outline.
Our next main result gives the scaling limit of the Busemann function. 

\begin{thm} \label{thm-busemann-conv}
The following two scaling limit results for the Busemann function $\mcl X$ introduced in Theorem~\ref{thm-busemann} hold with respect to the local Skorokhod topology on $\BB R$ (here, all stable processes are strictly stable):
\begin{itemize}
\item
In the LDP case, the re-scaled Busemann function $\{k^{-3/2} \mcl X( \lfloor k t \rfloor)\}_{t\in\BB R}$ converges in law  to a process $X : \BB R\to\BB R$ such that $-X|_{[0,\infty)}$ is a $2/3$-stable subordinator; and $X(-\cdot)|_{[0,\infty)}$ is a symmetric $2/3$-stable Lévy process independent of $X|_{[0,\infty)}$. 
\item  
In the SDP case, the re-scaled Busemann function $\{k^{-3/4} \mcl X(\lfloor k t \rfloor)\}_{t\in\BB R}$ converges in law  to a process $X : \BB R\to\BB R$ such that $X|_{[0,\infty)}$ is a $4/3$-stable Lévy process with only upward jumps; and $X(-\cdot)|_{[0,\infty)}$ is a symmetric $4/3$-stable Lévy process independent of $  X|_{[0,\infty)}$.
\end{itemize} 
\end{thm}

In Theorem~\ref{thm-busemann-conv}, the stable processes are specified only up to a positive multiplicative constant. Theorem~\ref{thm-busemann-tail} below shows that, in each case (LDP or SDP), the multiplicative constants for $X|_{[0,\infty)}$ and  $X(-\cdot)|_{[0,\infty)}$ are the same.

Theorem~\ref{thm-busemann-conv} may be viewed as the $\sqrt{4/3}$-directed LQG analog of the description of the law of the Busemann function for the directed landscape restricted to a horizontal line, as given in~\cite{bss-dl-semi}. See the end of Section~\ref{sec-lqg} and Section~\ref{sect:fpp-lpp} for further discussion.
To prove Theorem~\ref{thm-busemann-conv}, we will establish the following five properties of the Busemann function, then apply the heavy-tailed functional central limit theorem. 

\begin{thm} \label{thm-busemann-property}
Let $\XDP \in \{\op{LDP} , \op{SDP}\}$ and let $\mcl X$ be the Busemann function as in Theorem~\ref{thm-busemann}. The following properties hold:
\begin{enumerate}[label=(\roman*),ref=(\roman*)]
\item \label{item-busemann-ind} \textbf{Independent increments.} The increments $\{\mcl X(k) - \mcl X(k-1)\}_{k\in\BB Z}$ are independent.
\item \label{item-busemann-stationary} \textbf{Stationary increments.} The increments $\{\mcl X(k) - \mcl X(k-1)\}_{k \geq 1}$ all have the same distribution, and the increments $\{\mcl X(k) - \mcl X(k-1)\}_{k \leq 0}$ all have the same distribution.
\item \label{item-busemann-pos} \textbf{Bounds on the increments for $k\geq 1$.} In the LDP case, for $k\geq 1$ we have the bound $\mcl X(k) - \mcl X(k-1)  \leq -1$ . In the SDP case, for $k\geq 1$, we have the bound $\mcl X(k) - \mcl X(k-1) \geq -1$. 
\item \label{item-busemann-sym} \textbf{Symmetry for $k\leq 0$.} For $k\leq 0$, we have $\mcl X(k)-\mcl X(k-1) \eqD \mcl X(k-1) - \mcl X(k)$.
\end{enumerate} 
\end{thm}

The fifth property provides the following tail estimates for the increments of the Busemann function.

\begin{thm} \label{thm-busemann-tail}
Let $\XDP \in \{\op{LDP} , \op{SDP}\}$ and let $\mcl X$ be the Busemann function, as in Theorem~\ref{thm-busemann}. 
\begin{enumerate}[label=(v\alph*),ref=(v\alph*)]
 \item\label{item-busemann-tail-LDP} If $\XDP = \op{LDP}$, then there is a constant $c_1 > 0$ such that as $x\to\infty$, 
\begin{align} \label{eqn-busemann-tail-ldp}
&\BB P\left[ \mcl X(1) - \mcl X(0) \leq -x \right] = c_1 \, x^{-2/3} + o(x^{-2/3}) ,\notag\\
&\BB P\left[ |\mcl X(0) - \mcl X(-1)| \geq x \right] = 2c_1 \, x^{-2/3} + o(x^{-2/3}) .
\end{align} 
 \item\label{item-busemann-tail-SDP} If $\XDP = \op{SDP}$, then there is a constant $c_2 > 0$ such that as $x\to\infty$,
\begin{align} \label{eqn-busemann-tail-sdp}
&\BB P\left[ \mcl X(1) - \mcl X(0) \geq   x \right] = c_2 \, x^{-4/3} + o(x^{-4/3}),\notag\\  
&\BB P\left[ |\mcl X(0) - \mcl X(-1)| \geq x \right] = 2c_2 \, x^{-4/3} + o(x^{-4/3}) .
\end{align}  
Moreover, in this case $\BB E[\mcl X(k) - \mcl X(k-1)] = 0$ for all $k\in\BB Z$.  
\end{enumerate} 
\end{thm}

The proofs of Theorems~\ref{thm-busemann-property} and~\ref{thm-busemann-tail} constitute the bulk of the work in this paper. See Section~\ref{sec-outline} for a detailed outline. 
To briefly summarize, the independent increments property~\ref{item-busemann-ind} of Theorem~\ref{thm-busemann-property} will eventually follow from the fact that an $\XDP$ geodesic cuts a finite bipolar-oriented triangulation into two pieces that are conditionally independent given its length (Lemma~\ref{lem-lr-ind}). The stationary increments property~\ref{item-busemann-stationary}  will ultimately be a consequence of the stationarity of $M_{-\infty,\infty}$ under the KMSW procedure (Section~\ref{sec-busemann-ind}). The bounds~\ref{item-busemann-pos} for $k\geq 1$ are immediate from the fact that the right boundary of $M_{0,\infty}$ is a directed path in $M_{0,\infty}$. The symmetry~\ref{item-busemann-sym} for $k\leq 0$ comes from a non-trivial ``local reflection symmetry'' at the edges along the left boundary of $M_{0,\infty}$ (Section~\ref{sec-busemann-sym}). Theorem~\ref{thm-busemann-tail} is obtained by analyzing a recursive relation for the laws of $\mcl X(1) - \mcl X(0)$ and $\mcl X(0) - \mcl X(-1)$ which comes from applying one step of the KMSW procedure (Section~\ref{sec-recursive}).
In principle, it may be possible to extract from this recursive relation a complete characterization of the laws of $\mcl X(1)-\mcl X(0)$ and $\mcl X(0) - \mcl X(-1)$, up to a single unknown constant. But we only need the leading-order tail asymptotics.

\subsection{Directed distances in finite bipolar-oriented triangulations} 
\label{sec-finite-intro}

From the existence and scaling limit of the Busemann function on the UIBOT, one can deduce up-to-constants estimates for directed distances in various types of finite random bipolar-oriented triangulations. 
We will state and prove results to this effect in two settings: Boltzmann bipolar-oriented triangulations with fixed right boundary length (Section~\ref{sec-boltzmann-intro}) and submaps of the UIBOT explored by $n$ steps of the KMSW procedure (Section~\ref{sec-cell-intro}). We decided to work with these models of finite bipolar orientations in order to illustrate the main ideas while keeping the exposition concise. 
We expect that one can also obtain analogous results for, e.g., uniform bipolar-oriented triangulations with specified left boundary length, right boundary length, and total number of edges. However, proving such results would require some technical estimates for conditioned random walks, so we do not carry it out here (the estimates we need do not seem to immediately follow from the results of~\cite{dw-cones}).

\subsubsection{Directed distances in Boltzmann bipolar-oriented triangulations} 
\label{sec-boltzmann-intro}

\begin{defn} \label{def-bipolar-bdy}
For $\el , r \in \BB N$, a bipolar-oriented planar map with \textbf{left boundary length} $\el$ and \textbf{right boundary length} $r$ is a finite bipolar-oriented planar map $\frk m$ with the following properties. The source and the sink both lie on the boundary of the external face, and the oriented path from the source to the sink on the boundary of the external face, with the external face immediately to its left (resp.\ right), has length $\el$ (resp.\ $r$).
\end{defn}

For example, the bipolar-oriented triangulation on the left-hand side of Figure~\ref{fig-bipolar} has left boundary length 5 and right boundary length 4. 

\begin{defn}\label{def-boltzmann-right} \leavevmode\par

\begin{itemize}
\item
The \textbf{Boltzmann distribution} on bipolar-oriented triangulations \textbf{with right boundary length} $r\in\BB N$ is the law of the random bipolar-oriented triangulations $M(r)$ such that for each bipolar-oriented triangulation $\frk m$ with $r$ right boundary edges and $n\in\BB N$ total edges, 
\eqb  \label{eqn-boltzmann-right}
\BB P\left[ M(r) = \frk m \right] = \frac{3^{-n}}{C_r},  
\eqe 
where $C_r > 0$ is a normalizing constant.
\item 
The \textbf{Boltzmann distribution} on bipolar-oriented triangulations \textbf{with right boundary length} $r$ \textbf{and a marked left boundary vertex} is the law of the random pair $(M(r) , z(r))$, where $M(r)$ is a bipolar-oriented triangulation with $r$ right boundary edges and $z(r)$ is a vertex on its left boundary, such that the following is true. For each bipolar-oriented triangulation $\frk m$ with $r$ right boundary edges and $n\in\BB N$ total edges, and each vertex $\frk z$ on its left boundary, 
\eqb  \label{eqn-boltzmann-right-marked}
\BB P\left[ (M(r) , z(r))  = (\frk m , \frk z) \right] = \frac{3^{-n}}{\wt C_r}, 
\eqe  
where $\wt C_r > 0$ is a normalizing constant. 
\end{itemize}
\end{defn}

If $(M(r) , z(r))$ is sampled from the Boltzmann distribution on bipolar-oriented triangulations with right boundary length $r$ and a marked left boundary vertex, then the marginal law of $M(r)$ is equal to the Boltzmann distribution~\eqref{eqn-boltzmann-right} weighted by a constant times the number of possibilities for $z(r)$, which is equal to the left boundary length of $M(r)$ plus one.

It follows from the KMSW bijection that the distributions in Definition~\ref{def-boltzmann-right} are well-defined, in the sense that the total mass of the measure on triangulations is finite; see Lemmas~\ref{lem-kmsw-boltzmann} and~\ref{lem-boltzmann-finite}.

The normalization factor $3^{-n}$ in~\eqref{eqn-boltzmann-right} and~\eqref{eqn-boltzmann-right-marked} comes directly from the KMSW encoding: a bipolar-oriented triangulation with $n$ edges is encoded by a length-$n$ walk with \emph{three} possible step types, so a uniformly sampled encoding has probability $(1/3)^n$. Thus a map with $n$ edges naturally receives weight $3^{-n}$.
It should be possible to also argue that this choice is \emph{critical}. Let $a_{n,r}$ be the number of bipolar-oriented triangulations with right boundary length $r$ and $n$ edges. Via the KMSW representation, $a_{n,r}$ is the number of length-$n$ walks constrained to stay in a cone, and results on random walks in cones (see~\cite[Theorem~1]{dw-cones}) should yield $a_{n,r}=c_r\,3^{n}\,n^{-\alpha}\,(1+o(1))$ as $n\to\infty$, for some $\alpha>1$ and constant $c_r>0$. Consequently, under the Boltzmann weight $3^{-n}$, we obtain $\mathbb{P}\left(\#\mcl E(M(r))=n\right)\ \propto\ a_{n,r}\,3^{-n}\ \asymp\ n^{-\alpha}$. So the size distribution has a power-law tail.

Our main theorem for directed distances in Boltzmann bipolar-oriented triangulations is as follows. 

\begin{thm} \label{thm-boltzmann-ldp}
For $r\in\BB N$, let $(M(r) ,z(r))$ be sampled from the Boltzmann distribution on bipolar-oriented triangulations with right boundary length $r$ and a marked left boundary vertex (Definition~\ref{def-boltzmann-right}). 
Let $\op{Source}(r)$ and $\op{Sink}(r)$ denote the source and sink vertices of $M(r)$. 
For each $\delta \in (0,1)$, there exists $A = A(\delta)  > 1$ such that for each $r\in \BB N$, it holds with probability at least $1-\delta$ that the following is true.  
\begin{itemize}
\item $A^{-1} r^{3/2} \leq  \op{LDP}_{M(r)} \left(  \op{Source}(r)  , \op{Sink}(r)  \right) \leq A r^{3/2}$. 
\item $ \op{SDP}_{M(r)} \left(  \op{Source}(r)  , \op{Sink}(r)  \right) \leq A r^{3/4}$.
\item There exists a vertex $y$ on the right boundary of $M(r)$ such that 
\eqbn
\op{SDP}_{M(r)} \left(  \op{Source}(r)  , y \right) \geq A^{-1} r^{3/4} .
\eqen
\end{itemize}
\end{thm}

Thanks to the result in~\eqref{eq:edges-in-Mr}~of~Lemma~\ref{lem-uiqbot-boltzmann}, by possibly increasing $A$, we can arrange that with probability at least $1-\delta$,
\begin{equation*}
    A^{-1}r^{2} \leq \#\mcl E( M(r)) \leq Ar^2 
\end{equation*}
where $\mcl E(M(r))$ is the set of edges of $M(r)$.
Therefore, Theorem~\ref{thm-boltzmann-ldp} implies that LDP distances in $M(r)$ (resp.\ SDP distances in $M(r)$) are typically comparable to $\#\mcl E( M(r))^{3/4}$ (resp.\ $\#\mcl E( M(r))^{3/8}$); see also Theorems~\ref{thm-cell-ldp}~and~\ref{thm-cell-sdp} below.

We expect that also $\op{SDP}_{M(r)} \left(  \op{Source}(r)  , \op{Sink}(r)  \right) \geq A^{-1} r^{3/4}$ with high probability when $A$ is large, but we do not prove this here. The reason we work with the Boltzmann distribution with a marked left boundary vertex in Theorem~\ref{thm-boltzmann-ldp} is for technical convenience: weighting the law of $M(r)$ by its left boundary length allows us to relate it more directly to the uniform infinite quarter-plane bipolar-oriented triangulation of Definition~\ref{def-uiqbot}. We expect that Theorem~\ref{thm-boltzmann-ldp}, in the case where we do not have a left marked boundary vertex, can be deduced from the current version of Theorem~\ref{thm-boltzmann-ldp} together with some estimates for the KMSW encoding walk. However, for the sake of brevity, we do not carry this out here.

\subsubsection{Directed distances in size-$n$ KMSW cells} 
\label{sec-cell-intro}
 
For $n\in\BB N$, let $M_{0,n}$ be the submap\footnote{To be precise, this is a map with some missing edges on the boundary following our Definition~\ref{defn:planar-map-with-missing-edges}. See the precise  definition of the KMSW procedure in Definition~\ref{def-kmsw} for more details.} of the UIBOT $\uibot$ (Section~\ref{sec-busemann-intro}) explored by the KMSW procedure between time 0 and time $n$. In other words, $M_{0,n}$ is the directed map obtained by applying the procedure of Definition~\ref{def-kmsw} to the walk $\mcl Z|_{[0,n]}$, where $\mcl Z$ is a walk on $\BB Z^2$ with i.i.d.\ increments sampled uniformly from $\{(1,-1) , (-1,0) ,(0,1)\}$. The map $M_{0,n}$ is not bipolar-oriented since, for a vertex on the boundary of $M_{0,n}$, it could be that all of the incoming (or outgoing) edges for this vertex are not contained in $M_{0,n}$; see, for instance, the map $\frk{m}_6$ in Figure~\ref{fig-kmsw}. However, $M_{0,n}$ is believed to exhibit the same behavior in the bulk as a uniform bipolar-oriented triangulation of size $n$. In terms of LQG, the map $M_{0,n}$ is the discrete analog of a unit-size mated-CRT map cell. 

Our results for $M_{0,n}$ are as follows. 

\begin{thm} \label{thm-cell-ldp}
For each $\delta \in (0,1)$, there exists $A = A(\delta) >1$ such that for each $n\in\BB N$, it holds with probability at least $1-\delta$ that 
\eqb  \label{eqn-cell-ldp}
 A^{-1} n^{3/4} \leq \sup\left\{ \op{LDP}_{M_{0,n}}(x,y)  : x,y\in\mcl V(M_{0,n}) \right\} \leq A n^{3/4} ,   
\eqe 
where here we recall that $\op{LDP}_{M_{0,n}}(x,y) = -\infty$ if there is no directed path in $M_{0,n}$ from $x$ to $y$. 
\end{thm}

The map $M_{0,n}$ has a \textbf{lower boundary}, defined as its intersection with the boundary of $M_{0,\infty}$, and an \textbf{upper boundary} defined as the intersection of its boundary with the submap of $\uibot$ explored by the KMSW procedure after time $n$. See Definition~\ref{defn:bound-edges} for an alternative definition. 
  
\begin{thm} \label{thm-cell-sdp}
For each $\delta \in (0,1)$, there exists $A = A(\delta) > 1$ 
such that for each $n\in\BB N$,
it holds with probability at least $1-\delta$ that the following is true. 
\begin{itemize}
\item For every vertex $x$ on the lower boundary of $M_{0,n}$, there exists a vertex $y$ on the upper boundary of $M_{0,n}$ such that $\op{SDP}_{M_{0,n}}(x , y) \leq A n^{3/8}$.
\item For every vertex $y$ on the upper boundary of $M_{0,n}$, there exists a vertex $x$ on the lower boundary of $M_{0,n}$ such that $\op{SDP}_{M_{0,n}}(x , y) \leq A n^{3/8}$. 
\item There exists a vertex $y$ on the lower boundary of $M_{0,n}$, lying to the right of the root edge $\lambda_0=(x_0,x_1)$, such that $\op{SDP}_{M_{0,n}}(x_0 , y) \geq A^{-1} n^{3/8} $. 
\end{itemize}
\end{thm}

\begin{remark} \label{remark-nontrivial}
Prior to this work, the only known non-trivial bounds for undirected graph distances in random planar maps in the $\gamma$-LQG universality class, $\gamma\not=\sqrt{8/3}$, were obtained using LQG methods (see~\cite[Section 5.2]{ghs-mating-survey} for an explanation of how this is done).  
By (crudely) upper-bounding undirected graph distance by directed graph distance, Theorem~\ref{thm-boltzmann-ldp} (in the SDP case) and Theorem~\ref{thm-cell-sdp} provide the first non-trivial bounds for undirected distances in such maps which can be proven using only discrete arguments. 
However, known upper bounds for the dimension of LQG~\cite[Corollary 2.5]{gp-lfpp-bounds} show that the dimension of $\sqrt{4/3}$-LQG satisfies $0.304 \leq 1/d_{\sqrt{4/3}} \leq 0.326$. Hence the directed graph distance exponent $3/8$ provides a worse bound for the undirected graph distance exponent than what one can get from LQG. 
\end{remark}

\begin{comment}

dBound[gamma_] := (
 12 - Sqrt[6] gamma + 3 Sqrt[10] gamma + 3 gamma^2)/(4 + Sqrt[15])

dUpper[gamma_] :=
 If[gamma <= Sqrt[8/3],
  Min[1/3 (4 + gamma^2 + Sqrt[16 + 2 gamma^2 + gamma^4]) ,(* 
   Monotonicity *)
   2 + gamma^2/2 + Sqrt[2] gamma  (* Ball counting *) ],  
  dBound[gamma]  (* monotonicity *)
  
  ]

dLower[gamma_] :=
 If[gamma >= Sqrt[8/3],
   1/3 (4 + gamma^2 + Sqrt[16 + 2 gamma^2 + gamma^4]) ,  
  Max[2 gamma^2/(4 + gamma^2 - Sqrt[16 + gamma^4]), (* KPZ *) 
   dBound[gamma]   ] 
  
  ]

FullSimplify[1/dUpper[Sqrt[4/3]]]
N[%]

FullSimplify[1/dLower[Sqrt[4/3]]]
N[%]

\end{comment}

\subsection{Outline}
\label{sec-outline}

In Section~\ref{sec-related}, we discuss possible extensions to other models, relationships to directed versions of Liouville quantum gravity, and analogies to first- and last-passage percolation. This section exists only to provide additional context and is not needed for our proofs.

In Section~\ref{sec-prelim}, we review the KMSW bijection~\cite{kmsw-bipolar}, which will be the main tool in our proofs.
In particular, we explain that the UIBOT can be encoded via the KMSW bijection by a bi-infinite random walk $\mcl Z : \BB Z\to\BB Z^2$ whose increments are i.i.d.\ uniform samples from $\{(1,-1), (-1,0) , (0,1)\}$ (Proposition~\ref{prop-kmsw-uibot}). 
We also make the following observation (Lemma~\ref{lem-lr-ind}). Suppose that $M$ is a Boltzmann bipolar-oriented triangulation with specified left and right boundary lengths, and let $P$ be the leftmost $\XDP$ geodesic in $M$ from its source vertex to its sink vertex (Definition~\ref{def-geodesic}). Then the submaps of $M$ lying to the left and right of $P$ are conditionally independent given the length of $P$. This observation will ultimately be the source of the independent increments property of the Busemann function (Theorem~\ref{thm-busemann-property}). We also note that Lemma~\ref{lem-lr-ind} is not true for undirected graph distance geodesics, which is one reason why our proofs do not work for undirected distances.

In Section~\ref{sec-uiqbot}, we study the \textbf{uniform infinite quarter-plane bipolar-oriented triangulation (UIQBOT)} (Definition~\ref{def-uiqbot}), which is an infinite bipolar-oriented triangulation with infinite left and right boundary lengths, with a finite source and a sink at $\infty$. We show in Proposition~\ref{prop-cut-vertex} that a.s.\ there are infinitely many \textbf{cut vertices} $v$ of the UIQBOT with the property that removing $v$ disconnects the UIQBOT. The existence of such vertices will be very useful for the proof of Theorem~\ref{thm-boltzmann-ldp}, since long directed paths are forced to pass through them. The proof of Proposition~\ref{prop-cut-vertex} is based on estimates for the KMSW encoding walk. We also show in Proposition~\ref{prop-thin-ind} that the submaps of the UIQBOT lying to the left and right of its leftmost $\XDP$ geodesic from the source to $\infty$ are independent. This is done by applying the aforementioned conditional independence statement for finite bipolar-oriented triangulations (Lemma~\ref{lem-lr-ind}) and comparing such finite bipolar-oriented triangulations to the UIQBOT. However, some non-trivial work is involved to get rid in the limit of the conditioning on the length of the $\XDP$ geodesic (we use an argument based on the Hewitt-Savage zero-one law).

In Section~\ref{sec-busemann}, we prove Theorems~\ref{thm-busemann} and~\ref{thm-busemann-property} on the existence and basic properties of the Busemann function. To do this, we show that if $(\uibot,\lambda_0)$ is the UIBOT, then the submap $\wh M_{0,\infty}$ of $M_{0,\infty}$ induced by the vertices of $M_{0,\infty}$ which are reachable by a directed path started from the initial vertex of $\lambda_0$ has the law of the UIQBOT (Proposition~\ref{prop-future-map-law}). This allows us to apply the result on the existence of cut vertices in the UIQBOT (Proposition~\ref{prop-cut-vertex}) to show that infinite $\XDP$ geodesics in $M_{0,\infty}$ exist, and any two such geodesics have to have infinitely many points in common. This leads easily to the proof of Theorem~\ref{thm-busemann}. 

We deduce the stationarity and independent increments properties of Theorem~\ref{thm-busemann-property} using the independence property from Proposition~\ref{prop-thin-ind} and the stationarity of the law of the KMSW encoding walk $\mcl Z$. In particular, we show (Lemma~\ref{lem-geo-slice}) that $M_{0,\infty}$ can be decomposed into countably many independent ``slices'' bounded by leftmost $\XDP$ geodesics started from the boundary vertices $\{x_k\}_{k\in\BB Z}$. We remark that decompositions of planar maps by geodesic slices have been used extensively in probability and combinatorics, see, e.g.,~\cite{miermont-compact-survey,bgm-tight-map,ab-hypermap-slice} and the references therein. To our knowledge this is the first use of such decompositions in the setting of bipolar-oriented maps. 

To prove the symmetry property $\mcl X(0)-\mcl X(-1) \eqD \mcl X(-1)-\mcl X(0)$, we introduce (Definition~\ref{def-uihbot}) an infinite directed planar map with boundary, which we call the uniform infinite boundary-channeled half-plane bipolar-oriented triangulation (UIBHBOT). The UIBHBOT describes the local behavior of $M_{0,\infty}$ near a typical edge on its left boundary. We deduce the symmetry of $\mcl X$ from a non-obvious reflection symmetry of the UIBHBOT (Proposition~\ref{prop-uihbot-sym}). This reflection symmetry, in turn, follows from a symmetry of finite boundary-channeled bipolar-oriented triangulations (Lemma~\ref{lem-bdy-reverse}) and a local limit argument (Proposition~\ref{prop-disk-bs-conv}).

In Section~\ref{sec-recursive}, we deduce Theorem~\ref{thm-busemann-tail} from Theorem~\ref{thm-busemann-property}. This is the part of the argument that leads to the exact exponents in our theorem statements. If the reader takes Theorems~\ref{thm-busemann} and~\ref{thm-busemann-property} as a black box, Section~\ref{sec-recursive} can be read independently from the rest of the paper. The idea of the proof is as follows. Let $M_{1,\infty}$ be the map obtained from $M_{0,\infty}$ by applying one step of the KMSW procedure. We know that $M_{1,\infty} \eqD M_{0,\infty}$, so the Busemann function $\mcl X^1$ associated with $M_{1,\infty}$ has the same law as $\mcl X$. Furthermore, we can explicitly express $\mcl X$ in terms of $\mcl X^1$ (Lemmas~\ref{lem:LDP-relating-X1-X}~and~\ref{lem:SDP-relating-X1-X}). Using this and the properties in Theorem~\ref{thm-busemann-property}, we obtain in Propositions~\ref{prop:LDP-recursive-equation}~and~\ref{prop:SDP-recursive-equation} a recursive equation satisfied by the probabilities
\eqbn
f(x) := \BB P\left[ \mcl X(0) - \mcl X(-1) = x \right] \quad \text{and} \quad g(y):= \BB P\left[ \mcl X(1) - \mcl X(0) = y \right] .
\eqen
By multiplying by $e^{i t x+i s y}$ and summing over $x$ and $y$, we convert this equation into a functional equation for the characteristic functions $F$ and $G$ of  $\mcl X(0) -\mcl X(-1)$ and $\mcl X(1)-\mcl X(0)$ (see~\eqref{eq:big-rel} and \eqref{eq:big-rel2}). By analyzing the leading-order behavior of the terms in this equation, we can determine the asymptotic behavior of these characteristic functions as $t,s\to 0$.\footnote{
In the LDP case, the equation uniquely determines the leading order asymptotics of the characteristic functions as $t\to 0$. In the SDP case, matters are slightly more complicated. There are two possible leading order asymptotics, depending on the value of an a priori unknown parameter which we call $\kappa$. The $\kappa=0$ case leads to the correct $x^{-4/3}$ tail in Theorem~\ref{thm-busemann-tail}; the $\kappa > 0$ case instead leads to an $x^{-2/3}$ tail. We rule out the $\kappa > 0$ case by showing that it violates the trivial $n^{1/2}$ upper bound for SDPs in a size-$n$ KMSW cell.} 
By a Tauberian theorem (Proposition~\ref{prop:Tauberian}), this leads to Theorem~\ref{thm-busemann-tail}. Theorems~\ref{thm-busemann-property} and~\ref{thm-busemann-tail} immediately imply Theorem~\ref{thm-busemann-conv} by a standard heavy-tailed functional central limit theorem (see, e.g.,~\cite{js-limit-thm}).  

In Section~\ref{sec-finite}, we deduce our results for finite-volume planar maps from Theorems~\ref{thm-busemann-conv} and~\ref{thm-busemann-tail}. The results for the KMSW cell $M_{0,n}$ (Theorems~\ref{thm-cell-ldp} and~\ref{thm-cell-sdp}) follow from the results for the Busemann function and relatively straightforward geometric arguments. To prove Theorem~\ref{thm-boltzmann-ldp}, we find a submap $\wh M(r)$ of $M_{0,\infty}$ that has the same law as the map $M(r)$ of Theorem~\ref{thm-boltzmann-ldp} (see Lemma~\ref{lem-uiqbot-boltzmann}). We then compare $\wh M(r)$ to the size-$n$ cell $M_{0,n}$, with $n$ chosen to be comparable to $r^2$. 

In Section~\ref{sec-open-problems}, we discuss some open problems. In Appendix~\ref{sect:proofTau}, we prove the Tauberian-type result for characteristic functions (Proposition~\ref{prop:Tauberian}) needed in Section~\ref{sec-recursive}.

\begin{remark} \label{remark-separable}
The recent paper~\cite{abbds-separable-lis} obtains the exponent for longest increasing subsequences in Brownian separable permutons using a recursive equation. However, the proof strategy in~\cite{abbds-separable-lis} has essentially nothing in common with the proofs in this paper. In the setting of~\cite{abbds-separable-lis}, the recursive equation arises from an encoding of the permuton in terms of a labeled tree, rather than a stationarity property as in this paper. Moreover, in~\cite{abbds-separable-lis}, to extract the exponent from the recursive equation one has to show a priori that the distribution of the longest increasing subsequence has a regularly varying tail. In contrast, in this paper, we derive the exact tail behavior (Theorem~\ref{thm-busemann-tail}) directly from the recursive equation, with no a priori assumptions. 
See Remark~\ref{remark-critical-metric} for some discussion about the results of~\cite{abbds-separable-lis} in the context of directed LQG metrics.
\end{remark}

%% file: tex/related.tex
In Sections~\ref{sec-other} and~\ref{sec-special-bipolar}, we discuss other models to which the techniques of this paper might be applicable. In Section~\ref{sec-lqg}, we discuss conjectures related to directed versions of the LQG metric, which, in the special case when $\gamma=\sqrt{4/3}$ and $\wt\theta=\pi/2$, should describe the scaling limit of the directed distances considered in this paper. In Section~\ref{sect:fpp-lpp}, we discuss some parallels between the results of this paper and the study of directed first- and last-passage percolation and the directed landscape. Nothing in this section is needed to understand the proofs of our main results. 

\subsection{Other discrete models}
\label{sec-other}

It is likely possible to extend our techniques to obtain similar results for a variety of other discrete models. One natural class of models to consider is \textbf{bipolar-oriented planar maps with more general face degree distributions} (not just triangulations), e.g., bipolar-oriented quadrangulations and bipolar-oriented maps with unconstrained face degrees. Such models can also be encoded via the KMSW bijection, and we expect that results similar to those in this paper (with the same exponents) hold in this more general setting. 

A key technical difficulty in extending our results to this more general setting is as follows. If $M$ is not a triangulation, the map $M_{0,\infty}$ considered in Section~\ref{sec-busemann-intro} can have vertices on its left boundary that are not incident to any non-missing edge of $M_{0,\infty}$. This makes it unclear how to properly define the Busemann function. 
One possible way around this is to modify the KMSW bijection and the definition of $M_{0,\infty}$. 
Alternatively, one could try to compare directed distances in general bipolar-oriented maps to directed distances in bipolar-oriented triangulations. A possible strategy is to use a strong coupling of the KMSW encoding walks, as was used for undirected graph distances in~\cite{ghs-map-dist} (see also Problem~\ref{prob-universal}). 

A \textbf{Baxter permutation} of size $n\in\BB N$ is a permutation $\sigma \in \mcl S_n$ for which there do not exist indices $1 \leq i < j < k \leq n$ such that $\sigma(j+1) < \sigma(i) < \sigma(k) < \sigma(j)$ or $\sigma(j ) < \sigma(k) < \sigma(i) < \sigma(j+1)$. Bipolar-oriented planar maps with general face degree distribution, with $n$ total edges, are in bijection with Baxter permutations of size $n$~\cite{bbf-baxter}. Under this bijection, the longest increasing subsequence of the permutation corresponds to the longest directed path in the bipolar-oriented map. This leads to the following conjecture, which seems to be within reach given the results of this paper.

\begin{conj} \label{conj-baxter}
Let $\sigma_n$ be a uniform Baxter permutation of size $n$. Then the longest increasing subsequence in $\sigma_n$ grows like $n^{3/4}$ as $n\to\infty$.
\end{conj}
 
It may also be possible to extend our results to \textbf{biased bipolar-oriented planar maps} whose face degree distributions are biased based on the number of edges on the left and right boundaries of the faces. As explained in~\cite[Remark 5]{kmsw-bipolar}, such maps can be in the universality class of $\gamma$-LQG for any $\gamma \in (0,\sqrt 2)$. In particular, we expect that such maps have different exponents for directed distances compared to uniform bipolar-oriented triangulations. See Section~\ref{sec-special-bipolar} just below for a special class of biased bipolar-oriented planar maps that seem especially amenable to our techniques. 

More generally, our techniques can also potentially be applied to directed distances in other random planar map models. 
Indeed, the independence property for the maps to the left and right of an $\XDP$ geodesic seems to require that the orientation on our planar map is, in some sense, compatible with how the map is sampled (see, e.g., Lemma~\ref{lem-lr-ind}). Once we have this, the only other crucial tool in our proofs is the mating-of-trees bijection. 
For example, our techniques can likely be generalized to random planar maps decorated by \textbf{Schnyder woods}~\cite{lsw-schnyder-wood}, or more generally, grand-Schnyder woods~\cite{bernardi2025grand}. 
Another class of problems to which our techniques might apply is determining the minimal length of a path hit in order by some natural path on a random planar map. Some interesting examples include the contour exploration of a \textbf{uniform spanning tree}~\cite{mullin-maps,bernardi-maps,shef-burger,biane-mot} or a \textbf{percolation interface} (here, the peeling process would take the place of the mating-of-trees bijection~\cite{angel-peeling}). One could also attempt to extend our methods to determine the longest increasing subsequence exponent for other models of pattern-avoiding permutations which are related to random planar maps and LQG, such as \textbf{semi-Baxter and strong-Baxter permutations}~\cite{borga-strong-baxter}. 

The authors are currently working on two papers in the above direction. One such paper (which is also joint with Yuanzheng Wang) studies shortest directed paths in spanning-tree decorated maps. Another paper (which is also joint with Leonardo Bonanno) studies longest and shortest directed paths in triangulations decorated by Schnyder woods.

\subsection{Bipolar-oriented maps where faces have at most two right boundary edges}
\label{sec-special-bipolar}

Many of our proofs work for a more general class of models of bipolar-oriented maps than just triangulations. Indeed, consider a model of bipolar-oriented maps with a face degree distribution such that each of the faces is allowed to have at most two right boundary edges, and the number of left boundary edges has finite moments of all positive orders. 
The infinite-volume version of such a map (analogous to the UIBOT) is the rooted map $(\uibot,\lambda_0)$ which can be constructed by applying the KMSW procedure (Definition~\ref{def-kmsw}) to a bi-infinite walk $\mcl Z = (\mcl L , \mcl R) : \BB Z\to\BB Z^2$ with i.i.d.\ increments, whose increment distribution $\nu(x,y) = \BB P[\mcl Z(j) - \mcl Z(j-1)=(x,y)]$ satisfies the following conditions.
\begin{itemize}
\item The distribution $\nu$ is supported on $\{(1,-1)\} \cup \{(0, j)\}_{j\geq 1} \cup \{(-1, j)\}_{j\geq 0}$. 
\item The expectation of a sample from $\nu$ is $(0,0)$.
\item The second coordinate of a sample from $\nu$ has finite moments of all positive orders. 
\end{itemize}
Equivalently, $(\uibot,\lambda_0)$ is the local limit (around a uniform edge) of finite bipolar-oriented maps sampled from a probability measure where each face with $i\in \{1,2\}$ right boundary edges and $j \geq 1$ left boundary edges is assigned weight $\nu(1-i,j-1)$. 

If we define $\uibot$ as above and $M_{0,\infty}$ analogously to Definition~\ref{defn-future-map}, then every face of $M_{0,\infty}$ has at most one boundary edge which is not in $M_{0,\infty}$. So, every vertex on the boundary of $M_{0,\infty}$ is the initial vertex of an infinite directed path in $M_{0,\infty}$. Moreover, $\mcl L$ has the law of a lazy random walk. Of particular interest is the case where 
\eqb \label{eqn-schnyder}
\nu(1,-1) = \frac12 ,\quad \nu(-1,j) = 2^{-j-2} , \quad \forall j \geq 0, 
\eqe
which is closely related to uniformly sampled triangulations decorated by a Schnyder wood (see, for instance, \cite[Proposition 7.1]{ffno-baxter-family} or~\cite[Section 3.3.4]{ghs-map-dist}). Such triangulations are in the universality class of $\gamma$-LQG for $\gamma=1$~\cite{lsw-schnyder-wood}. In general, by varying $\nu$, one can get models of bipolar-oriented planar maps of the above type which belong to the $\gamma$-LQG universality class for any $\gamma \in (0,\sqrt{2})$. The value of $\gamma$ is determined by the correlation $\rho$ of the two coordinates of a sample from $\nu$ via the formula $\rho = -\cos(\pi\gamma^2/4)$, see~\cite[Theorem 1.9]{wedges}. 

Our proofs of Theorem~\ref{thm-busemann} and Properties~\ref{item-busemann-ind} through~\ref{item-busemann-pos} of Theorem~\ref{thm-busemann-property} carry over essentially\footnote{The description of the KMSW encoding walk $\wh{\mcl Z}'$ for $\wh M_{0,\infty}'$ in Lemma~\ref{lem-quadrant-walk} is not as nice for a general choice of $\nu$, but we still get the same description for the law of $\wh{\mcl Z}$ and the fact that $\wh{\mcl Z}$ and $\wh{\mcl Z}'$ are independent, which is all that is needed for our proofs.} verbatim to the above setting. The moment hypothesis is needed to apply the results of~\cite{dw-limit} (see Lemma~\ref{lem-cond-walk-bm}). Moreover, if one assumes appropriate analogs of Theorems~\ref{thm-busemann-property} and~\ref{thm-busemann-tail}, the proofs of our results in the finite-volume setting (Theorems~\ref{thm-boltzmann-ldp}--\ref{thm-cell-sdp}) also carry over without change. 

However, Property~\ref{item-busemann-sym} of Theorem~\ref{thm-busemann-property} is not true in the above more general setting. This makes it more difficult to derive tail asymptotics as in Theorem~\ref{thm-busemann-tail} from the recursive equation which comes from one step of the KMSW procedure. 
In general, we expect that the scaling limit of the Busemann function $\mcl X|_{(-\infty,0]}$ could be an asymmetric stable Lévy process. 

We plan to investigate bipolar-oriented planar maps of the above type further in future work.

\subsection{Directed LQG metrics}
\label{sec-lqg}

Directed distances in bipolar-oriented random planar maps are believed to be discrete analogs of directed versions of the LQG metric, in the same way that undirected distances in random planar maps are discrete analogs of the undirected LQG metric (as constructed in~\cite{dddf-lfpp,gm-uniqueness}). In this subsection, we explain what we mean by this, assuming that the reader has some basic familiarity with the theory of LQG and SLE. A reader who is not familiar with these topics can skip this subsection.
See Figure~\ref{fig-directed-lqg-graph} for a visual summary of this section.

\begin{figure}[ht!]
\begin{center}
\includegraphics[width=0.49\textwidth]{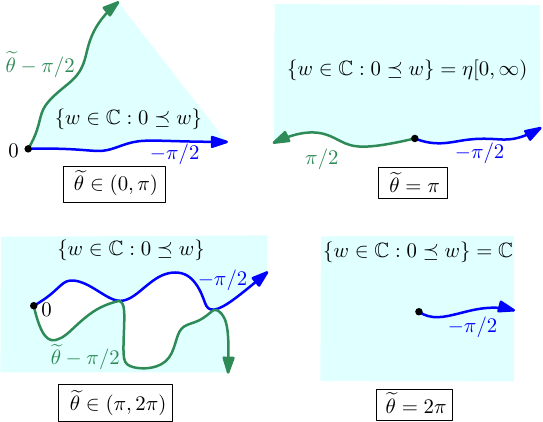}  
\includegraphics[width=0.49\textwidth]{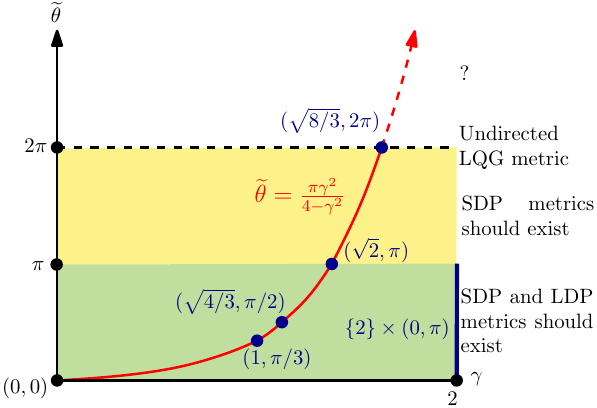}  
\caption{\label{fig-directed-lqg-graph} 
\textbf{Left:} The blue regions are the sets of points reachable by directed paths started from the origin for the hypothetical directed LQG metric in the cases when $\wt\theta \in (0,\pi)$, $\wt\theta =\pi$, $\wt\theta \in (\pi,2\pi)$, and $\wt\theta =2\pi$. The colored curves are flow lines (in the sense of Imaginary Geometry~\cite{ig4}) of a whole plane-GFF, with angles $-\pi/2$ and $\wt\theta-\pi/2$. In particular, they are each SLE$_{\gamma^2}(\gamma^2-2)$ curves.
\textbf{Right:} Graph of the possible parameter values for the directed LQG metrics. The parameter value $\wt\theta = 2\pi$ (dashed horizontal line) corresponds to the undirected LQG metric. When $\wt\theta \in (0,\pi)$ (green region), both shortest-path and longest-path directed LQG metrics should exist. When $\wt\theta \in [\pi,2\pi]$ (yellow region), only the shortest-path directed LQG metric should exist. We do not know whether there is any notion of directed LQG metrics when $\wt\theta > 2\pi$.
The red curve corresponds to parameter values where we expect additional solvability. The point $(\sqrt{4/3},\pi/2)$ corresponds to the scaling limit of directed distances in bipolar-oriented planar maps (i.e., the setting considered in this paper). The point $(\sqrt{8/3},2\pi)$ corresponds to the scaling limit of undirected distances on uniform planar maps (i.e., Brownian surfaces). The points $(\sqrt 2 , \pi)$ and $(1,\pi/3)$ correspond to other models which we expect can be analyzed using the techniques of this paper (spanning tree decorated maps and Schnyder-wood decorated maps, respectively); see Section~\ref{sec-other}. The segment $\{2\}\times (0,\pi)$ corresponds to the $\gamma\to 2$ case, which is closely related to Brownian separable permutons as studied in~\cite{abbds-separable-lis} (see Remark~\ref{remark-critical-metric}).}  
\end{center}
\end{figure}

Let $\gamma \in (0,2)$, $\kappa = 16/\gamma^2 > 4$, and $\wt\theta \in [0,2\pi]$. Our parameter $\wt\theta$ corresponds to $\theta+\pi/2$ in the setting\footnote{In \cite{borga-skew-permuton}, angles are measured in clockwise order, where the zero angle corresponds to the west direction. However, in the LQG/SLE literature (and also in this paper), angles are measured in counter-clockwise order, where the zero angle corresponds to the north direction. Hence, $\wt\theta$ here corresponds to $\pi-\theta$ in the setting of \cite{borga-skew-permuton}.} of~\cite{bhsy-baxter-permuton,bg-permuton-msrble}.  

Let $\Phi$ be the random generalized function on $\BB C$ corresponding to the $\gamma$-quantum cone~\cite[Definition 4.10]{wedges}. Then $\Phi$ is a minor modification of the whole-plane Gaussian free field (GFF). When $\gamma = \sqrt{4/3}$, the LQG surface described (heuristically) by the Riemannian metric tensor $e^{\gamma \Phi} \, (dx^2 + dy^2)$ on $\BB C$ is the continuum analog of the UIBOT (see, e.g.,~\cite[Section 4]{kmsw-bipolar}). 

Independently of $\Phi$, let $\eta$ and $\wt\eta$ be a pair of whole-plane space-filling SLE$_\kappa$ curves from $\infty$ to $\infty$ in $\BB C$. We assume that $\eta$ and $\wt\eta$ are coupled together so that they are space-filling counterflow lines of a common whole-plane GFF $\Psi$ in the sense of Imaginary Geometry~\cite{ig4}, with angles $0$ and $\wt\theta - \pi$, respectively. In the case when $\gamma = \sqrt{4/3}$ ($\kappa = 12$), the curve $\eta$ is the continuum analog of the gold Peano path on the right-hand side of Figure~\ref{fig-bipolar-boundaries} (see the discussion just above Definition~\ref{defn:KMSW-inv} for a precise definition). 
In the case when $\gamma=\sqrt{4/3}$ and $\wt\theta=\pi/2$, the curve $\wt\eta$ is the continuum analog of the ``dual'' Peano path of the same UIBOT, defined by considering the upper-right and lower-left trees instead of the upper-left and lower-right trees~\cite{ghs-bipolar}. Directed paths in the UIBOT are visited in order by both the primal and dual Peano paths. 

When $\wt\theta \in (0,\pi]$, we say that points $z,w \in \BB C$ are \textbf{ordered}, denoted $z\preceq w$, if they are hit in order by both $\eta$ and $\wt \eta$, i.e., there exists times $t_1 < t_2$ and $\wt t_1 < \wt t_2$ such that $\eta(t_1) =  \wt\eta(\wt t_1)  = z$ and $\eta(t_2) = \wt\eta(\wt t_2) = w$. When $\wt\theta \in (\pi,2\pi]$, we instead say that $z\preceq w$ if $z$ and $w$ are hit in order by either $\eta_1$ or $\eta_2$, i.e., either there exists times $t_1 < t_2$ such that $\eta(t_1) =  z$ and $\eta(t_2)   = w$ or the same is true with $\wt\eta$ in place of $\eta$. This makes it so that the set of points $w\in \BB C$ with $0\preceq w$ is bounded by two flow lines of $\Psi$ of angles $-\pi/2$ and $\wt\theta - \pi/2$, respectively. 
We note that $\preceq$ is not a partial order on $\BB C$ since there can exist pairs of distinct points $z,w\in \BB C$ such that $z\preceq w$ and $w\preceq z$.\footnote{When $\wt\theta\in (0,\pi]$, such pairs of points have zero Lebesgue measure since Lebesgue-a.e.\ point of $\BB C$ is hit only once by each of $\eta$ and $\wt\eta$. But, the set of such pairs does not have zero Lebesgue measure when $\theta \in (\pi,2\pi]$. } 
See the left-hand side of Figure~\ref{fig-directed-lqg-graph} for an illustration. 

We say that a path $P : [a,b] \to\BB C$ is \textbf{directed} if $P(s) \preceq P(t)$ for each $s,t\in [a,b]$ with $s < t$.

In the case when $\gamma=\sqrt{4/3}$  and $\wt\theta = \pi/2$, the relation $\preceq$ is the continuum analog of the partial ordering of the vertices of a bipolar-oriented planar map induced by the edge orientations. For general $\gamma \in (0,2)$ and $\wt\theta \in (0,\pi)$,  the relation $\preceq$ is closely related to the so-called \textbf{skew Brownian permutons}~\cite{borga-skew-permuton}. %(see ~\cite{bdg-skew-lis} for further explanation). 
When $\wt \theta=\pi$, it follows from the definition of space-filling SLE$_\kappa$~\cite[Theorem 1.16]{ig4} that $\eta=\wt\eta$, so a path is directed if and only if it is hit in order by $\eta$. 
When $\wt \theta = 2\pi$, the curves $\eta$ and $\wt\eta$ are time reversals of each other, so we get the trivial ordering where $z\preceq w$ for every $z,w\in\BB C$. That is, the case when $\wt\theta=2\pi$ corresponds to undirected distances.

%We say that a path $P : [0,1] \to\BB C$ is \textbf{directed} if it is hit in order by both $\eta$ and $\wt\eta$, in the sense that for each $s_1 , s_2 \in [0,1]$ with $s_1 < s_2$, there exist times $t_1 < t_2$ and $\wt t_1 < \wt t_2$ such that $\eta(t_1) = \wt\eta(\wt t_1) = P(s_1)$ and $\eta(t_2) = \wt\eta(\wt t_2) = P(s_2)$. 

\begin{conj}\label{conj-directed}
For each $\gamma \in (0,2)$ and $\wt\theta \in (0,\pi)$, there exist two distinct random functions $\frk L_\Phi^{\op{LDP}}$ and $\frk L_\Phi^{\op{SDP}}$, called the \textbf{longest-path and shortest-path directed LQG metrics}, where for each $z,w\in \BB C$, the distance $\frk L_\Phi^{\op{LDP}}(z,w)$ (resp.\ $\frk L_\Phi^{\op{SDP}}(z,w)$) is a limit of regularized versions of the longest (resp.\ shortest) LQG length of a directed path from $z$ to $w$. 
In the case when $\gamma=\sqrt{4/3}$ and $\wt \theta=\pi/2$, these directed LQG metrics describe the scaling limits of longest and shortest directed paths, respectively, in uniform bipolar-oriented triangulations.
The shortest-path directed LQG metric $\frk L_\Phi^{\op{SDP}}$ also exists for $\gamma \in (0,2)$ and $\wt\theta \in [\pi , 2\pi]$ (and it coincides with the undirected $\gamma$-LQG metric when $\wt\theta=2\pi$).
\end{conj}

The reason why there is not an interesting longest-path directed LQG metric for $\wt\theta\in [\pi,2\pi]$ is that in this case, the space-filling SLE$_\kappa$ curves $\eta$ and $\wt\eta$ are themselves directed paths.
%In Conjecture~\ref{conj-directed}, we are deliberately vague about exactly what sort of object the directed LQG metrics are. An initial guess is that they should be similar types of objects as the directed landscape from~\cite{dov-dl}. 
 
Although the directed LQG metrics $\frk L_\Phi^{\op{LDP}}$ and $\frk L_\Phi^{\op{SDP}}$ have not been constructed, our results give information about how these directed metrics should behave when $\gamma = \sqrt{4/3}$ and $\wt\theta = \pi/2$. 
By Theorems~\ref{thm-cell-ldp} and~\ref{thm-cell-sdp}, for $\gamma=\sqrt{4/3}$ and $\wt\theta = \pi/2$, the scaling exponents for $\frk L_\Phi^{\op{LDP}}$ and $\frk L_\Phi^{\op{SDP}}$ should be $3/4$ and $3/8$, respectively, in the sense that scaling LQG areas by $C$ corresponds to scaling $\frk L_\Phi^{\op{LDP}}$-distances by $C^{3/4}$ and $\frk L_\Phi^{\op{SDP}}$-distances by $C^{3/8}$. 
These scaling exponents are the directed analogs of $1/d_{\sqrt{4/3}}$, where $d_{\sqrt{4/3}}$ is the Hausdorff dimension of the undirected $\sqrt{4/3}$-LQG metric space. 
We do not have predictions for the scaling exponents for directed LQG distances for other values of $\gamma$ and $\theta$, but we obtain some bounds in the forthcoming work~\cite{bdg-skew-lis} (joint with Sayan Das). 
 
Assume that the whole-plane space-filling SLE$_\kappa$ curve $\eta$ is parametrized so that the $\gamma$-LQG area of $\eta[a,b]$ is $b-a$ for each $-\infty  <a < b < \infty$. 
Theorem~\ref{thm-busemann-conv} suggests that for $\gamma=\sqrt{4/3}$ and $\wt \theta=\pi/2$, for each $t \in \BB R$, the Busemann function for $\frk L_\Phi^{\op{LDP}}$ (resp.\ $\frk L_\Phi^{\op{SDP}}$) along the boundary of $\eta(-\infty,t]$ should be a $2/3$- (resp.\ $4/3$-) stable Lévy process. It is an interesting open question to determine the joint law of such processes for different values of $t$; see Problem~\ref{prob-busemann-joint}. 

Directed LQG metric should have additional ``solvability'' properties in the special case when
\eqb  \label{eqn-special-theta}
\wt\theta = \wt\theta_*(\gamma) := \frac{\pi \gamma^2}{4-\gamma^2} .
\eqe  
Note that by~\cite[Theorems 1.2 and 1.5]{wedges}, \eqref{eqn-special-theta} is equivalent to the condition that the LQG surface obtained by restricting $\Phi$ to $\{w\in\BB C \,:\, 0\preceq w\}$ is an LQG wedge of weight $\gamma^2/2$ (log singularity $2/\gamma+\gamma/2$). In the forthcoming work~\cite{bdg-skew-lis}, we will prove several special symmetry and solvability properties for the ordering $\preceq$ when $\wt\theta=\wt\theta_*(\gamma)$, working exclusively in the continuum. 
 
At least in the case when $\gamma \in (0,\sqrt 2]$, i.e.\ $\wt\theta_*(\gamma) \in (0,\pi]$, we expect that when $\wt\theta = \wt\theta_*(\gamma)$, the Busemann functions for $\frk L_\Phi^{\op{LDP}}$ and $\frk L_\Phi^{\op{SDP}}$ along the interface $\bdy \eta(-\infty,t]$ should be a stable L\'evy process (with index and skewness parameters depending on $\gamma$ and the choice of LDP or SDP).  See Remark~\ref{remark-special-theta} just below for a heuristic justification of this. We do not have any reason to believe that the laws of the Busemann functions for the directed LQG metrics have a nice description when $\wt\theta\not=\wt\theta_*(\gamma)$.
 
%The reason is that, as we will show in the forthcoming work~\cite{bdg-skew-lis}, the ordering $\preceq$ has several special symmetries and solvability properties for this value of $\wt\theta$ (which are useful for bounding exponents related to directed distances).  One example is as follows. When $\wt\theta= \wt\theta_*(\gamma)$, the LQG surface obtained by restricting $\Phi$ to $\{w\in\BB C : 0 \preceq w\}$ has the law of an LQG wedge of weight $\gamma^2/2$ (this is the critical weight for an LQG wedge to have cut points~\cite[Section 4.4]{wedges}). The left/right boundary length process for the curve obtained by concatenating the excursions made by $\eta$ into this region has a much simpler description when $\wt\theta = \wt\theta_*(\gamma)$ than when $\wt\theta\not=\wt\theta_*(\gamma)$. %~\cite[Proposition 3.3]{bdg-skew-lis}. 

We have $\wt\theta_*(\sqrt{4/3}) = \pi/2$, which corresponds to the conjectural scaling limit of directed distances in bipolar-oriented triangulations. Moreover, $\wt\theta_*(\sqrt{8/3}) = 2\pi$, which corresponds to the fact that undirected distances in uniform planar maps are solvable. Additional special cases include $\wt\theta_*(\sqrt 2) = \pi$ (corresponding to paths hit in order by the contour exploration on a spanning-tree decorated map) and $\wt\theta_*(1) = \pi/3$ (corresponding to directed distances in planar maps decorated by Schnyder woods). We plan to investigate these latter two cases in future work, see Section~\ref{sec-other}. When $\gamma > \sqrt{8/3}$, we have $\wt\theta_*(\gamma)  > 2\pi$. We do not know whether it is possible to make sense of a directed variant of the LQG metric with $\wt\theta > 2\pi$. 
%FullSimplify[(3 gamma^2 - 4) Pi / (2 (4 - gamma^2)) + Pi/2, Assumptions -> {0 < gamma < 2}]; Plot[(gamma^2  Pi )/(4 - gamma^2), {gamma, 0, 2}]

\begin{remark} \label{remark-special-theta}
Here is a heuristic justification of why $\wt\theta=\wt\theta_*(\gamma)$ is special, at least in the case when $\gamma \in (0,\sqrt{2})$. The argument uses results from Sections~\ref{sec-uiqbot} and~\ref{sec-busemann}, so the reader may want to return to this remark after reading those sections.
Consider an infinite bipolar-oriented map $\uibot^\nu$ with a biased face degree distribution, encoded by a bi-infinite walk with increment distribution $\nu$ as in Section~\ref{sec-special-bipolar}.  
By varying $\nu$, we can get walks whose coordinates have any correlation in $(-1,0)$, and hence maps in the universality class of $\gamma$-LQG for any $\gamma\in (0,\sqrt 2)$. 

As explained in Section~\ref{sec-special-bipolar}, the proofs of the independent and stationary increments properties for the Busemann function in Theorem~\ref{thm-busemann-property} carry over essentially verbatim to $\uibot^\nu$. 
Hence, if $\XDP$ distances on $\uibot^\nu$ converge to a directed LQG metric $\frk L_\Phi^{\XDP}$, it should hold for each $t \in \BB R$ that the Busemann function of $\frk L_\Phi^{\XDP}$ on $\bdy \eta((-\infty,t])$ is a stable L\'evy process. 
We will argue that the scaling limit of $\XDP$ distances on $\uibot^\nu$ should in fact correspond to $\wt\theta = \wt\theta_*(\gamma)$. In fact, we suspect that one gets $\wt\theta=\wt\theta_*(\gamma)$ for the scaling limit of directed distances in bipolar-oriented planar maps encoded by the KMSW bijection for \emph{any} face degree distribution with a sufficiently light tail, but we will not justify this here.

Let $\wh{\mcl Z}^\nu : \BB N_0 \to \BB Z$ be a bi-infinite walk with increment distribution $\nu$ started from $(0,0)$ and conditioned so that its first coordinate stays non-negative.
Let $\wh M^\nu$ be the infinite planar map obtained from $\wh{\mcl Z}^\nu$ via the KMSW procedure (i.e., the analog of the UIQBOT of Section~\ref{sec-uiqbot}). 
It follows from~\cite[Theorems 1.3 and 3.5]{ag-disk} that an LQG wedge of weight $\gamma^2/2$ can be encoded by a 2d Brownian motion of correlation $-\cos(\pi\gamma^2/4)$ conditioned to stay in the right half-plane. Hence, $\wh M^\nu$ is a discrete analog of an LQG wedge of weight $\gamma^2/2$. 

Define $M_{0,\infty}^\nu$ as in Definition~\ref{defn-future-map} but with $\uibot^\nu$ in place of $\uibot$. By the analog of Proposition~\ref{prop-future-map-law} for $\uibot^\nu$, the submap of $M_{0,\infty}^\nu$ induced by the set of vertices reachable by a directed path in $M_{0,\infty}^\nu$ started from the initial vertex of the root edge has the same law as $\wh M^\nu$. By the previous paragraph, for the scaling limit of $\XDP$ distances on $\uibot^\nu$, the set of points reachable by a directed path from the origin should be an LQG wedge of weight $\gamma^2/2$. As explained just after~\eqref{eqn-special-theta}, this is equivalent to $\wt\theta=\wt\theta_*(\gamma)$.  
\end{remark}

\begin{remark} \label{remark-critical-metric}
If we fix $\theta \in (0,\pi)$ and send $\gamma\to 2$, we expect that the longest-path directed LQG metrics $\frk L_\Phi^{\op{LDP}}$ should converge, in some sense, to a random directed metric $\wh{\frk L}^{\op{LDP}}$ depending on $\theta$. The longest directed paths for $\wh{\frk L}^{\op{LDP}}$ should be closely related to longest increasing subsequences (LIS) in  Brownian separable permutons~\cite{bassino-separable-permuton}.\footnote{The Brownian separable permutons depend on a parameter $q \in (0,1)$ which should be related to $\wt\theta$ by the formula $q  = ( \pi -  \wt\theta)/\pi$.
%$q/(1-q) = ( \pi -  \wt\theta)/\wt\theta$
This relation is obtained by sending $\gamma\to 2$ in~\cite[Theorem 1.1]{asy-p-theta} and recalling that $\theta=\wt\theta -\pi/2$.} Substantial progress on constructing $\wh{\frk L}^{\op{LDP}}$ has been made in the recent paper~\cite{abbds-separable-lis}, which computes (as a function of $p=1-q$) the growth exponent for the length of the LIS in permutations sampled from the skew Brownian permuton and shows that the length of the LIS, re-scaled appropriately, converges in law to a non-constant random variable. As explained in Remark~\ref{remark-separable}, the techniques of~\cite{abbds-separable-lis} are completely different from those of the present paper.
  
The limiting metric $\wh{\frk L}^{\op{LDP}}$ should have a coupling with $\gamma=2$-LQG decorated by CLE$_4$ due to the critical mating-of-trees theorem of~\cite{ahps-critical-mating}. Consequently, it can be interpreted as a directed version of the 2-LQG metric. However, in contrast to the directed metrics in Conjecture~\ref{conj-directed}, the directed metric $\wh{\mcl L}^{\op{LDP}}$ should not be defined on $\BB C$, but rather on a tree-like space obtained by identifying each CLE$_4$ loop to a point. See~\cite[Section 1.3]{abbds-separable-lis} for further discussion.

We are less certain about the behavior of the shortest-path directed LQG metrics $\frk L_\Phi^{\op{SDP}}$ as $\gamma\to 2$. 
We expect that, at least for $\wt\theta \in (0,\pi)$, the scaling exponent converges to $1/2$ as $\gamma \to 2 $ with $\wt\theta$ fixed, but we are not sure whether there is an interesting limiting directed metric.
\end{remark}

%For these directed metrics, the two space-filling SLE$_\kappa$ curves $\eta$ and $\wt\eta$ will be replaced by two different explorations of the same CLE$_4$, defined using different Bernoulli random variables associated with the CLE$_4$ loops; see~\cite[Section 3]{ahps-critical-mating} for the definition of such CLE$_4$ explorations. 

\subsection{Parallels with directed first- and last-passage percolation}
\label{sect:fpp-lpp}

We explain why directed distances on uniform bipolar-oriented planar maps are analogous to directed first- or last-passage percolation on $\BB Z^2$; see also Figure~\ref{fig-bipolar-map-LPP}. We will then discuss some parallels between directed LQG metrics and the directed landscape. 

In two-dimensional directed first- or last-passage percolation, each vertex $(x,y) \in \BB Z^2$ is assigned an i.i.d.\ non-negative random weight $\omega_{x,y}$. Fix now a reference point $(\ol x,\ol y)$. A directed path from $(x,y)$ to $(\ol x,\ol y)$ is a sequence of lattice points that starts at $(x,y)$, ends at $(\ol x,\ol y)$, and moves only in unit steps either to the right or upward. The passage time $T(\pi)$ of a directed path $\pi$ is the sum of the random weights of the lattice points it visits. For a point $(x,y)$, the last- (resp.\ first-) passage time is defined as
    \begin{equation}\label{eq:LPPFPP-defn}
       \op{LPT}(x,y)=\max_{\pi:(x,y)\to (\ol x,\ol y)} T(\pi)\qquad\text{and}\qquad \op{FPT}(x,y)=\min_{\pi:( x,y)\to(\ol x,\ol y)} T(\pi).
    \end{equation}
The quantities $\op{LDP}$ and $\op{SDP}$ from Definition~\ref{def-ldp} are the planar map analogs of $\op{LPT}$ and $\op{FPT}$. In the case of bipolar-oriented triangulations, we do not need extra randomness, such as the weights \(\omega_{x,y}\) used in the lattice setting, since directed distances are already random due to the map and orientation being random. On the other hand, we expect that adding i.i.d.\ weights to the vertices (with sufficiently light tails) would not change the limiting behavior of directed distances; see for instance~\cite{curien-legall-fpp} for a similar property for undirected planar map distances.

The case when the vertex weights $\omega_{x,y}$ have an exponential distribution --- henceforth \textbf{exponential LPP/FPP} --- is more solvable than other cases.\footnote{\ewain{The cases of geometric and Poisson distributed vertex weights are also solvable, but we focus on the exponential case for concreteness.}}. Exponential LPP has been studied more extensively than directed exponential FPP, so we henceforth focus on this case.
The existence of the Busemann function for exponential LPP (i.e., the analog of our Theorem~\ref{thm-busemann}) was established in~\cite{fp-second-class} \ewain{following earlier work~\cite{newman-fpp-surface} which proves a conditional existence result for a more general weight distribution.}
In LPP, each semi-infinite geodesic is associated with a direction $\xi\in\BB R$, so there is a one-parameter family of Busemann functions, one for each $\xi $. This is in contrast to the case of the UIBOT, where there is only one ``direction'' for semi-infinite geodesics (see, e.g., Lemma~\ref{lem-infty-cut}), hence only one Busemann function. 

The $1/3$ shape fluctuation exponent for exponential LPP was first obtained in~\cite{johansson-shape-fluc} using the RSK correspondence. This is roughly analogous to the $3/4$ scaling exponents for LDP distances in the setting of the present paper. 
The Busemann function for exponential LPP has i.i.d.\ increments, analogously to the case of the UIBOT (Theorem~\ref{thm-busemann-property}). This was shown in~\cite{sepp-corner-growth} using the so-called Burke property from~\cite{bcs-burke-property}. %\cite[Theorems 4.2 and 4.3]{sepp-corner-growth}. 
It was shown in~\cite{busani-busemann-lpp} that the scaling limit of the Busemann function on a horizontal line is the so-called \textbf{stationary horizon}, a family of coupled Brownian motions indexed by the direction $\xi$. 
This is analogous to our Theorem~\ref{thm-busemann-conv}, but in our setting the scaling limit is a stable Lévy process instead of a family of Brownian motions. 

\begin{figure}[ht!]
\begin{center}
\includegraphics[width=0.95\textwidth]{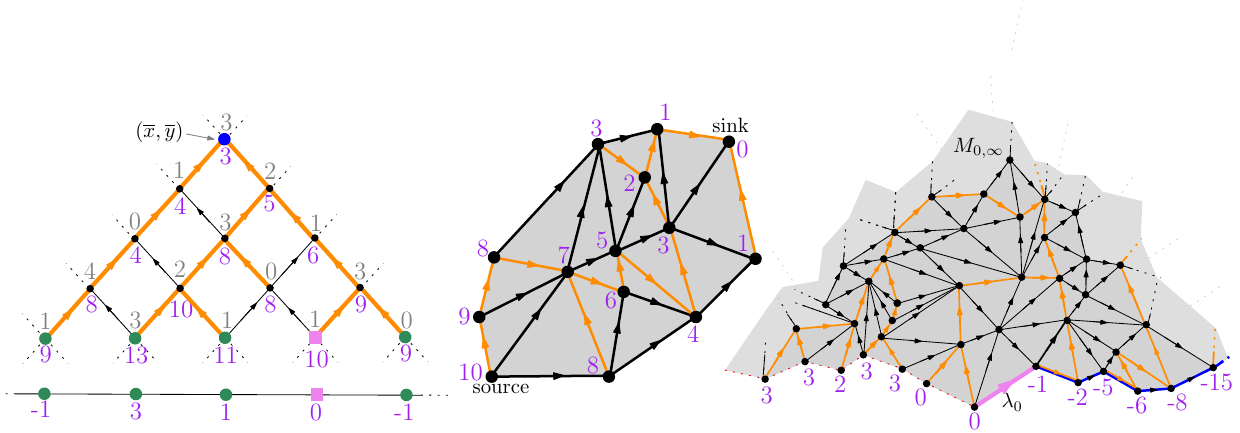}
\caption{\label{fig-bipolar-map-LPP} 
\textbf{Left:} Last passage percolation on $\BB Z^2$. A portion of $\BB Z^2$, rotated counterclockwise by 45 degrees, is shown. A reference point $(\ol x, \ol y)$ is shown in blue. The grey labels associated with each vertex $(x,y)$ are the weights $\omega_{x,y}$, while the purple labels are the last-passage times $\op{LPT}(x,y)$ introduced in \eqref{eq:LPPFPP-defn}. The labels on the bottom of the picture are the $\op{LPT}(x,y)$ times of the corresponding green vertices renormalized so that the $\op{LPT}$ time for the squared violet vertex is zero. The orange paths are $\op{LPT}$-geodesics, i.e.\ directed paths that realize the $\op{LPT}$ times. 
\textbf{Middle:} A bipolar-oriented triangulation with its vertices labeled by their $\op{LDP}$ distance to the sink vertex, as introduced in Definition~\ref{def-ldp}. The orange paths are  $\op{LDP}$ geodesics from each vertex to the sink. \textbf{Right:} A portion of the infinite rooted submap $(M_{0,\infty},\lambda_0)$ of the UIBOT introduced in Definition~\ref{defn-future-map}. The boundary vertices are labeled in purple by the values of the Busemann function $\mcl X$ for $\op{LDP}$ introduced in Theorem~\ref{thm-busemann}. These labels are the analogue of the bottom labels in the picture on the left when one sends the blue reference point $(\ol x, \ol y)$ to infinity (recall that the limiting labels depend on the direction $\xi$ in which we send $(\ol x, \ol y)$ to infinity compared to the squared violet vertex).  The orange paths are  $\op{LDP}$ geodesics from the boundary vertices to the sink ``at infinity'' in the sense of Definition \ref{def-infinite-geo}. The two models in the middle and right picture are a natural version of last-passage percolation on planar maps, one in finite and one in infinite volume.}
\end{center}
\end{figure}

The full limit of exponential LPP when we re-center and re-scale appropriately (analogous to our Conjecture~\ref{conj-directed}) was obtained in~\cite{dv-dl-lis}. 
The limiting object is the \textbf{directed landscape}, a random directed metric $\frk L : \BB R^2 \times \BB R^2 \to \BB R \cup \{-\infty\}$ constructed in~\cite{dov-dl}. It is closely connected to the KPZ equation~\cite{kpz-fluctuation}. Intuitively speaking, for $x,y\in\BB R$ and $t < s$, $\frk L(x,t ; y,s)$ describes the length of the longest directed path from $(x,t)$ to $(y,s)$ in a certain random geometry, where here a path is considered to be directed if its second coordinate is increasing (we set $\frk L(x,t ; y , s) = -\infty$ if $s < t$). See, e.g.,~\cite{ganguly-dl-survey,gm-dl-survey} for expository articles on the directed landscape. 

Directed first-passage percolation is believed to converge (under appropriate re-centering and re-scaling) to $-\frk L$. This is in contrast to the setting of uniform bipolar-oriented triangulations, where LDP and SDP distances are believed to have fundamentally different scaling limits (see Conjecture~\ref{conj-directed}). 
%In the directed landscape, each semi-infinite geodesic is associated with a direction $\xi\in\BB R$, so there is a one-parameter family of Busemann functions, one for each $\xi $. In contrast, for uniform bipolar-oriented triangulations, there is only one ``direction'' for semi-infinite geodesics (see, e.g., Lemma~\ref{lem-infty-cut}), hence only one Busemann function. 

\ewain{It follows from the local Brownian behavior of the Airy-2 process~\cite[Corollary 4.2]{ch-brownian-gibbs} that} the Busemann function in the directed landscape on a fixed horizontal line $ \BB R \times \{t\} $, with geodesics going in a fixed direction, is a Brownian motion (see, e.g.,~\cite[Corollary 3.22]{rv-infinite-geo}). The joint law of the Busemann functions on $\BB R\times \{t\}$, for all directions, has the law of the stationary horizon~\cite[Theorem 5.3(iii)]{bss-dl-semi}(see the above discussion on Busemann functions in different directions for exponential LPP).  The coupling for the Busemann functions with different values of $t$ is described in terms of the KPZ fixed point of~\cite{mqr-kpz-fixed-pt}, see~\cite[Theorem 5.1(iv)]{bss-dl-semi}). 

For the directed $\sqrt{4/3}$-LQG metrics as in Conjecture~\ref{conj-directed}, the interface $\bdy \eta(-\infty,t]$ (which, for $t=0$, is the continuum analog of the discrete interface $\bdy M_{0,\infty}$) plays a role analogous to $\BB R\times \{t\}$. 
Theorem~\ref{thm-busemann-conv} suggests that the Busemann function along this interface for a fixed value of $t$ is a certain stable Lévy process.
We do not yet have a $\sqrt{4/3}$-LQG analog of the KPZ dynamics which governs how these Busemann functions evolve when $t$ varies (see Problem~\ref{prob-busemann-joint}). 
 
We expect that the directed LQG metrics have deterministic limits when $\gamma\to 0$, but it is natural to also look at the second order behavior. We have the following tantalizing question. 

\begin{prob} \label{prob-dl}
Do the directed LQG metrics of Conjecture~\ref{conj-directed} converge, in some sense, to the directed landscape when we send $\gamma\to 0$ (with some appropriate behavior for $\wt\theta$) and re-center and re-scale appropriately?
\end{prob} 

A possible guess for the appropriate behavior of $\wt\theta$ in Problem~\ref{prob-dl} is to take $\wt\theta = \wt\theta_*(\gamma)$ as in~\eqref{eqn-special-theta} and send $\gamma\to 0$. Since $\wt\theta_*(\gamma) \to 0$, directed paths in this setting will converge to straight lines, but the second-order fluctuations around these straight lines could be interesting. If the directed LQG metric does indeed converge to the directed landscape in this regime, then we would have, conjecturally, a one-parameter family of solvable models which interpolate between the directed landscape at $\gamma=0$ and Brownian surfaces at $\gamma=\sqrt{8/3}$ (see the solid red curve in Figure~\ref{fig-directed-lqg-graph}, right).

%% file: tex/prelim.tex
Sections~\ref{sec-kmsw} and~\ref{sec-kmsw-random} contain a review of the Kenyon-Miller-Sheffield-Wilson~\cite{kmsw-bipolar} (KMSW) bijection and its versions for various particular directed triangulations. Section~\ref{sect-cond-indp} gives a conditional independence statement when we cut a Boltzmann bipolar-oriented triangulation by an $\XDP$ geodesic (Lemma~\ref{lem-lr-ind}). This statement is a key ingredient in the proof of the independent increments property in Theorem~\ref{thm-busemann-property}. Section~\ref{sec-kmsw-left} discusses a specialization of the KMSW bijection to so-called boundary-channeled bipolar-oriented triangulations, which will be needed for the proof of the symmetry property in Theorem~\ref{thm-busemann-property}.

\subsection{The Kenyon-Miller-Sheffield-Wilson bijection} 
\label{sec-kmsw}

In this subsection, we only consider deterministic objects. Recall that we assume all planar maps in this paper are equipped with an acyclic orientation of their edges and are drawn in the plane with the edges oriented from south-west to north-east. 

For a planar map $\frk m$, the \textbf{external face} is the unbounded complementary connected component of the union of the edges of $\frk m$. This is well-defined since we are viewing planar maps modulo orientation-preserving homeomorphisms $\BB C \to \BB C$ (instead of $\BB C\cup\{\infty\} \to \BB C\cup\{\infty\}$), recall Section~\ref{sec-overview}. 
We write $\mcl V(\frk m)$, $\mcl E(\frk m)$, and $\mcl F(\frk m)$ for the sets of vertices, edges, and non-external faces of $\frk m$, respectively. 
  
\begin{defn}\label{defn:planar-map-with-missing-edges}
    A \textbf{planar map with missing edges} is a planar map $\frk m$ with two types of edges: \textbf{missing} and \textbf{non-missing} edges (we often refer to non-missing edges simply as edges). The missing edges are always on the boundary of the external face of $\frk m$ and are \emph{not} considered to be part of $\mcl E(\frk m)$. 
    %This means that $\mcl E(\frk m)$ does not include all of the edges on the boundaries of the faces in $\mcl F(\frk m)$.
    A \textbf{directed planar map with missing edges}  is a planar map $\frk m$ with missing edges equipped with an orientation of the edges in $\mcl E(\frk m)$ that is acyclic. In particular, missing-edges are \emph{not} oriented.
    See the left-hand side of Figure~\ref{fig-bipolar-boundaries} for an example.
\end{defn}

The planar maps $\frk m$ with missing edges considered in this paper will always have the boundary of their external face divided into four (possibly empty)  sets of consecutive all missing or all non-missing edges. When the four sets are all non-empty, they alternate between missing and non-missing types. See again the map on the left-hand side of Figure~\ref{fig-bipolar-boundaries} for an example.

\begin{defn}\label{defn:bound-edges}
    Given an oriented planar map $\frk m$ with missing edges, make the following definitions (see the map on the left-hand side of Figure~\ref{fig-bipolar-boundaries}). 
\begin{itemize}
    \item The \textbf{upper-left} (resp.\ \textbf{lower-right}) \textbf{boundary} of $\frk m$ is defined to be the set of edges $e\in\mcl E(\frk m)$ lying on the boundary of the external face of $\frk m$  which are oriented so that the external face of $\frk m$ is  immediately to the left (resp.\ right) of $e$. For all of the maps we consider, the upper-left and lower-right boundaries of $\frk m$ will be (possibly empty) simple directed paths.
    \item The \textbf{upper-right} (resp.\ \textbf{lower-left}) \textbf{boundary} of $\frk m$ is defined to be the set of consecutive \emph{missing} edges $e\notin\mcl E(\frk m)$ between the terminal (resp.\ initial) vertex of the upper-left boundary and the terminal (resp.\ initial) vertex of the lower-right boundary.  
\end{itemize}
When the orientation of  $\frk m$ is bipolar and there are no missing edges, we simply call the upper-left boundary \textbf{left} boundary and the lower-right boundary \textbf{right} boundary.
\end{defn}
Our oriented planar maps (with possibly missing edges) will often have a single distinguished edge on the upper-left boundary of the map, called the \textbf{active edge} of the map. Note that, in particular, the active edge is in $\mcl E(\frk m)$.

\begin{figure}[t]
\begin{center}
\includegraphics[width=1\textwidth]{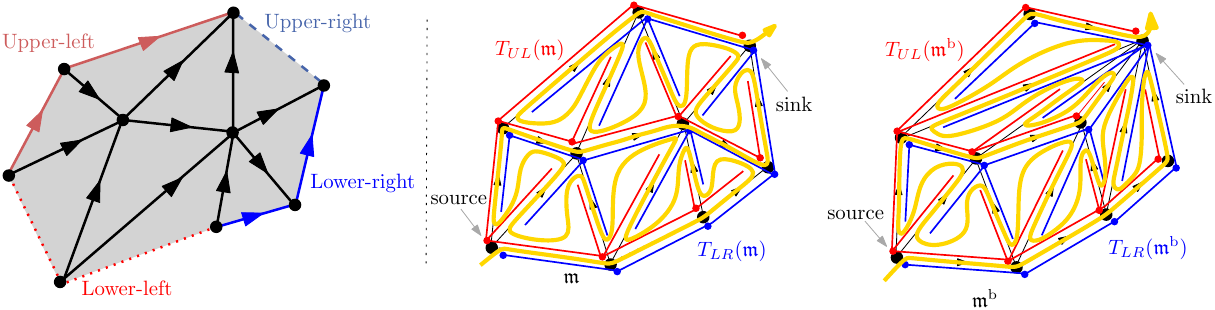}  
\caption{\label{fig-bipolar-boundaries} 
\textbf{Left:} An oriented planar triangulation with missing edges. The missing edges are the dotted (non-oriented) edges on the external face. The upper-left/lower-left/upper-right/lower-right boundaries of the triangulation on the left are shown in darkred/red/lightblue/blue. Note that the upper-right and lower-left boundaries are formed by missing edges. \textbf{Right:} The first bipolar-oriented triangulation $\frk m$ is the  map on the left picture where we fixed an orientation of the missing edges (turning them into non-missing edges) to obtain a bipolar orientation. The second map $\frk m^{\bc}$ is a boundary-channeled bipolar-oriented triangulation (Definition~\ref{def-reverse-map}). The trees $T_{UL}(\frk m)$ and $T_{LR}(\frk m)$ (resp. $T_{UL}(\frk m^{\bc})$ and $T_{LR}(\frk m^{\bc})$) introduced in the description of the inverse KMSW procedure (below Lemma~\ref{lem-kmsw-bdy}) are  shown in red and blue in both maps.  The corresponding Peano paths are shown in gold. Note that the right-most branch of $T_{UL}(\frk m)$ has two edges that correspond to the first two right-boundary edges of $\frk m$, while the right-most branch of $T_{UL}(\frk m^\bc)$ has four edges that correspond to the four (and so all) right-boundary edges of $\frk m^\bc$.
}
\end{center}
\end{figure}

\medskip
  
A key tool in the study of bipolar-oriented planar maps is the \textbf{Kenyon-Miller-Sheffield-Wilson (KMSW)} bijection~\cite{kmsw-bipolar}, which encodes a bipolar-oriented planar map by means of a walk on $\BB Z^2$. We will only need the case of oriented triangulations, in which case the encoding walk has increments in\footnote{The map obtained via the KMSW procedure in our setting is reflected so that the left and right boundaries are swapped as compared to the map in~\cite{kmsw-bipolar}. 
This is to make it so that the first (resp.\ second) coordinate of the encoding walk corresponds to the left (resp.\ right) boundary length, as in the continuum setting of~\cite{wedges}. Moreover, in this paper, we decided to draw all the maps with the orientation going from west to east (instead of from south to north as in~\cite{kmsw-bipolar}) in order to have horizontal (instead of vertical) boundaries in the figures.
} 
$\left\{(1,-1), (-1,0) , (0,1)\right\}$.
The following definition explains how to construct a map from a walk. 

\begin{figure}[t]
\begin{center}
\includegraphics[width=1\textwidth]{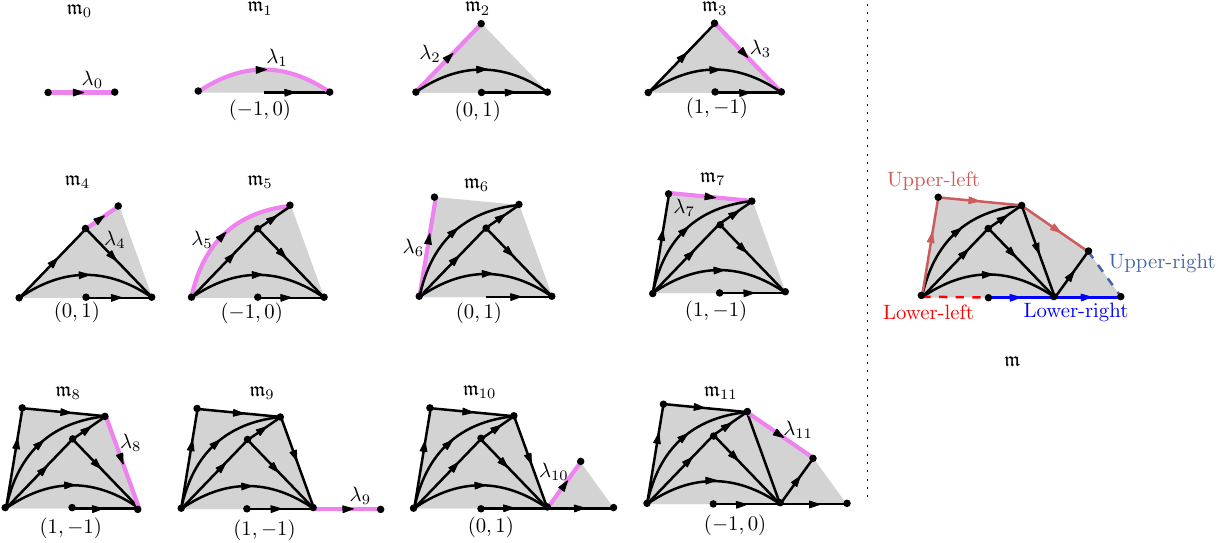}  
\caption{\label{fig-kmsw}  
\textbf{Left:} The directed triangulations with missing edges $\frk{m}_0,\dots, \frk{m}_{11}$ built from the KMSW procedure for a walk with 11 steps. For convenience, the lower boundary edges are drawn along a horizontal line. The steps of the walk are shown underneath the maps. The active edges $\lambda_{j}$ are shown in violet. All edges are oriented from west to east.
\textbf{Right:} The final map $\frk{m} = \frk{m}_{11}$ with its four special boundary segments shown in different colors as in Figure~\ref{fig-bipolar-boundaries}. Note that edges along the upper-left and lower-right boundary belong to $\frk{m}$, but edges along the lower-left and lower-right boundary do not.} 
\end{center}
\end{figure}

\begin{defn}[KMSW procedure] \label{def-kmsw}
Let $\frk Z : [0,n]\cap\BB Z \to\BB Z^2$ be a two-dimensional walk with increments 
\[\frk Z(j) -\frk Z(j-1) \in \Big\{(1,-1), (-1,0) , (0,1)\Big\}.\] 
The \textbf{KMSW procedure} builds a triangulation $\frk m = \frk m_n$ (possibly with some missing edges), equipped with an acyclic orientation, from $\frk Z$ through the following inductive procedure. See Figure~\ref{fig-kmsw} for an illustration. 

Let $(\frk m_0,\lambda_0)$ be the map consisting of a single directed active edge $\lambda_0$. The only non-empty boundaries of this map are the upper-left and lower-right boundaries (which coincide). We assume that this single edge is drawn in the plane so that it is oriented from west to east.
%We define the \textbf{upper-left (resp.\ lower-right) boundary} of $\frk m_{j-1}$ to be the set of edges $e$ lying on the boundary of the external face of $\frk m_{j-1}$ which are oriented so that the external face of $\frk m_{j-1}$ is to the left (resp.\ right) of $e$. We define the \textbf{upper-right (resp.\ lower-left) boundary} of $\frk m_{j-1}$ to be the set of edges $e$ lying on the boundary of the external face of $\frk m_{j-1}$ which do not belong to $\frk m_{j-1}$, and are oriented so that the external face of $\frk m_{j-1}$ is immediately to the left (resp.\ right) of $e$. 

Inductively, suppose $j\in [1,n]\cap\BB Z$ and a directed map $(\frk m_{j-1},\lambda_{j-1})$ with an active edge $\lambda_{j-1} \in \mcl E(\frk m_{j-1})$ (and possibly missing edges) have been defined and drawn in the plane so that all its edges are oriented from south-west to north-east. 
We let $(\frk m_j,\lambda_{j})$ be the directed triangulation with an active edge (and possibly missing edges) obtained from $(\frk m_{j-1},\lambda_{j-1})$ as follows:
\begin{itemize}
\item If $\frk Z(j) - \frk Z(j-1) = (1,-1)$, we construct $(\frk m_j,\lambda_{j})$ in the following way:
\begin{itemize}
    \item If the upper-right boundary of $\frk m_{j-1}$ is not empty, we consider the missing edge of the upper-right boundary of $\frk m_{j-1}$ immediately to the east of $\lambda_{j-1}$ and construct $(\frk m_j,\lambda_{j})$ by adding this edge  to the set of edges of $\frk m_{j-1}$, orienting this edge from west to east and setting $\lambda_{j}$ equal to this new edge. See, for instance, how the triangulation $\frk m_3$ is obtained from the triangulation $\frk m_2$ in \cref{fig-kmsw}.
    
    \item Otherwise, we construct $(\frk m_j,\lambda_{j})$ by adding to the set of edges of $\frk m_{j-1}$ a new active edge $\lambda_{j}$ oriented from west to east, whose initial vertex coincides with the terminal vertex of $\lambda_{j-1}$. Note that $\lambda_{j}$ does not lie on the boundary of any triangle of $\frk m_{j}$. See, for instance, how the triangulation $\frk m_9$ is obtained from the triangulation $\frk m_8$ in \cref{fig-kmsw}.
\end{itemize}

\item If $\frk Z(j) - \frk Z(j-1)= (-1,0)$, we construct $(\frk m_j,\lambda_j)$ by adding to $\frk m_{j-1}$ a new triangular face $t_j$ with a new directed edge $\lambda_{j}$ on the boundary of $t_j$ with the following properties:  The edges $\lambda_{j-1}$ and $\lambda_{j}$ both lie on the boundary of $t_j$, and their terminal vertices coincide. Moreover,
\begin{itemize}
    \item if the upper-left boundary of $\frk m_{j-1}$, excluding $\lambda_{j-1}$, is not empty,  the third edge on the boundary of $t_j$ is the edge of the upper-left boundary of $\frk m_{j-1}$, which lies immediately to the west of $\lambda_{j-1}$. In particular, the initial vertex of this edge coincides with the initial vertex of $\lambda_{j}$. See, for instance, how the triangulation $\frk m_5$ is obtained from the triangulation $\frk m_4$ in \cref{fig-kmsw};
    \item otherwise, the third edge on the boundary $t_j$ is a new \emph{missing} edge,  with one vertex equal to the initial vertex of $\lambda_{j-1}$. Moreover, the other vertex of this missing edge coincides with the initial vertex of $\lambda_{j}$. See, for instance, how the triangulation $\frk m_1$ is obtained from the map $\frk m_0$ in \cref{fig-kmsw};
\end{itemize}

\item If $\frk Z(j) - \frk Z(j-1) = (0,1)$, we construct $\frk m_j$ by adding to $\frk m_{j-1}$ a new triangle $t_j$ and a new directed edge $\lambda_{j}$ with the following properties. The edges $\lambda_{j-1}$ and $\lambda_{j}$ both lie on the boundary of $t_j$, and their initial vertices coincide. The third edge on the boundary of $\lambda_{j}$ is a new \emph{missing} edge, whose initial vertex coincides with the terminal vertex of $\lambda_{j}$.  See, for instance, how the triangulation $\frk m_2$ is obtained from the triangulation $\frk m_1$ in \cref{fig-kmsw};
\end{itemize}
We finally set $\frk m = \frk m_n$. %See, for instance, the triangulation $\frk m$ on the right-hand side of~\cref{fig-kmsw}.
\end{defn}

The map $\frk m$ produced by the KMSW procedure depends only on the increments of $\frk Z$, so it is unchanged if we replace $\frk Z$ by $\frk Z  + z$ for any $z\in\BB Z^2$.

\begin{remark} \label{remark-kmsw-infinite} 
Let $n\in\BB Z$. The KMSW procedure can still be defined if $\frk Z : \BB Z_{\geq n} \to \BB Z^2$  with increments 
 in $\{(1,-1), (-1,0) , (0,1)\}$. In this case, we take $\frk{m}_{n,\infty}$ to be the increasing union of the triangulations $(\frk{m}_j)_{j\in\BB Z_{\geq n}}$ built during the KMSW procedure of Definition~\ref{def-kmsw}, where the maps are labeled so that $(\frk{m}_n,\lambda_n)$ denotes the initial map consisting of a single active oriented edge $\lambda_n$. Moreover, we set the edge in $\frk{m}_{n,\infty}$ corresponding to the initial active edge $\lambda_n$ to be the root edge of $\frk{m}_{n,\infty}$. The boundary of $\frk{m}_{n,\infty}$ is divided into four (possibly infinite or possibly empty) intervals of consecutive edges: the upper-left (resp.\ lower-left, lower-right, upper-right) boundary of $\frk m_{n,\infty}$ is the set of edges that belong to the upper-left (resp.\ lower-left, lower-right, upper-right) boundary of $\frk m_j$ for infinitely many $j$.

The KMSW procedure can also be defined for bi-infinite walks $\frk Z : \BB Z \to \BB Z^2$ with increments 
 in $\{(1,-1), (-1,0) , (0,1)\}$. In this case, we take  $\frk{m}_{-\infty,\infty}$ to be the union of the triangulations $(\frk{m}_{n,\infty})_{n\in\BB Z}$ -- note that these triangulations increase when $n$ decreases -- where the maps $\frk{m}_{n,\infty}$ are obtained from $\frk Z |_{[n,\infty)}$ using the KMSW procedure described in the previous paragraph. 
 Moreover, we set the edge in $\frk{m}_{-\infty,\infty}$ corresponding to the active edge $\lambda_0$ to be the root edge of $\frk{m}_{-\infty,\infty}$. 
 The boundary of $\frk{m}_{-\infty,\infty}$ is divided into four (possibly-infinite and possibly-empty) intervals of consecutive edges: the upper-left (resp.\ lower-left, lower-right, upper-right) boundary of $\frk m_{-\infty,\infty}$ is the set of edges that belong to the upper-left (resp.\ lower-left, lower-right, upper-right) boundary of $\frk{m}_{n,\infty}$ for infinitely many $n$.  Note that it is possible for all four boundary intervals of $\frk{m}_{-\infty,\infty}$ to be empty; in this case, $\frk{m}_{-\infty,\infty}$ has no boundary.
 \end{remark}

The following is the special case of~\cite[Theorem 2]{kmsw-bipolar} when we restrict to triangulations.

\begin{thm}[\cite{kmsw-bipolar}]\label{thm-kmsw}
Let $\el ,r , n \in \BB N$. The KMSW procedure restricts to a bijection between the following two sets.
\begin{itemize}
\item Walks with increments in $\{(1,-1) , (-1,0) , (0,1)\}$ which start at $(0,0)$, end at $(\el -1 ,-r + 1)$, have $n-1$ total steps, and stay in $[0,\infty) \times [-r+1,\infty)$. 
\item Bipolar-oriented triangulations with $n$ total edges,\footnote{We stress that these triangulations have no missing edges.} $\el $ left boundary edges, and $r $ right boundary edges (Definition~\ref{def-bipolar-bdy}). 
\end{itemize}
\end{thm}

Note that the walk used in Figure~\ref{fig-kmsw} is not a valid walk for the theorem above, since the walk does not stay in the requested domain.
In the setting of Definition~\ref{def-kmsw}, the lengths (and several other properties) of the four special boundary segments of $\mathfrak{m}$ can be recovered from the encoding walk as follows.

\begin{lem} \label{lem-kmsw-bdy}
Let $\frk{Z}=(\frk{L},\frk{R}) : [0,n]\cap\BB Z \to \BB Z^2$ be a walk with increments in $\{(1,-1), (-1,0) , (0,1)\}$ and started at $\frk{Z}(0) = (0,0)$ and let $\frk{m}$ be its associated map via the KMSW procedure as in Definition~\ref{def-kmsw}. 
Define the four boundary segments of $\frk{m}$ as done in Definition~\ref{defn:bound-edges}. 
\begin{enumerate}[label=(\roman*),ref=(\roman*)]
    \item \label{item-lower-left} The time $\min\left\{ m \in \BB N : \frk L (m) = -k \right\}$, is the time when the $k$th lower-left (missing) boundary edge of $\frk{m}$ (counting starting from the initial vertex of $\lambda_0$ and moving along the lower-left boundary)  is added to $\frk{m}$ by the KMSW procedure. Moreover, this is the first time that the west vertex of this new edge is the initial vertex of an active edge (and all the other times the same fact happens are the times $m$ such that $\frk L (m)= -k$ and $\frk L (j) \geq  -k$ for all $j\leq m$).
 
    \item \label{item-lower-right} The time $\min\left\{ m \in \BB N : \frk R (m) = -k \right\}$,  is the time when the $k+1$th (counting respecting the orientation) lower-right boundary edge of $\frk{m}$ is added to $\frk{m}$ by the KMSW procedure. Moreover, this is the first time that the terminal vertex of this new edge is the terminal vertex of an active edge (and all the other times the same fact happens are the times $m$ such that $\frk R (m)= -k$ and $\frk R (j) \geq  -k$ for all $j\leq m$).
    
    \item \label{item-upper-left} \jbb{The $k$th smallest $m\in\{0,\dots, n\}$} such that ${\frk L} (j) \geq {\frk L} (m)+1$ for all $j\in [m+1,n]\cap\BB Z$ is the time when the $k$th (counting respecting the orientation)  upper-left boundary edge of $\frk{m}$ is added to $\frk{m}$ by the KMSW procedure. Moreover, for every time $m'\leq m$ such that ${\frk L} (m')={\frk L} (m)$ and ${\frk L} (j) \geq {\frk L} (m')$ for all $j\in [m', m]\cap\BB Z$, the corresponding active edge $\lambda_{m'}$ has the same initial vertex as the upper-left boundary edge $\lambda_{m}$. 
    
    \item \label{item-upper-right} \jbb{The $k$th smallest $m\in\{1,\dots, n\}$} such that ${\frk R} (j) \geq {\frk R} (m-1)+1$ for all $j\in [m,n]\cap\BB Z$ is the time when the $k$th upper-right (missing) boundary edge of $\frk{m}$ (counting starting from the terminal vertex of the lower-right boundary and moving along the upper-right boundary) is added to $\frk{m}$ by the KMSW procedure.  
\end{enumerate}
In particular, denoting the set of edges on the upper-left, upper-right, lower-left, and lower-right boundary segments by $\ol\bdy_L \frk{m}$, $\ol\bdy_R \frk{m}$, $\ul\bdy_L \frk{m}$, and $\ul\bdy_R \frk{m}$, respectively, we have that
\allb \label{eqn-kmsw-bdy}
\# \ol \bdy_L \frk{m} &= \frk{L}(n) - \min_{j\in [0,n]\cap\BB Z} \frk{L}(j)+ 1  ,\quad \qquad
\# \ol \bdy_R \frk{m} = \frk{R}(n) - \min_{j\in [0,n]\cap\BB Z} \frk{R}(j) ,\notag\\ 
\# \ul \bdy_L \frk{m} &=   - \min_{j\in [0,n]\cap\BB Z} \frk{L}(j) , \quad 
\qquad\qquad\qquad\# \ul \bdy_R \frk{m} =   - \min_{j\in [0,n]\cap\BB Z} \frk{R}(j) + 1 .
\alle 

\end{lem}
\begin{proof}
We prove the statements for the lower-left and upper-left boundaries. The proof for the lower-right and upper-right boundaries are similar.

For $j\in [0,n]\cap\BB Z$, let $\frk m_j$ be the time-$j$ map in the KMSW procedure, as in Definition~\ref{def-kmsw}. By considering each of the three possible values of $\frk Z(j) -\frk Z(j-1)$ in Definition~\ref{def-kmsw}, we see that for $j\geq 1$,
\begin{align}
    \frk L(j)  - \frk L(j-1)&= (\# \ol \bdy_L \frk{m}_j  - \# \ol \bdy_L \frk{m}_{j-1}) - (\# \ul \bdy_L \frk{m}_{j}  - \# \ul \bdy_L \frk{m}_{j-1})\notag\\
    &= (\# \ol \bdy_L \frk{m}_j  - \# \ul \bdy_L \frk{m}_j) - (\# \ol \bdy_L \frk{m}_{j-1}  - \# \ul \bdy_L \frk{m}_{j-1}) ,
\end{align}
and $\frk L(0)=0= (\# \ol \bdy_L \frk{m}_0  - \# \ul \bdy_L \frk{m}_0)-1 $.
(We also have the analogous relation with $\frk R$ in place of $\frk L$ and ``right'' in place of ``left'', but with starting value $\frk R(0)=0= (\# \ol \bdy_R \frk{m}_0  - \# \ul \bdy_R \frk{m}_0)+1$.) Summing these relations over $i\in [0,j]\cap\BB Z$ gives 
\eqb \label{eqn-kmsw-bdy-diff}
\frk L(j) = \# \ol \bdy_L \frk{m}_j  - \# \ul \bdy_L \frk{m}_j-1 \quad (\text{and} \quad 
\frk R(j) = \# \ol \bdy_R \frk{m}_j  - \# \ul \bdy_R \frk{m}_j+1.)
\eqe

By construction, the lower-left boundary length $\ul\bdy_L \frk m_j$ is non-decreasing in $j$ (once a lower-left boundary edge is added, it can never be removed). 
If $j\in [0,n]\cap\BB Z$ such that $\# \ul \bdy_L \frk{m}_j >  \# \ul \bdy_L \frk{m}_{j-1}$, then at time $j$ we must add a triangle to $\frk m_{j-1}$ which has a new missing edge that becomes part of the lower-left boundary of $\frk m_j$ and a new active edge whose initial vertex $v$ coincides with one vertex of the new missing edge (the one farthest from $\lambda_0$). By Definition~\ref{def-kmsw}, this happens if and only if $\frk Z(j) - \frk Z(j-1) = (-1,0)$ and the upper-left boundary of $\frk m_{j-1}$ consists of the single edge $\lambda_{j-1}$. By~\eqref{eqn-kmsw-bdy-diff}, this is the case if and only if $\frk L(j) = - \# \ul \bdy_L \frk{m}_j = \min_{i \in [0,j]\cap\BB Z} \frk L(i)$.  
 
Hence, the time $\tau_{-k}$ at which $\frk L$ attains the running minimum $-k$ corresponds to the times at which the $k$th lower-left boundary edge is added to the map. 
In particular, this gives the first part of Assertion~\ref{item-lower-left} and the formula for $\#\ul\bdy_L \frk m$ in~\eqref{eqn-kmsw-bdy}. The formula for $\# \ol\bdy_L \frk m$ in~\eqref{eqn-kmsw-bdy} follows from the formula for $\#\ul\bdy_L \frk m$ and~\eqref{eqn-kmsw-bdy-diff}. 
 
We will now prove the rest of Assertion~\ref{item-lower-left}.
The time $\tau_{-k}$ is obviously the first time that the vertex $v$ (which is the west vertex of the $k$th lower-left boundary edge) is the initial vertex of an active edge. 
For all subsequent times $m\geq \tau_{-k}$ such that $\frk{L}(j)\geq -k$ for all $j\leq m$, we have that $\# \ul \bdy_L \frk{m}_m = - \min_{i \in [0,m]\cap\BB Z} \frk L(i)$ since $\ul\bdy_L \frk m_j$ is non-decreasing in $j$, and so by \eqref{eqn-kmsw-bdy-diff} and~\eqref{eqn-kmsw-bdy} for $\ul\bdy_L \frk m_m$, we have that 
\begin{equation}\label{eq:vertex-bnd-times}
    \# \ol \bdy_L \frk{m}_j -1=\frk{L}(m)+ \# \ul \bdy_L \frk{m}_m =\frk{L}(m)-\min_{j\in [0,m]\cap\BB Z} \frk{L}(j)=\frk{L}(m)-(-k)
\end{equation} 
is equal to the distance along the upper-left boundary between $v$ and the initial vertex of $\lambda_m$. Hence, the vertex $v$ is still part of the upper-left boundary of the corresponding map $\frk{m}_m$ and we have that when $\frk{L}(j)\geq -k$ for all $j\leq m$ and $\frk{L}(m) = -k$, the initial vertex of $\lambda_m$ must be the vertex $v$. These are the only times this happens. Indeed, at the first subsequent time $m\geq \tau_{-k}$ such that $\frk{L}(m)= -(k+1)$, using the same argument in the paragraph below \eqref{eqn-kmsw-bdy-diff}, we have that the vertex $v$ is no longer part of the upper-left boundary of the corresponding map $\frk{m}_m$.  
This completes the proof of Assertion~\ref{item-lower-left}.

We now consider Assertion~\ref{item-upper-left}. \jbb{By the construction in Definition~\ref{def-kmsw} and \eqref{eqn-kmsw-bdy} for the upper-left boundary of this map, for $m\in [0,n]\cap\BB Z$, the edge $\lambda_m$ is part of the upper-left boundary of $\frk m$ if and only if $ \frk L  (j) \geq  \frk L (m)+1$ for all $j\in [m+1,n]\cap\BB Z$.} This gives the first part of Assertion~\ref{item-upper-left}. The statement about the initial vertex of $\lambda_m$ follows from a similar argument as in the case of Assertion~\ref{item-lower-left}. 
\end{proof}

We now describe the inverse KMSW procedure, i.e., the procedure that, given a bipolar-oriented triangulation $\frk m$ with missing edges, returns the corresponding encoding walk  $\frk Z$ with increments in $\{(1,-1) , (-1,0) , (0,1)\}$. We stress that the inverse KMSW procedure will be defined only for bipolar-oriented maps, in particular, maps having only non-missing edges.
In what follows, \textbf{disconnecting} an edge from a vertex means deleting the connection between the edge and the vertex, resulting in an edge that is missing one of its vertices. See the right-hand side of Figure~\ref{fig-bipolar-boundaries} for examples of this procedure.

Assume that $\frk m$ has $n$ total edges. For each vertex $v$ of $\frk m$, we disconnect from $v$ every incoming edge to $v$ (including the missing ones) except for the leftmost incoming edge to $v$. This turns the map $\frk m$ into a plane tree $T_{UL}(\frk m)$ rooted at the source, which we call the \textbf{upper-left tree} of $\frk m$ (see the right-hand side of \cref{fig-bipolar-boundaries}). The tree $T_{UL}(\frk m)$ contains every edge of $\frk m$. We denote by $\lambda_1,\dots,\lambda_n$ the edges of $\frk m$ in the order that they are first visited by the counterclockwise contour exploration of $T_{UL}(\frk m)$.
The tree $T_{LR}(\frk m)$, called the \textbf{lower-right tree}, can be obtained similarly from $\frk m$ by disconnecting every outgoing edge from each vertex but the rightmost, and is rooted at the sink. The following  facts hold (see~\cite[Section 2.1]{kmsw-bipolar}):
The clockwise contour exploration of $T_{LR}(\frk m)$ also visits the edges of $\frk m$ in the order $ \lambda_{1},\ldots,\lambda_n$. Moreover, one can draw $T_{UL}(\frk m)$ and $T_{LR}(\frk m)$ in the plane, one next to the other, in such a way that the interface between the two trees traces a path, called the \textbf{Peano path},\footnote{The Peano path coincides with the counter-clockwise contour exploration of $T_{UL}(\frk m)$.} from the source to the sink visiting edges $\lambda_1,\ldots, \lambda_{n}$ in this order. Moreover, every branch of $T_{UL}(\frk m)$ (resp.\
 $T_{LR}(\frk m)$) is a leftmost (resp.\ rightmost) directed path from the source (resp.\ to the sink). 

\begin{defn}[Inverse KMSW procedure]\label{defn:KMSW-inv}
Let $n\geq 1$ and $\frk m$ be a bipolar-oriented triangulation with $n$ total edges, with $\el$ edges on its left boundary and $r$ edges on its right boundary. We define the corresponding KMSW walk $\frk  Z = (\frk  L(j),\frk  R(j))_{0\leq j\leq n-1}$ as follows: for $0\leq j\leq n-1$, 
\begin{itemize}
    \item $\frk  L(j)$ is the height in the tree $T_{UL}(\frk  m)$ of the initial vertex of $\lambda_{j+1}$ (i.e.\ its distance in $T_{UL}(\frk  m)$ from the source);
    \item $\frk  R(j)$ is the height in the tree $T_{LR}(\frk m)$ of the terminal vertex of $\lambda_{j+1}$ (i.e.\ its distance in $T_{LR}(\frk m)$ from the sink) minus $r-1$.
\end{itemize}
\end{defn}

By~\cite[Theorem 2]{kmsw-bipolar}, the function $\frk m\to \frk Z$ of Definition~\ref{defn:KMSW-inv} is the inverse of the bijection from Theorem~\ref{thm-kmsw}. That is, if $\frk m$ and $\frk Z$ are as in Definition~\ref{defn:KMSW-inv}, then $\frk Z$ is a walk with increments in $\{(1,-1) , (-1,0) , (0,1)\}$ which starts at $(0,0)$, ends at $(\el -1 ,-r + 1)$, has $n-1$ total steps, and stay in $[0,\infty) \times [-r+1,\infty)$; and applying the KMSW procedure to $\frk Z$ gives $\frk m$.

\subsection{KMSW encodings of random bipolar-oriented triangulations} 
\label{sec-kmsw-random}
 
Recall from Section~\ref{sec-busemann-intro} that the uniform infinite bipolar-oriented triangulation (UIBOT) is the directed, edge-rooted infinite bipolar-oriented triangulation $(\uibot,\lambda_0)$ obtained as the Benjamini-Schramm~\cite{benjamini-schramm-topology} local limit in law of uniform bipolar-oriented triangulations with $n$ edges and a fixed number of left and right boundary edges, rooted at a uniformly chosen edge. 
We recall the definition of this topology for directed planar maps.

\begin{defn} \label{def-bs-topology}
Let $(M,e)$ and $(\wt M , \wt e)$ be directed planar maps, each equipped with a (directed) root edge. 
For $r \in \BB N $, define the graph-distance ball $\mcl B_r(M)$ to be the directed submap of $M$ consisting of the vertices of $M$ which lie at (undirected) graph distance at most $r$ from the initial vertex of $e$, together with the (directed) edges whose vertices both lie in this vertex set and the faces whose boundary vertices all lie in this vertex set. 
Similarly define $\mcl B_r(\wt M)$. The \textbf{Benjamini-Schramm local distance} between $(M,e)$ and $(\wt M , \wt e)$ is 
\eqbn
D_{\op{loc}}((M,e) , (\wt M , \wt e)) = \inf\left\{ 2^{-r} : \mcl B_r(M) \cong \mcl B_r(\wt M)\right\} 
\eqen
where $\cong$ means that the two balls are isomorphic as directed planar maps.  
\end{defn}

The UIBOT can be described in terms of the KMSW procedure as follows (see~\cite[Proposition 3.16]{ghs-map-dist} for the case of bipolar-oriented maps with unconstrained face degrees -- the triangulation case can be treated identically). 

\begin{prop}[\cite{ghs-map-dist}] \label{prop-kmsw-uibot}
Let $\mcl Z : \BB Z\to\BB Z^2$ be a bi-infinite random walk whose increments are i.i.d.\ uniform samples from $\{(1,-1), (-1,0) , (0,1)\}$, normalized so that $\mcl Z(0) = (0,0)$. Applying the KMSW procedure to $\mcl Z$, as detailed in  Remark~\ref{remark-kmsw-infinite}, gives the UIBOT $\uibot$.
\end{prop}

\jb{Recall the definition of leftmost and rightmost directed paths from Section~\ref{sec-busemann-intro}. The following lemma  describes the boundary of the infinite submaps $(M_{n,\infty},\lambda_n)$ of the UIBOT.

\begin{lem} \label{lem-uibot-bdy}
Let $\mcl Z : \BB Z\to\BB Z^2$ be a bi-infinite random walk whose increments are i.i.d.\ uniform samples from $\{(1,-1), (-1,0) , (0,1)\}$, normalized so that $\mcl Z(0) = (0,0)$.
For $n\in\BB Z$, let $(M_{n,\infty},\lambda_n)$ be the rooted submap of the UIBOT obtained by applying, as detailed in  Remark~\ref{remark-kmsw-infinite}, the KMSW procedure to $\mcl Z|_{[n,\infty)}$. 

The boundary of $M_{n,\infty}$ is the union of the root edge $\lambda_n$, the rightmost directed path in $M_{-\infty,\infty}$ from the terminal vertex of $\lambda_n$ to $\infty$, and the leftmost directed path in $M_{-\infty,\infty}$ from $-\infty$ to the initial vertex of $\lambda_n$. The edge $\lambda_n$ and the edges on the rightmost directed path are part of $M_{n,\infty}$ and form the right boundary of $M_{n,\infty}$, whereas the edges of the leftmost directed path are not, i.e.\ are missing edges, and form the left boundary of $M_{n,\infty}$. 

Furthermore, every vertex on the left boundary of $M_{n,\infty}$, including the initial vertex of $\lambda_n$, has at least one outgoing edge but no incoming edges (recall that missing edges are not oriented), while every other vertex of $M_{n,\infty}$ has at least one incoming and one outgoing edge. Finally, every pair of rightmost directed paths $P_x$ and $P_{x'}$ started from a pair of boundary vertices $x$ and $x'$ of $M_{n,\infty}$ coalesce into each other, i.e.\ there is some $s,s'>0$ such that $P_x(s+t)=P_{x'}(s'+t)$ for all $t\geq 0$.
\end{lem}
\begin{proof}
This follows combining:  ($i$) the construction  in Remark~\ref{remark-kmsw-infinite}   of $M_{n,\infty}$ as the increasing union of the triangulations $(M_{0,j})_{j\in\BB Z_{\geq n}}$ built during the KMSW procedure; ($ii$) the construction of $\frk{m}_{-\infty,\infty}$ in the same remark; and ($iii$) the description of the inverse KMSW procedure in Definition~\ref{defn:KMSW-inv}, recalling that every branch of $T_{UL}(M_{0,j})$ (resp.\
 $T_{LR}(M_{0,j})$) is a leftmost (resp.\ rightmost) directed path from the source (resp.\ to the sink).
\end{proof}}

We have the following description of the KMSW encoding walk for Boltzmann bipolar-oriented planar maps with a given right boundary length (Definition~\ref{def-boltzmann-right}).

\begin{lem} \label{lem-kmsw-boltzmann}
Let $  r \in \BB N$. Let $\mcl Z = (\mcl L , \mcl R) : \BB N_0 \to\BB Z^2$ be a walk started from $(0,0)$ with i.i.d.\ increments sampled uniformly from $\{(1,-1), (-1,0) , (0,1)\}$. Let $\tau$ be the smallest $n\in\BB N$ for which $\mcl R(n) = -r$ and let
\eqbn
E := \left\{ \mcl L(n) \geq 0,\: \forall n \in [0,\tau]\cap\BB Z\right\} .
\eqen 
Let $M$ be the map obtained by applying the KMSW procedure to $\mcl Z|_{[0,\tau-1]}$.  
The conditional law of $M$ given $E$ is that of a Boltzmann bipolar-oriented triangulation with right boundary length $r$ (Definition~\ref{def-boltzmann-right}).  
\end{lem}
\begin{proof}
Let $\frk m$ be a bipolar-oriented triangulation with $\el$ left boundary edges, $r$ right boundary edges, and $n$ total edges. 
Let $\frk Z$ be the corresponding encoding walk under the KMSW bijection as in Theorem~\ref{thm-kmsw}. Then $\frk Z$ is a walk that starts at $(0,0)$, ends at $(\el-1,-r+1)$, has $n-1$ total steps, and stays in $[0,\infty) \times [-r+1,\infty)$. 
We have $M = \frk m$ if and only if $\mcl Z|_{[0,\tau-1]} = \frk Z$. This is the case if and only if $\mcl Z|_{[0,n-1]} = \frk Z$ and $\mcl Z(n) - \mcl Z(n-1) = (1,-1)$ (in which case $E$ occurs).
Since the increments of $\mcl Z$ are i.i.d.,  
\eqbn
\BB P[M = \frk m] 
= \BB P\left[ \mcl Z|_{[0,n-1]} = \frk Z ,\, \mcl Z(n) - \mcl Z(n-1) = (1,-1) \right] 
= \left( \frac13 \right)^n . 
\eqen
Since $\{M = \frk m\} \subset E$, it follows that
\eqbn
\BB P[M = \frk m \,|\, E]  
= \BB P[E]^{-1} \left( \frac13 \right)^n   
\eqen
as required. 
\end{proof}

As a consequence of Lemma~\ref{lem-kmsw-boltzmann}, we see that the first Boltzmann distribution defined in Definition~\ref{def-boltzmann-right} is well-defined, in the sense that it has finite total mass. We can also establish the analogous statement for the second distribution defined in Definition~\ref{def-boltzmann-right}.

\begin{lem} \label{lem-boltzmann-finite}
Assume that we are in the setting of Lemma~\ref{lem-kmsw-boltzmann}. 
Then 
\eqb \label{eqn-boltzmann-finite}
\BB E\left[ \mcl L(\tau) + 1  \,|\, E \right] < \infty .
\eqe 
In particular, the Boltzmann distribution on bipolar-oriented triangulations with right boundary length $r$ and a marked boundary edge (Definition~\ref{def-boltzmann-right}) is well-defined.
\end{lem}

\begin{proof}
\jb{Let $\sigma := \min\{n\geq 0 : \mcl L(n) = -1\}$, so that $E = \{\sigma > \tau\}$. Let
\[
Y(n) := (\mcl L(n)+1) \BB 1_{\{\sigma>n\}}
      = \mcl L(n\wedge \sigma)+1.
\]
Since $\mcl L$ is a martingale, $Y$ is a non-negative martingale w.r.t.\ the filtration generated by $\mcl Z|_{[0,n]}$.

The time $\tau$ is an a.s.\ finite stopping time for this filtration. By Fatou's lemma and the optional stopping theorem,
\alb
\BB E\left[ (\mcl L(\tau) +1) \BB 1_E \right]
= \BB E\left[ Y(\tau) \right]
\leq \liminf_{n\to\infty} \BB E\left[ Y(\tau\wedge n) \right]
= \BB E[Y(0)] = 1.
\ale
This gives~\eqref{eqn-boltzmann-finite}.} As explained just after Definition~\ref{def-boltzmann-right}, the Boltzmann distribution on bipolar-oriented triangulations with right boundary length $r$ and a marked left boundary vertex is obtained by weighting the Boltzmann distribution on bipolar-oriented triangulations with right boundary length $r$ by the left boundary length plus one. By Lemma~\ref{lem-kmsw-boltzmann}, this corresponds to weighting the conditional law of $\mcl Z|_{[0,\tau]}$ given $E$ by $\mcl L(\tau) + 1$. This is a valid weighting by~\eqref{eqn-boltzmann-finite}.
\end{proof}

\subsection{Conditional independence for Boltzmann bipolar-oriented triangulations}\label{sect-cond-indp}

Recall the notation $\op{LDP}$ and $\op{SDP}$ from Definition~\ref{def-ldp}.

\begin{defn} \label{def-geodesic}  
Let $\XDP \in \{\op{LDP}, \op{SDP}\}$. 
Let $\frk{m}$ be a finite planar map equipped with an acyclic orientation of its edges and let $x,y\in\mcl V(\frk{m})$. 
We say that a directed path $\frk{p} : [1,|\frk{p}|]\cap\BB Z \to \mcl E(\frk{m})$ from $x$ to $y$ is an \textbf{$\XDP$ geodesic} if the length $|\frk{p}|$ of $\frk{p}$ is equal to $\XDP_\frk{m}(x,y)$ (Definition~\ref{def-ldp}). \jbb{If $\frk{m}$ is a finite bipolar-oriented map, we} say that $\frk{p}$ is the \textbf{leftmost $\XDP$ geodesic} from $x$ to $y$ if every $\XDP$ geodesic in $\frk{m}$ from $x$ to $y$ lies weakly to the right of $\frk{p}$ if we stand at $x$ and look toward $y$. That is, every $\XDP$ geodesic from $x$ to $y$ lies in the submap of $\frk{m}$ bounded by $\frk{p}$, the rightmost directed path started from $x$, and the rightmost directed path ending at $y$ (including the vertices and edges lying on the three paths).
\end{defn}
 
It is easy to see that if there is a directed path from $x$ to $y$, then there is a unique leftmost $\XDP$ geodesic from $x$ to $y$. 
%Let $P$ be any LDP geodesic from $x$ to $y$. If $P$ is leftmost, we are done. Otherwise, let $P'$ be an LDP geodesic which does not lie weakly to the right of $P$. Let $z$ and $w$ be vertices hit by $P$ and $P'$ such that the segment of $P'$ from $z$ to $w$ lies strictly to the left of the segment of $P$ from $z$ to $w$. Then the segments of $P$ and $P'$ from $z$ to $w$ have the same length. For each such pair of points $z$ and $w$, we replace the segment of $P$ from $z$ to $w$ by the segment of $P'$ from $z$ to $w$. This gives a path $P''$ which lies weakly to the left of both $P$ and $P'$. If $P''$ is a leftmost LDP geodesic, we are done. Otherwise, we repeat the above procedure for $P''$. Since $M$ is finite, this process will terminate after finitely many iterations, yielding a leftmost LDP geodesic. 

\begin{defn} \label{def-boltzmann-lr}
For $\el , r \in \BB N$, let $\mcl M_{\el ,r}$ be the set of finite bipolar-oriented triangulations with boundary having $\el$ edges on their left boundary and $r$ edges on their right boundary. 
We define the \textbf{Boltzmann distribution} on bipolar-oriented triangulations \textbf{with left boundary length $\el$ and right boundary length $r$} exactly as in Definition~\ref{def-boltzmann-right}, but with both the left and right boundary lengths fixed. That is, it is the law of a random bipolar-oriented triangulation $M(\el,r)$ such that for every bipolar-oriented triangulation $\frk{m}$ with $\el$ left boundary edges, $r$ right boundary edges, and $n$ total edges,  
\eqb \label{eqn-boltzmann}
\BB P\left[ M(\el,r) = \frk m \right] = \frac{3^{-n} }{C_{\el,r}},    
\eqe
where $C_{\el,r}$ is a normalizing constant. Equivalently, $M(\el,r)$ is a sample from the conditional law of a Boltzmann bipolar-oriented triangulation with $r$ right boundary edges (Definition~\ref{def-boltzmann-right}) given that its left boundary length is $\el$. 
\end{defn}

The following statement is the starting point of the proof that the Busemann function of Theorem~\ref{thm-busemann} has independent increments; see Figure~\ref{fig-fig-bipolar-map-cond-ind}.

\begin{figure}[t]
\begin{center}
\includegraphics[width=0.35\textwidth]{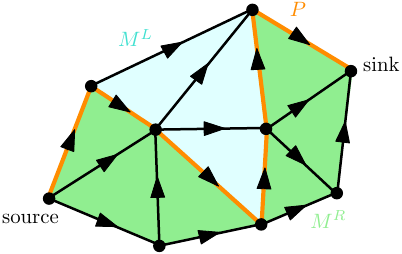}  
\caption{\label{fig-fig-bipolar-map-cond-ind}  
A Boltzmann bipolar-oriented triangulation $M(\el,r)$ with left boundary length $\el=3$ and right boundary length $r=4$. In orange, the leftmost $\op{LDP}$ geodesic $P$ from the source to the sink, as introduced in Definition~\ref{def-geodesic}. The map $M^L$ (resp.\ $M^R$) is the submap of $M(\el,r)$ consisting of the vertices, edges, and faces of $M(\el,r)$ in the lightblue (resp.\ lightgreen) region including the vertices and edges which lie on the orange leftmost $\XDP$ geodesic $P$. Note that the lightblue map $M^L$ does not contain any other geodesic from the source to the sink other than $P$, while the lightgreen map $M^R$ contains additional geodesics. By Lemma~\ref{lem-lr-ind}, the submaps $M^L$ and $M^R$ are conditionally independent given the length of $P$.}
\end{center}
\end{figure}
 
\begin{lem} \label{lem-lr-ind}
Fix $\XDP \in \{\op{LDP} , \op{SDP}\}$ and $\el,r\in\BB N$. Let $M(\el,r)$ be a Boltzmann bipolar-oriented triangulation with left boundary length $\el$ and right boundary length $r$. 
Let $P$ be the leftmost $\XDP$ geodesic from the source of $M(\el,r)$ to the sink of $M(\el,r)$ (Definition~\ref{def-geodesic}). 
Let $M^L$ (resp.\ $M^R$) be the submap of $M(\el,r)$ consisting of the vertices, edges, and faces of $M(\el,r)$ that lie to the left (resp.\ right) of $P$, including the vertices and edges that lie on $P$. 
Then $M^L$ and $M^R$ are conditionally independent given $|P| = \XDP_{M(\el,r)}(\mathrm{source}, \mathrm{sink})$.
\end{lem}
\begin{proof}
Fix $s \geq \max\{\el,r\}$ (if $\XDP = \op{LDP}$) or $s \leq \min\{\el,r\}$ (if $\XDP = \op{SDP}$). We will show that $M^L$ and $M^R$ are conditionally independent given $\{|P| = s\}$. 
Let $\mcl M^L_{\el,s}$ be the set of bipolar-oriented triangulations with $\el$ edges on their left boundary and $s$ edges on their right boundary, with the further property that the right boundary is the unique $\XDP$ geodesic from the source to the sink.
Let $\mcl M^R_{s,r}$ be the set of bipolar-oriented triangulations with $s$ edges on their left boundary and $r$ edges on their right boundary, with the further property that the left boundary is an $\XDP$ geodesic from the source to the sink (not necessarily unique). The asymmetry in the definitions is because we consider the leftmost $\XDP$ geodesic. Finally, let $\mcl M_{\el,r,s}$ be the set of all bipolar-oriented triangulations in $\mcl M_{\el,r}$ with the property that $\XDP(\mathrm{source} , \mathrm{sink})= s$. 
 
Consider the map $\Phi$ from $M^L_{\el,s} \times M^R_{s,r}$ to $M_{\el,r }$ that takes $(\frk m^L , \frk m^R) \in M^L_{\el,s} \times M^R_{\el , s}$ to the map $\Phi(\frk m^L,\frk m^R) \in M_{\el,r}$ obtained by identifying the right boundary of $\frk m^L$ with the left boundary of $\frk m^R$. One can easily check that $\Phi$ is a bijection from $M^L_{\el,s} \times M^R_{  s, r}$ to $\mcl M_{\el,r,s}$. The inverse bijection is given by cutting an element of $\mcl M_{\el,r,s}$ along its leftmost $\XDP$ geodesic from the source to the sink. 
%We first argue that $\Phi$ maps any $(\frk m^L , \frk m^R) \in M^L_{\el,s} \times M^R_{\el , s}$ to an element of $\mcl M_{\el,r,s}$. Indeed, write $\frk m = \Phi(\frk m^L , \frk m^R)$ and let $\frk p$ be the length-$s$ directed path from the source to the sink in $\frk m$ corresponding to the right boundary of $\frk m^L$ (equivalently, the left boundary of $\frk m^R$). Also let $\frk p'$ be any directed path from the source to the sink in $\frk m$. Then $\frk p'$ can be written as the concatenation of its excursions to the left and right of $\frk p$ (we count segments of $\frk p'$ which trace $\frk p$ as being part of excursions to the right of $\frk p$. By the definition of $\mcl M^L_{\el,s}$ (resp.\ $\mcl M^R_{s,r}$), for each excursion of $\frk p'$ to the left (resp.\ right) of $\frk p$, the length of the excursion is strictly smaller than (resp.\ at most) the length of the segment of $\frk p$ between the endpoints of the excursion. From this, we get that the length of $\frk p'$ is at most the length of $\frk p$, so $\frk m\in \mcl M_{\el,r,s}$. Similarly in the SDP case.  

If $\frk m \in \mcl M_{\el,r,s}$ and $(\frk m^L , \frk m^R) = \Phi^{-1}(\frk m)$, then the edge set of $\frk m$ is the union of the edge sets of $\frk m^L$ and $\frk m^R$. This is a disjoint union, except that the $s$ edges along the leftmost LDP geodesic of $\frk m$ are counted twice, hence
\eqb \label{eqn-edgescount}
\# \mcl E(\frk m) %=s+(\# \mcl E(\frk m^L)-s)+(\# \mcl E(\frk m^R)-s)
=\# \mcl E(\frk m^L)+\# \mcl E(\frk m^R)-s.  
\eqe
From this, we obtain  
\alb
\BB P\left[ M(\el,r)  = \frk m \,\Big|\, |P| = s \right] 
 \stackrel{\eqref{eqn-boltzmann}}{=} \frac{3^{- \# \mcl E(\frk m) } }{C_{\el, r} \BB P[|P| = s]}    
 \stackrel{\eqref{eqn-edgescount}}{=}   \frac{3^s }{C_{\el, r} \BB P[|P| = s]}   3^{- \# \mcl E(\frk m^L) } 3^{- \#\mcl E(\frk m^R)}.
\ale
This is the product of a constant times something depending only on $\frk m^L$ times something depending only on $\frk m^R$, which gives the desired conditional independence. 
\end{proof}

\begin{remark} \label{remark-ind-general}
Exactly the same proof gives analogs of Lemma~\ref{lem-lr-ind} for bipolar-oriented planar maps with other face degree distributions. This includes the biased face degree distributions considered in~\cite[Remark 5]{kmsw-bipolar} (which are expected to be in the universality class of $\gamma$-LQG for $\gamma \in (0,\sqrt 2)$ with $\gamma$ depending on the particular model).
\end{remark}

\subsection{KMSW bijection for boundary-channeled bipolar-oriented triangulations}
\label{sec-kmsw-left}

Recall from Definition~\ref{def-boltzmann-lr} that for $\el , r \in\BB N$, we denote by $\mcl M_{\el,r}$ the set of bipolar-oriented triangulations with $\el$ left boundary edges and $r$ right boundary edges. In this section, we discuss a special subset of $\mcl M_{\el,r}$ that will be used in Section~\ref{sec-uihbot-sym} for the purpose of proving the symmetry property from Theorem~\ref{thm-busemann-property}. The reader can skip this section on a first read. 

\begin{defn} \label{def-reverse-map}
Let $\mcl M_{\el,r}^\bc$ be the set of bipolar-oriented triangulations $\frk m \in \mcl M_{\el,r}$ \jb{with simple boundary (also allowing the single-edge map)} with the following property: each vertex on the boundary of $\frk m$, except for the sink vertex, has no incoming edges other than boundary edges. Equivalently, every boundary vertex, other than the source and the sink, has exactly one incoming edge. We call the maps in $\mcl M_{\el,r}^\bc$ \textbf{boundary-channeled} bipolar-oriented triangulations with left boundary length $\el$ and right boundary length $r$.
\end{defn}

The reason for our interest in $\mcl M_{\el,r}^\bc$ is the following symmetry, which will be used in Section~\ref{sec-uihbot-sym} to prove Property~\ref{item-busemann-sym} of Theorem~\ref{thm-busemann-property}.

\begin{lem} \label{lem-bdy-reverse}
Let $\mcl M_{\el,r}^\bc$ be as above. For $\frk m \in \mcl M_{\el,r}^\bc$, let $\frk x_0,\dots,\frk x_r$ be the vertices on the right boundary of $\frk m$, enumerated in order so that $\frk x_0$ is the source and $\frk x_r$ is the sink. For $k\in \{1,\dots,r-1\}$, let $\Phi_k(\frk m)$ be the directed planar map obtained from $\frk m$ by reversing the orientation of the edges $(\frk x_0,\frk x_1),\dots,(\frk x_{k-1},\frk x_k)$. Then $\Phi_k$ is a bijection from $\mcl M_{\el,r}^\bc$ to $\mcl M_{\el+k,r-k}^\bc$.
\end{lem}

\begin{proof}
We claim that $\Phi_k(\frk m)$ is a bipolar-oriented map with source $\frk x_k$ and sink $\frk x_r$. Since $\Phi_k$ does not change the orientation of any non-boundary edge, \jb{preserves the simple boundary, and changes the left and right boundary lengths to $\el+k$ and $r-k$, respectively,} this implies that $\Phi_k$ maps $\mcl M_{\el,r}^\bc$ to $\mcl M_{\el+k,r-k}^\bc$. It is clear that $\Phi_k$ has an inverse map \jb{obtained by reversing the first $k$ left-boundary edges of the image map}, so this will prove the lemma.

It remains to prove the claim. Since $\Phi_k$ does not alter the orientations of any edges incident to the sink $\frk x_r$, there are no outgoing edges from $\frk x_r$ on $\Phi_k(\frk m)$.
Since $\frk x_k$ does not have any non-boundary incoming edges on $\frk m$ (by the definition of $\mcl M_{\el,r}^\bc$), it follows that $\frk x_k$ has no incoming edges at all on $\Phi_k(\frk m)$.
Every boundary vertex of $\Phi_k(\frk m)$ other than $\frk x_k$ and $\frk x_r$ has at least one incoming edge and at least one outgoing edge, namely the two boundary edges incident to it. Since $\Phi_k$ does not change the orientations of the non-boundary edges, every non-boundary vertex of \jb{$\Phi_k(\frk m)$} still has at least one incoming and at least one outgoing edge.

It remains to show that the orientation on $\Phi_k(\frk m)$ is acyclic. Suppose by way of contradiction that there are edges $e_1,\dots,e_m$ which form a directed cycle on $\Phi_k(\frk m)$. Since the orientation on $\frk m$ is acyclic, at least one of the edges $e_1,\dots,e_m$ must be one of the boundary edges whose orientation we changed. The boundary of $\Phi_k(\frk m)$ is \jb{simple and is} not a directed cycle, so at least one of $e_1,\dots,e_m$ must be a non-boundary edge. Moreover, none of these edges can be incident to the sink $\frk x_r$ since it only has incoming edges incident to it. Hence, by re-labeling, we can assume without loss of generality that $e_1$ is a non-boundary edge and $e_2$ is a boundary edge. But then $e_1$ is \jb{a} non-boundary edge whose terminal vertex is a boundary vertex different from $\frk x_r$. This contradicts the fact that $\frk m \in \mcl M_{\el,r}^\bc$.
\end{proof}

We now explain how elements of $\mcl M_{\el,1}^\bc$ can be obtained by applying the KMSW procedure to certain specific walks.

\begin{lem} \label{lem-kmsw-left}
Let $\el,n\in\BB N$. The KMSW procedure restricts to a bijection between the following two sets:
\begin{itemize}
\item Walks $\frk Z^{\bc}_{\el}$ with increments in $\{(1,-1),(-1,0),(0,1)\}$ and $n$ total steps, which start at $(\el-1,0)$ and end at $(0,0)$, and stay in $[0,\infty)\times[0,\infty)$.
\item Maps $\frk m'$ \jbb{obtained from a map $\frk m\in\mcl M_{\el,1}^\bc$ other than the single-edge map,} having $n-1$ non-boundary edges, i.e.\ \jbb{$N=n+\el$} total edges, \jbb{by declaring} all of its left boundary edges except for the left boundary edge whose terminal vertex is the sink vertex \jbb{to be} missing.
\end{itemize}
\jbb{Separately, the zero-step walk at $(0,0)$ corresponds to the single-edge map, whose left and right boundary lengths are both $1$.}
\end{lem}

\begin{proof}
Let $\widetilde{\frk{m}}$ be any bipolar-oriented triangulation. By the definition of the tree $T_{UL}(\widetilde{\frk{m}})$ (above Definition~\ref{defn:KMSW-inv}), we have that the rightmost branch of $T_{UL}(\widetilde{\frk{m}})$ coincides with the right boundary of $\widetilde{\frk{m}}$ if and only if every right boundary vertex of $\widetilde{\frk{m}}$, other than the source and the sink, has exactly one incoming edge (necessarily a boundary edge); see the right-hand side of Figure~\ref{fig-bipolar-boundaries}.

In particular, by Definition~\ref{defn:KMSW-inv} and Theorem~\ref{thm-kmsw}, given $\overline{\frk m}\in\mcl M_{1,r}^\bc$ with $N$ total edges, the corresponding KMSW encoding walk $\overline{\frk Z}$ is a walk with increments in $\{(1,-1),(-1,0),(0,1)\}$ and $N-1$ total steps, starting at $(0,0)$ and ending at $(0,-(r-1))$, staying in $[0,\infty)\times[-(r-1),\infty)$; with the additional constraint that \jbb{the first $r-1$ steps are $(1,-1)$-steps}. Removing the first $r-1$ steps from $\overline{\frk Z}$ and shifting the walk so that it ends at $(0,0)$, we obtain a walk $\overline{\frk Z}^{\bc}_{r}$ with increments in $\{(1,-1),(-1,0),(0,1)\}$ and $N-r$ total steps, starting at $(r-1,0)$ and ending at $(0,0)$, staying in $[0,\infty)\times[0,\infty)$.
Moreover, the triangulation $\overline{\frk m}$ with all its right boundary edges declared to be missing except for the edge having the sink vertex as the terminal vertex, coincides with the triangulation obtained from $\overline{\frk Z}^{\bc}_{r}$ by applying the KMSW procedure.

\jbb{When $\el=1$, we use the convention $\Phi_0=\mathrm{id}$.}
Now, given $\frk m\in\mcl M_{\el,1}^\bc$ \jbb{as in the second item, with $N=n+\el$ total edges,} we can apply the bijection \jbb{$\Phi_{\el-1}^{-1}$} from Lemma~\ref{lem-bdy-reverse}, to obtain a triangulation \jbb{$\Phi_{\el-1}^{-1}(\frk m)\in\mcl M_{1,\el}^\bc$}. Using the bijection in the last paragraph, we obtain an encoding walk $\frk Z^{\bc}_{\el}$ with increments in $\{(1,-1),(-1,0),(0,1)\}$ and $N-\el=n$ total steps, starting at $(\el-1,0)$ and ending at $(0,0)$, staying in $[0,\infty)\times[0,\infty)$. Moreover, the triangulation \jbb{$\Phi_{\el-1}^{-1}(\frk m)$} with all its right boundary edges declared to be missing except for the edge having the sink vertex as the terminal vertex, coincides with the triangulation $\frk m'$ obtained from $\frk Z^{\bc}_{\el}$ by applying the KMSW procedure. This gives the lemma statement.
\jbb{The exceptional zero-step correspondence follows directly from the initialization of the KMSW procedure with a single edge.}
\end{proof}

%% file: tex/uiqbot.tex
\subsection{Definitions and results}
\label{sec-uiqbot-def}

Throughout this section, we fix $\XDP \in \{\op{LDP} , \op{SDP}\}$. 

Let $\mcl Z=(\mcl L , \mcl R) : \BB N_0 \to \BB Z^2$ be a walk whose increments are i.i.d.\ uniform samples from the set of steps $\{( 1, -1), (-1,0) , (0, 1)\}$. 
For $z=(x,y)\in\BB Z^2$, we write $\BB P_z$ for the law of $\mcl Z$ starting from $z$.
Let $\wh{\mcl Z} = (\wh{\mcl L} , \wh{\mcl R})$ be the random walk obtained from $\mcl Z$ by conditioning $\mcl L$ to stay non-negative for all time in the sense of~\cite{bd-conditioning}. To describe this process more explicitly, for $z\in \BB N_0 \times \BB Z$, we write $\wh{\BB P}_z$ for the law of $\wh{\mcl Z}$ starting from $z$. 
The process $\wh{\mcl Z}$ is a Markov process, and for each $n \geq 0$, the $\wh{\BB P}_z$-law of $\wh{\mcl Z}|_{[0,n]}$ is obtained by weighting the $\BB P_z$-law of $\mcl Z|_{[0,n]}$ by  
\eqb \label{eqn-cond-walk-rn}
(x+1)^{-1} ( \mcl L(n) + 1)  \BB 1\{\mcl L(j) \geq 0 ,\, \forall j\in [0,n]\cap\BB Z\}  . 
\eqe 
In this section, we will study the following infinite bipolar-oriented triangulation. 

\begin{defn} \label{def-uiqbot}
Assume that $\wh{\mcl Z}$ starts from $(0,0)$, i.e.\ we are working under $\wh{\BB P}_{(0,0)}$. Let $\wh M$ be the infinite bipolar-oriented triangulation obtained from $\wh{\mcl Z}$ via the KMSW procedure (Definition~\ref{def-kmsw} and Remark~\ref{remark-kmsw-infinite}). We call $\wh M$ the \textbf{uniform infinite quarter-plane bipolar-oriented triangulation (UIQBOT)}. The root edge of $\wh M$ is the edge $\wh \lambda_0$ where the KMSW procedure starts.
\end{defn}

The map $\wh M$ is equipped with an acyclic orientation with a unique source, which is the initial vertex of the root edge of $\wh M$, and no sink (the sink is ``at $\infty$'').
We stress that the map $\wh M$ has no missing edges (as opposed to $M_{0,\infty}$); this follows from Assertions~\ref{item-lower-left}~and~\ref{item-upper-right} of Lemma~\ref{lem-kmsw-bdy}.
Analogously to Definition~\ref{defn:bound-edges},
we define the \textbf{left} (resp.\ \textbf{right}) \textbf{boundary} of $\wh M$ to be the path on the boundary of the external face of $\wh M$ from the source vertex to $\infty$ with the external face immediately to the left (resp.\ right). 
For later convenience, we note that the left (resp.\ right) boundary edges of $\wh M$ are upper-left (resp.\ lower-right) boundary edges according to Definition~\ref{defn:bound-edges}.

The reason for the name ``quarter plane'' is that we will eventually show that the UIQBOT has the same law as the ``first quadrant'' in the UIBOT bounded by the rightmost directed path from the terminal vertex of the root edge (i.e.\ the right boundary of $M_{0,\infty}$) and the leftmost directed path from the initial vertex of the root edge (Proposition~\ref{prop-future-map-law}). 

We expect that $\wh M$ arises as the Benjamini-Schramm local limit around the source vertex of Boltzmann bipolar-oriented triangulations with given left and right boundary lengths (see~\eqref{eqn-boltzmann}) as the left and right boundary lengths approach $\infty$. We will not prove this here, but see Lemma~\ref{lem-past-map}.

The first main result of this subsection is that the UIQBOT a.s.\ has infinitely many cut vertices, defined as follows. 

\begin{defn} \label{def-cut-vertex}
A vertex $v$ of $\wh M$ is a \textbf{cut vertex} of $\wh M$ if $v$ lies on both the left and right boundaries of $\wh M$. Equivalently, $v$ is a cut vertex if and only if either $v$ is the source vertex or removing $v$ disconnects $\wh M$.  
\end{defn}

\begin{prop} \label{prop-cut-vertex} 
Almost surely, the UIQBOT $\wh M$ has infinitely many cut vertices. 
\end{prop}

We will prove Proposition~\ref{prop-cut-vertex} in Section~\ref{sec-cut-vertex} using estimates for the walk $\wh{\mcl Z}$. 
Proposition~\ref{prop-cut-vertex} will play an essential role in the proof of the existence of the Busemann function (Theorem~\ref{thm-busemann}). The reason is that the existence of cut vertices forces any two sufficiently long directed paths in $\wh M$ to have a vertex in common. Taking these paths to be $\XDP$ geodesics will allow us to verify the condition in~\eqref{eqn-busemann-lim} of Theorem~\ref{thm-busemann}.

\begin{remark}
Proposition~\ref{prop-cut-vertex} may be surprising in light of known facts in the continuum. 
Indeed, from mating of trees theory (see, e.g.,~\cite[Theorems 1.3 and 3.5]{ag-disk}) the map $\wh M$ is the discrete analog of a weight $\gamma^2/2 = 2/3$ LQG wedge, with $\gamma=\sqrt{4/3}$. The walk $\wh{\mcl Z}$ is the discrete analog of the left/right boundary length process for an SLE$_{12}(2;0)$ curve on this LQG wedge. Weight $\gamma^2/2$ is exactly the critical weight for an LQG wedge to have cut points: an LQG wedge of weight $W \geq \gamma^2/2$  is homeomorphic to the half-plane, whereas an LQG wedge of weight $W < \gamma^2/2$ is homeomorphic to a countable string of domains that are homeomorphic to a disk with two marked points, glued together end-to-end~\cite[Sections 4.2 and 4.4]{wedges}. Intuitively, the reason why our infinite planar map $\wh M$ has cut vertices whereas its continuum counterpart is homeomorphic to the half-plane is that the cut vertices of $\wh M$ are very sparse, so they disappear when one passes to the scaling limit.
\end{remark}

As an easy consequence of Proposition~\ref{prop-cut-vertex}, we get that infinite $\XDP$ geodesics in the UIQBOT exist.
The following definition generalizes Definition~\ref{def-geodesic} to infinite planar maps. 

\begin{defn} \label{def-infinite-geo}
Let $\frk{m}$ be an infinite planar map equipped with an acyclic orientation. A directed path $\frk{p} : \BB N \to \mcl E(\frk{m})$ started from $x\in\mcl V(\frk{m})$ is an \textbf{$\XDP$ geodesic from $x$ to $\infty$} if for any integers $1 \leq a \leq b < \infty$, the path $\frk{p}|_{[a,b]}$ is the longest (if $\XDP = \op{LDP}$) or shortest (if $\XDP = \op{SDP}$) directed path between its starting and terminal vertices. We say that $\frk{p}$ is a \textbf{leftmost $\XDP$ geodesic from $x$ to $\infty$} if any other infinite $\XDP$ geodesic from $x$ to $\infty$ lies (weakly) to the right of $\frk{p}$. That is, any such $\XDP$ geodesic is contained in the submap of $\frk{m}$ lying to the right of $\frk{p}$ and to the left of the rightmost directed path started from $x$ (including the vertices and edges on these paths). 
\end{defn}

\begin{lem} \label{lem-geo-exist}
Let $\wh M$ be the UIQBOT as in Definition~\ref{def-uiqbot}. \jbb{Almost surely,} for each $x\in \mcl V(\wh M)$, there is a (unique) leftmost \jbb{$\XDP$} geodesic from $x$ to $\infty$ in $\wh M$. 
\end{lem}

Lemma~\ref{lem-geo-exist} will be proven at the end of Section~\ref{sec-cut-vertex}.
The second main result of this section is the following proposition, which extends Lemma~\ref{lem-lr-ind} to the infinite map $\wh M$. 

\begin{prop} \label{prop-thin-ind}
Let $\wh P$ be the leftmost $\XDP$ geodesic from the source vertex to $\infty$ in the UIQBOT $\wh M$, as provided by Lemma~\ref{lem-geo-exist}. 
Let $\wh M^L$ and $\wh M^R$ be the submaps of $\wh M$ lying (weakly) to the left and right of $\wh P$, respectively. 
Then $\wh M^L$ and $\wh M^R$ are independent. 
\end{prop}

Proposition~\ref{prop-thin-ind} will be proven in Section~\ref{sec-thin-ind}.

\subsection{Existence of cut vertices in the UIQBOT}
\label{sec-cut-vertex}
  
In this subsection, we will prove Proposition~\ref{prop-cut-vertex}. 
For $n\in\BB N_0$, let $\wh\lambda_n$ be the time-$n$ active edge of the KMSW procedure for $\wh M $ (Definition~\ref{def-kmsw}). 
Just below, we will prove the following statement using basic observations about the KMSW procedure. 

\begin{lem} \label{lem-cut-equiv}
For $n\in \BB N_0$, let $\wh\lambda^-_n$ be the initial vertex of $\wh\lambda_n$. The following are equivalent. 
\begin{itemize}
\item $\wh\lambda^-_n$ is a cut vertex for $\wh M$ and $\wh\lambda^-_j \not= \wh\lambda^-_n$ for each $j\leq n-1$. 
\item The event $A_n\cap B_n$ occurs, where 
\eqb  \label{eqn-walk-bdy-event}
A_n := \left\{ \wh{\mcl R}(n) \leq \wh{\mcl R}(j) -1 ,\: \forall j \leq  n-1 \right\} \quad \text{and} \quad 
B_n := \left\{ \wh{\mcl L}(j) \geq \wh{\mcl L}(n)   ,\: \forall j \geq n   \right\} .
\eqe  
\end{itemize}
\end{lem}  

\begin{proof}
Assume that $A_n\cap B_n$ occurs. By Assertion~\ref{item-lower-right} of Lemma~\ref{lem-kmsw-bdy} (recall also Remark~\ref{remark-kmsw-infinite}), the occurrence of $A_n$ implies that $\wh\lambda^-_n$ lies on the right boundary of $\wh M$ and $\wh\lambda^-_j \not=\wh\lambda^-_n$ for each $j\leq n-1$. Furthermore, by Assertion~\ref{item-upper-left} of Lemma~\ref{lem-kmsw-bdy}, 
the occurrence of $B_n$ implies that $\wh\lambda^-_n$ lies on the left boundary of $\wh M$. The converse is checked similarly.
\end{proof}
 
By Lemma~\ref{lem-cut-equiv}, Proposition~\ref{prop-cut-vertex} is equivalent to the statement that a.s.\ there are infinitely many $n\in\BB N$ such that $A_n\cap B_n$ occurs. To prove this statement, we will first observe that the set of such $n$ is a discrete-time renewal process (Lemma~\ref{lem-renewal}). We will then show using basic estimates for random walk that $\sum_{n=1}^\infty \BB P[A_n\cap B_n] = \infty$ (Lemma~\ref{lem-walk-bdy-sum}). This implies that a.s.\ $\sum_{n=1}^\infty \BB 1_{A_n\cap B_n} = \infty$ by a standard argument for discrete-time renewal processes.

\begin{lem} \label{lem-renewal} 
Let $T_0 = 0$ and for $m \in\BB N$, let $T_m$ be the $m$th smallest time $n \in \BB N $ for which $A_n\cap B_n$ occurs, or $T_m = \infty$ if there are fewer than $m$ such times. Equivalently, by Lemma~\ref{lem-cut-equiv}, $T_m$ is the $m$th smallest $n\in\BB N$ for which $\wh\lambda^-_n$ is a cut vertex of $\wh M$ which is visited by $\wh\lambda^-$ for the first time at time $n$. 
On the event $\{T_m < \infty\}$, the conditional law of $(\wh{\mcl Z} (\cdot +T_m)- \wh{\mcl Z} (T_m))|_{[0,\infty)}$ given $\wh{\mcl Z} |_{[0,T_m]}$ is the same as the law $\wh{\BB P}_{(0,0)}$ of $\wh{\mcl Z}$. 
\end{lem}
\begin{proof}
We have $T_m = n$ if and only if the following is true:
\begin{itemize}
\item $\wh{\mcl R} $ attains a running minimum at time $n$, i.e., $A_n$ occurs. 
\item There are exactly $m$ times $j \in [1,n] \cap\BB Z$ such that $\wh{\mcl R} $ attains a running minimum at time $j$ and $\wh{\mcl L} (i) \geq \wh{\mcl L}(j)$ for each $i \in [j,n]\cap\BB Z$.
\item $\wh{\mcl L} (j) \geq \wh{\mcl L} (n)  $ for each $j\geq n $, i.e., $B_n$ occurs.  
\end{itemize}
The first two itemized conditions depend only on $\wh{\mcl Z} |_{[0,n]}$. Hence, conditioning on $\wh{\mcl Z} |_{[0,n]}$ and $\{T_m = n\}$ is the same as conditioning on $\wh{\mcl Z} |_{[0,n]}$ and the event $B_n$. By the Markov property, the conditional law of $(\wh{\mcl Z} (\cdot + n)  - \wh{\mcl Z} (n) )|_{[0,\infty)}$ given $\wh{\mcl Z} |_{[0,n]}$ and $\{T_m = n\}$ is that of a random walk started at $(0,0)$ with the same increment distribution as $\mcl Z$ conditioned so that its first component stays non-negative. That is, its conditional law is the same as the law $\wh{\BB P}_{(0,0)}$ of $\wh{\mcl Z} $. 
\end{proof}
%If we start at $\mcl L(0) = x$ and we want to condition $\mcl L$ to stay above $x$, then by Lemma~\ref{lem-walk-pos-prob} with $s = \mcl L(n) - x$ and Bayes' rule, we need to re-weight the law $(x+1)^{-1} (\mcl L(n) + 1) \BB 1\{\mcl L(j)\geq 0 ,\forall j \in [0,n]\} \, d\mcl L$ by $(x+1) (\mcl L(n) - x + 1) / (\mcl L(n) + 1) \BB 1\{\mcl L(n) \geq x , \forall j \in [0,n]\}$. This gives the law $(\mcl L(n)-x+1) \BB 1\{\mcl L(n) - x\geq 0 ,\, \forall j\in [0,n]\} \,d\mcl L$. 

As noted above, to show that a.s.\ there exist infinitely many $n$ such that $A_n\cap B_n$ occurs, we want to show that $\sum_{n=1}^\infty \BB P[A_n\cap B_n] = \infty$. To accomplish this, we will prove a lower bound for the conditional probability of $B_n$ given $A_n$ (Lemma~\ref{lem-walk-pos-prob}) and then combine this with a lower bound for the number of values of $n$ for which $A_n$ occurs (Lemma~\ref{lem-walk-bdy-sum}).

\begin{lem} \label{lem-walk-pos-prob}
Recall that $\wh{\BB P}_{(x,y)}$ denotes the law of the (conditioned) walk $\wh{\mcl Z}$ started from $(x,y)$.
For $x\geq 0$ and $s \in [0,x]\cap\BB Z$,
\eqbn
\wh{\BB P}_{(x,y)}\left[ \wh{\mcl L}(j)  \geq x  - s , \, \forall j \geq 0 \right] 
= \frac{ s+1}{x+1}. 
\eqen 
\end{lem}
\begin{proof}
For $k\in \BB N$, let
\eqbn
\wh\sigma_k = \min\left\{n\in\BB N : \wh{\mcl L}(n) \geq k \right\}, 
\eqen
and similarly define $\sigma_k$ with the unconditioned walk $\mcl L$ in place of $\wh{\mcl L}$. 
By the definition of $\wh{\mcl Z}$ from~\eqref{eqn-cond-walk-rn}, for $k\geq x+1$, 
\alb
\wh{\BB P}_{(x,y)}\left[ \wh{\mcl L}(j)  \geq x - s, \, \forall j \in [0,\wh\sigma_k] \cap \BB Z \right] 
&= \frac{1}{x+1} \BB E_{(x,y)}\left[  ( \mcl L(\sigma_k) + 1)  \BB 1\left\{  {\mcl L}(j)  \geq x - s , \, \forall j \in [0,\sigma_k] \cap \BB Z \right\} \right]  \\
&= \frac{k+1}{x + 1} \BB P_{(x,y)}\left[  {\mcl L}(j)  \geq x - s , \, \forall j \in [0,\sigma_k] \cap \BB Z  \right].  
\ale 
Since $\mcl L$ is a lazy random walk, by the standard Gambler's ruin formula, the probability in the last line above is equal to
$(s+1) / (k-x+s+1)$.
Sending $k\to\infty$ and using~\cite[Theorem 1]{bd-conditioning} concludes the proof. 
\end{proof}

We will need the following fact, which is a special case of the much more general result~\cite[Theorem 2]{dw-limit}. 

\begin{lem}[\cite{dw-limit}] \label{lem-cond-walk-bm}
Under $\wh{\BB P}_{(0,0)}$, the re-scaled walk $\{n^{-1/2} \wh{\mcl Z}(\lfloor t n \rfloor) \}_{t \geq 0}$ converges in law with respect to the topology of uniform convergence on compact subsets of $[0,\infty)$. The limiting process is $\sqrt{2/3}$ times a 2d Brownian motion of correlation $-1/2$ started from $(0,0)$ and conditioned to stay in the right half-plane, i.e., it is described as follows. Let $\mcl A$ be a 3d Bessel process and let $\mcl B$ be an independent standard linear Brownian motion.
Then consider the process $\sqrt{2/3} (\mcl A , \frac{\sqrt3}{2} \mcl B - \frac12 \mcl A)$. 
\end{lem}

\begin{proof}
Apply~\cite[Theorem 2]{dw-limit} with the cone equal to the half-plane $\{(x,y) \in \BB R^2 : x \geq 0\}$, to the random walk obtained by applying a linear transformation to make the coordinates of $\mcl Z$ uncorrelated. The constant $\sqrt{2/3}$ appears because the variance of each coordinate of $\mcl Z$ is $2/3$.
\end{proof}

\begin{lem} \label{lem-walk-bdy-sum} 
With $A_n$ and $B_n$ as in~\eqref{eqn-walk-bdy-event},
\eqbn
\sum_{n=1}^\infty \wh{\BB P}_{(0,0)} [A_n\cap B_n] = \infty. 
\eqen
\end{lem}
\begin{proof}
For $n\in\BB N$, let 
\eqb  \label{eqn-running-min-reg}
F_n := A_n \cap \left\{ \wh{\mcl L} (n) \in \left[ \frac{1}{100} n^{1/2} , 100 n^{1/2} \right]  \right\} .
\eqe 
In words, $F_n$ is the event that $\wh{\mcl R}$ attains a running minimum at time $n$ and $\wh{\mcl L} (n)$ is comparable to $n^{1/2}$. 

By the Markov property, if we condition on $\wh{\mcl Z}|_{[0,n]}$, then the conditional law of 
$\wh{\mcl Z}(\cdot + n)  |_{[0,\infty)}$ is $\wh{\BB P}_{\wh{\mcl Z}(n)}$.  By Lemma~\ref{lem-walk-pos-prob} applied with $s=0$, the conditional probability given $\wh{\mcl Z}|_{[0,n]}$ that the event $B_n$ occurs is $1/ (\wh{\mcl L}(n)+1) $.  
Since $\wh{\mcl L} (n) \leq 100  n^{ 1/2}$ on $F_n$,
\eqb \label{eqn-pos-given-reg}
\wh{\BB P}_{(0,0)}\left[ B_n \,|\, F_n   \right] \geq \frac{1}{101} n^{-1/2} .
\eqe

We will now argue that there are many times $n$ for which $\wh{\BB P}_{(0,0)} [F_n]$ is bounded below.   
For $N\in\BB N$, let $E_N$ be the event that the following is true:
\begin{enumerate}
\item $\wh{\mcl L} (n) \in \left[ \frac{1}{100} n^{1/2} , 100 n^{1/2} \right]$ for each $n\in [N/2,N] \cap\BB Z$;
\item $\min_{n \in [0,N/2]\cap\BB Z} \wh{\mcl R} (n)  \geq - N^{1/2}$;
\item $\wh{\mcl R} (N) \leq -2 N^{1/2}$. 
\end{enumerate}
\jbb{For every continuous path $f:[0,1]\to[0,\infty)\times\mathbb R$ with $f(0)=(0,0)$, and every $\varepsilon>0$, the limiting process in Lemma~\ref{lem-cond-walk-bm} has positive probability of lying within $\varepsilon$ of $f$ in uniform distance.}
By Lemma~\ref{lem-cond-walk-bm}, there is a universal constant $p\in (0,1)$ such that $\wh{\BB P}_{(0,0)} [E_N] \geq p$ for each $N \geq 100$, say (we can decrease $p$ to deal with finitely many small values of $N$).
If $E_N$ occurs, then there are at least $N^{1/2}$ times $n\in [N/2,N]$ at which $\wh{\mcl R} $ attains a running minimum. By~\eqref{eqn-running-min-reg} and the definition of $E_N$, the event $F_n$ occurs for each of these times $n$. Consequently,
\eqb  \label{eqn-min-count}
p N^{1/2} 
\leq \wh{\BB P}_{(0,0)} [E_N] N^{1/2} 
\leq \wh{\BB E}_{(0,0)} \left[ \sum_{n=\lceil N/2 \rceil }^N  \BB 1_{F_n} \right] 
=  \sum_{n=\lceil N/2 \rceil }^N \wh{\BB P}_{(0,0)} \left[  F_n \right]  .
\eqe 

Recall that $F_n\subset A_n$. 
By combining~\eqref{eqn-pos-given-reg} and~\eqref{eqn-min-count}, we find that there is a universal constant $c > 0$ such that for every $N \geq 100$, 
\eqbn
 c \leq \sum_{n=\lceil N/2 \rceil }^N \wh{\BB P}_{(0,0)} \left[ B_n\,|\, F_n \right]\, \wh{\BB P}_{(0,0)}[F_n] 
 \leq  \sum_{n=\lceil N/2 \rceil }^N \wh{\BB P}_{(0,0)}\left[ A_n \cap B_n \right]  . 
\eqen
Applying this with $N = 2^k$ and summing over $k$ concludes the proof. 
\end{proof}

\begin{proof}[Proof of Proposition~\ref{prop-cut-vertex}] 
Let $S $ be the set of $n\in\BB N_0$ for which $A_n \cap B_n$ occurs. By Lemma~\ref{lem-cut-equiv}, it suffices to show that $\wh{\BB P}_{(0,0)}$-a.s.\ $\# S = \infty$.  
With $T_m$ as in Lemma~\ref{lem-renewal}, 
\eqbn
\# S = \min\left\{ m \in \BB N : T_m = \infty \right\}. 
\eqen
By Lemma~\ref{lem-renewal}, if $\wh{\BB P}_{(0,0)}[\# S < \infty] > 0$, then $\wh{\BB P}_{(0,0)}[T_1 = \infty] > 0$ and for each $m\in\BB N$, on the event $\{T_{m-1} < \infty\}$, 
\eqb  \label{eqn-renewal-geometric}
\wh{\BB P}_{(0,0)}\left[ T_m = \infty \,\middle|\, \mcl Z|_{[0,T_{m-1}]} \right]  = \wh{\BB P}_{(0,0)}[T_1 = \infty] > 0 .
\eqe 
By~\eqref{eqn-renewal-geometric}, if $\wh{\BB P}_{(0,0)}[\# S < \infty] > 0$, then $\# S$ has a geometric distribution with positive success probability, so $\wh{\BB E}_{(0,0)}[\# S] < \infty$. 
But, Lemma~\ref{lem-walk-bdy-sum} implies that $\wh{\BB E}_{(0,0)}[\# S] = \infty$. 
Hence we conclude that  $\wh{\BB P}_{(0,0)}[\# S < \infty] = 0$. 
\end{proof}

We conclude this section by providing the proof of Lemma~\ref{lem-geo-exist} regarding the existence of infinite geodesics in $\wh M$.

\begin{proof}[Proof of Lemma~\ref{lem-geo-exist}]
Define the times $T_m$ as in Lemma~\ref{lem-renewal}, so that $\wh\lambda^-_{T_m}$ is the $m$th cut vertex of $\wh M$ (we know that $T_m < \infty$ a.s.\ for every $m \in\BB N$ by Proposition~\ref{prop-cut-vertex}). 

Fix $x\in \mcl V(\wh M)$  and consider the infinite rightmost directed path from $x$ in $\wh M$. This path has to visit a cut vertex. Let $m_0$ be the smallest $m\in\BB N$ for which this path visits the cut vertex $\wh\lambda^-_{T_m}$.
%Let $m_0$ be the smallest $m\in\BB N$ for which $x$ can be joined to $\wh\lambda^-_{T_m}$ by a directed path.
For $m\geq m_0$, let $\wh P_m$ be the leftmost $\XDP$ geodesic from $x$ to $\wh\lambda^-_{T_m}$ (Definition~\ref{def-geodesic}).   By the definition of a cut vertex, for $m' \geq m \geq m_0$, the path $\wh P_{m'}$ must visit the vertex $\wh\lambda^-_{T_m}$. By the uniqueness of leftmost $\XDP$ geodesics between finite points, the portion of $\wh P_{m'}$ up to the time when it hits $\wh\lambda^-_{T_m}$ coincides with $\wh P_m$. Hence we can define a leftmost $\XDP$ geodesic $\wh P $ from $x$ to $\infty$ by declaring that $\wh P $ coincides with $\wh P_m$ until the first time when it hits $\wh\lambda^-_{T_m}$, for each $m\geq m_0$. This leftmost $\XDP$ geodesic is unique since the leftmost $\XDP$ geodesic from $x$ to $\wh\lambda^-_{T_m}$ is unique for each $m\geq m_0$. 
\end{proof}

\subsection{Independence in the UIQBOT}
\label{sec-thin-ind}

In this subsection, we will only consider walks started from $(0,0)$, so we just write $\BB P = \wh{\BB P}_{(0,0)}$. 
The idea of the proof of Proposition~\ref{prop-thin-ind} is as follows. Let $\wh P$ be the leftmost $\XDP$ geodesic from the source vertex to $\infty$ in the UIQBOT $\wh M$, as provided by Lemma~\ref{lem-geo-exist}. Using the existence of cut vertices of $\wh M$ (Proposition~\ref{prop-cut-vertex}), we can find arbitrarily large finite submaps of $\wh M$ that have the law of a Boltzmann bipolar-oriented triangulation conditional on their left and right boundary lengths (Lemma~\ref{lem-past-map}). By Lemma~\ref{lem-lr-ind}, the leftmost $\XDP$ geodesic $\wh P$ of $\wh M$ divides each of these submaps into two pieces that are conditionally independent given the length of the portion of $\wh P$ in the submap (Lemma~\ref{lem-finite-ind}). It remains to send the size of the submap to $\infty$ and get rid of the conditioning. We do this using a combination of the backward martingale convergence theorem and the Hewitt-Savage zero-one law (the latter is applied to the i.i.d.\ submaps of $\wh M$ between its successive cut vertices). 

Define the times $T_m$ as in Lemma~\ref{lem-renewal}, so that $\wh\lambda^-_{T_m}$ is the $m$th cut vertex of $\wh M$. For $k \in\BB N_0$, let
\eqb \label{eqn-R-inf}
\tau_k^R = \min\left\{ n\in\jbb{\BB N_0} :
\wh{\mcl R}(n) = -k \right\},
\eqe
so that $\wh\lambda_{\tau_k^R}$ is the $k+1$th edge on the
right boundary of $\wh M$ thanks to Assertion~\ref{item-lower-right}
in Lemma~\ref{lem-kmsw-bdy}.
For $k\in\BB N_0$, let
\eqb \label{eqn-last-cut}
m_k := \max\left\{ m\jbb{\in\BB N_0} : T_m \leq \tau_k^R \right\}.
\eqe
That is, $\wh\lambda^-_{T_{m_k}}$ is the last cut vertex along
the right boundary of $\wh M$ which comes before
$\wh\lambda_{\tau_k^R}$ in directed order.
\jbb{On the event $\{T_{m_k}>0\}$, let}
\[
\wh M(T_{m_k}) := \wh M_{0,T_{m_k}-1}
\]
be the submap of $\wh M$ obtained by applying the KMSW procedure
to $\wh Z|_{[0,T_{m_k}-1]}$.
Then $\wh M(T_{m_k})$ is a bipolar-oriented triangulation with
the same source as $\wh M$ and sink equal to
$\wh\lambda^-_{T_{m_k}}$.
\jbb{On the event $\{T_{m_k}=0\}$, we instead define $\wh M(0)$
to be the pointed map consisting only of the vertex
$\wh\lambda^-_0$, with no edges; its source and sink both
coincide with $\wh\lambda^-_0$.}
See Figure~\ref{fig-past-map} for an illustration.

\begin{figure}[ht!]
\begin{center}
\includegraphics[width=0.6\textwidth]{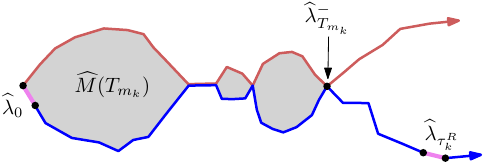}  
\caption{\label{fig-past-map}  
Illustration of the map $\wh M(T_{m_k})$, in grey, defined just after~\eqref{eqn-last-cut}. We show in Lemma~\ref{lem-past-map} that the conditional law of $\wh M(T_{m_k})$ given its left and right boundary lengths is that of a Boltzmann bipolar-oriented triangulation with given left and right boundary lengths (Definition~\ref{def-boltzmann-lr}). 
} 
\end{center}
\end{figure}

\begin{lem} \label{lem-past-map}
Let $k \in \BB N$ and let $\wh M(T_{m_k})$ be the map defined just above. If we condition on
$\wh{\mcl Z}(\cdot + T_{m_k})|_{[0,\infty)}$
(which in particular determines $\wh{\mcl Z}(T_{m_k})$),
then \jbb{on the event $\{T_{m_k}>0\}$,} the conditional law of
$\wh{\mcl Z}|_{[0,T_{m_k}]}$ is that of a random walk with the
law of $\mcl Z$ started at $(0,0)$, conditioned to first hit
$\BB R \times \{\wh{\mcl R}(T_{m_k})\}$ at
$(\wh{\mcl L}(T_{m_k}),\wh{\mcl R}(T_{m_k}))$ and to stay in
$[0,\infty)\times[\wh{\mcl R}(T_{m_k}),\infty)$ until this
hitting time.
Under the same conditioning, \jbb{on the event $\{T_{m_k}>0\}$,}
the conditional law of $\wh M(T_{m_k})$ is that of a Boltzmann
bipolar-oriented triangulation with left and right boundary
lengths $\wh{\mcl L}(T_{m_k})$ and
$-\wh{\mcl R}(T_{m_k})$, respectively
(Definition~\ref{def-boltzmann-lr}).
\jbb{On the event $\{T_{m_k}=0\}$, the stopped walk consists
only of $(0,0)$ at time zero, and $\wh M(T_{m_k})$ consists
only of its source vertex.}
\end{lem}

\begin{proof}  
Define the times $\tau_k^R$ as in~\eqref{eqn-R-inf}. 
By the definition of $T_m$ from Lemma~\ref{lem-renewal}, $\{T_m\}_{m\in\BB N_0} \subset \{\tau_k^R\}_{k\in\BB N_0}$. 
Fix $\el , r \geq 1$ with $r \leq k$. Define the event 
\eqbn
E_{\el,r}=E_{\el,r}(k) := \left\{\wh{\mcl Z}(T_{m_k}) = (\el  , -r)\right\} = \left\{T_{m_k} = \tau_r^R ,\, \wh{\mcl L}(\tau_r^R)  =\el  \right\} .
\eqen
By unpacking the definitions of $\tau_r^R$ and $T_{m_k}$ and noting that $T_{m_k} \leq \tau_k^R$ by definition, we see that $E_{\el,r}$ is the same as the event that the following is true: 
\begin{itemize}
\item $\wh{\mcl L}(\tau_r^R) = \el $;
\item $\wh{\mcl L}(j) \geq \el $ for each $j \geq \tau_r^R$ (combined with the first item, this is the event $B_{\tau_r^R}$ in \eqref{eqn-walk-bdy-event} and, by the definitions of $\tau_r^R$ and $T_m$, this ensures that $\tau_r^R = T_m$ for some $m\in \BB N_0$); 
\item there are no numbers $s\in [r+1,k ]\cap\BB Z$ such that $\wh{\mcl L}(n) \geq \wh{\mcl L}(\tau_s^R)$ for each $n \geq \tau_s^R  $ (by~\eqref{eqn-last-cut}, this ensures that $T_{m_k} = \tau_r^R$, instead of $T_{m_k} = \tau_s^R$ for some  $s\in[r+1,k ] \cap\BB Z$ ). 
\end{itemize}
The first itemized condition depends only on $\wh{\mcl Z}(\tau_r^R)$ and the other two itemized conditions depend only on $\wh{\mcl Z}(\cdot + \tau_r^R)|_{[0 ,\infty)}$. By the strong Markov property, $\wh{\mcl Z}|_{[0, \tau_r^R]}$ is conditionally independent from $\wh{\mcl Z}(\cdot + \tau_r^R)|_{[0 ,\infty)}$ given $\wh{\mcl Z}(\tau_r^R)$. Hence, on the event $E_{\el,r}$, the conditional law of $\wh{\mcl Z}|_{[0,\tau_r^R]} =  \wh{\mcl Z}|_{[0,T_{m_k}]}$ given $\wh{\mcl Z}(\cdot + \tau_r^R)|_{[0 ,\infty)}$ is the same as the conditional law of $\wh{\mcl Z}|_{[0,\tau_r^R]}$ given only $\{\wh{\mcl L}(\tau_r^R) = \el\}$. By the definition~\eqref{eqn-R-inf} of $\tau_r^R$, we get the desired description of the conditional law of $\wh{\mcl Z}|_{[0,T_{m_k}]}$. The description of the conditional law of $\wh M(T_{m_k}) $ then follows by applying Lemma~\ref{lem-kmsw-boltzmann} and conditioning on the left boundary length. 
\end{proof}

By Definition~\ref{def-cut-vertex} of a cut vertex, the leftmost $\XDP$ geodesic  $\wh P$ from the source vertex to $\infty$ in $\wh M$ must pass through the cut vertex $\wh\lambda^-_{T_m}$ for each $m\in\BB N$.  
Note that $\wh\lambda^-_{T_0} = \wh\lambda^-_0$ is the source vertex of $\wh M$. 
For $k \geq 1$, define the $\sigma$-algebra  
\eqb \label{eqn-cut-filtration} 
\mcl F_k := \sigma\left(   \wh{\mcl Z}( \cdot + T_{m_k})|_{[0 ,\infty)} ,\, \XDP_{\wh M}\left(\wh\lambda^-_0 ,  \wh\lambda^-_{T_{m_k} } \right) \right). 
\eqe

\begin{lem} \label{lem-finite-ind}
Let $\wh M^L(T_{m_k})$ and $\wh M^R(T_{m_k})$ be the submaps of $\wh M(T_{m_k})$ lying to the left and right of $\wh P$, respectively, both including the vertices and edges which lie on $\wh P \cap \wh M(T_{m_k})$. Then $\wh M^L(T_{m_k})$ and $\wh M^R(T_{m_k})$ are conditionally independent given the $\sigma$-algebra $\mcl F_k$ of~\eqref{eqn-cut-filtration}.
\end{lem}
\begin{proof}
\jbb{On the event $\{T_{m_k}=0\}$, both $\wh M^L(T_{m_k})$ and $\wh M^R(T_{m_k})$ consist only of the source vertex, so their conditional independence is immediate. Henceforth assume that $T_{m_k}>0$.}
Recall that the source of $\wh M(T_{m_k})$ is the same as the source $\wh\lambda^-_0$ of $\wh M$, and the sink of $\wh M(T_{m_k})$ is the vertex $\wh\lambda^-_{T_{m_k}}$.  
Since $\wh P$ must visit the cut vertex $\wh\lambda^-_{T_{m_k} }$, the portion of $\wh P$ which lies in $\wh M(T_{m_k})$ is the leftmost $\XDP$ geodesic from the source to the sink of $\wh M(T_{m_k})$, and the length of this portion of $\wh P$ is $\XDP_{\wh M}(\wh\lambda^-_0 ,  \wh\lambda^-_{T_{m_k} } )$.
The lemma then follows by combining the description of the conditional law of $\wh M(T_{m_k})$ given $ \wh{\mcl Z}( \cdot + T_{m_k})|_{[0 ,\infty)}$ from Lemma~\ref{lem-past-map} with the conditional independence statement of Lemma~\ref{lem-lr-ind}.  
\end{proof}

\begin{lem} \label{lem-filtration-mono}
The $\sigma$-algebras $\mcl F_k$ of~\eqref{eqn-cut-filtration} are decreasing in $k$, i.e., $\mcl F_{K} \subset\mcl F_k$ whenever $K  > k$. Moreover, any event in the intersection $\sigma$-algebra $\bigcap_{k \geq 1} \mcl F_k$ has probability zero or one. 
\end{lem}
\begin{proof}
The monotonicity is proven essentially by inspection, but we give the details. Let $K > k$. Define the walk 
\eqbn
\wh{\mcl Z}' = (\wh{\mcl L}',\wh{\mcl R}') := \wh{\mcl Z}(\cdot + T_{m_k})|_{[0 ,\infty)} .   
\eqen
By definition, $\wh{\mcl Z}'$ is $\mcl F_k$-measurable.
The time $T_{m_{K}} - T_{m_k}$ is determined by $\wh{\mcl Z}'$: indeed, to recover $T_{m_{K}} - T_{m_k}$ from $\wh{\mcl Z}'$, we first let $\acute\tau_K^R$ be the smallest $n\in\BB N$ for which $\wh{\mcl R}'(n)  = -K$. Then $T_{m_{K}} - T_{m_k}$ is the largest $n\leq \acute\tau_K^R$ at which the event $A_n\cap B_n$ of~\eqref{eqn-walk-bdy-event} occurs with $\wh{\mcl Z}'$ in place of $\wh{\mcl Z}$. 
We also note that $\wh\lambda^-_{T_{m_k}}$ is the source vertex for the map constructed from $\wh{\mcl Z}'$ via the KMSW procedure. 
From this, it follows that $\XDP_{\wh M}( \wh\lambda^-_{T_{m_k}} , \wh\lambda^-_{T_{m_K}}    )$ is determined by $\wh{\mcl Z}'$. 
Since $\wh P$ must pass through each of the cut vertices $\wh\lambda^-_{T_m}$, 
\eqb \label{eqn-filtration-ldp}
\XDP_{\wh M}\left(\wh\lambda^-_0  , \wh\lambda^-_{T_{m_K}}  \right) =\XDP_{\wh M}\left(\wh\lambda^-_0  , \wh\lambda^-_{T_{m_k}} \right) +\XDP_{\wh M}\left(\wh\lambda^-_{T_{m_k}}  , \wh\lambda^-_{T_{m_K}} \right) .
\eqe 
Since $\XDP_{\wh M}( \wh\lambda^-_{T_{m_k}} , \wh\lambda^-_{T_{m_K}}    )$ is determined by $\wh{\mcl Z}'$, it follows from~\eqref{eqn-filtration-ldp} and the definition~\eqref{eqn-cut-filtration} of $\mcl F_k$ that $\XDP_{\wh M}\left(\wh\lambda^-_0  , \wh\lambda^-_{T_{m_K}}  \right)$ is $\mcl F_k$-measurable. Clearly, $\wh{\mcl Z}(\cdot + T_{m_K})|_{[0 ,\infty)}$ is also $\mcl F_k$-measurable. Therefore, $\mcl F_K \subset\mcl F_k$. 

To prove the tail triviality, we will argue using the Hewitt-Savage zero-one law. We will apply this result to the collection of i.i.d.\ random variables $(Z_j)_{j\in\BB N}$, given by the walk segments 
\eqb \label{eqn-cut-increments}
Z_j:=\left(\wh{\mcl Z}(\cdot + T_{j-1}) - \wh{\mcl Z}(T_{j-1})\right)|_{[0,T_j - T_{j-1}]}  ,\quad j \in \BB N,
\eqe
which are i.i.d.\ by Lemma~\ref{lem-renewal}.
The $j$th walk segment $Z_j$ determines the submap of $\wh M$ between the cut vertices $\wh \lambda^-_{T_{j-1}}$ and $\wh\lambda^-_{T_j}$ via the KMSW procedure, so in particular it determines $\XDP_{\wh M}\left(\wh\lambda^-_{T_{j-1}}  , \wh\lambda^-_{T_j} \right)$.
Since the $\XDP$ geodesic $\wh P$ must pass through each of the cut vertices $\wh\lambda^-_{T_m}$,
\eqb
\wh{\mcl Z}(T_m) = \sum_{j=1}^m (\wh{\mcl Z}(T_j) - \wh{\mcl Z}(T_{j-1}) )  \quad \text{and} \quad 
\XDP_{\wh M}\left(\wh\lambda^-_0  ,  \wh\lambda^-_{T_m} \right) = \sum_{j=1}^m \XDP_{\wh M}\left(\wh\lambda^-_{T_{j-1}}  , \wh\lambda^-_{T_j} \right) .
\eqe
Consequently, $\wh{\mcl Z}(T_m)$ and $ \XDP_{\wh M}(\wh\lambda^-_0 ,  \wh\lambda^-_{T_m} )$ remain unchanged if we make any finite permutation of the first $m$ walk segments $(Z_j)_{j\in[1,m]\cap \BB Z}$.
This implies that $\wh{\mcl Z}( \cdot + T_{m })|_{[0 ,\infty)}$ also remains unchanged if we make any finite permutation of the first $m$ walk segments $(Z_j)_{j\in[1,m]\cap \BB Z}$. 
The $\sigma$-algebra $\mcl F_k$ of~\eqref{eqn-cut-filtration}  is generated by  $\wh{\mcl Z}( \cdot + T_{m })|_{[0 ,\infty)}$ and $ \XDP_{\wh M}(\wh\lambda^-_0 ,  \wh\lambda^-_{T_m} )$ and with $m = m_k$, so any $\mcl F_k$-measurable event is invariant under finite permutations of the first $m_k$ walk segments $(Z_j)_{j\in[1,m_k]\cap \BB Z}$; indeed, such permutations cannot change the value of $m_k$. 
Since $m_k\to\infty$ as $k\to\infty$ by Proposition~\ref{prop-cut-vertex}, any $\bigcap_{k=0}^\infty \mcl F_k$-measurable event must be invariant under finite permutations of the walk segments $(Z_j)_{j\in\BB N}$.
Since these walk segments are i.i.d., the Hewitt-Savage zero-one law implies that any such event has probability zero or one. 
\end{proof}

\begin{figure}[ht!]
\begin{center}
\includegraphics[width=0.6\textwidth]{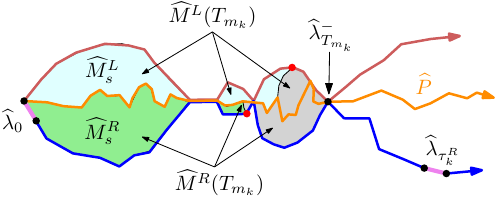}  
\caption{\label{fig-thin-ind}  
Illustration for the proof of Proposition~\ref{prop-thin-ind}. To get that the submaps $\wh M^L$ and $\wh M^R$ lying (weakly) to the left and right of $\wh P$ are independent, we first fix $s$ and send $k\to\infty$ to deduce the independence of $\wh M_s^L$ and $\wh M_s^R$; highlighted in light blue and green, respectively (the red vertices are the $s$th vertices along the left and right boundaries of $\wh M$). We then send $s\to\infty$. 
}
\end{center}
\end{figure}

\begin{proof}[Proof of Proposition~\ref{prop-thin-ind}]
See Figure~\ref{fig-thin-ind} for an illustration. 
Fix $s\in\BB N$. Let $\wh M_s^L$ be the submap of $\wh M^L$ consisting of all vertices of $\wh M^L$ which can be joined to the $s$th vertex on the left boundary of $\wh M^L$ by a directed path in $\wh M^L$, and all edges and faces of $\wh M^L$ whose vertices are all of this type. Similarly define $\wh M_s^R$ with ``right'' in place of ``left''. On the event 
\eqb \label{eqn-min-coord}
E_k = E_k(s) := \left\{\min\{ \wh{\mcl L}(T_{m_k}), - \wh{\mcl R}(T_{m_k}) \} \geq s+1 \right\} ,
\eqe
the cut vertex $\wh\lambda^-_{T_{m_k}}$ comes after the $s$th vertex on both the left and right boundaries of $\wh M$. Hence on $E_k$, we have $\wh M_s^L \subset \wh M^L(T_{m_k})$ and $\wh M_s^R \subset \wh M^R(T_{m_k})$, where we are using the same notation as in Lemma~\ref{lem-finite-ind}. Moreover, on $E_k$ the map $\wh M_s^L$ is determined by $\wh M^L(T_{m_k})$: indeed, the definition of $\wh M_s^L$ with $\wh M^L(T_{m_k})$ in place of $\wh M^L$ gives us the same map. The same is true for $\wh M^R_s$. 
 
Let $\mcl F_k$ be as in~\eqref{eqn-cut-filtration} and note that the event $E_k$ of~\eqref{eqn-min-coord} is $\mcl F_k$-measurable. 
By Lemma~\ref{lem-finite-ind} and the previous paragraph, we get that on $E_k$, the maps $\wh M_s^L$ and $\wh M_s^R$ are conditionally independent given $\mcl F_k$.
Hence, if $X^L$ and $X^R$ are bounded real-valued random variables which are a.s.\ given by measurable functions of $\wh M_s^L$ and $\wh M_s^R$, respectively, then 
\eqb \label{eqn-use-cond-ind}
\BB E\left[ X^L X^R \,\middle|\, \mcl F_k \right] \BB 1_{E_k}  
= \BB E\left[ X^L \,\middle|\, \mcl F_k \right] \BB E\left[ X^R \,\middle|\, \mcl F_k \right]  \BB 1_{E_k}   .
\eqe

By Lemma~\ref{lem-filtration-mono} and the backward martingale
convergence theorem, we get the a.s.\ convergence
$\BB E[X^L X^R \,|\, \mcl F_k] \to \BB E[X^L X^R]$
as $k\to\infty$, and similarly for the other two conditional
expectations in~\eqref{eqn-use-cond-ind}.
\jbb{By the definition of $T_m$ in Lemma~\ref{lem-renewal},
the event $A_{T_m}$ occurs for each $m\geq1$, so the step at
$T_m$ is $(1,-1)$. Moreover, $B_{T_{m-1}}$ occurs, with
$B_0$ holding automatically. Consequently,
$\wh{\mcl L}(T_m)\geq\wh{\mcl L}(T_{m-1})+1$.
The strict running-minimum condition $A_{T_m}$ also gives
$-\wh{\mcl R}(T_m)\geq-\wh{\mcl R}(T_{m-1})+1$.
Since $T_0=0$ and $\wh{\mcl Z}(0)=(0,0)$, it follows that
$\min\{\wh{\mcl L}(T_{m_k}),-\wh{\mcl R}(T_{m_k})\}\geq m_k$.
By Proposition~\ref{prop-cut-vertex}, $T_m<\infty$ a.s.\
for every $m$, and hence~\eqref{eqn-last-cut} and
$\tau_k^R\to\infty$ imply that $m_k\to\infty$ a.s.\
as $k\to\infty$. Therefore,
$\min\{\wh{\mcl L}(T_{m_k}),-\wh{\mcl R}(T_{m_k})\}$
goes to $\infty$ a.s.\ as $k\to\infty$, and so $\BB P[E_k] \to 1$ as $k\to\infty$.}
Taking the limit as $k\to\infty$ in~\eqref{eqn-use-cond-ind}
therefore gives
$\BB E[X^L X^R]=\BB E[X^L]\BB E[X^R]$.
Since this holds for any possible choices of $X^L$ and $X^R$,
we get that $\wh M_s^L$ and $\wh M_s^R$ are independent.
As $s\to\infty$, $\wh M_s^L$ and $\wh M_s^R$ increase to
$\wh M^L$ and $\wh M^R$, respectively.
Sending $s\to\infty$ now concludes the proof.
\end{proof}

%% file: tex/busemann.tex
\subsection{Setup and outline}
\label{sec-busemann-setup}

In this section, we will prove Theorems~\ref{thm-busemann} and~\ref{thm-busemann-property}. 
We consider the following setup.
Fix $\XDP \in \{\op{LDP} , \op{SDP}\}$. 
Let $\mcl Z = (\mcl L , \mcl R) : \BB Z^2 \to \BB Z$ be the encoding walk for the UIBOT $(\uibot,\lambda_0)$ as in Proposition~\ref{prop-kmsw-uibot}. For $n\in\BB Z$, let
\eqbn
\lambda_n = (\lambda^-_n , \lambda^+_n)
\eqen
be the active edge of the KMSW procedure at time $n$, with the indexing chosen so that $\lambda_0  $ is the root edge. 

For $n\in\BB Z$, let $M_{n,\infty}$  be the infinite acyclic-oriented triangulation with missing edges obtained by applying the KMSW procedure to $(\mcl Z(\cdot+n) - \mcl Z(n))|_{[n,\infty)}$.  By Remark~\ref{remark-kmsw-infinite}, $\uibot$ can be recovered from $(M_{n,\infty})_{n\in\BB Z}$ by taking their union.  Boundary edges to the west of $\lambda_n$ are not part of $M_{n,\infty}$, i.e.\ are missing edges.

Let $\wh M_{n,\infty}$ be the submap of $M_{n,\infty}$ induced by the vertices of $M_{n,\infty}$ which are reachable by a directed path in $M_{n,\infty}$ started from the initial vertex $\lambda^-_n$. That is, $\wh M_{n,\infty}$ consists of this vertex set, all of the edges of $M_{n,\infty}$ whose vertices are both in this vertex set, and all of the faces of $M_{n,\infty}$ whose three boundary vertices are all in this vertex set.

The left and right boundaries of $\wh M_{n,\infty}$ are the leftmost and rightmost directed paths in $M_{n,\infty}$ from $\lambda^-_n$ to $\infty$ (note that these two paths are \emph{not} disjoint in general). In particular, $\wh M_{n,\infty}$ has no missing edges. 
The rightmost directed path in $M_{n,\infty}$ from $\lambda^-_n$ to $\infty$ is just the set of non-missing boundary edges of $M_{n,\infty}$, i.e.\ all the boundary edges to the east of $\lambda_n$, including $\lambda_n$ itself.
%lying to the right of $\lambda^-_n$. 
We denote the leftmost directed path from $\lambda^-_n$ to $\infty$ in $M_{n,\infty}$ by 
\eqb  \label{eqn-north-flow-line}
\beta_n : \BB N \to \mcl E(M_{n,\infty}) ,  
\eqe 
noting that such a path exists thanks to the final claim in Lemma~\ref{lem-uibot-bdy}.
By definition, all of the edges of $\beta_n$ are included in $\wh M_{n,\infty}$. 

Let $\wh M_{n,\infty}'$ be the submap of $M_{n,\infty}$ obtained by removing from $M_{n,\infty}$ all of the vertices, edges, and faces of $\wh M_{n,\infty}$ (including the ones on $\beta_n$). Then $\wh M_{n,\infty}'$ is a directed triangulation with missing edges (every boundary edge of $\wh M_{n,\infty}'$ is a missing edge). See Figure~\ref{fig-flow-line-def} for an illustration of the above definitions.

\begin{figure}[ht!]
\begin{center}
\includegraphics[width=0.8\textwidth]{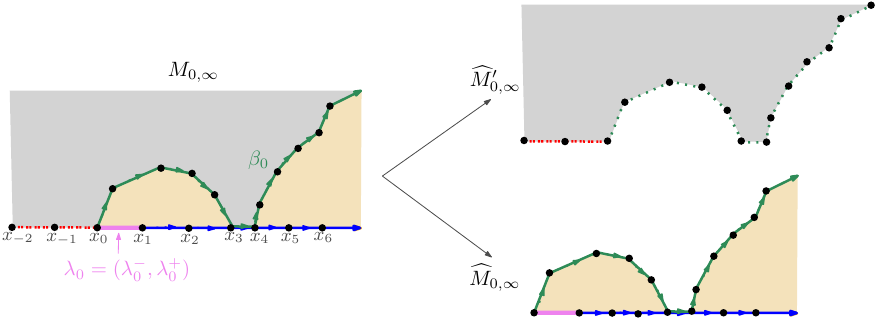}  
\caption{\label{fig-flow-line-def}  
Illustration of the definitions of the maps $\wh M_{n,\infty}$, $\wh M_{n,\infty}'$ obtained from $M_{n,\infty}$ and $\beta_n$, in the case when $n =0$. The edges on $\beta_0$ belong to $\wh M_{0,\infty}$ but not to $\wh M_{0,\infty}'$. 
}
\end{center}
\end{figure}

The following proposition makes the connection between the results of Section~\ref{sec-uiqbot} and the UIBOT. 
 
\begin{prop} \label{prop-future-map-law}
For each $n\in\BB Z$, the maps $\wh M_{n,\infty}$ and $\wh M_{n,\infty}'$ are independent. Furthermore, the map $\wh M_{n,\infty}$ has the same law as the UIQBOT $\wh M$ of Definition~\ref{def-uiqbot}.  
\end{prop}
 
We will prove Proposition~\ref{prop-future-map-law} in Section~\ref{sec-future-map-law} using the KMSW bijection. 
 
As in Theorem~\ref{thm-busemann}, we let $\{x_k\}_{k\in\BB Z}$ be the boundary vertices of $M_{0,\infty}$, enumerated in directed order in such a way that $x_0 = \lambda_0^-$ and $x_1 = \lambda_0^+$ are the initial and terminal vertices of $\lambda_0$, respectively. 
For $k\in\BB Z$, let 
\eqb \label{eqn-bdy-vertex-hit}
\tau_k := \min\left\{ n\in\BB N_0 : \text{$x_k$ is the initial vertex of $\lambda_n$} \right\} .
\eqe
Recalling Definition~\ref{defn:bound-edges}, it is easy to see from the Assertions~\ref{item-lower-left} and~\ref{item-lower-right} of Lemma~\ref{lem-kmsw-bdy} (note that in Assertions~\ref{item-lower-left} we are counting in the reversed order compared to how boundary vertices are enumerated in $M_{0,\infty}$; this is why we have $ \mcl L(n) = k $ instead of $  \mcl L(n) = -k$ in the last case of \eqref{eqn-bdy-vertex-kmsw}) that
\eqb \label{eqn-bdy-vertex-kmsw}
\tau_k = \begin{cases}
0 ,\quad &k = 0 \\
\min\left\{ n\in\BB N : \mcl R(n) = -k   \right\} ,\quad & k\geq 1 \\
\min\left\{ n\in\BB N : \mcl L(n) = k   \right\} ,\quad & k\leq -1  .
\end{cases}
\eqe
In particular, $\tau_k$ is increasing for $k\geq 0$ and decreasing for $k\leq 0$.
 
In Section~\ref{sec-busemann-exist}, we will prove Theorem~\ref{thm-busemann} using the existence of cut vertices in the UIQBOT (Proposition~\ref{prop-cut-vertex}). More precisely, we will force long directed paths started from boundary vertices of $M_{0,\infty}$ to pass through the cut vertices of one of the maps $\wh M_{\tau_k,\infty}$ (which has the law of the UIQBOT by Proposition~\ref{prop-future-map-law} and the strong Markov property of $\mcl Z$). 
   
In Section~\ref{sec-busemann-ind}, we will use the left and right independence property of   Proposition~\ref{prop-thin-ind} (applied to the UIQBOTs $\wh M_{\tau_k}$), together with the strong Markov property of $\mcl Z$, to prove the independent and stationary increments properties in Theorem~\ref{thm-busemann-property} (Properties~\ref{item-busemann-ind} and~\ref{item-busemann-stationary}). Property~\ref{item-busemann-pos} follows from elementary monotonicity considerations. 

What remains is then to prove Property~\ref{item-busemann-sym}. To do this, in Section~\ref{sec-busemann-sym} we introduce a new infinite acyclically-oriented triangulation with boundary, which we call the uniform infinite boundary-channeled half-plane bipolar-oriented triangulation (UIBHBOT; Definition~\ref{def-uihbot}). The UIBHBOT describes the local limit of uniformly sampled finite boundary-channeled triangulations (Definition~\ref{def-reverse-map}) around a uniform boundary edge (Proposition~\ref{prop-disk-bs-conv}). Moreover, the UIBHBOT admits submaps  which have the same law as $M_{0,\infty}$ (Lemma~\ref{lem-uihbot-future}). Using the symmetry in Lemma~\ref{lem-bdy-reverse}, one can show that the law of the UIBHBOT is invariant under applying an orientation-reversing homeomorphism from $\BB C\to\BB C$ (Proposition~\ref{prop-uihbot-sym}). Combining these last two facts leads to Property~\ref{item-busemann-sym}.

\subsection{Existence of the Busemann function}
\label{sec-busemann-exist}

\begin{figure}[ht!]
\begin{center}
\includegraphics[width=0.9\textwidth]{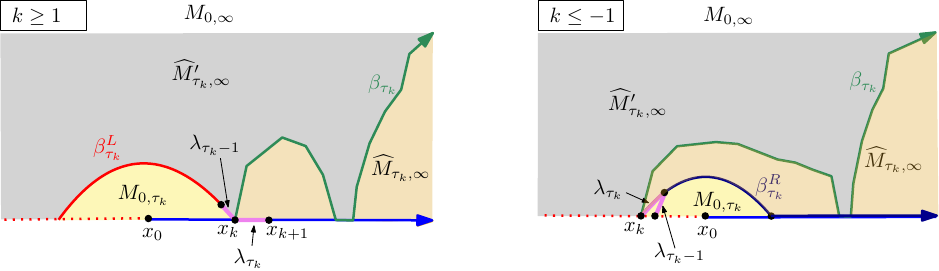}  
\caption{\label{fig-directed-contain}  
Illustration of the proof of Lemma~\ref{lem-directed-contain} in the case when $k\geq 1$ (left) and $k\leq -1$ (right). 
}
\end{center}
\end{figure}
 
\begin{lem} \label{lem-directed-contain}
Define $\tau_k$ as in~\eqref{eqn-bdy-vertex-hit}. 
For each $k\in\BB Z$, the map $M_{\tau_k,\infty}$ has the same law as the map $M_{0,\infty}$ and the map $\wh M_{\tau_k,\infty}$ has the same law as the UIQBOT (Definition~\ref{def-uiqbot}). Furthermore, every directed path in $M_{0,\infty}$ started from $x_k$ is contained in the map $\wh M_{\tau_k,\infty}$ defined just above~\eqref{eqn-north-flow-line}. 
\end{lem}

\begin{proof}
See Figure~\ref{fig-directed-contain} for an illustration of the proof.
By~\eqref{eqn-bdy-vertex-kmsw}, each $\tau_k$ is a stopping time for $\mcl Z$. 
By the strong Markov property, the map $M_{\tau_k,\infty}$ (which is determined by $(\mcl Z(\cdot+\tau_k) - \mcl Z(\tau_k))|_{[0,\infty)}$) has the same law as $M_{0,\infty}$. By this and Proposition~\ref{prop-future-map-law}, $\wh M_{\tau_k,\infty}$ has the law of the UIQBOT. 

It remains to prove the statement about directed paths. This statement holds by definition for $k=0$, so we only need to prove it for $k\not=0$.  
We first argue that for $k\in\BB Z \setminus \{0\}$, every directed path in $M_{0,\infty}$ started from $x_k$ is contained in $M_{\tau_k,\infty}$. Let $P$ be such a directed path. 

\smallskip

First, consider the case when $k \geq 1$.
Let $\beta_{\tau_k}^L$ be the path formed by the missing edges on the boundary of $M_{\tau_k,\infty}$ lying to the west of $\lambda_{\tau_k}^- = x_k$. 
By Lemma~\ref{lem-uibot-bdy}, $\beta_{\tau_k}^L$ is the leftmost directed path in $M_{-\infty,\infty}$ ending at $x_k$ and every vertex on $\beta_{\tau_k}^L$ has no incoming edges in $M_{\tau_k,\infty}$. Since $\beta_{\tau_k}^L$ is the leftmost directed path ending at $x_k$, all of the edges of $M_{0,\infty}\setminus M_{\tau_k,\infty}$ incident to $x_k$ must be oriented toward $x_k$. Therefore, the first step of $P$ must be from $x_k$ to a vertex of $M_{\tau_k,\infty}$, and $P$ cannot visit any vertex on $\beta_{\tau_k}^L$ after its first step. The path $\beta_{\tau_k}^L$ disconnects $M_{0,\infty}\setminus M_{\tau_k,\infty}$ from $\infty$ in $M_{0,\infty}$. Hence $P$ cannot enter $M_{0,\infty}\setminus M_{\tau_k,\infty}$, so $P$ must be contained in $M_{\tau_k,\infty}$.

\smallskip

Next, consider the case when $k \leq -1$. Let $\jbb{\beta_{\tau_k}^R}$ be the path formed by the edges in the  boundary of $M_{\tau_k,\infty}$ lying to the east of the terminal vertex $\lambda_{\tau_k}^+$ of $\lambda_{\tau_k}$ (see the dark blue path on the right-hand side of Figure~\ref{fig-directed-contain}). By Lemma~\ref{lem-uibot-bdy}, $\beta_{\tau_k}^R$ is the rightmost directed path in $M_{0,\infty}$ started from $\lambda^+_{\tau_k} $. It follows that $ P$ cannot cross $\beta_{\tau_k}^R$: indeed, otherwise concatenating part of $\beta_{\tau_k}^R$ with the portion of $ P$ lying to the right of $\beta_{\tau_k}^R$ would yield a directed path started from $\lambda^+_{\tau_k}$ which lies partially to the right of $\beta_{\tau_k}^R$.

By \eqref{eqn-bdy-vertex-hit}, $\tau_k$ is the \emph{first}
non-negative time at which $x_k$ is the initial vertex of
$\lambda_n$.
\jbb{Since every (non-missing) edge is first added as an active edge by the KMSW procedure,
no (non-missing) edge added before time $\tau_k$ can have initial
vertex $x_k$. Moreover, $x_k$ lies on the left boundary of
$M_{0,\infty}$ and has no incoming (non-missing) edges
by Lemma~\ref{lem-uibot-bdy}.}
Consequently, there are no (non-missing) edges of
\jbb{$M_{0,\tau_k-1}$} incident to $x_k$.

In particular, the first edge of $P$ is
\jbb{not a (non-missing) edge of $M_{0,\tau_k-1}$}.
Since $P$ is a directed path and the orientation on
$M_{0,\infty}$ is acyclic, $P$ can visit $x_k$ only once.
\jbb{By the KMSW boundary identification and planarity,
reaching a (non-missing) edge of $M_{0,\tau_k-1}$ after the
first step would require either revisiting $x_k$ or
crossing $\beta_{\tau_k}^R$.}
We already saw that $P$ cannot cross $\beta_{\tau_k}^R$,
so $P$ cannot visit \jbb{any (non-missing) edge} of
$M_{0,\tau_k-1}$.
\jbb{Every (non-missing) edge of $M_{0,\infty}$ which is not
present in $M_{\tau_k,\infty}$ was added before time
$\tau_k$, and hence is a (non-missing) edge of $M_{0,\tau_k-1}$.
Thus $P$ is contained in $M_{\tau_k,\infty}$.}

\smallskip

Recall that $\wh M_{\tau_k,\infty}$ is the submap of $M_{\tau_k,\infty}$ induced by the set of vertices of $M_{\tau_k,\infty}$ reachable by a directed path in $M_{\tau_k,\infty}$ started from $x_k$. Since every such directed path  in $M_{0,\infty}$ started from $x_k$ is contained in $M_{\tau_k,\infty}$, any such directed path is in fact contained in $\wh M_{\tau_k,\infty}$. 
\end{proof}

The following lemma tells us that any finite collection of sufficiently long directed paths started from vertices on the boundary of $M_{0,\infty}$ must have a vertex in common. This will be the key tool for the proof of the existence of the Busemann function (Theorem~\ref{thm-busemann}).

\begin{lem} \label{lem-infty-cut}
Let $\{x_k\}_{k\in\BB Z}$ be the sequence of boundary vertices as in Theorem~\ref{thm-busemann}.  
For each $k , k'\in\BB Z$ with $k < k'$, there exists a sequence of distinct vertices $\{v_m\}_{m\in\BB N}$ of $M_{0,\infty}$ with the following property. Let $\el \in [k,k']\cap\BB Z$ and let $\{w_j\}_{j\in\BB N}$ be any sequence of distinct vertices of $M_{0,\infty}$ which can each be reached by a directed path in $M_{0,\infty}$ started from $x_\el$. Then for each $m\in\BB N$, it holds for each sufficiently large $j\in\BB N$ that every directed path in $M_{0,\infty}$ from $x_\el$ to $w_j$ visits $v_m$. 
\end{lem}
\begin{proof}
The proof is based on the fact that the UIQBOT has infinitely many cut vertices (Proposition~\ref{prop-cut-vertex}).
Let $\tau_k$ be as in~\eqref{eqn-bdy-vertex-hit}. 
By Lemma~\ref{lem-directed-contain}, the map $\wh M_{\tau_k,\infty}$ has the same law as the UIQBOT.

Let $k , k' \in \BB Z$ with $k < k'$. 
Since the KMSW procedure only explores finitely many edges between times $\tau_k$ and $\tau_{k'}$, the right boundaries of the maps $M_{\tau_\el ,\infty}$ for $\el \in [k,k']\cap\BB Z$ (equivalently, the right boundaries of the maps $\wh M_{\tau_\el,\infty}$ for $\el\in [k,k']\cap\BB Z$) coincide except for finitely many edges. 
In particular, there is an infinite directed path $\beta_{\tau_k , \tau_{k'}}^R$ which is part of the right boundary of $\wh M_{\tau_\el,\infty}$ for each $\el \in [k,k']\cap\BB Z$. Almost surely, the right boundary of $\wh M_{\tau_k,\infty}$ has infinitely many cut vertices (Proposition~\ref{prop-cut-vertex}). So, a.s.\ there are infinitely many cut vertices of $\wh M_{\tau_k,\infty}$ which lie on $\beta_{\tau_k , \tau_{k'}}^R$. Let $\{v_m\}_{m\in\BB N}$ be the ordered sequence of these cut vertices. 
 
For $\el \in [k,k']\cap\BB Z$, the leftmost directed path $\beta_{\tau_\el }$ of~\eqref{eqn-north-flow-line} (which is the left boundary of $\wh M_{\tau_{\el},\infty}$) cannot cross $\beta_{\tau_k}$ (which is the left boundary of $\wh M_{\tau_k,\infty}$). Hence every cut vertex of $\wh M_{\tau_k,\infty}$ which lies on $\beta_{\tau_k,\tau_{k'}}^R$ (i.e., each of the vertices $v_m$) is also a cut vertex of $\wh M_{\tau_\el,\infty}$.
By Lemma~\ref{lem-directed-contain}, every directed path in $M_{0,\infty}$ started from $x_\el$ must be contained in $\wh M_{\tau_\el,\infty}$. For each $m\in\BB N$, there are only finitely many vertices of $\wh M_{\tau_\el,\infty}$ which are not disconnected from $x_\el$ when we remove the cut vertex $v_m$. It follows that there are only finitely many vertices of $M_{0,\infty}$ which are reachable by a directed path started from $x_\el$ in $M_{0,\infty}$ which does not visit $v_m$. If $\{w_j\}_{j\in\BB N}$ is as in the lemma statement, then if $j$ is sufficiently large, $w_j$ is not equal to any of these finitely many vertices, for any $\el \in [k,k']\cap\BB Z$. This implies the lemma statement. 
\end{proof}

\begin{proof}[Proof of Theorem~\ref{thm-busemann}]
Let $\XDP \in \{\op{LDP} , \op{SDP}\}$ and $K\in\BB N$. By Lemma~\ref{lem-infty-cut}, there exists a vertex $z = z(K)$ of $  M_{0,\infty}$ with the following property. Let $k \in [-K,K]\cap\BB Z$ and let $\{w_j\}_{j\in\BB N}$ be a sequence of distinct vertices of $M_{0,\infty}$ which can each be reached from $x_k$ by a directed path in $M_{0,\infty}$.   
Then, it holds for each sufficiently large $j\in\BB N$ that every directed path in $M_{0,\infty}$ from $x_k$ to $w_j$ visits $z$. Indeed, we can take $z$ to be one of the vertices $v_m$ in Lemma~\ref{lem-infty-cut} (with $-K$ and $K$ in place of $k$ and $k'$). 
We define
\eqb \label{eqn-busemann-vertex}
\mcl X(k) := \XDP_{M_{0,\infty}} (x_k , z) - \XDP_{M_{0,\infty}} (x_0 , z) ,\quad \forall k \in [-K,K]\cap\BB Z .
\eqe

We claim that the function defined in~\eqref{eqn-busemann-vertex} satisfies the defining property~\eqref{eqn-busemann-lim} from the theorem statement for all $k,k' \in [-K,K]\cap\BB Z$.  
Indeed, let $\{w_j\}_{j\in\BB Z}$ be a (new) sequence of distinct vertices which can be reached by a directed path in $M_{0,\infty}$ started from $x_k$ or started from $x_{k'}$  
Then for each sufficiently large $j$, each $\XDP$ geodesic in $M_{0,\infty}$ from $x_k$ to $w_j$ visits $z$, which implies that
\eqbn
\XDP_{M_{0,\infty}} (x_k , w_j) = \XDP_{M_{0,\infty}} (x_k ,z) + \XDP_{M_{0,\infty}} (z , w_j)
\eqen
The analogous relation also holds with $k'$ in place of $k$. By subtracting these two relations, we obtain that for each sufficiently large $j$, 
\eqbn
\mcl X(k') -\mcl X(k) = \XDP_{M_{0,\infty}} (x_{k'} , w_j) - \XDP_{M_{0,\infty}} (x_k , w_j),
\eqen
which is~\eqref{eqn-busemann-lim}. 

It is clear that there is a at most one function satisfying $\mcl X(0) = 0$ and~\eqref{eqn-busemann-lim} with $k,k'$ restricted to lie in $ [-K,K]\cap\BB Z$. Since we have constructed such a function on $[-K,K] \cap\BB Z$ for an arbitrary $K\in\BB N$, we have proven the theorem statement. 
\end{proof}

\subsection{Independent and  stationary increments}
\label{sec-busemann-ind}

In this section, we will prove Properties~\ref{item-busemann-ind} through~\ref{item-busemann-pos} of Theorem~\ref{thm-busemann-property}. We first observe that leftmost $\XDP$ geodesics to $\infty$ in $M_{0,\infty}$ (Definition~\ref{def-infinite-geo}) exist and satisfy a coalescence property.

\begin{lem} \label{lem-geo-infty}
Let $\{x_k\}_{k\in\BB Z}$ be the ordered sequence of boundary vertices as in Theorem~\ref{thm-busemann}. 
Almost surely, for each $k\in\BB Z$, there exists a (unique) leftmost $\XDP$ geodesic $P_{x_k}$ from $x_k$ to $\infty$ in $M_{0,\infty}$. These leftmost $\XDP$ geodesics coalesce, in the sense that for any $k,k'\in\BB Z$, there exists $\theta, \theta' \in \BB N$ such that $ P_{x_k}(t + \theta) = P_{x_{k'}}(t + \theta')$ for each $t \in \BB N$.  
\end{lem}
\begin{proof}%[Proof of Lemma~\ref{lem-geo-infty}]
Define $\{\tau_k\}_{k\in\BB Z}$ as in~\eqref{eqn-bdy-vertex-hit}.
Since $\wh M_{\tau_k,\infty} $ has the law of the UIQBOT (Lemma~\ref{lem-directed-contain}), Lemma~\ref{lem-geo-exist} implies that there is a unique leftmost $\XDP$ geodesic $P_{x_k}$ from $x_k$ to $\infty$ in $\wh M_{\tau_k,\infty}$. Since every directed path from $x_k$ to $\infty$ in $M_{0,\infty}$ is contained in $\wh M_{\tau_k}$ (Lemma~\ref{lem-directed-contain}), the path $P_{x_k}$ is in fact the unique leftmost $\XDP$ geodesic from $x_k$ to $\infty$ in $M_{0,\infty}$. 
 
By Lemma~\ref{lem-infty-cut}, if we define the sequence of vertices $\{v_m\}_{m\in\BB N}$ for $x_k$ and $x_{k'}$ as in that lemma, then each of $P_{x_k}$ and $P_{x_{k'}}$ must visit $v_m$ for every $m\in\BB N$.  
In particular, $P_{x_k}$ and $P_{x_{k'}}$ have at least one vertex in common. 
By the uniqueness of leftmost $\XDP$ geodesics, $P_{x_k}$ and $P_{x_{k'}}$ coincide after the first time they meet. 
This gives the coalescence property.   
\end{proof}

The following lemma is the main input in the proofs of the independent and stationary increments properties in Theorem~\ref{thm-busemann-property} (Properties~\ref{item-busemann-ind} and~\ref{item-busemann-stationary}).
It is a consequence of Proposition~\ref{prop-thin-ind}.

\begin{figure}[ht!]
\begin{center}
\includegraphics[width=0.95\textwidth]{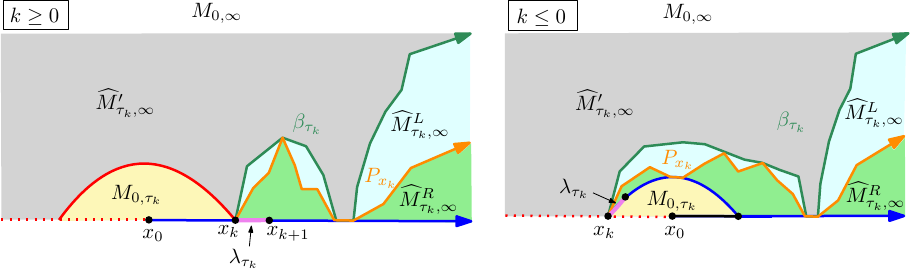}  
\caption{\label{fig-geo-infty-ind}  
Illustration of the submaps used in the proof of Lemma~\ref{lem-geo-infty-ind} in the case when $k\geq 0$ (left) and the case when $k\leq 0$ (right). 
}
\end{center}
\end{figure}

\begin{lem} \label{lem-geo-infty-ind}
For $k\in\BB Z$, let $\rng M_{x_k,\infty}^L$ (resp.\ $\rng M_{x_k,\infty}^R$) be the submap of $M_{0,\infty}$ lying to the left (resp.\ right) of the leftmost $\XDP$ geodesic $P_{x_k}$ of Lemma~\ref{lem-geo-infty}, including the vertices and edges lying on $P_{x_k}$.
Then $\rng M_{x_k,\infty}^L$ and $\rng M_{x_k,\infty}^R$ are independent. Furthermore, if $k\geq 0$ then $\rng M_{x_k,\infty}^R \eqD \rng M_{x_0,\infty}^R$ and if $k\leq 0$ then $\rng M_{x_k,\infty}^L \eqD \rng M_{x_0,\infty}^L$. 
\end{lem}
\begin{proof}
See Figure~\ref{fig-geo-infty-ind} for an illustration.

Define $\tau_k$ as in~\eqref{eqn-bdy-vertex-hit}. Let $M_{0,\tau_k}$ be the submap of $M_{0,\infty}$ obtained by applying the KMSW procedure to $\mcl Z|_{[0,\tau_k]}$ (Definition~\ref{def-kmsw}). By~\eqref{eqn-bdy-vertex-kmsw}, each $\tau_k$ is a stopping time for $\mcl Z$. Since $M_{0,\tau_k}$ is determined by $\mcl Z|_{[0,\tau_k]}$ and $M_{\tau_k,\infty}$ is determined by $(\mcl Z(\cdot+\tau_k) - \mcl Z(\tau_k))|_{[0,\infty)}$, the maps $M_{0,\tau_k}$ and $M_{\tau_k,\infty}$ are independent.

Define the maps $\wh M_{\tau_k,\infty}$ and $\wh M_{\tau_k,\infty}'$ as in the discussion just above Proposition~\ref{prop-future-map-law}. By Proposition~\ref{prop-future-map-law} and Lemma~\ref{lem-directed-contain}, the maps $\wh M_{\tau_k,\infty}$ and $\wh M_{\tau_k,\infty}'$ are independent and their joint law does not depend on $k$. Since the maps $\wh M_{\tau_k,\infty}$ and $\wh M_{\tau_k,\infty}'$ are a.s.\ determined by $M_{\tau_k,\infty}$, combining this with the previous paragraph shows that the triple $(M_{0,\tau_k} , \wh M_{\tau_k,\infty} , \wh M_{\tau_k,\infty}')$ is independent.

By Lemma~\ref{lem-directed-contain}, $P_{x_k}$ is the leftmost $\XDP$ geodesic from $x_k$ to $\infty$ in $\wh M_{\tau_k,\infty}$. 
By Lemma~\ref{lem-directed-contain}, the map $\wh M_{\tau_k,\infty}$ has the same law as the UIQBOT. 
Hence, Proposition~\ref{prop-thin-ind} implies the following. Let $\wh M_{\tau_k,\infty}^L$ (resp.\ $\wh M_{\tau_k,\infty}^R$) be the submap of $\wh M_{\tau_k,\infty}$ consisting of the vertices, edges, and faces lying to the left (resp.\ right) of $P_{x_k}$, including the vertices and edges lying on $P_{x_k}$. Then $\wh M_{\tau_k,\infty}^L$ and $\wh M_{\tau_k,\infty}^R$ are independent. By combining this with the previous paragraph, we get that the 4-tuple $(M_{0,\tau_k} , \wh M_{\tau_k,\infty}^L , \wh M_{\tau_k,\infty}^R , \wh M_{\tau_k,\infty}')$ is independent. Moreover, the laws of $\wh M_{\tau_k,\infty}^L ,\wh M_{\tau_k,\infty}^R$, and $\wh M_{\tau_k,\infty}'$ do not depend on $k$. 

When $k\geq 0$, the map $\rng M_{x_k,\infty}^R$ is equal to $\wh M_{\tau_k,\infty}^R$. The map $\rng M_{x_k,\infty}^L$ is obtained by removing the edge $(x_k , x_{k+1})$ from $M_{0,\tau_k}$  to get $M_{0,\tau_k-1}$, then identifying the maps $M_{0,\tau_k-1}$, $\wh M_{\tau_k,\infty}^L$, and $\wh M_{\tau_k,\infty}'$ together along their boundaries. In particular, the upper boundary of $M_{0,\tau_k-1}$ (Definition~\ref{def-kmsw}) is identified with a segment of the left boundary of $\wh M_{\tau_k,\infty}'$ starting from the vertex $x_k$, and the right boundary of $\wh M_{\tau_k,\infty}'$ is identified with the left boundary of $\wh M_{\tau_k,\infty}^L$. By the independence statement from the previous paragraph and the fact that the law of $\wh M_{\tau_k,\infty}^R$ does not depend on $k$, we get the lemma statement in the case when $k \geq 0$. 

In the case when $k\leq 0$, the map $\rng M_{x_k,\infty}^L$ is obtained by identifying the right boundary of $\wh M_{\tau_k}'$ with the left boundary of $\wh M_{\tau_k,\infty}^L$. The map $\rng M_{x_k,\infty}^R$ is obtained by identifying the upper boundary of $M_{0,\tau_k}$ with a segment of the right boundary of $\wh M_{\tau_k,\infty}^R$ started at $x_k$. Thus the lemma statement similarly follows in this case.
\end{proof}

For $k\in\BB Z$, let $P_{x_k}$ be the leftmost $\XDP$ geodesic from $x_k$ to $\infty$ in $M_{0,\infty}$, as in Lemma~\ref{lem-geo-infty}.
We will now use Lemma~\ref{lem-geo-infty-ind} to decompose $M_{0,\infty}$ into countably many independent ``geodesic slices'' lying between successive leftmost $\XDP$ geodesics $P_{x_{k-1}}$ and $P_{x_k}$. For $k \in \BB Z$, let $\theta_k^- , \theta_k \in \BB N_0$ be the smallest numbers such that
\eqb \label{eqn-coalesce}
P_{x_{k-1}}(t + \theta_k^-) = P_{x_k}(t + \theta_k) ,\quad \forall t \in \BB N .
\eqe
Such numbers exist by Lemma~\ref{lem-geo-infty}.

\begin{figure}[ht!]
\begin{center}
\includegraphics[width=0.6\textwidth]{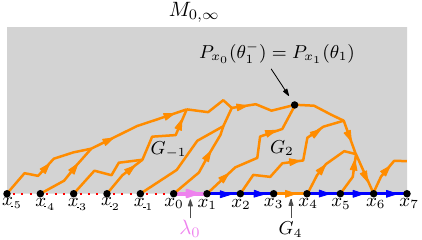}  
\caption{\label{fig-geo-slice}  
Illustration of the $\XDP$ geodesic slices $G_k$ defined in Lemma~\ref{lem-geo-slice} (regions between the orange paths). For $k\geq 0$, $G_k$ can consist of a single edge (which means that $\theta_k^- = 1$ and $\theta_k = 0$) as shown for $G_4$, but for $k \leq -1$ this is not possible since $M_{0,\infty}$ does not include the boundary edges to the left of $x_0$.
}
\end{center}
\end{figure}

\begin{lem} \label{lem-geo-slice}
Let $G_k$ be the (directed) submap of $M_{0,\infty}$ consisting of the vertices, edges, and faces of $M_{0,\infty}$ which lie between $P_{x_{k-1}}([1,\theta_k^-])$ and $P_{x_k}([1,\theta_k])$, including the vertices and edges which lie on these $\XDP$ geodesic segments. We call $G_k$ an \textbf{$\XDP$ geodesic slice} (see Figure~\ref{fig-geo-slice}). The maps $\{G_k\}_{k\in\BB Z}$ are independent. Furthermore, the maps $\{G_k\}_{k \geq 1}$ all have the same law, and the maps $\{G_k\}_{k \leq 0}$ all have the same law. 
\end{lem}
\begin{proof}
Define the maps $\rng M_{x_k,\infty}^L$ and $\rng M_{x_k,\infty}^R$ as in Lemma~\ref{lem-geo-infty-ind}. Since leftmost $\XDP$ geodesics cannot cross each other, for $k\in\BB Z$, the maps $\{G_j\}_{j \leq k }$ are a.s.\ determined by $\rng M_{x_{k },\infty}^L$ and the maps $\{G_j\}_{j \geq k+1}$ are a.s.\ determined by $\rng M_{x_k,\infty}^R$. By the independence part of Lemma~\ref{lem-geo-infty-ind}, the sequence of maps $\{G_j\}_{j\leq k }$ is independent from the sequence of maps $\{G_j\}_{ j\geq k+1}$. Since this holds for every $k\in\BB Z$, the sequence of maps $\{G_k\}_{k\in\BB Z}$ are independent from each other.
%We know $\{G_j\}_{j\leq k-1}$ is independent from $\{G_j\}_{j \geq k}$ and $\{G_j\}_{j\leq k}$ is independent from $\{G_j\}_{j \leq k+1}$. If $A,B,C$ are events depending on $\{G_j\}_{j\leq k-1}$, $G_k$, and $\{G_j\}_{j \leq k+1}$, then $A\cap B$ is independent from $C$ and $A$ is independent from $B$, so $\BB P[A\cap B\cap C] = \BB P[A\cap B] \BB P[C] = \BB P[A]\BB P[B]\BB P[C]$. 
To get the stationarity property, we observe that $G_k$ is determined by each of $\rng M_{x_k,\infty}^L$ and $\rng M_{x_{k-1},\infty}^R$ in a manner which does not depend on $k$. The stationarity property for $G_k$ therefore follows from the stationarity part of Lemma~\ref{lem-geo-infty-ind}. 
\end{proof}

The Busemann function can be expressed in terms of the geodesic slices $G_k$ as follows.

\begin{lem} \label{lem-busemann-cls}
Let $\theta_k^-$ and $\theta_k$ be as in~\eqref{eqn-coalesce}. Then the Busemann function (Theorem~\ref{thm-busemann}) satisfies
\eqb \label{eqn-busemann-cls} 
\mcl X(k) - \mcl X(k-1) = \theta_k  - \theta_k^- .
\eqe 
\end{lem}
\begin{proof}
Let $w_0$ be the terminal vertex of $P_{x_k}(\theta_k)$, equivalently, the terminal vertex of $P_{x_{k-1}}(\theta_k^-)$. 
Let $\{w_j\}_{j\in\BB N}$ be the ordered sequence of vertices visited by $P_{x_k} $ (equivalently, $P_{x_{k-1}} $) after $w_0$. 
Since $P_{x_k}$ and $P_{x_{k-1}}$ are $\XDP$ geodesics, for each $j\in\BB N$, 
\eqbn
\XDP_{M_{0,\infty}}(x_k , w_j) = \XDP_{M_{0,\infty}}(x_k , w_0) + \XDP_{M_{0,\infty}}(w_0 , w_j)
\eqen
and the same is true with $x_{k-1}$ in place of $x_k$. Therefore, 
\eqbn
\XDP_{M_{0,\infty}}(x_k , w_j) - \XDP_{M_{0,\infty}}(x_{k-1} , w_j)
= \XDP_{M_{0,\infty}}(x_k , w_0) - \XDP_{M_{0,\infty}}(x_{k-1} , w_0)
= \theta_k - \theta_k^- .
\eqen
By the definition of $\mcl X$ from Theorem~\ref{thm-busemann}, this implies~\eqref{eqn-busemann-cls}.
\end{proof}

\begin{proof}[Proof of Theorem~\ref{thm-busemann-property}, Properties~\ref{item-busemann-ind} through~\ref{item-busemann-pos}]
By Lemma~\ref{lem-busemann-cls}, the increment $\mcl X(k) - \mcl X(k-1)$ is a function of the XDP geodesic slice map $G_k$ of Lemma~\ref{lem-geo-slice}. Therefore, the independent increments and stationary increments properties \ref{item-busemann-ind} and~\ref{item-busemann-stationary} follow from Lemma~\ref{lem-geo-slice}. 
To get Property~\ref{item-busemann-pos}, let $k\geq 1$. Use Theorem~\ref{thm-busemann} to find $w\in M_{0,\infty}$ such that \eqb \label{eqn-busemann-pos-dist}
\mcl X(k) - \mcl X(k-1)
= \XDP_{M_{0,\infty}}(x_k , w) - \XDP_{M_{0,\infty}}(x_{k-1} , w) .
\eqe
Since the directed edge $(x_{k-1} , x_k)$ is part of $M_{0,\infty}$, Property~\ref{item-busemann-pos} is immediate from~\eqref{eqn-busemann-pos-dist}.  
\end{proof}

\subsection{Symmetry via the UIBHBOT}
\label{sec-busemann-sym}

To prove Theorem~\ref{thm-busemann-property}, it remains to prove Property~\ref{item-busemann-sym}. 
For this purpose, we will introduce an auxiliary infinite directed planar map, which we now define. 
Let $\mcl Z^\bc = (\mcl L^\bc ,\mcl R^\bc) : \BB Z\to\BB Z^2$ be a bi-infinite walk such that $\mcl Z^\bc(0)  =(0,0)$; $\mcl Z^\bc|_{[0,\infty)} \eqD \mcl Z|_{[0,\infty)}$; and 
\[\text{$\mcl Z^\bc |_{(-\infty,0]}$ has the law of $\mcl Z|_{(-\infty,0]}$ conditioned so that $\mcl L(n) \geq 1$ for every $n \leq -1$}.\]
To be concrete, one way to construct $\mcl Z^\bc|_{(-\infty,0]}$ is as follows. Let $\{\wt{\mcl Z}_k(\cdot)\}_{k\in \BB Z}$ be an i.i.d.\ collection of random walk paths such that $\wt{\mcl Z}_k$ has the same law as $\mcl Z|_{[0,\tau_{-1}]}$, where $\tau_{-1} = \min\{n\in\BB N : \mcl L(n) = -1\}$ (as in~\eqref{eqn-bdy-vertex-kmsw}). Let $\mcl Z^\bc : \BB Z\to\BB Z^2$ be obtained by concatenating the paths $\wt{\mcl Z}_k$ end-to-end in order, in such a way that $\mcl Z^\bc(0) = (0,0)$ and 0 is the time at which we start traversing the path $\wt{\mcl Z}_0$. Then for $k\in\BB Z$, the path $\wt{\mcl Z}_k$ is the segment of $\mcl Z^\bc$ between the first time that $\mcl L^\bc$ hits $-k$ and the first time that $\mcl L^\bc$ hits $-k-1$. 
It follows from the strong Markov property that $\mcl Z^\bc|_{[0,\infty)} \eqD \mcl Z|_{[0,\infty)}$. It can be seen from Lemma~\ref{lem-quadrant-walk} below that $\mcl Z^\bc|_{(-\infty,0]}$ has the same law as the process obtained from $\mcl Z|_{(-\infty,0]}$ by conditioning $\mcl L|_{(-\infty,1]}$ to stay positive in the sense of~\cite{bd-conditioning} (but we will not need to use this latter fact).  
%\footnote{To be more precise, one can define a walk $\wh{\mcl Z}^+ = (\wh{\mcl L}^+ , \wh{\mcl R}^+)$ with the law of $\mcl Z$ started from $(0,0)$ and conditioned so that $\mcl R(n)\geq 1$ for every $n\geq 1$ as follows. Let $\wh{\mcl Z}$ be a walk with the law of $\mcl Z$ started from $(0,0)$ conditioned so that $\wh{\mcl R}(n) \geq 0$ for every $n\geq 0$, as in~\eqref{eqn-cond-walk-rn} but with $\mcl R$ in place of $\mcl L$. Then let $\wh{\mcl Z}^+(0) = (0,0)$ and let $\wh{\mcl Z}^+(n) = \wh{\mcl Z}(n-1) + (0,1)$ for every $n\geq 1$. This definition makes sense since if $\wh{\mcl R}^+(n) \geq 1$ for each $n\geq 1$, then the first step of $\wh{\mcl Z}^+$ must be from $(0,0)$ to $(0,1)$. The walk $ \mcl Z^\bc |_{(-\infty,0]}$ has the same law as $(\wh{\mcl R}^+(-\cdot) , \wh{\mcl L}^+(-\cdot))$. }

By Definition~\ref{def-kmsw} and Remark~\ref{remark-kmsw-infinite}, we can apply the KMSW procedure to $\mcl Z^\bc$ to produce an infinite bipolar-oriented triangulation $M^\bc$. As we will explain just below, $M^\bc$ has an infinite boundary, and all of its boundary edges are missing (i.e.\ it only has an infinite lower-left boundary); see Figure~\ref{fig-uihbot} for an illustration of $M^\bc$.

We describe $M^\bc$ more precisely.
For $k\in\BB Z$, let
\eqb \label{eqn-uihbot-bdy-time}
\tau_k^\bc := \min\left\{ n\in\BB Z  : \mcl L^\bc(n) =   k\right\}  .
\eqe 
By the definition of $\mcl Z^\bc$, $\{\tau_k^\bc\}_{k\in\BB Z}$ is a decreasing sequence of times and $\tau_0^\bc = 0$.
By Assertion~\ref{item-lower-left} of Lemma~\ref{lem-kmsw-bdy}, for each $k\in\BB Z$, the time $\tau_k^\bc$ is the time when the triangle of $M^\bc$ whose boundary includes the $k$th  lower-left (missing) boundary edge of $M^\bc$ is added to $M^\bc$ by the KMSW procedure, with the lower-left boundary edges enumerated from west to east. Hence, the (missing) boundary edges are added from east to west by the KMSW procedure.  Since, a.s.,  
\eqbn
\liminf_{n\to\infty} \mcl L^\bc(n) = \liminf_{n\to\infty} \mcl R^\bc(n) = \liminf_{n\to - \infty} \mcl R^\bc(n) = -\infty,
\eqen
Lemma~\ref{lem-kmsw-bdy} and Remark~\ref{remark-kmsw-infinite} imply that a.s.\ all of the boundary edges of $M^\bc$ belong to the lower-left boundary. 

Let $\{x^\bc_k\}_{k\in\BB Z}$ be the vertices along the (lower-left) boundary of $M^\bc$, enumerated so that $(x^\bc_{k},x^\bc_{k+1})$ is the edge corresponding to time $\tau_k^\bc$. Since, as explained above, the edges $\{(x^\bc_{k},x^\bc_{k+1})\}_{k\in\BB Z}$ are enumerated from west to east, the vertices $\{x_k^\bc\}_{k\in\BB Z}$ are also enumerated from west to east. 
We write $e^\bc = (x^\bc_{0}, x^\bc_1)$ and call it the \textbf{root edge} of $M^\bc$. Note that here we are not respecting the usual convention that maps are rooted at $\lambda_0$, the reason being that $\lambda_0$ is not a boundary edge for $M^\bc$ (rather, its initial vertex is $x^\bc_0$ and its terminal vertex is not on the boundary). 
Also note that $e^\bc$ is a missing edge of $M^\bc$.  
 
\begin{defn} \label{def-uihbot}
The \textbf{uniform infinite boundary-channeled half-plane bipolar-oriented triangulation} (UIBHBOT) is the edge-rooted triangulation with boundary $(M^\bc,e^\bc)$ defined just above.
\end{defn}

\begin{figure}[ht!]
\begin{center}
\includegraphics[width=0.6\textwidth]{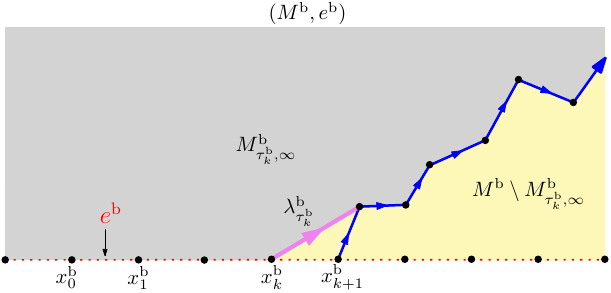}  
\caption{\label{fig-uihbot}  
The UIBHBOT $M^\bc$ and the submap $M^\bc_{\tau_k^\bc , \infty} \eqD M_{0,\infty}$ for a typical value of $k\in\BB Z$. The KMSW procedure first explores the yellow region and then the grey region. The triangles containing a boundary edge are explored in east to west order. Also shown is the active edge for the KMSW procedure at time $\tau_k^\bc$ with the corresponding triangle having a (missing) boundary edge. Note that all the boundary edges of the UIBHBOT are missing. By, e.g., Proposition~\ref{prop-disk-bs-conv} below, each of the boundary vertices of the UIBHBOT has no incoming edges. 
}
\end{center}
\end{figure}

The reason for the name is that the UIHBHBOT arises as the Benjamini-Schramm local limit of certain uniform boundary-channeled bipolar-oriented triangulations, as defined in Definition~\ref{def-reverse-map}, around a uniformly chosen boundary edge (see Proposition~\ref{prop-disk-bs-conv} below). 
For the purposes of proving Theorem~\ref{thm-busemann-property}, the main fact about the UIBHBOT which we need is the following non-trivial symmetry.

\begin{prop} \label{prop-uihbot-sym} 
Let $(M^\bc ,e^\bc)$ be the UIBHBOT.  
Let $(\wt M^\bc , \wt e^\bc)$ be the directed, edge-rooted triangulation with missing boundary edges obtained from $(M^\bc,e^\bc)$ by applying an orientation-reversing homeomorphism $\BB C\to\BB C$. 
Then $(\wt M^\bc , \wt e^\bc)$ has the same law as $(M^\bc ,e^\bc)$.  
\end{prop}

Recall that a planar map is viewed modulo \emph{orientation-preserving} homeomorphisms from $\BB C$ to $\BB C$.
Applying an orientation-reversing homeomorphism to $(M^\bc , e^\bc)$ gives us a different planar map. 
Proposition~\ref{prop-uihbot-sym} tells us that this new planar map has the same law as $(M^\bc,e^\bc)$. 

In the setting of Proposition~\ref{prop-uihbot-sym}, each (non-missing) edge of $\wt M^\bc$ is equipped with the same orientation as the corresponding edge of $M^\bc$ (it is only the orientation of the plane which is reversed, not the orientation of the edges of the map). 

Proposition~\ref{prop-uihbot-sym} is not obvious from the definition of the UIBHBOT in terms of the KMSW procedure. 
Indeed, the KMSW encoding walk corresponding to $(\wt M^\bc , \wt e^\bc)$ is coupled with the original KMSW encoding walk $\mcl Z^\bc$ in a highly non-trivial way. 
Reversing the orientation of $\BB C$ swaps west and east, so if $\{x^\bc_k\}_{k\in\BB Z}$ is the west to east ordered sequence of boundary vertices of $M^\bc$, as in the discussion just above~\eqref{def-uihbot}, then $\{x^\bc_{-k}\}_{k\in\BB Z}$ is the west to east ordered sequence of boundary vertices of $\wt M^\bc$. 

We will prove Proposition~\ref{prop-uihbot-sym} in Section~\ref{sec-uihbot-sym} using a symmetry for finite boundary-channeled bipolar-oriented triangulations (Lemma~\ref{lem-bdy-reverse}) and a limiting argument. For the time being, we use Proposition~\ref{prop-uihbot-sym} to conclude the proof of Theorem~\ref{thm-busemann-property}. 

\begin{lem} \label{lem-uihbot-future} 
Let $(M^\bc , e^\bc)$ be the UIBHBOT (Definition~\ref{def-uihbot}). 
For $k\in\BB Z$, let $\tau_k^\bc$ be as in~\eqref{eqn-uihbot-bdy-time} and let $\{x_k^\bc\}_{k\in\BB Z}$ be the sequence of boundary vertices as in the discussion just after~\eqref{eqn-uihbot-bdy-time}.
Let $M^\bc_{\tau_k^\bc,\infty}$ be the submap of $M^\bc$ obtained by applying the KMSW procedure to $(\mcl Z^\bc(\cdot + \tau_k^\bc) - \mcl Z^\bc(\tau_k^\bc))|_{[0,\infty)}$. 
Then $M_{\tau_k^\bc , \infty} \eqD M_{0,\infty}$. 
Furthermore, for each $\el \leq k$, each directed path in $M^\bc$ from $x_{\el}^\bc$ to $\infty$ is contained in $M_{\tau_k^\bc,\infty}$. 
\end{lem}
\begin{proof}
In the construction of $\mcl Z^\bc$ described at the beginning of Section~\ref{sec-busemann-sym}, the process $(\mcl Z^\bc(\cdot + \tau_k^\bc)- \mcl Z^\bc(\tau_k^\bc))|_{[0,\infty)} $ is obtained by concatenating the walk paths $\wt{\mcl Z}_j$ for $j \geq -k$. By the strong Markov property of $\mcl Z$, this implies that
$(\mcl Z^\bc(\cdot + \tau_k^\bc)- \mcl Z^\bc(\tau_k^\bc))|_{[0,\infty)} \eqD \mcl Z|_{[0,\infty)}$.  Hence $M^\bc_{\tau_k^\bc , \infty} \eqD M_{0,\infty}$. 
The statement that each directed path in $M^\bc$ from $x_\el^\bc$ to $\infty$ is contained in $M_{\tau_k^\bc,\infty}$ follows from a similar argument as in the $k\leq -1$ case of Lemma~\ref{lem-directed-contain}. 
\end{proof}

From Lemma~\ref{lem-uihbot-future}, we can obtain the analog of Theorem~\ref{thm-busemann} for the UIBHBOT.

\begin{prop} \label{prop-busemann-uihbot}
Let $\{x_k^\bc\}_{k\in\BB Z}$ be the boundary vertices of the UIBHBOT as in the discussion just after~\eqref{eqn-uihbot-bdy-time}. There exists a unique function $\mcl X^\bc :  \BB Z \to \BB Z$, called the \textbf{Busemann function}, such that $\mcl X^\bc(0) = 0$ and the following is true. Let $k,k' \in \BB Z$ and let $\{w_j\}_{j\in\BB N}$ be a sequence of distinct vertices of $M^\bc$ such that for each $j\in\BB N$, each of $x_k^\bc$ and $x_{k'}^\bc$ can be joined to $w_j$ by a directed path in $M^\bc$. Then
\eqb \label{eqn-busemann-lim-uihbot}
\mcl X^\bc(k') -\mcl X^\bc(k) = \XDP_{M^\bc}(x_{k'} , w_j) - \XDP_{M^\bc}(x_{k } , w_j) ,\quad \text{for each sufficiently large $j\in\BB N$.}
\eqe
Furthermore, the increments $\mcl X^\bc(k) - \mcl X^\bc(k-1)$ are i.i.d., and if $\mcl X$ is the Busemann function for $M_{0,\infty}$ as in Theorem~\ref{thm-busemann}, then 
\eqb \label{eqn-uihbot-busemann-law} 
\mcl X^\bc(k) - \mcl X^\bc(k-1) \eqD   \mcl X(0) - \mcl X(-1) .
\eqe  
\end{prop} 
\begin{proof}
For $K \in\BB Z$, define the time $\tau_K^\bc$ and the map $M_{\tau_K^\bc,\infty}^\bc$ as in Lemma~\ref{lem-uihbot-future}. By Lemma~\ref{lem-uihbot-future}, $M^\bc_{\tau_K^\bc,\infty} \eqD M_{0,\infty}$ and for $ k \leq K$, each directed path in $M^\bc$ from $x_{ k}^\bc$ to $\infty$ is contained in $M^\bc_{\tau_K^\bc,\infty}$. 
By applying Theorem~\ref{thm-busemann} to $M^\bc_{\tau_K^\bc,\infty}$, we find that there exists a unique $\mcl X^\bc : (-\infty,K]\cap\BB Z\to\BB Z$ satisfying~\eqref{eqn-busemann-lim-uihbot} for $k,k' \leq K$. 
Furthermore, by applying Properties~\ref{item-busemann-ind} and~\ref{item-busemann-stationary} of Theorem~\ref{thm-busemann-property} (which we have already proven) to $M^\bc_{\tau_K^\bc , \infty}$, we find that the increments $\mcl X^\bc(k)-\mcl X^\bc(k-1)$ for $k \leq K$ are i.i.d.\ and have the same law as the increment $\mcl X(0) - \mcl X(-1)$. Since $K$ can be made arbitrarily large, this proves the proposition.  
\end{proof}

\begin{proof}[Proof of Theorem~\ref{thm-busemann-property}, Property~\ref{item-busemann-sym}]
Let $(M^\bc,e^\bc)$ be the UIBHBOT and let $\mcl X^\bc$ be its Busemann function as in Proposition~\ref{prop-busemann-uihbot}. 
By Proposition~\ref{prop-uihbot-sym}, the law of $(M^\bc,e^\bc)$ is invariant under applying an orientation-reversing homeomorphism from $\BB C$ to $\BB C$ (while keeping the directions of the edges fixed). Applying such an orientation-reversing homeomorphism has the effect of reversing the order of the boundary vertices of $M^\bc$, which by~\eqref{eqn-busemann-lim-uihbot} has the effect of replacing $\mcl X^\bc$ by its time reversal $\mcl X^\bc(-\cdot)$. Consequently, the increments of $\mcl X^\bc$ and its time reversal have the same law, i.e., $\mcl X^\bc(k) -\mcl X^\bc(k-1) \eqD \mcl X^\bc(k-1) - \mcl X^\bc(k)$. By~\eqref{eqn-uihbot-busemann-law}, this implies that $\mcl X (0) - \mcl X (-1) \eqD \mcl X (-1) - \mcl X (0)$. Since the negative-time increments of $\mcl X$ all have the same law (Property~\ref{item-busemann-stationary}), this implies Property~\ref{item-busemann-sym}. 
\end{proof}

\subsection{Proof of Proposition~\ref{prop-future-map-law}: Relation between the UIBOT and the UIQBOT}
\label{sec-future-map-law}

The idea of the proof of Proposition~\ref{prop-future-map-law} is to look at the excursions that the KMSW exploration $\lambda_n$ makes to the left and right of the leftmost directed path $\beta_0$ of~\eqref{eqn-north-flow-line}. In Lemma~\ref{lem-flow-line-times}, we give a description in terms of the KMSW encoding walk $\mcl Z$ of the times $n$ when $\lambda_n \in \wh M_{0,\infty}$ and times $n$ when $\lambda_n \in \wh M_{0,\infty}'$ (recall the definitions of $\wh M_{0,\infty}$ and $\wh M_{0,\infty}'$ from the beginning of Section~\ref{sec-busemann-setup}). We then concatenate the increments of $\mcl Z$ corresponding to times when $\lambda_n \in \wh M_{0,\infty}$ (resp.\ $\lambda_n \in \wh M_{0,\infty}'$) to obtain two KMSW encoding walks $\wh{\mcl Z}$ and $\wh{\mcl Z}'$ for $\wh M_{0,\infty}$ and $\wh M_{0,\infty}'$, respectively (Lemma~\ref{lem-quad-kmsw}). In Lemma~\ref{lem-quadrant-walk}, we describe the joint law of $(\wh{\mcl Z}, \wh{\mcl Z}')$, with the help of an elementary lemma about lazy random walks (Lemma~\ref{lem-lazy-walk-flip}). This description immediately yields Proposition~\ref{prop-future-map-law}. 

We will often use in this section the following two aspects of the definitions of $\wh M_{0,\infty}$ and $\wh M'_{0,\infty}$: $(i)$ $\wh M_{0,\infty}$ is the submap of $M_{0,\infty}$ induced by the vertices of $M_{0,\infty}$ which are reachable by a directed path in $M_{0,\infty}$ started from the initial vertex $\lambda^-_0$; $(ii)$ the edges of the leftmost directed path $\beta_0$ of~\eqref{eqn-north-flow-line} are non-missing edges in $\wh M_{0,\infty}$ and missing edges in $\wh M'_{0,\infty}$.

\begin{figure}[ht!]
\begin{center}
\includegraphics[width=0.6\textwidth]{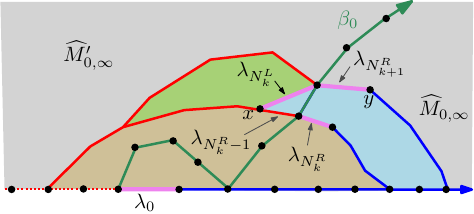}  
\caption{\label{fig-flow-line-times}  
The times $N_k^R , N_k^L,$ and $N_{k+1}^R$ appearing in Lemma~\ref{lem-flow-line-times} and the active edges at these times. The KMSW procedure explores the light brown region, then the light blue region, then the light green region, and then the gray region.  
}
\end{center}
\end{figure}

\begin{lem} \label{lem-flow-line-times}
Let $N_1^R = 0$. Inductively, for $k \in \BB N$, let $N_k^L$ be the smallest $n \geq N_k^R + 1$ for which $\lambda_n \notin \wh M_{0,\infty}$ (i.e., $\lambda_n \in \wh M_{0,\infty}'$) and let $N_{k+1}^R$ be the smallest $n \geq N_k^L + 1$ for which $\lambda_n \in \wh M_{0,\infty}$. Then
\allb \label{eqn-flow-line-times} 
N_k^L &= \min\left\{ n \geq N_k^R + 1 : \mcl L(n) = \mcl L(N_k^R) - 1 \right\} \notag\\
N_{k+1}^R &= \min\left\{ n \geq N_k^L + 1 :  \mcl R(n) = \mcl R(N_k^L) - 1 \right\} .
\alle
In particular: 
\begin{enumerate}[label=(\roman*),ref=(\roman*)]
    \item\label{item:flow-line-times-1} At each of the times $N_k^{R}$ for $k\geq 2$, 
    \begin{equation}\label{eq:ref-for-lat-disc2}
        \mcl Z(N_k^R) - \mcl Z(N_k^R-1) = (1,-1),
    \end{equation}
     and the KMSW procedure for $M_{0,\infty}$ adds at time $N_k^R$ a new single non-missing edge (rather than a triangular face) to $M_{0,N_k^R-1}$. This edge is $\lambda_{N_k^R}$.
    \item\label{item:flow-line-times-2} At each of the times $N_k^{L}$ for $k\geq 1$, 
    \begin{equation}\label{eq:ref-for-lat-disc}
        \mcl Z(N_k^L) - \mcl Z(N_k^L-1) = (-1,0),
    \end{equation}
     and the KMSW procedure for $M_{0,\infty}$ creates a triangular face $T$ contained in $\wh M_{0,\infty}'$ \jb{whose left boundary consists of  the (new active) edge $\lambda_{N_k^L}$ and whose right boundary consist of the edge $\lambda_{ N_k^R - 1}$ followed by the (previously active) edge $ \lambda_{N_k^L-1}$ along the path $\beta_0$.} \jb{For $k=1$, the edge $\lambda_{N_k^R-1}$ in this description is replaced by the missing boundary edge of $M_{0,\infty}$ immediately to the west of $\lambda_0$.}
\end{enumerate}
 
\end{lem}

\begin{proof}
The lemma is essentially a re-statement of~\cite[Lemma 2.1]{ghs-bipolar}, but~\cite{ghs-bipolar} has a slightly different setup as compared to this paper, so we will give a proof. Our proof is based on the forward KMSW bijection, in contrast to~\cite{ghs-bipolar} which uses the reverse KMSW bijection.  
See Figure~\ref{fig-flow-line-times} for an illustration. 

We proceed by induction on $k$. Assume that $ k\in \BB N$ and the claimed description of the time $N_k^R$ has been proven.  
Let $x$ be the vertex on the boundary of $M_{N_k^R,\infty}$ lying immediately to the west of the initial vertex $\lambda^-_{N_k^R}$ of $\lambda_{N_k^R}$. We claim that $x \notin \wh M_{0,\infty}$. 
For $k = 1$, in which case $N_1^R = 0$, the claim follows from the fact that $\beta_0$ cannot hit the boundary of $M_{0,\infty}$ strictly to the west of its initial vertex $\lambda^-_0$, which in turn follows from the fact that vertices on the left boundary of $M_{0,\infty}$ have no incoming edges in $M_{0,\infty}$ (Lemma~\ref{lem-uibot-bdy}).
  
If $k\geq 2$, then~\eqref{eqn-flow-line-times} with $k-1$ in place of $k$ implies \eqref{eq:ref-for-lat-disc2} since $\mcl Z$ has increments in $\{(1,-1), (-1,0) , (0,1)\}$.

By the KMSW procedure, this means that at time $N_k^R$ we add a new non-missing edge (rather than a triangular face) to $M_{0,N_k^R-1}$. This edge is $\lambda_{N_k^R}$.
By the definition of $N_k^R$, the edge $\lambda_{N_k^R}$ lies in $\wh M_{0,\infty}$. Furthermore, the edge on the boundary of $M_{N_k^R , \infty}$ immediately to the west of $\lambda_{N_k^R}$ is $\lambda_{N_k^R-1}$, which (by the definition of $N_k^R$) is not in $\wh M_{0,\infty}$. 
Hence $\lambda_{N_k^R-1} = (x, \lambda^-_{N_k^R}) \notin \mcl E(\wh M_{0,\infty})$ and so $x \notin \mcl V(\wh M_{0,\infty})$. 
 
By the definition of $N_k^L$, the time $N_k^L$ is at most the first time $n\geq N_k^R+1$ at which $x$ is a vertex of $\lambda_n$. From Assertion~\ref{item-lower-left} of Lemma~\ref{lem-kmsw-bdy} (applied to the walk $(\mcl Z(N_k^R+\cdot) - \mcl Z(N_k^R))|_{[0,\infty)}$ and the corresponding map $M_{N_k^R,\infty}$), this time is 
\eqb \label{eqn-flow-line-candidate-L}
 \min\left\{ n \geq N_k^R + 1 : \mcl L(n) = \mcl L(N_k^R) - 1 \right\} .
\eqe 
Conversely, assume that $n\geq N_k^R + 1$ and $n$ is strictly smaller than the time~\eqref{eqn-flow-line-candidate-L}.  
 We make the following preliminary observation: Fix $m\geq 0$. Let $N>m$ be the first time such that $\mcl L(N) < \mcl L(m)$. Thanks to Assertion~\ref{item-lower-left} of Lemma~\ref{lem-kmsw-bdy} (applied to the walk $(\mcl Z(m+\cdot) - \mcl Z(m))|_{[0,N-1)}$ and the corresponding map $M_{m,N-1}$), the map $M_{m,N-1}$ does not have any edge on the lower-left boundary.
Hence, the vertex $\lambda^-_{m}$  remains on the lower-left boundary of the map $M_{m',\infty}$ for all times $m'$ such that $m\leq m'<N$.

Since we assume that $n$ is strictly smaller than the time~\eqref{eqn-flow-line-candidate-L}, we deduce from the latter observation that $\lambda^-_{N_k^R}$ lies on the lower-left boundary of $M_{n,\infty}$, i.e.\ to the boundary to the west of $\lambda_n$ (possibly $\lambda^-_{N_k^R}$ is the initial vertex of $\lambda_{n}$).
Moreover, there is a directed path in $M_{0,\infty}$ from $\lambda^-_{N_k^R}$ to the initial vertex of $\lambda_n$, namely a segment of the boundary of $M_{n,\infty}$. Therefore, $\lambda_n \in \wh M_{N_k^R,\infty}$ because $\wh M_{N_k^R,\infty}$ is, by definition, the submap of $M_{N_k^R,\infty}$ induced by the vertices of $M_{N_k^R,\infty}$ which are reachable by a directed path in $M_{N_k^R,\infty}$ started from the initial vertex $\lambda^-_{N_k^R}$. Since $\wh M_{N_k^R,\infty}\subset \wh M_{0,\infty}$, we conclude that $n < N_k^L$. Hence $N_k^L$ is equal to the time~\eqref{eqn-flow-line-candidate-L}, as stated in \eqref{eqn-flow-line-times}, and we also proved Assertion~\ref{item:flow-line-times-1}.

\medskip

Now we consider $N_{k+1}^R$. Let $y$ be the vertex on the boundary of $M_{N_k^L,\infty}$ lying immediately to the east of the terminal vertex $\lambda^+_{N_k^L}$ of $\lambda_{N_k^L}$. By the description of $N_k^L$ proven just above, we have \eqref{eq:ref-for-lat-disc} since $\mcl Z$ has increments in $\{(1,-1), (-1,0) , (0,1)\}$.
By the KMSW procedure, at time $N_k^L$, we add a triangle with two of its edges on the boundary of $M_{N_k^L-1 ,\infty}$, and $\lambda_{N_k^L}$ is the other edge of this triangle. In particular, $\lambda^+_{N_k^L}$ is equal to $\lambda^+_{N_k^L-1}$, which is in $ \wh M_{0,\infty}$. Since $\lambda_{N_k^L} \notin \wh M_{0,\infty}$, the vertex $\lambda^+_{N_k^L}$ must in fact lie on the left boundary $\beta_0$ of $\wh M_{0,\infty}$. Furthermore, we note the following fact for future reference: since $\beta_0$ cannot cross the left or right boundaries of $M_{N_k^L,\infty}$, the part of $\beta_0$ contained in $M_{N_k^L,\infty}$ consists only of edges of $\beta_0$ traversed after it hits $\lambda^+_{N_k^L}$. 

The edge $(\lambda^+_{N_k^L} , y)$ is directed, and so, since $\lambda^+_{N_k^L} \in \wh M_{0,\infty}$, this edge also belongs to $\wh M_{0,\infty}$.
By the definition of $N_{k+1}^R$, we thus have that $N_{k+1}^R$ is at most the first time $n\geq N_k^L + 1$ for which $\lambda_n = (\lambda^+_{N_k^L} , y)$. By Assertion~\ref{item-lower-right} of Lemma~\ref{lem-kmsw-bdy} (applied to the walk $(\mcl Z(N_k^L+\cdot) - \mcl Z(N_k^L))|_{[0,\infty)}$ and the corresponding map $M_{N_k^L,\infty}$), this time is 
\eqb \label{eqn-flow-line-candidate-R}
\min\left\{ n \geq N_k^L + 1 :  \mcl R(n) = \mcl R(N_k^L) - 1 \right\} . 
\eqe

Conversely, if $n\geq N_k^L + 1$ and $n$ is strictly smaller than the time~\eqref{eqn-flow-line-candidate-R}, 
then $\lambda^+_{N_k^L}$ lies on the lower-left boundary of $M_{n ,\infty}$, i.e.\ to the boundary to the east of $\lambda_{ n }$ (possibly $\lambda^+_{N_k^L}$ is the terminal vertex of $\lambda_{n}$); this follows using the same idea we used below \eqref{eqn-flow-line-candidate-L}. 
We also make the following observation for future reference: the segment of the boundary of $M_{n  ,\infty}$ from $\lambda^+_{n }$ to $\lambda^+_{N_k^L}$ is a directed path between these two vertices in $M_{0,\infty}$. Since the orientation on $\uibot$ is acyclic, there is no directed path in $\uibot$ from $\lambda^+_{N_k^L}$ to $\lambda^+_{n }$. Since $(\lambda^-_{n} , \lambda^+_{n})$ is directed, there is also no directed path in $\uibot$ from $\lambda^+_{N_k^L}$ to $\lambda^-_{n }$.

We now show that $N_{k+1}^R> n$. Since $\lambda_{N_{k+1}^R-1} \notin \wh M_{0,\infty}$, $\lambda_{N_{k+1}^R} \in \wh M_{0,\infty}$, and $\lambda_{N_{k+1}^R-1}$ and $\lambda_{N_{k+1}^R}$ share a vertex, at least one of the vertices of $\lambda_{N_{k+1}^R}$ must lie on $\beta_0$ (we will see a posteriori that $\lambda_{N_{k+1}^R}^-$ has to be the shared vertex, but for the moment it may be $\lambda_{N_{k+1}^R}^-$ or $\lambda_{N_{k+1}^R}^+$). As we noted at the end of the paragraph after~\eqref{eq:ref-for-lat-disc}, $\lambda^+_{N_k^L}$ is the first vertex of $M_{N_k^L,\infty}$ hit by $\beta_0$. So, one of the vertices of $\lambda_{N_{k+1}^R}$ is hit by $\beta_0$ after $\lambda^+_{N_k^L}$. Since $\beta_0$ is a directed path, this gives a directed path from $\lambda^+_{N_k^L}$ to either $ \lambda^-_{N_{k+1}^R} $ or $ \lambda^+_{N_{k+1}^R} $, which combined with the result in the previous paragraph shows that $N_{k+1}^R> n$. Hence $N_{k+1}^R$ is equal to the time~\eqref{eqn-flow-line-candidate-R}, as stated in \eqref{eqn-flow-line-times}, and we also proved Assertion~\ref{item:flow-line-times-2}.
\end{proof}

We will now describe the KMSW encoding walks for the maps $\wh M_{0,\infty}$ and $\wh M_{0,\infty}'$ in terms of $\mcl Z$.
For $m \in \BB N_0$, let $S_m , S_m' \in \BB N_0$ be the smallest integers such that
\eqb \label{eqn-walk-quadrant-time}
m = \#\left\{ \jb{1 \leq} n \leq S_m : \lambda_n \in \wh M_{0,\infty} \right\}
  = \#\left\{ \jb{1 \leq} n \leq S_m' : \lambda_n \in \wh M_{0,\infty}' \right\}.
\eqe
Define $\wh{\mcl Z} = (\wh{\mcl L}, \wh{\mcl R}) : \BB N_0 \to \jb{\BB Z^2}$ by $\wh{\mcl Z}(0) = (0,0)$ and
\eqb \label{eqn-walk-quadrant}
\wh{\mcl Z}(m) - \wh{\mcl Z}(m-1) := \mcl Z(S_m) - \mcl Z(S_m - 1), \quad\forall m\in\BB N.
\eqe
In other words, $\wh{\mcl Z}(m)$ is obtained by summing the first $m$ increments of $\mcl Z$ corresponding to times \jb{$n\geq 1$} such that $\lambda_n \in \wh M_{0,\infty}$.

Similarly, with $N_k^L$ as in Lemma~\ref{lem-flow-line-times}, we define $\wh{\mcl Z}' = (\wh{\mcl L}', \wh{\mcl R}') : \BB N_0 \to \BB Z$ so that $\wh{\mcl Z}'(0) = (0,0)$ and
\eqb  \label{eqn-walk-quadrant'}
\wh{\mcl Z}'(m) - \wh{\mcl Z}'(m-1) := 
\begin{cases}
\mcl Z(S'_m) - \mcl Z(S'_m - 1)  ,\quad  S'_m\notin \{N_k^L\}_{k \in \BB N} \\
(0,1)  ,\qquad\qquad\qquad\quad\,\,\,\,  S'_m\in \{N_k^L\}_{k \in \BB N}  ,
\end{cases}   
\quad\forall m\in\BB N .
\eqe 
In other words, $\wh{\mcl Z}'(m)$ is obtained by summing the first $m$ increments of $\mcl Z$ corresponding to times $n$ such that $\lambda_n \in \wh M_{0,\infty}'$, except that for each of the times $N_k^L$, we replace the corresponding increment 
$\mcl Z(N_k^L) - \mcl Z(N_k^L-1) = (-1,0)$ by $(0,1)$. The reason we do this is explained in the proof of the next lemma.

\begin{lem} \label{lem-quad-kmsw}
The maps $\wh M_{0,\infty}$ and $\wh M_{0,\infty}'$ are obtained by applying the KMSW procedure to the walks $\wh{\mcl Z}$ and $\wh{\mcl Z}'$, respectively\jb{, with the initial edge in the second construction subsequently declared to be missing}.
\end{lem}
\begin{proof}
\jb{For the construction from $\wh{\mcl Z}'$, identify the initial edge with the missing boundary edge of $M_{0,\infty}$ immediately to the west of $\lambda_0$, temporarily oriented from $x_{-1}$ to $x_0$. The first increment $(0,1)$ creates the triangle explored at time $N_1^L$, whose new active edge is $\lambda_{N_1^L}$ and whose other missing edge lies on $\beta_0$. After carrying out the KMSW procedure, we declare the initial edge to be missing.}

The fact that $\wh{\mcl Z}$ (resp.\ $\wh{\mcl Z}'$) is the correct KMSW walk for $\wh M_{0,\infty}$ (resp.\ $\wh M_{0,\infty}'$) is obvious when $S_m\notin \{N_k^R\}_{k \geq 2}$ (resp.\ $S'_m\notin \{N_k^L\}_{k \geq 1}$), i.e. when the KMSW exploration of $M_{0,\infty}$ is not crossing $\beta_0$.
We now justify why the definitions in \eqref{eqn-walk-quadrant} and \eqref{eqn-walk-quadrant'} are the correct definitions also when $S_m\in \{N_k^R\}_{k \geq 2}$ and  $S'_m\in \{N_k^L\}_{k \geq 1}$. We first look at the case $S'_m\in \{N_k^L\}_{k \geq 1}$.

By Assertion~\ref{item:flow-line-times-2} of Lemma~\ref{lem-flow-line-times}, at each of the times $N_k^L$ for \jb{$k\geq2$},, the KMSW procedure for $M_{0,\infty}$ creates a triangular face $T$ contained in $\wh M_{0,\infty}'$ whose boundary consists of the edge $\lambda_{N_k^R - 1}$, the (new active) edge $\lambda_{N_k^L}$, and the (previously active) edge $e = \lambda_{N_k^L-1}$ along the path $\beta_0$; see the triangle to the west of $\beta_0$ in the light-blue region in Figure~\ref{fig-flow-line-times}. Now, since the edges along $\beta_0$ are missing edges in $\wh M_{0,\infty}'$, the KMSW procedure for $\wh M_{0,\infty}'$ never activates the missing edge $e$. On the other hand,  the KMSW procedure for $\wh M_{0,\infty}'$ adds the  triangular face $T$ by adding it directly to the edge $\lambda_{N^R_k-1}$ and creating the new (non-missing) edge $\lambda_{N^L_k}$ and the new (missing) edge $e$. This corresponds to a $(0,1)$ step for the KMSW encoding walk of $\wh M_{0,\infty}'$ (instead of a $(-1,0)$ step for the KMSW encoding walk of $ M_{0,\infty}$).

We finally explain why, for the walk $\wh{\mcl Z}$ in \eqref{eqn-walk-quadrant}, we do not have to make a modification similar to the one we had to make for $\wh{\mcl Z}'$ in \eqref{eqn-walk-quadrant'}.  
By Assertion~\ref{item:flow-line-times-1} of Lemma~\ref{lem-flow-line-times}, at each of the times $N_k^R$ for $k\geq 2$, the KMSW procedure for $M_{0,\infty}$ adds at time $N_k^R$ a new single non-missing edge (rather than a triangular face) to $M_{0,N_k^R-1}$. This edge is $\lambda_{N_k^R}$; see the pink edge on the lower boundary of the light-blue region in Figure~\ref{fig-flow-line-times}. Now, noting that the KMSW procedure for $\wh M_{0,\infty}$ must add the edge $\lambda_{N_k^R}$ using the same procedure, i.e.\ adding a new single non-missing edge (rather than a triangular face), we get that the step (1,-1) taken by the $\mcl Z$ walk is the correct step also for the $\wh{\mcl Z}$ walk.
This completes the proof of the lemma.
\end{proof}

The walks $\wh{\mcl Z}$ and $\wh{\mcl Z}'$ can equivalently be described in terms of the times $N_k^R , N_k^L$ for $k\geq 1$ defined in Lemma~\ref{lem-flow-line-times}. Indeed, by definition of $N_k^R $ and $ N_k^L$, the collections of times $n$ such that $\lambda_n\in\wh M_{0,\infty}$ and $\lambda_n\in\wh M'_{0,\infty}$ are respectively
\eqb  \label{eqn-quad-pre}
%\lambda^{-1}(\wh M_{0,\infty}) = 
\bigcup_{k=1}^\infty [N_k^R , N_k^L-1]  \cap\BB Z    \quad \text{and} \quad
%\lambda^{-1}( \wh M_{0,\infty}') =
\bigcup_{k=1}^\infty [N_k^L , N_{k+1}^R-1] \cap\BB Z .
\eqe   
By~\eqref{eqn-walk-quadrant} and~\eqref{eqn-quad-pre}, $\wh{\mcl Z}$ can be obtained by concatenating the random walk paths 
\begin{equation}\label{eqn-quad-pre2}
    \mcl Z|_{[0,N_1^L-1]}\qquad\text{and}\qquad\left(\mcl Z(\cdot + N_k^R-1) - \mcl Z(N_k^R-1)\right)|_{[0,N_k^L - N_k^R]},\quad \text{for } k \geq 2.
\end{equation}
That is, to get $\wh{\mcl Z}(m)$ we sum the first $m$ increments of these paths in order. Similarly, by~\eqref{eqn-walk-quadrant'} and~\eqref{eqn-quad-pre},  $\wh{\mcl Z}'$ can be obtained by concatenating the random walk paths 
\begin{equation}\label{eqn-quad-pre3}
\left(\mcl Z(\cdot + N_k^L-1) - \mcl Z(N_k^L-1) + (1,1) \,\BB 1_{[1,\infty)}(\cdot)\right)|_{[0,N_{k+1}^R - N_k^L]}  \quad\text{for }k\geq 1 .
\end{equation}
Here, adding $(1,1) \,\BB 1_{[1,\infty)}(\cdot)$ is equivalent to replacing the first walk step $\mcl Z(N_k^L) - \mcl Z(N_k^L-1)\stackrel{\eqref{eq:ref-for-lat-disc}}{=}(-1,0)$ of each walk path by $(0,1)$; which accounts for the special rule at the times $N_k^L$ in~\eqref{eqn-walk-quadrant'}.

The following lemma is the key input in the proof of Proposition~\ref{prop-future-map-law}.

\begin{lem} \label{lem-quadrant-walk}
The walks $\wh{\mcl Z}$ and $\wh{\mcl Z}'$ of~\eqref{eqn-walk-quadrant} and~\eqref{eqn-walk-quadrant'} are independent. The walk $\wh{\mcl Z}$ has the same law as the walk obtained by conditioning $\mcl Z$ so that its first coordinate stays non-negative (defined as in~\eqref{eqn-cond-walk-rn}). As for the walk $\wh{\mcl Z}'$, we have $\wh{\mcl Z}'(0) = (0,0)$, $\wh{\mcl Z}'(1) = \jb{(0,1)}$, and $(\wh{\mcl Z}'(\cdot + 1) - \wh{\mcl Z}'(1))|_{[0,\infty)}$ has the same law as the walk obtained by conditioning $\mcl Z$ so that its second coordinate stays non-negative (defined as in~\eqref{eqn-cond-walk-rn} with $\mcl R$ in place of $\mcl L$).
\end{lem}

For the proof of Lemma~\ref{lem-quadrant-walk}, we need the following random walk analog of Pitman's theorem for Brownian motion.

\begin{lem} \label{lem-lazy-walk-flip}
Let $\mcl L$ be a random walk \jbb{started from $0$} with i.i.d.\
increments which are uniform on $\{-1,0,1\}$.
For $n\in\jbb{\BB N_0}$, define
\eqb
\wh{\mcl L}(n) := \mcl L(n)
- 2 \min_{j \in [0,n]\cap\BB Z} \mcl L(j).
\eqe
In words, to obtain $\wh{\mcl L}$, every time $\mcl L$ attains
a \jbb{new} running minimum (at which time it has a $-1$ step),
we flip the $-1$ step to a $+1$ step.
Then $\wh{\mcl L}$ is a random walk started from 0 and
conditioned to stay non-negative in the sense
of~\cite{bd-conditioning} (see~\eqref{eqn-cond-walk-rn}).
\end{lem}
\begin{proof}
\jbb{For $n\in\BB N_0$, let} $\mcl W_n$ be the set of
$n$-step paths $\frk L : [0,n]\cap\BB Z \to \BB Z$ such that
$\frk L(0)=0$ and $\frk L(j)-\frk L(j-1)\in\{-1,0,1\}$
for each $j\in\{1,\dots,n\}$.
Let $\wh{\mcl W}_n$ be the set of $\frk L\in\mcl W_n$ such
that $\frk L(j)\geq0$ for each $j\in\{0,\dots,n\}$.
We define a function $\Phi_n:\mcl W_n\to\wh{\mcl W}_n$ by
\eqbn
\Phi_n\frk L(j)
:= \frk L(j)-2\min_{i\in[0,j]\cap\BB Z}\frk L(i).
\eqen
In words, to obtain $\Phi_n\frk L$, every time $\frk L$
attains a \jbb{new} running minimum (at which time it has
a $-1$ step), we flip the $-1$ step to a $+1$ step.

For $\wh{\frk L}\in\wh{\mcl W}_n$ and
$k\in\{0,\dots,\wh{\frk L}(n)\}$, we also define
$\Psi_{n,k}\wh{\frk L}\in\mcl W_n$ by
\eqbn
\Psi_{n,k}\wh{\frk L}(j)
:= \wh{\frk L}(j)
-2\min\left\{k,\min_{i\in[j,n]\cap\BB Z}\wh{\frk L}(i)\right\}.
\eqen
In words, to obtain $\Psi_{n,k}\wh{\frk L}$, for
$r\in\{0,\dots,k-1\}$, we look at the last time that
$\wh{\frk L}$ hits $r$ before time $n$,
\jbb{and flip the $+1$ step immediately after this time
to a $-1$ step.}

From the step-flipping descriptions of $\Phi_n$ and
$\Psi_{n,k}$, for any $\wh{\frk L}\in\wh{\mcl W}_n$ and
$k\in\{0,\dots,\wh{\frk L}(\jbb{n})\}$, we have
$\Phi_n\Psi_{n,k}\wh{\frk L}=\wh{\frk L}$.
Furthermore, if $\jbb{\frk L}\in\mcl W_n$ and
$\jbb{k=-\min_{j\in[0,n]\cap\BB Z}\frk L(j)}$, then
$\Psi_{n,k}\Phi_n\frk L=\frk L$.
These two facts imply that if
$\wh{\frk L}\in\wh{\mcl W}_n$, then
\eqbn
\Phi_n^{-1}\wh{\frk L}
=
\left\{\Psi_{n,k}\wh{\frk L}:
k\in\{0,\dots,\wh{\frk L}(n)\}\right\}.
\eqen
\jbb{These preimages are distinct, since
$(\Psi_{n,k}\wh{\frk L})(n)=\wh{\frk L}(n)-2k$.}
Consequently,
$\#\Phi_n^{-1}\wh{\frk L}=\wh{\frk L}(n)+1$.

For each $n\in\jbb{\BB N_0}$, the walk
$\mcl L|_{[0,n]}$ is a uniform sample from $\mcl W_n$.
We have
$\wh{\mcl L}|_{[0,n]}=\Phi_n(\mcl L|_{[0,n]})$,
so the previous paragraph implies that
$\wh{\mcl L}|_{[0,n]}$ is sampled from the uniform measure
on $\wh{\mcl W}_n$ weighted by $\wh{\frk L}(n)+1$.
By the definition~\eqref{eqn-cond-walk-rn}, this means that
$\wh{\mcl L}$ is the walk obtained by conditioning
$\mcl L$ to stay non-negative.
\end{proof}

\begin{proof}[Proof of Lemma~\ref{lem-quadrant-walk}]
By Lemma~\ref{lem-flow-line-times} and the strong Markov property, the sequences of random walk paths
\eqb \label{eqn-quad-increment}
\left\{ \left(\mcl Z(\cdot + N_k^L) - \mcl Z(N_k^L)\right)|_{[0,N_{k+1}^R - N_k^L]} \right\}_{k\geq 1}
\quad \text{and} \quad
\left\{ \left(\mcl Z(\cdot + N_k^R) - \mcl Z(N_k^R)\right)|_{[0,N_k^L - N_k^R]} \right\}_{k\geq 1}
\eqe
are independent. By definition of $N_k^L$ and $N_k^R$ in Lemma~\ref{lem-flow-line-times}, the first (resp.\ second) sequence in~\eqref{eqn-quad-increment} has the law of a collection of i.i.d.\ samples from the law of $\mcl Z$ stopped at the first time that $\mcl R$ (resp.\ $\mcl L$) hits $-1$.
Consequently, concatenating the elements of the first (resp.\ second) sequence in~\eqref{eqn-quad-increment} in order gives a walk $\rng{\mcl Z}' = (\rng{\mcl L}',\rng{\mcl R}')$ (resp.\ $\rng{\mcl Z} = (\rng{\mcl L},\rng{\mcl R})$) with the same law as $\mcl Z|_{[0,\infty)}$. Furthermore, the walks $\rng{\mcl Z}$ and $\rng{\mcl Z}'$ are independent.

The $k$th path $\left(\mcl Z(\cdot + N_k^L) - \mcl Z(N_k^L)\right)|_{[0,N_{k+1}^R - N_k^L]}$ from the first sequence in~\eqref{eqn-quad-increment} coincides with the segment of $\rng{\mcl Z}'$ between the first time that $\jb{\rng{\mcl R}'}$ hits $-(k-1)$ and the first time that $\jb{\rng{\mcl R}'}$ hits $-k$.
Similarly, the $k$th path $\left(\mcl Z(\cdot + N_k^R) - \mcl Z(N_k^R)\right)|_{[0,N_k^L - N_k^R]}$ from the second sequence in~\eqref{eqn-quad-increment} coincides with the segment of $\rng{\mcl Z}$ between the first time that $\rng{\mcl L}$ hits $-(k-1)$ and the first time that $\rng{\mcl L}$ hits $-k$.

From \eqref{eq:ref-for-lat-disc2} and \eqref{eq:ref-for-lat-disc} in Lemma~\ref{lem-flow-line-times}, we have that the last step of the two paths
\[
\left(\mcl Z(\cdot + N_k^L) - \mcl Z(N_k^L)\right)|_{[0,N_{k+1}^R - N_k^L]}
\qquad\text{and}\qquad
\left(\mcl Z(\cdot + N_k^R) - \mcl Z(N_k^R)\right)|_{[0,N_k^L - N_k^R]}
\]
are respectively
\[
\mcl Z(N_{k+1}^R) - \mcl Z(N_{k+1}^R-1) = (1,-1)
\qquad\text{and}\qquad
\mcl Z(N_k^L) - \mcl Z(N_k^L-1) = (-1,0).
\]
Hence, recalling \eqref{eqn-quad-pre3}, $\wh{\mcl Z}'$ can be obtained from $\rng{\mcl Z}'$ in the following manner. We set $\wh{\mcl Z}'$ to be the walk starting at $(0,0)$, taking a first $\jb{(0,1)}$ step (this is because the first walk path in \eqref{eqn-quad-pre3} has an additional starting $\jb{(0,1)}$ step compared to the first walk path in the first sequence in \eqref{eqn-quad-increment}), and then exactly following the steps of the following modified version of $\rng{\mcl Z}'$. For each of the times at which $\rng{\mcl R}'$ attains a \jb{new} running minimum, replace the corresponding step $\mcl Z(N_{k+1}^R) - \mcl Z(N_{k+1}^R-1)=(1,-1)$, which is the last step of the $k$th path in the first sequence in~\eqref{eqn-quad-increment}, by the step $\mcl Z(N_{k+1}^L) - \mcl Z(N_{k+1}^L-1) + (1,1)=(0,1)$, which is the first step of the $k+1$th excursion in~\eqref{eqn-quad-pre3}.

Similarly, recalling \eqref{eqn-quad-pre2}, we find that $\wh{\mcl Z}$ can be obtained from $\rng{\mcl Z}$ in the following manner. For each of the times at which $\rng{\mcl L}$ attains a \jb{new} running minimum, replace the corresponding step $\mcl Z(N_k^L) - \mcl Z(N_k^L-1)=(-1,0)$, which is the last step of the $k$th excursion in the second sequence in~\eqref{eqn-quad-increment}, by the step $\mcl Z(N_{k+1}^R) - \mcl Z(N_{k+1}^R-1)=(1,-1)$, which is the first step of the $k+1$th excursion in~\eqref{eqn-quad-pre2}. Note that in this case we do not need to do any special modification for the first step of the walk since the first step of the first walk path in \eqref{eqn-quad-pre2} is $\mcl Z(1)-\mcl Z(0)$ and the first step of the first walk path in the second sequence in \eqref{eqn-quad-increment} is $\mcl Z(N_1^R+1)-\mcl Z(N_1^R)=\mcl Z(1)-\mcl Z(0)$ since $N_1^R=0$ by Lemma~\ref{lem-flow-line-times}.

Since $\rng{\mcl Z}$ and $\rng{\mcl Z}'$ are independent and $\wh{\mcl Z}$ and $\wh{\mcl Z}'$ are functions of $\rng{\mcl Z}$ and $\rng{\mcl Z}'$, respectively, $\wh{\mcl Z}$ and $\wh{\mcl Z}'$ are independent.

The marginal law of $\rng{\mcl R}'$ is that of a random walk with i.i.d.\ increments each sampled uniformly from $\{-1,0,1\}$. Therefore, the \jb{construction above} together with Lemma~\ref{lem-lazy-walk-flip} implies that \jb{$(\wh{\mcl R}'(\cdot+1)-1)|_{[0,\infty)}$ has the same law as $\mcl R$ conditioned to stay non-negative, while $\wh{\mcl R}'(0)=0$ and $\wh{\mcl R}'(1)=1$.} Since the steps of $\wh{\mcl R}'$ determine the steps of $\wh{\mcl L}'$, this implies that $\wh{\mcl Z}'$ has the desired law.
Similarly, we obtain that $\wh{\mcl Z}$ has the desired law.
\end{proof}

\begin{proof}[Proof of Proposition~\ref{prop-future-map-law}] 
Thanks to Lemma~\ref{lem-quad-kmsw}, the maps $\wh M_{0,\infty}$ and $\wh M_{0,\infty}'$ can be built using the KMSW procedure on the walks $\wh{\mcl Z}$ and $\wh{\mcl Z}'$ of Lemma~\ref{lem-quadrant-walk}, respectively. Hence, Lemma~\ref{lem-quadrant-walk} implies that $\wh M_{0,\infty}$ and $\wh M_{0,\infty}'$ are independent.  Furthermore, by Definition~\ref{def-uiqbot} and the description of the law of $\wh{\mcl Z}$ from Lemma~\ref{lem-quadrant-walk}, we get that $\wh M_{0,\infty}$ has the law of the UIQBOT. 
\end{proof}

\subsection{Proof of Proposition~\ref{prop-uihbot-sym}: Symmetry for the UIBHBOT}
\label{sec-uihbot-sym}

Recall the set of boundary-channeled bipolar-oriented triangulations $\mcl M_{\el,r}^\bc$ with $\el$ left boundary edges and $r$ right boundary edges from Definition~\ref{def-reverse-map}.
For $n\in\jb{\BB N_0}$, let $\mcl M_{\el,r}^\bc(n)$ be the set of maps in $\mcl M_{\el,r}^\bc$ having $n$ non-boundary edges.
To prove Proposition~\ref{prop-uihbot-sym}, we will first show that the UIBHBOT arises as the local limit of uniformly sampled elements of $\mcl M_{\el,1}^\bc(n)$. Proposition~\ref{prop-uihbot-sym} will then be an easy consequence of the left/right symmetry property of $\mcl M_{\el,r}^\bc$ (Lemma~\ref{lem-bdy-reverse}).

The following proposition justifies the terminology in Definition~\ref{def-uihbot}.

\begin{prop} \label{prop-disk-bs-conv}
Let $\{\el_n\}_{n\in\BB N}$ be a sequence of positive integers. Assume there exists $A > 1$ such that $\el_n \in [A^{-1} n^{1/2}, A n^{1/2}]$ for each $n\in\BB N$. \jb{Assume also that $n-\el_n\in3\BB N_0$ for every $n\in\BB N$.}
For $n\in\BB N$, let $M_n^\bc$ be the map obtained by starting with a uniform sample from $\mcl M_{\el_n+1,1}^\bc(n-1)$ (as defined just above), then declaring all of its left boundary edges, except for the left boundary edge whose terminal vertex is the sink vertex, to be missing.
Conditional on $M_n^\bc$, let $e_n^\bc$ be a uniform sample from the $\el_n$ missing edges on the (lower-left) boundary of $M_n^\bc$.
Then $(M_n^\bc,e_n^\bc)$ converges in law to the UIBHBOT of Definition~\ref{def-uihbot} with respect to the (directed) Benjamini-Schramm topology (Definition~\ref{def-uihbot}).
\end{prop}

We will prove Proposition~\ref{prop-disk-bs-conv} via a standard argument based on Cram\'er's large deviation theorem~\cite[Theorem 2.1.24]{dembo-ld}; see, e.g.,~\cite[Section 4.2]{shef-kpz} for a similar argument.

Let $\mcl Z_n^\bc=(\mcl L_n^\bc,\mcl R_n^\bc)$ be the walk which encodes $M_n^\bc$ via the KMSW bijection as in Lemma~\ref{lem-kmsw-left}.
By Lemma~\ref{lem-kmsw-left}, the walk $\mcl Z_n^\bc$ is a uniform sample from the set of walks with increments in the set $\{(1,-1),(-1,0),(0,1)\}$ which start at $(\el_n,0)$, stay in $[0,\infty)^2$ until time $n$, and end up at $(0,0)$ at time $n$. Let $K_n\in[1,\el_n]\cap\BB Z$ be chosen so that $e_n^\bc$ is the $K_n$th missing edge in order on the lower-left boundary of $M_n^\bc$ (starting from the missing edge next to the unique lower-right boundary edge). Since $e_n^\bc$ is chosen uniformly from the missing boundary edges of $M_n^\bc$, the integer $K_n$ is a uniform sample from $[1,\el_n]\cap\BB Z$, independent from $\mcl Z_n^\bc$. Let $T_n$ be the stopping time
\eqb \label{eqn-disk-walk-time}
T_n := \min\left\{j\in[1,n]\cap\BB Z:\mcl L_n^\bc(j)-\mcl L_n^\bc(0)=-K_n\right\}.
\eqe
By Assertion~\ref{item-lower-left} of Lemma~\ref{lem-kmsw-bdy}, $T_n$ is the same as the time at which the KMSW exploration associated with $\mcl Z_n^\bc$ explores the triangle of $M_n^\bc$ with $e_n^\bc$ on its boundary.

The main technical step in the proof of Proposition~\ref{prop-disk-bs-conv} is to show that the probability of the event that we condition on to get the law of $\mcl Z_n^\bc$ decays slower than exponentially in $\el_n$. We will in fact prove a power-law lower bound.

\begin{lem} \label{lem-cone-prob-lower}
Recall from Section~\ref{sec-uiqbot} that $\BB P_{(\el,0)}$ denotes the law of a walk $\mcl Z$ with i.i.d.\ increments sampled uniformly from $\{(1,-1),(-1,0),(0,1)\}$, started from $(\el,0)$.
For $n\in\BB N$, let
\eqb \label{eqn-walk-cone-event}
E_n := \left\{\mcl Z[0,n]\subset[0,\infty)^2,\:\mcl Z(n)=(0,0)\right\}.
\eqe
For each $A>1$, there exist $c>0$ such that
\eqb \label{eqn-cone-prob-lower}
\BB P_{(\el_n,0)}[E_n]\geq c n^{-3},
\quad\forall n\in\BB N,
\quad\forall\el_n\in[A^{-1}n^{1/2},An^{1/2}]\cap\BB Z
\jb{\text{ such that }n-\el_n\in3\BB N_0}.
\eqe
\end{lem}

\begin{proof}
This is a fairly straightforward consequence of the results of~\cite{dw-cones,dw-limit}, but these results do not exactly apply in our setting so we will explain how to extract~\eqref{eqn-cone-prob-lower} from the results of~\cite{dw-cones,dw-limit}. Whenever we apply results from~\cite{dw-cones,dw-limit}, we need to apply an affine transformation which maps $\mcl Z$ to an uncorrelated random walk, see~\cite[Example 2]{dw-cones}. In our setting, the correlation of the coordinates of $\mcl Z$ is $-1/2$, so the parameter $p$ from~\cite{dw-cones,dw-limit} is $3$.

\jb{It suffices to prove the lemma for all sufficiently large $n$, depending on $A$. Indeed, for every admissible pair $(n,\el_n)$, the event $E_n$ has positive probability: one may take $\el_n$ steps $(-1,0)$ followed by $(n-\el_n)/3$ copies of the three-step path $(0,1),(1,-1),(-1,0)$. Once the bound is proved for sufficiently large $n$, decreasing $c$ handles the finitely many remaining admissible pairs.}
Let $n_1,n_2,n_3\in\BB N$ be chosen so that
\eqb \label{eqn-n-split}
n_1+n_2+n_3=n,\quad n_1,n_2,n_3\geq n/4,\quad n_2\in3\BB Z.
\eqe
The reason for requiring that $n_2\in3\BB Z$ is that $\mcl Z$ is periodic of period 3.

We will lower-bound the probabilities of three events, depending on the restrictions of $\mcl Z$ to intervals of length $n_1,n_2,n_3$, whose intersection is contained in $E_n$.
To this end, we define
\alb
\mcl Z^1(j)&:=\mcl Z(j),\quad j\in[0,n_1]\cap\BB Z\\
\mcl Z^2(j)&:=\mcl Z(j+n_1)-\mcl Z(n_1),\quad j\in[0,n_2]\cap\BB Z\\
\mcl Z^3(j)&:=\mcl Z(n-j)-\mcl Z(n),\quad j\in[0,n_3]\cap\BB Z.
\ale
Then $\mcl Z^1,\mcl Z^2$, and $\mcl Z^3$ are independent, and the laws of $\mcl Z^2$ and $\mcl Z^3$ do not depend on the starting point for $\mcl Z$.

Define the events
\alb
E_n^1&:=\left\{\mcl Z^1(j)\in[0,\infty)^2,\:\forall j\in[0,n_1]\cap\BB Z\right\}
\cap\left\{\mcl Z^1(n_1)\in[n^{1/2},2n^{1/2}]^2\right\}\notag\\
E_n^2&:=\left\{\mcl Z^2(j)+\mcl Z^1(n_1)\in[0,\infty)^2,\:\forall j\in[0,n_2]\cap\BB Z\right\}
\cap\left\{\mcl Z^2(n_2)=\mcl Z^3(n_3)-\mcl Z^1(n_1)\right\}\notag\\
E_n^3&:=\left\{\mcl Z^3(j)\in[0,\infty)^2,\:\forall j\in[0,n_3]\cap\BB Z\right\}
\cap\left\{\mcl Z^3(n_3)\in[3n^{1/2},4n^{1/2}]^2\right\}.
\ale
Then $E_n^1\cap E_n^2\cap E_n^3\subset E_n$.
We also note that $E_n^1$ and $E_n^3$ are determined by $\mcl Z^1$ and $\mcl Z^3$, respectively, so these events are independent.

We first consider the event $E_n^1$. By a standard estimate for one-dimensional random walk, there exists $c_1>0$ such that for all $\el,n\in\BB N$,
\eqb \label{eqn-walk-cone1'}
\BB P_{(0,0)}\left[F_n^1\right]\geq c_1n^{-1/2},
\quad\text{where}\quad
F_n^1:=\left\{\mcl R(j)\geq0,\:\forall j\in[0,n_1]\cap\BB Z\right\}.
\eqe
By~\cite[Theorem 1]{dw-limit}, applied with the cone equal to the upper half-plane, under $\BB P_{(0,0)}[\cdot\,|\,F_n^1]$, the re-scaled walk $n_1^{-1/2}\mcl Z(\lfloor n_1\cdot\rfloor)|_{[0,1]}$ converges as $n_1\to\infty$ to a correlated two-dimensional Brownian motion (with correlation $-1/2$) started at $(0,0)$ and conditioned to stay in the upper half-plane for one unit of time.
For any $\el_n\in[A^{-1}n^{1/2},An^{1/2}]$ and each small enough $\ep>0$ (depending on $A$), such a conditioned Brownian motion has a positive chance (uniformly over the choice of $\el_n$) to stay in $[-n_1^{-1/2}\el_n+\ep,\infty)\times[0,\infty)$ until time 1 and end up in $(-n_1^{-1/2}\el_n,0)+[2(1+\ep),\sqrt{2}(2-\ep)]^2$ at time 1.
By shifting by $(\el_n,0)$, noting that $n^{1/2}/2\leq n_1^{1/2}\leq n^{1/2}/\sqrt{2}$ by \eqref{eqn-n-split}, and applying the preceding sentence together with~\eqref{eqn-walk-cone1'}, we find that there exists $c_2>0$ such that for \jb{all sufficiently large} $n\in\BB N$ and $\el_n\in[A^{-1}n^{1/2},An^{1/2}]\cap\BB Z$,
\eqb \label{eqn-walk-cone1}
\BB P_{(\el_n,0)}\left[E_n^1\right]\geq c_2n^{-1/2}.
\eqe

Next we consider $E_n^3$.
By~\cite[Theorem 1]{dw-cones}, applied to the time reversal of $\mcl Z$, there exists $c_3>0$ such that for any $n\in\BB N$,
\eqb \label{eqn-reverse-cone}
\BB P\left[\mcl Z^3(j)\in[0,\infty)^2,\:\forall j\in[0,n_3]\cap\BB Z\right]\geq c_3n^{-3/2}.
\eqe
By~\cite[Theorem 1]{dw-limit}, applied as in the argument just above~\eqref{eqn-walk-cone1}, this implies that we can find $c_4>0$ such that for \jb{all sufficiently large} $n\in\BB N$ and $\el_n\in[A^{-1}n^{1/2},An^{1/2}]\cap\BB Z$,
\eqb \label{eqn-walk-cone2}
\BB P_{(\el_n,0)}\left[E_n^3\right]\geq c_4n^{-3/2}.
\eqe

To deal with $E_n^2$, we first observe that on $E_n^1\cap E_n^3$,
\eqbn
\mcl Z^3(n_3)-\mcl Z^1(n_1)\in[n^{1/2},3n^{1/2}]^2.
\eqen
Recall that $\mcl Z^2$ is independent from $\jb{(\mcl Z^1,\mcl Z^3)}$.
\jb{Let
\[
\Lambda:=\{(x,y)\in\BB Z^2:x-y\in3\BB Z\}.
\]
The increments of $\mcl Z^2(3\cdot)$ take values in $\Lambda$, and their support contains $(0,0)$, $(1,1)$, and $(-1,2)$, with the latter two vectors generating $\Lambda$. Hence this walk is strongly aperiodic when viewed as a walk on $\Lambda$. Moreover, since every forward increment changes the difference of the coordinates by $-1$ modulo $3$, the assumptions $n-\el_n\in3\BB N_0$ and $n_2\in3\BB Z$ imply that $\mcl Z^3(n_3)-\mcl Z^1(n_1)\in\Lambda$.}
Since $n_2$ was chosen to be a multiple of 3, we can apply a standard local limit result for unconditioned random walk \jb{on $\Lambda$} (see, e.g.,~\cite[Section 7, P9]{spitzer-walks}) to get that there exists $c_5>0$ such that on $E_n^1\cap E_n^3$,
\eqb \label{eqn-walk-cone-endpt}
\BB P_{(\el_n,0)}\left[
\mcl Z^2(n_2)=\mcl Z^3(n_3)-\mcl Z^1(n_1)
\,\middle|\,\jb{\mcl Z^1,\mcl Z^3}
\right]\geq c_5n^{-1}.
\eqe
We now use the convergence of random walk conditioned on its endpoints to two-dimensional Brownian bridge (see, e.g.,~\cite[Theorem 3]{liggett-bridge-vector}), together with the fact that for any $S>1$ and $z,w\in[1/S,S]^2$, a (correlated) Brownian bridge from $z$ to $w$ has a uniformly positive chance to stay in $[0,\infty)^2$. We infer from~\eqref{eqn-walk-cone-endpt} that there exists $c_6>0$ such that on $E_n^1\cap E_n^3$,
\eqb \label{eqn-walk-cone3}
\BB P_{(\el_n,0)}\left[E_n^2\,\middle|\,\jb{\mcl Z^1,\mcl Z^3}\right]\geq c_6n^{-1}.
\eqe
% Let $p_{n,x}$ be the probability for random walk bridge from $0$ to $x$ in $n$ steps to stay in $[-n^{1/2}/S,\infty)^2$. If $n^{-1/2}x_n\to z$, then $p_{n,x_n}$ converges to the corresponding probability for Brownian bridge. This probability is uniformly positive over $z\in[1/S,S]^2$. If $p_{n,x_n}\to0$ for some sequence of $x_n$, we can pass to a subsequence along which $n^{-1/2}x_n$ converges to some $z\in[1/S,S]^2$, which gives a contradiction.
Combining~\eqref{eqn-walk-cone1}, \eqref{eqn-walk-cone2}, and~\eqref{eqn-walk-cone3} and recalling that $E_n^1$ and $E_n^3$ are independent concludes the proof.
\end{proof}

Recall the discussion above \eqref{eqn-disk-walk-time}.

\begin{lem} \label{lem-disk-walk-conv}
For $n\in\BB N$, let $E_n$ be as in~\eqref{eqn-walk-cone-event}.
Let $\{\el_n\}_{n\in\BB N}$ be a sequence of positive integers. Assume there exists $A>1$ such that $\el_n\in[A^{-1}n^{1/2},An^{1/2}]$ for each $n\in\BB N$. \jb{Assume also that $n-\el_n\in3\BB N_0$ for every $n\in\BB N$.}
Let $\mcl Z_n^\bc=(\mcl L_n^\bc,\mcl R_n^\bc)$ be a walk from $(\el_n,0)$ to $(0,0)$ in $[0,\infty)^2$ sampled from the $\BB P_{(\el_n,0)}$-conditional law of $\mcl Z$ given $E_n$.
Also let $K_n$ be sampled uniformly from $[1,\el_n]\cap\BB Z$, independently from $\mcl Z_n^\bc$ and let $T_n$ be as in \eqref{eqn-disk-walk-time}.
% \eqb \label{eqn-disk-walk-time}
% T_n := \min\left\{j\in[1,n]\cap\BB Z:\mcl L_n(j)\jb{-\mcl L_n(0)}=-K_n\right\}.
% \eqe
\jb{In this lemma and its proof, walks are extended constantly outside their time domains whenever needed.}
For each fixed $N\in\BB N$, the law of
\eqb \label{eqn-disk-walk-shift}
\Big(\mcl Z_n^\bc(\cdot+T_n)-\mcl Z_n^\bc(T_n)\Big)|_{[-N,N]}
\eqe
converges in the total variation sense as $n\to\infty$ to the law of $\mcl Z^\bc|_{[-N,N]}$, where $\mcl Z^\bc$ is the encoding walk for the UIBHBOT, as defined at the beginning of Section~\ref{sec-busemann-sym}.
\end{lem}

\begin{proof}
Fix $m\in\BB N$. We will send $n\to\infty$ then send $m\to\infty$. We first study the unconditional law of $\mcl Z$ sampled from $\BB P_{(\el_n,0)}$.
Let $\sigma_0=0$ and for $k\in\BB N$, let
\eqb \label{eq:stop-cramer}
\sigma_k:=\min\left\{j\in\BB N:\mcl L(j)-\mcl L(0)=-k\right\}.
\eqe
Note that $T_n$ in \eqref{eqn-disk-walk-time} is equal to $\sigma_{K_n}$ (with $\mcl L$ replaced by $\mcl L_n^\bc$).
Furthermore, for $s\in\BB N$, the stopped walk paths
\eqb \label{eqn-cramer-increment}
Z_s:=\left(\mcl Z(\cdot+\sigma_{(s-1)m})-\mcl Z(\sigma_{(s-1)m})\right)|_{[0,(\sigma_{sm}-\sigma_{(s-1)m})\wedge m^{100}]}
\eqe
are i.i.d.\ paths. The reason why we stop after $m^{100}$ steps is to ensure that $Z_s$ takes values in a finite set (namely, the set of paths in $\BB Z^2$ started from $(0,0)$ with increments in $\{(1,-1),(0,1),(-1,0)\}$ which go for at most $m^{100}$ steps). We also note that the definition of the stopping times $\sigma_k$ in~\eqref{eq:stop-cramer} and the increments in~\eqref{eqn-cramer-increment} does not depend on the starting point for $\mcl Z$.

By Cram\'er's larger deviation theorem~\cite[Theorem 2.1.24]{dembo-ld}, the empirical distribution for $Z_1,\dots,Z_{\lfloor\el_n/m\rfloor}$ is exponentially concentrated around the actual distribution of $Z_1$ as $n\to\infty$. That is, for any possible realization $\frk z$ of $Z_1$, let $p_{l_n}(\frk z)$ be the fraction of values of $s\in[1,\el_n/m]\cap\BB Z$ for which $Z_s=\frk z$. Then for each $\delta>0$, there exists $c=c(m,\delta)>0$ such that\footnote{The law of $Z_s$ in~\eqref{eqn-cramer-increment} does not depend on \jb{the starting point for $\mcl Z$}, so~\eqref{eqn-cramer} is equivalent to the analogous statement with $\BB P_{(r,0)}$ in place of $\BB P_{(\el_n,0)}$ for any $r\in\BB Z$.}
\eqb \label{eqn-cramer}
\BB P_{(\el_n,0)}\left[
\sup_{\frk z}\left|p_{l_n}(\frk z)-\BB P_{(\el_n,0)}[Z_1=\frk z]\right|>\delta
\right]
=O(e^{-c\,\el_n}),\quad\text{as $n\to\infty$},
\eqe
where the supremum is over all possible realizations $\frk z$ of $Z_1$.

By Lemma~\ref{lem-cone-prob-lower}, $\BB P_{(\el_n,0)}[E_n]$ \jb{is bounded below by a negative power of $n$}. In particular, since $\el_n$ is comparable to $n^{1/2}$, this probability decays slower than exponentially in $\el_n$. Hence~\jb{\eqref{eqn-cramer}} implies that
\eqb \label{eqn-cramer-cond}
\BB P_{(\el_n,0)}\left[
\sup_{\frk z}\left|p_{l_n}(\frk z)-\BB P_{(\el_n,0)}[Z_1=\frk z]\right|>\delta
\,\middle|\,E_n
\right]
=O(e^{-c'\el_n}),\quad\text{as $n\to\infty$},
\eqe
for a slightly smaller constant $c'=c'(\delta,m)>0$.

Let $K_n$ be as in the lemma statement and \jb{set $S_n=\lceil K_n/m\rceil$ when $K_n\le m\lfloor\el_n/m\rfloor$, and otherwise take $S_n$ to be a graveyard point. On the former event, set $U_n=K_n-(S_n-1)m$. Conditional on this event, $U_n$ is uniform on $[1,m]\cap\BB Z$ and is independent of $S_n$ and the walk, also conditional on $E_n$. Events below involving $S_n$ as an integer are understood to require that it is not the graveyard point.}
\jb{By~\eqref{eqn-cramer-cond}, the conditional law of $(Z_{S_n},U_n)$ given $E_n$ converges in total variation to the product of the unconditional law of $Z_1$ and the uniform law on $[1,m]\cap\BB Z$, where we assign the pair a fixed graveyard value when $S_n$ is the graveyard point. Indeed, the set of possible realizations of $Z_1$ is finite, and the probability of the graveyard event is at most $m/\el_n$.}

It remains to transfer from a statement about the law of $Z_{S_n}$ to a statement about the law of~\eqref{eqn-disk-walk-shift}.
To this end, let $\ep\in(0,1)$ and $N\in\BB N$.
Since $\sigma_{sm}$ is the hitting time of $-sm$ for the lazy random walk $\mcl L-\mcl L(0)$, if we choose $m$ large enough (depending on $\ep$) then for every $s\in[1,\el_n/m]\cap\BB Z$, the unconditional $\BB P_{(\el_n,0)}$-probability that $\sigma_{sm}-\sigma_{(s-1)m}$ is larger than $m^{100}$ is at most $\ep/2$. By the previous paragraph, if $m$ is chosen to be large enough (depending on $\ep$) and $n$ is large enough (depending on $m$ and $\ep$), then
\eqb \label{eqn-walk-conv-trunc}
\BB P_{(\el_n,0)}\left[
\sigma_{S_nm}-\sigma_{(S_n-1)m}\leq m^{100}
\,\middle|\,E_n
\right]\geq1-\ep.
\eqe
Furthermore, if $m$ is large enough (depending on $\ep$ and $N$) and $n$ is large enough (depending on $N$, $m$, and $\ep$), then the conditional probability given $E_n$ that $[K_n-N,K_n+N]\subset[(S_n-1)m,S_nm]$ is at least $1-\ep$.
By this and~\eqref{eqn-walk-conv-trunc}, we can choose $m=m(\ep,N)\in\BB N$ large enough so that when $n$ is sufficiently large (recall that $T_n$ in \eqref{eqn-disk-walk-time} is equal to $\sigma_{K_n}$),
\eqbn
\BB P_{(\el_n,0)}\left[
[T_n-N,T_n+N]\subset
\left[\sigma_{(S_n-1)m},
\sigma_{(S_n-1)m}+(\sigma_{S_nm}-\sigma_{(S_n-1)m})\wedge m^{100}\right]
\,\middle|\,E_n
\right]\geq1-2\ep.
\eqen
In other words, under the conditional law given $E_n$, it holds with probability at least $1-2\ep$ that the walk path~\eqref{eqn-disk-walk-shift} (defined with $\mcl Z$ in place of $\mcl Z_n^\bc$) is determined by $Z_{S_n}$ and $K_n-(S_n-1)m$.
\jb{Combining this with the joint convergence above and recalling that $\mcl Z_n^\bc$ is sampled from the conditional law of $\mcl Z$ given $E_n$, we obtain the following. For sufficiently large $m$ and then sufficiently large $n$, the total variation distance between the law of~\eqref{eqn-disk-walk-shift} and the law of $(\mcl Z(\cdot+\sigma_M)-\mcl Z(\sigma_M))|_{[-N,N]}$ is $O(\ep)$, where $M$ is sampled uniformly from $[1,m]\cap\BB Z$, independently of $\mcl Z$. Here we use the same block-containment and truncation estimates for the unconditional comparison path. For each deterministic $M\ge N$, this centered window has exactly the law of $\mcl Z^\bc|_{[-N,N]}$: its past is covered by the last $N$ independent descending-ladder excursions, and its future is an independent walk, as in the construction at the beginning of Section~\ref{sec-busemann-sym}. Since $\BB P[M<N]\le N/m$, sending $n\to\infty$, then $m\to\infty$ and $\ep\to0$, concludes the proof.}
\end{proof}

\begin{proof}[Proof of Proposition~\ref{prop-disk-bs-conv}]
This is an easy consequence of Lemma~\ref{lem-disk-walk-conv}, but we spell out the details for completeness.
Recall the discussion above \eqref{eqn-disk-walk-time}.
Let $(M^\bc,e^\bc)$ be the UIBHBOT and let $\mcl Z^\bc$ be its KMSW encoding walk, as in the discussion just above Definition~\ref{def-uihbot}.
For $a,b\in\BB Z$ with $a<b$, let $M_{a,b}^\bc$ be the submap of $M^\bc$ obtained by applying the KMSW procedure to $(\mcl Z^\bc(\cdot+a)-\mcl Z^\bc(a))|_{\jb{[0,b-a]}}$.
For $r\in\BB N$, define the \textbf{graph distance ball} $\mcl B_r(M^\bc)$ to be the submap of $M^\bc$ consisting of the vertices of $M^\bc$ which lie at graph distance at most $r$ from the initial vertex of $e^\bc$, along with the edges of $M^\bc$ whose vertices both satisfy this condition, and the faces of $M^\bc$ whose boundary vertices all satisfy this condition.
For any $\ep>0$ and $r\in\BB N$, we can choose $N=N(r,\ep)\in\BB N$ sufficiently large so that
\eqbn
\BB P\left[\mcl B_r(M^\bc)\subset M_{-N,N}^\bc\right]\geq1-\ep.
\eqen
By the total variation convergence in Lemma~\ref{lem-disk-walk-conv}, this implies that for large enough $n\in\BB N$, it holds with probability at least $1-2\ep$ that the graph distance ball $\mcl B_r(M_n^\bc)$ of radius $r$ in $M_n^\bc$ centered at the initial vertex of $e_n^\bc$ is completely explored by the KMSW procedure for $M_n^\bc$ during the time interval $[T_n-N,T_n+N]\cap\BB Z$ (the latter is determined by the walk $\mcl Z_n^\bc|_{[T_n-N,T_n+N]}$). Note that Lemma~\ref{lem-disk-walk-conv} guarantees that we can choose the same $N$ for all values of $n$.
By applying Lemma~\ref{lem-disk-walk-conv} again, this implies the desired Benjamini-Schramm convergence.
\end{proof}

\begin{proof}[Proof of Proposition~\ref{prop-uihbot-sym}]
Assume that we are in the setting of Proposition~\ref{prop-disk-bs-conv}.
Let $(\wt M_n^\bc,\wt e_n^\bc)$ be the edge-rooted directed map obtained from $(M_n^\bc,e_n^\bc)$ by applying an orientation-reversing homeomorphism from $\BB C$ to $\BB C$.
By the definition of $M_n^\bc$ (see Proposition~\ref{prop-disk-bs-conv}), the map $\wt M_n^\bc$ can be obtained by starting with a uniform sample from $\mcl M_{1,\el_n+1}^\bc(\jb{n-1})$ (as defined just above Proposition~\ref{prop-disk-bs-conv}), then declaring all of the right boundary edges to be missing except for the one whose terminal vertex is the sink vertex.
Furthermore, $\wt e_n^\bc$ is a uniformly sampled missing boundary edge of $\wt M_n^\bc$.
By Lemma~\ref{lem-bdy-reverse} (applied with $(\el,r)=(1,\el_n+1)$ and $k=\el_n$),
\eqb \label{eqn-use-bdy-reverse}
(M_n^\bc,e_n^\bc)\eqD(\wt M_n^\bc,\wt e_n^\bc).
\eqe

By Proposition~\ref{prop-disk-bs-conv}, $(M_n^\bc,e_n^\bc)\to(M^\bc,e^\bc)$ in law with respect to the Benjamini-Schramm topology. By applying an orientation-reversing homeomorphism, we get that also $(\wt M_n^\bc,\wt e_n^\bc)\to(\wt M^\bc,\wt e^\bc)$ in law with respect to the Benjamini-Schramm topology, where $(\wt M^\bc,\wt e^\bc)$ is as in the proposition statement.
By~\eqref{eqn-use-bdy-reverse}, this shows that $(M^\bc,e^\bc)\eqD(\wt M^\bc,\wt e^\bc)$, as required.
\end{proof}

%% file: tex/recursive.tex
In this section, we prove Theorem~\ref{thm-busemann-tail} about the tail asymptotics of the increments of the Busemann function $\mcl X$  when $\XDP \in\{ \op{LDP},\op{SDP}\}$. We will prove Items~\ref{item-busemann-tail-LDP}~and~\ref{item-busemann-tail-SDP} of Theorem~\ref{thm-busemann-tail} separately in Sections~\ref{sec-recursive-LDP}~and~\ref{sec-recursive-SDP}, respectively.

Recall from Lemma~\ref{lem-uibot-bdy} that for $n\in\BB Z$, the infinite rooted bipolar-oriented triangulation with missing edges $(M_{n,\infty},\lambda_n)$ is the submap of the UIBOT obtained by applying the KMSW procedure to $\mcl Z|_{[n,\infty)}$, where  $\mcl Z : \BB Z\to\BB Z^2$ is a bi-infinite random walk whose increments are i.i.d.\ uniform samples from $\{(1,-1), (-1,0) , (0,1)\}$, normalized so that $\mcl Z(0) = (0,0)$. 

By the stationarity of $\mcl Z$, we have that 
\begin{equation}\label{eq:stationarity-rel}
    M_{1,\infty} \eqD M_{0,\infty}.
\end{equation}
Therefore, the Busemann function $\mcl X^1$ associated with $M_{1,\infty}$ has the same law as $\mcl X$.  Also, recall the notation 
\eqb\label{eq:den-dist-f-g}
f(x) := \BB P\left[ \mcl X(0) - \mcl X(-1) = x \right] \quad \text{and} \quad g(y):= \BB P\left[ \mcl X(1) - \mcl X(0) = y \right] .
\eqe
Since the Busemann function $\mcl X$ introduced in Theorem~\ref{thm-busemann} is normalized so that $\mcl X(0)=0$, we equivalently have that $f(x) = \BB P\left[ - \mcl X(-1) = x \right]$ and $g(y) = \BB P\left[ \mcl X(1)  = y \right] $.
Our main goal will be to determine the tail behavior of $f$ and $g$ for $\XDP \in\{ \op{LDP},\op{SDP}\}$. 

Using the fact that $\mcl X^1 \eqD \mcl X$ and the properties in Theorem~\ref{thm-busemann-property} proved in the previous section, we will obtain a recursive equation for $f$ and $g$ (see Proposition~\ref{prop:LDP-recursive-equation} for the $\op{LDP}$ case and Proposition~\ref{prop:SDP-recursive-equation} for the $\op{SDP}$ case), which will later be converted into a functional equation for the characteristic functions $F$ and $G$ of $f$ and $g$, respectively. 
We then reach our goal by first studying the asymptotics of the characteristic functions $F$ and $G$ (see Proposition~\ref{prop:asympt-exp2} for the $\op{LDP}$ case and Proposition~\ref{prop:asympt-exp} for the $\op{SDP}$ case) and then transferring the estimates to $f$ and $g$ using the following Tauberian type result, whose proof can be found in Appendix~\ref{sect:proofTau}.

\begin{prop}\label{prop:Tauberian}
Let $X$ be a non-negative integer-valued random variable. Define the characteristic function $\varphi(t) = \BB E[e^{itX}]$. Consider the following two sets of assumptions:
\begin{enumerate}
    \item\label{ass:1} For some $c \in \BB C \setminus i \BB R$ and $\nu \in (0,1)$,  
\eqb \label{eqn-tauberian2}
\varphi(t) = 1 + (c \BB 1_{(t > 0)} + \ol c \BB 1_{(t < 0)} ) |t|^\nu  + o(t^\nu)  ,\quad \text{as $t \to 0$}.
\eqe
    \item\label{ass:2} For some $c_1 > 0$, $c_2 \in \BB C \setminus  \BB R$, and $\nu \in (0,1)$,  
\eqb \label{eqn-tauberian'2}
\varphi(t) = 1 + i c_1 t +     (c_2 \BB 1_{(t > 0)} + \ol c_2 \BB 1_{(t < 0)} ) |t|^{1+\nu}       + o(t^{1+\nu})  ,\quad \text{as $t \to 0$}. 
\eqe
\end{enumerate}
Then:
\begin{enumerate}
    \item\label{ass1-tau} In the case of Assumption~\ref{ass:1}, $\re(c)<0$ and $\im(c)>0$, and setting $a=-\frac{\re(c)}{\Gamma(1-\nu)}\sec\Big(\tfrac{\pi\nu}{2}\Big)>0$,
    \eqb \label{eqn-prob-asymp2}
\BB P[X > R] = a R^{-\nu}  +  o(R^{-\nu}) ,\quad \text{as $R\to\infty$}. 
\eqe 
    \item\label{ass2-tau} In the case of Assumption~\ref{ass:2}, $\im(c_2)<0$ and setting $a=-\frac{\nu\im(c_2)}{\Gamma(1-\nu)}\sec\left(\tfrac{\pi\nu}{2}\right)>0$,
    \eqb \label{eqn-prob-asymp'2}
\BB P[X > R] = a R^{-\nu-1}  +  o(R^{-\nu-1}) ,\quad \text{as $R\to\infty$}. 
\eqe
\end{enumerate}
\end{prop}

\subsection{The LDP case}\label{sec-recursive-LDP}

We start by looking at the case when $\XDP = \op{LDP}$. Thanks to Items~\ref{item-busemann-pos}~and~\ref{item-busemann-sym} of Theorem~\ref{thm-busemann-property}, we know that $f$ is a symmetric distribution on $\BB Z$, and $g$ is a distribution on $\BB Z_{\leq -1}$. 
We introduce the three characteristic functions
 \begin{equation}\label{eq:FGdefs-new}
 	F(t):=\sum_{x\in\mathbb Z} f(x)\,e^{itx},\qquad
 	F^{-}(t):=\sum_{x \in\BB Z_{\leq -1}} f(x)\,e^{itx},\qquad
 	G(s):=\sum_{y\in \BB Z_{\leq -1}} g(y)\,e^{isy}.
 \end{equation}
 By the symmetry of $f$, we have $F(t)\in\mathbb R$. Moreover, since every characteristic function $\varphi(t)$ satisfies $\varphi(-t)=\overline{\varphi(t)}$, in what follows, we only focus on the case $t\geq 0$.

 \begin{prop}[Asymptotics of the characteristic functions, LDP case]\label{prop:asympt-exp2}
 	Setting 
 	\[\kappa:=3+f(0)>0,\] 
    it holds that as $t\downarrow 0$,
 		\begin{align*}
 			&F(t)= 1-\left(\frac{\kappa}{2}\right)^{1/3}\,t^{2/3}+o(t^{2/3}),\\
 			&F^-(t)=\frac{1}{2}(1-f(0))+\frac{-1 - i\sqrt{3}}{2} \left(\frac{\kappa}{2}\right)^{1/3}\,t^{2/3}+o(t^{2/3}),\\
 			&G(t)=1+\frac{-1 - i\sqrt{3}}{2}\left(\frac{\kappa}{2}\right)^{1/3}\,t^{2/3}+o(t^{2/3}).
 		\end{align*}
 \end{prop}

Before proving Proposition~\ref{prop:asympt-exp2}, we note that it immediately implies Item~\ref{item-busemann-tail-LDP} of Theorem~\ref{thm-busemann-tail}.

\begin{proof}[Proof of Item~\ref{item-busemann-tail-LDP} of Theorem~\ref{thm-busemann-tail} (assuming Proposition~\ref{prop:asympt-exp2})]
We combine Proposition~\ref{prop:asympt-exp2} with the Tauberian result in Proposition~\ref{prop:Tauberian}. The asymptotic expansion of $F(t)$ in Proposition~\ref{prop:asympt-exp2} ensures that $f(0)\neq 1$. Recalling the notation in \eqref{eq:den-dist-f-g}, in the case of $f$, we apply Proposition~\ref{prop:Tauberian} to the random variable $X$ with $\BB P[X=x] = \frac{2}{1-f(0)} f(x)$ for all $x\geq1$, so that $X$ is non-negative (recall that $f(x) = f(-x)$). Since the characteristic function of $X$ for $t>0$ is 
\begin{align*}
    \varphi_X(t)&=\frac{2}{1-f(0)}(F(t)-F^{-}(t)-f(0))\\
    &=1+
    \frac{-1+i\sqrt{3}}{1-f(0)}\left(\frac{3+f(0)}{2}\right)^{1/3}\,t^{2/3}+o(t^{2/3}),
\end{align*}
we obtain that
\begin{align*}
    \BB P [X>R]&=\frac{1}{1-f(0)}\left(\frac{3+f(0)}{2}\right)^{1/3}\frac{1}{\Gamma(1-2/3)}\sec\Big(\tfrac{\pi}{3}\Big) R^{-2/3}+o(R^{-2/3})\\
    &=\frac{\sqrt{3}\,\Gamma(2/3)}{\pi(1-f(0))}\left(\frac{3+f(0)}{2}\right)^{1/3}R^{-2/3}+o(R^{-2/3})
\end{align*}
Hence, recalling that $f(x) = \BB P\left[ \mcl X(0) - \mcl X(-1) = x \right]$ for all $x\in \BB Z$,
\begin{equation}\label{eq:conc-asympt-1}
    \BB P \left[|\mcl X(0)-\mcl X(-1)|>R\right]=\frac{\sqrt{3}\,\Gamma(2/3)}{\pi}\left(\frac{3+f(0)}{2}\right)^{1/3}R^{-2/3}+o(R^{-2/3}).
\end{equation}

In the case of $g$, we apply Proposition~\ref{prop:Tauberian} to the random variable $Y$ with $\BB P[Y=y] = g(-y)$ for all $y\geq1$, so that $Y$ is non-negative.
Since the characteristic function of $Y$ for $t>0$ is 
\begin{equation*}
    \varphi_Y(t)=G(-t)=\overline{G(t)}=1+\frac{-1 + i\sqrt{3}}{2}\left(\frac{3+f(0)}{2}\right)^{1/3}\,t^{2/3}+o(t^{2/3})
\end{equation*}
we obtain that
\begin{align*}
    \BB P [Y>R]&=\frac{1 }{2}\left(\frac{3+f(0)}{2}\right)^{1/3}\frac{1}{\Gamma(1-2/3)}\sec\Big(\tfrac{\pi}{3}\Big) R^{-2/3}+o(R^{-2/3})\\
    &=\frac{\sqrt{3}\,\Gamma(2/3)}{2\pi}\left(\frac{3+f(0)}{2}\right)^{1/3} R^{-2/3}+o(R^{-2/3})
\end{align*}
Hence, recalling that $g(y) = \BB P\left[ \mcl X(1) - \mcl X(0) = y \right]$ for all $y\in \BB Z_{\leq -1}$,
\begin{equation}\label{eq:conc-asympt-2}
    \BB P \left[\mcl X(1)-\mcl X(0)<-R\right]=\frac{\sqrt{3}\,\Gamma(2/3)}{2\pi}\left(\frac{3+f(0)}{2}\right)^{1/3}R^{-2/3}+o(R^{-2/3}).
\end{equation}
Setting $c_1=\frac{\sqrt{3}\,\Gamma(2/3)}{2\pi}\left(\frac{3+f(0)}{2}\right)^{1/3}$ in \eqref{eq:conc-asympt-1}~and~\eqref{eq:conc-asympt-2} gives the expressions in the statement of the theorem.
\end{proof}

\begin{remark}
    Note that we do not know the exact value of $c_1=\frac{\sqrt{3}\,\Gamma(2/3)}{2\pi}\left(\frac{3+f(0)}{2}\right)^{1/3}$ because we do not know the value of $f(0)$.
\end{remark}

To prove Proposition~\ref{prop:asympt-exp2}, we first express $\mcl X$ in terms of the Busemann function $\mcl X^1$ of $M_{1,\infty}$.

\begin{lem}[Express $\mcl X$ in terms of $\mcl X^1$ in the LDP case]\label{lem:LDP-relating-X1-X}
Let $\XDP = \op{LDP}$. Define the KMSW encoding walk $\mcl Z  :\BB Z\to\BB Z^2$, the infinite directed triangulations $M_{0,\infty} \eqD M_{1,\infty}$, and their associated Busemann functions $\mcl X$ and $\mcl X^1$ as at the beginning of Section~\ref{sec-recursive}.
    % For $n\in\{0,1\}$, let $(M_{n,\infty},\lambda_n)$  be the infinite bipolar oriented triangulation obtained by applying the KMSW procedure (Definition~\ref{def-kmsw}) to $\mcl Z|_{[n,\infty)}$, where  $\mcl Z : \BB Z\to\BB Z^2$ is a bi-infinite random walk whose increments are i.i.d.\ uniform samples from $\{(1,-1), (-1,0) , (0,1)\}$, normalized so that $\mcl Z(0) = (0,0)$; in particular $(M_{0,\infty},\lambda_0)$ and $(M_{1,\infty},\lambda_1)$ are coupled together.  Let $\mcl X = \mcl X^0$ and $\mcl X^1$  be the Busemann functions for $(M_{0,\infty},\lambda_0)$ and $(M_{1,\infty},\lambda_1)$, respectively; as in Theorem~\ref{thm-busemann}.Then, the following three assertions hold:
    \begin{itemize}
        \item If $\mcl Z(1)-\mcl Z(0)=(1,-1)$, then 
            \eqb \label{eqn-ldp-relating1}
            \Big(\mcl X(-1)\,,\,\mcl X(1)\Big)=\left(\mcl X^1(-2) - \max\{\mcl X^1(-1),1\}\,,\,  - \max\{\mcl X^1(-1),1\}\right). 
            \eqe
        \item If $\mcl Z(1)-\mcl Z(0)=(-1,0)$, then 
           \eqb \label{eqn-ldp-relating2}
           \Big(\mcl X(-1)\,,\,\mcl X(1)\Big)=\left( - \mcl X^1(1)-1\,,\,-1\right).
           \eqe
        \item If $\mcl Z(1)-\mcl Z(0)=(0,1)$, then 
           \eqb  \label{eqn-ldp-relating3}
           \Big(\mcl X(-1)\,,\,\mcl X(1)\Big)=\left(\mcl X^1(-1) \,,\, \mcl X^1(2) \right).
           \eqe
    \end{itemize}
\end{lem}

\begin{figure}[ht!]
\begin{center}
\includegraphics[width=1\textwidth]{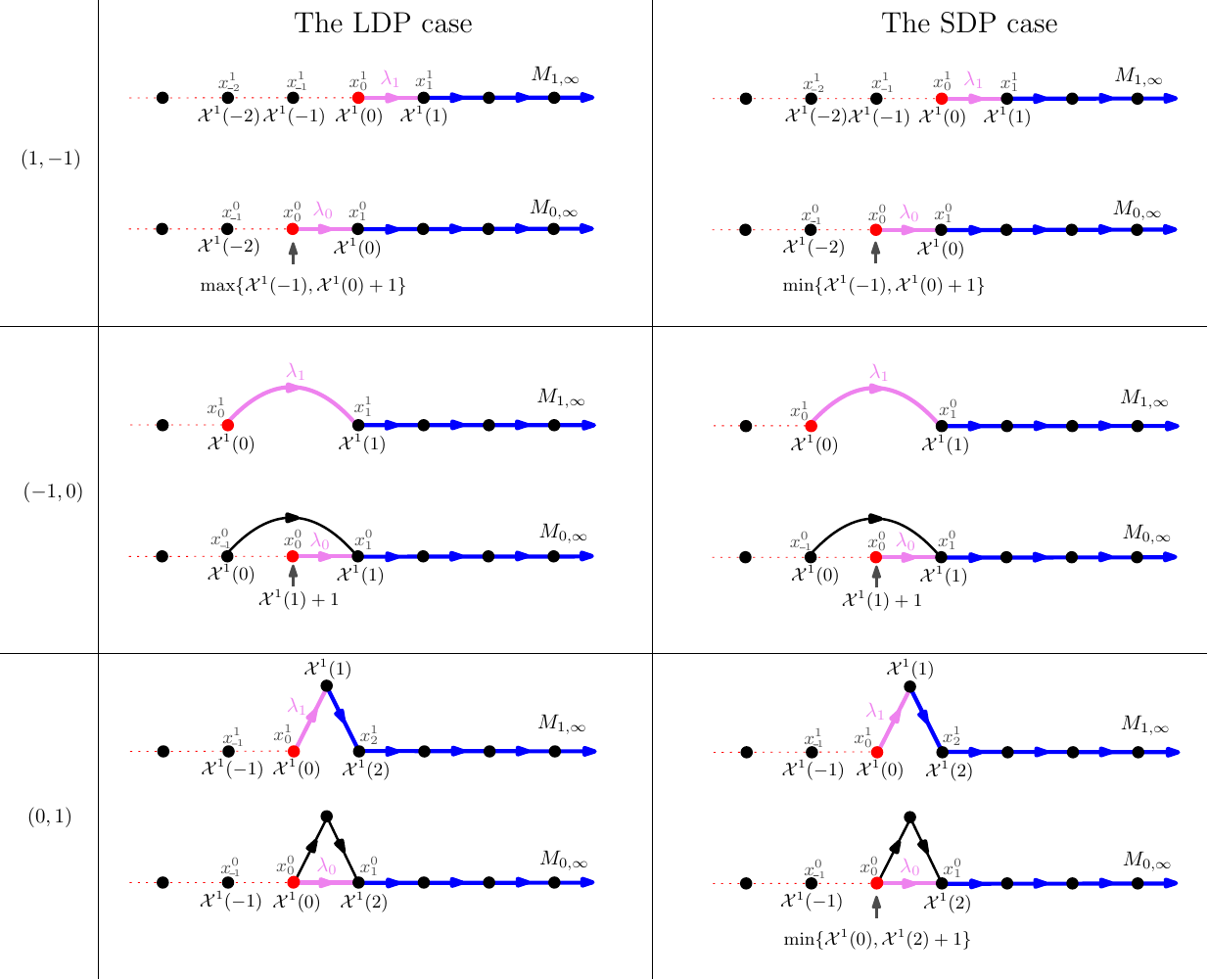}  
\caption{\label{fig-boundary-ev-LDP-SDP}  
A schema for the proof of Lemmas~\ref{lem:LDP-relating-X1-X}~and~\ref{lem:SDP-relating-X1-X}. The left-hand side is for the LDP case, while the right-hand side is for the SDP case. The three rows correspond to the three different steps $\{(1,-1), (-1,0) , (0,1)\}$ the walk $\mcl Z$ can take between time 0 and 1. The violet edges correspond to the root edges of the maps, the red vertex is the vertex $x^0_0$ (resp.\ $x^1_0$ for the map $M_{1,\infty}$), with the vertices $(x^0_i)_{i\geq 1}$ (resp.\ $(x^1_i)_{i\geq 1}$) on the right and the vertices $(x^0_i)_{i\leq 1}$ (resp.\ $(x^1_i)_{i\leq 1}$) on the left. On the maps $M_{1,\infty}$, we show some of the values of the Busemann function $\mcl X^1$ below each vertex. On the maps $M_{0,\infty}$, we show some of the values of the shifted Busemann function $\wt{\mcl X}$ of~\eqref{eqn-ldp-busemann-renormalize} in terms of $\mcl X^1$. We have $\mcl X(\cdot) = \wt{\mcl X}(\cdot) -\wt{\mcl X}(0)$.}
%in terms of  the Busemann function $\mcl X^1$, but \emph{before} renormalizing the values so that $\mcl X(0)=0$. 
\end{center}
\end{figure}

\begin{proof}
%Recall the description of the KMSW procedure from Definition~\ref{def-kmsw}. We will express $\mcl X$ in terms of $\mcl X^1$. We distinguish three cases according to the three possible values of  $\mcl Z(1)-\mcl Z(0)$. See the left column in \cref{fig-boundary-ev-LDP-SDP} for an illustration.
Recall that $ M_{0,\infty} $ and $ M_{1,\infty} $ are constructed via the KMSW procedure from $\mcl Z|_{[0,\infty)}$ and $\mcl Z|_{[1,\infty)}$, respectively. 
We label the boundary vertices of  $M_{0,\infty}$ from left to right by $\{x^0_k\}_{k\in\BB Z}$ so that $\lambda_0 = (x_0^0,x_1^0)$ and the boundary vertices of  $M_{1,\infty}$ from left to right by  $\{x^1_k\}_{k\in\BB Z}$ so that $\lambda_1 = (x_0^1,x_1^1)$. 
By Theorem~\ref{thm-busemann} (applied to, e.g., a sequence of vertices along the right boundary of $M_{1,\infty}$), there exists $w \in \mcl V(M_{1,\infty})$ such that 
\eqb \label{eqn-LDP-relating-w}
\op{LDP}_{M_{j,\infty}}(x^{ j }_{k'} , w) - \op{LDP}_{M_{j,\infty}}(x^{ j }_k , w) = \mcl X^j(k') - \mcl X^j(k) , \quad \forall j \in \{0,1\}, \quad\forall k,k' \in [-2,2]\cap\BB Z .
\eqe 
We will analyze how the Busemann function changes when we perform one backward step of the KMSW procedure (Definition~\ref{def-kmsw}) to get $M_{0,\infty}$ from $M_{1,\infty}$. 
Since $\mcl X^1$ is normalized so that $\mcl X^1(0) = 0$, to facilitate the comparison, it is convenient to introduce
\eqb  \label{eqn-ldp-busemann-renormalize}
\wt{\mcl X}(k) := \mcl X(k) + \op{LDP}_{M_{0,\infty}}(x^0_0 , w) - \op{LDP}_{M_{1,\infty}}(x^1_0 , w)
=  \op{LDP}_{M_{0,\infty}}(x^0_k , w) - \op{LDP}_{M_{1,\infty}}(x^1_0 , w) .
\eqe 
Then $\mcl X(\cdot) = \wt{\mcl X}(\cdot) - \wt{\mcl X}(0)$. We will express $\wt{\mcl X}(-1)$, $\wt{\mcl X}(0)$, and $\wt{\mcl X}(1)$ in terms of $\mcl X^1$, then subtract $\wt{\mcl X}(0)$ from  $\wt{\mcl X}(-1)$ and $\wt{\mcl X}(1)$ to get an expression for $\mcl X(-1)$ and $\mcl X(1)$ in terms of $\mcl X^1$. 

\smallskip
 
If $\mcl Z(1)-\mcl Z(0)=(1,-1)$, then the edges on the boundary of $M_{0,\infty}$ are the same as the edges on the boundary of $M_{1,\infty}$, with the only difference being that the root edge $\lambda_0$ of $M_{0,\infty}$ is the missing edge to the left of the root edge $\lambda_1$ of $M_{1,\infty}$; see the top-left diagram in \cref{fig-boundary-ev-LDP-SDP}. 
In this case, an LDP geodesic in $M_{0,\infty}$ from the vertex $x^0_{-1} = x^1_{-2}$ to $w$ must be an LDP geodesic in $M_{1,\infty}$ from the vertex $x^1_{-2}$ to $w$ because the vertex $x^0_{0}$ has only outgoing edges in $M_{0,\infty}$ thanks to the last claim in Lemma~\ref{lem-uibot-bdy}. Hence, in the notation~\eqref{eqn-ldp-busemann-renormalize}, 
\[
\wt{\mcl X}(-1) =   \mcl X^1(-2)  .
\]
An LDP geodesic in $M_{0,\infty}$ from $x_0^0 = x_{-1}^1$ to $w$ must be the longest between an LDP geodesic in $M_{1,\infty}$ from $x^1_{-1}$ to $w$ and an LDP geodesic in $M_{1,\infty}$ from $x^1_{0}$ to $w$ with the additional edge $\lambda_0$ prepended. Recalling that $\mcl X^1(0) = 0$, this gives
\[
\wt{\mcl X}(0) = \max\left\{ \mcl X^1(-1) ,   1 \right\} . 
\]
An LDP geodesic in $M_{0,\infty}$ started from $x_1^0 = x_0^1$ cannot use the edge $\lambda_0$ (which is incoming \jb{to $x_1^0$}), so $\wt{\mcl X}(1) = \mcl X^1(0) = 0$. 
Writing $\mcl X(-1) = \wt{\mcl X}(-1) - \wt{\mcl X}(0)$ and $\mcl X(1) = \wt{\mcl X}(1) - \wt{\mcl X}(0)$ and plugging in the above relations gives~\eqref{eqn-ldp-relating1}.

\smallskip

If $\mcl Z(1)-\mcl Z(0)=(-1,0)$, then the edges on the boundary of $M_{0,\infty}$ are the same as the edges on the boundary of $M_{1,\infty}$, with the only difference being that the root edge $\lambda_1$ of $M_{1,\infty}$ is replaced in the boundary of $M_{0,\infty}$ by a path of length two with the following properties. The first edge in the path is a missing edge whose initial vertex coincides with the initial vertex of $\lambda_1$. The second edge in the path is the new root edge $\lambda_0$ whose terminal vertex coincides with the terminal vertex of $\lambda_1$; see the middle-left diagram in \cref{fig-boundary-ev-LDP-SDP}. From this, we obtain
\eqbn
\wt{\mcl X}(-1) = \mcl X^1(0) = 0 ,\quad \wt{\mcl X}(0) = \mcl X^1(1) + 1 , \quad \wt{\mcl X}(1) = \mcl X^1(1) 
\eqen  
where in the case of $\wt{\mcl X}(0)$ we used that an LDP geodesic in $M_{0,\infty}$ from $x_0^0$ to $w$ must be an LDP geodesic in $M_{1,\infty}$ from $x^1_{1}$ to $w$, with the additional edge $\lambda_0$ prepended. Subtracting gives~\eqref{eqn-ldp-relating2}. 

\smallskip
  
If $\mcl Z(1)-\mcl Z(0)=(0,1)$, then the edges on the boundary of $M_{0,\infty}$ are the same as the edges on the boundary of $M_{1,\infty}$, with the only difference being that the root edge $\lambda_1$ of $M_{1,\infty}$ and the edge $e$ immediately to the right of it are replaced in the boundary of $M_{0,\infty}$ by the new root edge $\lambda_0$ whose initial vertex coincides with the initial vertex of $\lambda_1$ and whose terminal vertex coincides with the terminal vertex of $e$; see the bottom-left diagram in \cref{fig-boundary-ev-LDP-SDP}. 
This leads to
\eqbn
\wt{\mcl X}(-1) = \mcl X^1(-1) ,\quad \wt{\mcl X}(0) = \mcl X^1(0) = 0,\quad \wt{\mcl X}(1) = \mcl X^1(2), 
\eqen
where here we used that there is a directed path of length two in $M_{1,\infty}$ between the two vertices of $\lambda_0$, so no LDP geodesic in $M_{0,\infty}$ can traverse $\lambda_0$. Subtracting gives~\eqref{eqn-ldp-relating3}. 
\end{proof}
 
Lemma~\ref{lem:LDP-relating-X1-X} immediately implies the following recursive equation for the distributions $f$ and $g$.
 
 \begin{prop}[LDP recursive equation]\label{prop:LDP-recursive-equation}
 	For all $x\in\mathbb Z$ and $y\in\mathbb Z_{\le -1}$,
    \begin{equation}\label{eq:master-eq2}
        3f(x)g(y) =f(x)f(y)+f(x)\sum_{j=1}^{-y-1}g(y+j)g(-j)+\BB{1}_{\{y=-1\}}\left(\sum_{j\geq0}  f(j+x+1) f(-j)+ g(-x-1)\right).
    \end{equation}
 \end{prop}
\begin{proof} 
    Fix $x\in\mathbb Z$ and $y\in\mathbb Z_{\le -1}$. Using the fact that $\mcl X^1 \eqD \mcl X$ (which follows from the fact that $M_{0,\infty} \eqD M_{1,\infty}$), we can write
    \begin{align*}
        \BB P &\left[ \Big(\mcl X^1(-1)\,,\,\mcl X^1(1)\Big)=(x,y)\right]=\BB P \left[ \Big(\mcl X(-1)\,,\,\mcl X(1)\Big)=(x,y)\right].
    \end{align*}
    Thanks to Lemma~\ref{lem:LDP-relating-X1-X} \jb{and using that the map $M_{1,\infty}$ is independent from $\mcl Z(1) - \mcl Z(0)$ by definition of the KMSW procedure (Definition~\ref{def-kmsw}),} we have that 
    \begin{align}\label{eq:splitting-proba}
        \BB P \left[\Big(\mcl X(-1)\,,\,\mcl X(1)\Big)=(x,y)\right]&=\frac{1}{3}\BB P \left[ \left(\mcl X^1(-2) - \max\{\mcl X^1(-1),1\}\,,\, - \max\{\mcl X^1(-1),1\}\right)
    =(x,y)\right]\notag\\
    &\,+
    \frac{1}{3}\BB P \left[ \left( - \mcl X^1(1)-1\,,\,-1\right)
    =(x,y)\right]\notag\\
    &\,+
    \frac{1}{3}\BB P \left[ \left(\mcl X^1(-1) \,,\, \mcl X^1(2) \right)
    =(x,y)\right].
    \end{align}
    By stationarity (recall \eqref{eq:stationarity-rel} and \eqref{eq:den-dist-f-g}) and the fact that $\mcl X(0) = \mcl X^1(0) = 0$,
    \eqbn
    f(x) = \BB P\left[   - \mcl X(-1) = x \right] = \BB P\left[  - \mcl X^1(-1) = x \right]
    \quad \text{and} \quad 
    g(y)= \BB P\left[ \mcl X(1)  = y \right]  = \BB P\left[ \mcl X^1(1)  = y \right]  .
    \eqen 
    Furthermore, the increments $\{\mcl X^1(k) - \mcl X^1(k-1)\}_{k\in\BB Z}$ are independent, the increments $\{\mcl X^1(k) - \mcl X^1(k-1)\}_{k \geq 1}$ all have the same distribution, and the increments $\{\mcl X^1(k) - \mcl X^1(k-1)\}_{k \leq 0}$ all have the same distribution (Properties~\ref{item-busemann-ind}~and~\ref{item-busemann-stationary} in Theorem~\ref{thm-busemann-property}). Applying these facts in~\eqref{eq:splitting-proba} gives
    \begin{align*}
        3f(-x)g(y) &=f(-x)f(y)+\BB{1}_{\{y=-1\}}\sum_{j\leq0}  f(j-x-1) f(-j)\\
        &\,\,\,\,\,+ \BB{1}_{\{y=-1\}}\,\,g(-x-1)\\
        &\,\,\,\,\,+f(-x)\sum_{j=1}^{-y-1}g(y+j)g(-j),
    \end{align*}
    where the first probability on the right-hand side of \eqref{eq:splitting-proba} involving the maximum was computed by distinguishing the cases when $\mcl X^1(-1)\leq 0$ and $\mcl X^1(-1)\geq 1$. Using the fact that $f(x) = f(-x)$, we obtain the expression stated in the proposition. 
\end{proof}

 We now proceed with the proof of Proposition~\ref{prop:asympt-exp2} by first establishing some preliminary results. 

 \begin{lem}\label{lem:Y-not-const}
    Let $Y\sim f$. Then $Y$ is not a.s.\ constant and $\BB P(Y=0)=f(0)>0$.
 \end{lem}

 \begin{proof}
We first prove that $Y$ is not a.s.\ constant. Since $f(x) = f(-x)$, it is enough to show that $f(0)=1$ is not a solution to \eqref{eq:master-eq2}.
Assume for contradiction that $f(0)=1$. Then $f(k)=0$ for all $k\neq 0$.
Fixing $x=0$ in \eqref{eq:master-eq2}, we obtain that for all $y\le -1$,
\begin{equation}\label{eq:main-simpl}
3 g(y)= f(y)+ \sum_{j=1}^{-y-1} g(y+j)g(-j)
+ \mathbf{1}_{\{y=-1\}}\!\left(\sum_{j\ge 0} f(j+1)f(-j) + g(-1)\right).
\end{equation}
First, take $y=-1$. Since $f(-1)=0$ and $f(j+1)=0$ for all $j\ge 0$, \eqref{eq:main-simpl} gives $3 \,g(-1)= g(-1)$, and so $g(-1)=0$.
Now let $y\le -2$. The indicator term vanishes and $f(y)=0$ since $y\neq 0$, so \eqref{eq:main-simpl} becomes
\begin{equation}\label{eq:rec}
3 g(y)= \sum_{j=1}^{-y-1} g(y+j)g(-j).
\end{equation}
For $y=-2$, the right-hand side is $g(-1)g(-1)=0$, hence $g(-2)=0$.
Inductively, one obtains that $g(y)=0$ for all $y\le -1$. A contradiction since $g$ is a probability distribution supported on $\BB Z_{\leq -1}$.

We now prove that $f(0)>0$. Fixing $y=-1$ and summing \eqref{eq:master-eq2} over all $x\in\mathbb Z$, we get  using the symmetry of $f$,
\begin{equation}\label{eq:g-bound}
3g(-1)= f(-1) + \frac{1+f(0)}{2} + 1,
\end{equation}
and so $g(-1)\ge \frac12$.
Now take $x=0$, $y=-1$ in \eqref{eq:master-eq2}:
\[
3 f(0)g(-1) = f(0)f(-1) + \sum_{j\ge0} f(j+1)f(j) + g(-1),
\]
again using the symmetry of $f$. If $f(0)=0$, this gives $g(-1)\le 0$, resulting in a contradiction.
\end{proof}

 To simplify computations, we introduce the auxiliary characteristic function:
 \begin{equation}\label{eq:G-bar-above}
     \overline{G}(s):=\sum_{y\in \BB Z_{\geq 1}} g(-y)\,e^{isy},
 \end{equation}
 that is, $\overline{G}(-s)=G(s)$, where $G(s)$ was introduced in \eqref{eq:FGdefs-new}.
 Multiplying  \eqref{eq:master-eq2} by $e^{i(tx-sy)}$ and summing over $x\in\mathbb Z$ and $y\in\BB Z_{\le -1}$, we obtain the functional equation
 \begin{equation}\label{eq:big-rel}
 	3\overline{G}(s)F(t)
 	=F(t)\left(F^{-}(-s)+\overline{G}(s)^2\right)+e^{is-it}\left(F(t)\Big(F^{-}(t)+f(0)\Big)+\overline{G}(t)\right).
 \end{equation}  
 Setting $s=0$ in \eqref{eq:big-rel} and using $F^{-}(0)=\frac{1}{2}(1-f(0))$ and $\overline{G}(0)=1$, gives
 \begin{equation}\label{eq:s0}
 	3F(t)=F(t)\left(\tfrac{3-f(0)}{2}\right)+e^{-it}\left(F(t)\Big(F^{-}(t)+f(0)\Big)+\overline{G}(t)\right).
 \end{equation}
 Solving \eqref{eq:s0} for $\overline{G}(t)$ yields
 \begin{equation}\label{eq:Ghat}
 	\overline{G}(t)=-\,F(t)\,\widehat F^{-}(t),\qquad
 	\widehat F^{-}(t):=F^{-}(t)+c(t),\qquad
 	c(t):=f(0)-\frac{\kappa}{2}\,e^{it},
 \end{equation}
 where we recall that  $\kappa=3+f(0)>0$.
 Now set $t=0$ in \eqref{eq:big-rel} and substitute the expression for $\overline{G}(t)$ in \eqref{eq:Ghat}. After some algebra, we obtain the quadratic equation\footnote{Just performing the substitution, one gets a minus sign on the left-hand side of \eqref{eq:quadratic-new} that we removed for simplicity.} in the variable $\widehat F^{-}$: 
 \begin{equation}\label{eq:quadratic-new}
 	F^2\big(\widehat F^{-}\big)^2 + (3F-1)\,\widehat F^{-} + F=0.
 \end{equation}
 
 \begin{lem}[$\widehat{F}^{-}(t)$ has  non-trivial imaginary part]\label{lem:nonreal2}
 Let $\Delta(F):=(3F-1)^2-4F^3$ be the discriminant of \eqref{eq:quadratic-new}. Then there exists $\delta>0$ such that for all $t$ with $0<t<\delta$, 
 \[\Delta(F(t))<0,\] 
 hence the two roots of \eqref{eq:quadratic-new} are non-real complex conjugates. In particular, writing 
 \begin{equation}\label{eq:fhatexpr2}
     \widehat{F}^{-}(t)=\alpha(t)+i\beta(t)
 \end{equation} 
 with a continuous choice of branch, one has $\beta(t)\ne 0$ for $0<t<\delta$.
 \end{lem}
 \begin{proof}
 We have $F(0)=1$, $F$ is continuous, and $\Delta(1)=\Delta'(1)=0$, $\Delta''(1)=-6<0$. Thus 
 \begin{equation}\label{eq:delta-est2}
 	\Delta(1+\eps)\le - C\,\eps^2
 \end{equation}
  for sufficiently small $\eps \in \BB R\setminus \{0\}$ and some $C>0$. Let $Y\sim f$ and $d$ be the lattice span of $Y$, i.e.
 \[
 d := \gcd\{x-y : f(x)>0,\ f(y)>0\}.
 \]
Since $Y$ is not a.s.\ constant by Lemma~\ref{lem:Y-not-const}, we have $d<\infty$. Since $f(0)>0$ thanks to Lemma~\ref{lem:Y-not-const}, a simple argument for the characteristic function of integer-valued random variables gives
 \begin{equation}\label{eq:iff-span}
      F(t)=1 \quad \text{if and only if} \quad t \in \tfrac{2\pi}{d}\mathbb{Z}.
 \end{equation}
 
\begin{comment}
Indeed, since $f(0)>0$, we can write $Y=dZ$ with $Z$ integer-valued. By the
definition of $d$, the random variable $Z$ has lattice span $1$.
Moreover, $F(t)=\mathbb E\big(e^{itY}\big)=\mathbb E\big(e^{itdZ}\big)$.

Now, suppose $F(t_0)=1$. Because $|e^{it_0dZ}|=1$ a.s., the triangle inequality gives
\[
1=|F(t_0)|=\big|\mathbb E(e^{it_0dZ})\big|\le \mathbb E\big|e^{it_0dZ}\big|=1,
\]
so equality holds. Hence $e^{it_0dZ}$ is a.s.\ constant, and therefore
\begin{equation}\label{eq:exp-equal-1}
    e^{it_0d(y-y')}=1, \quad \text{for all $y,y'$ in the support of $Z$.}
\end{equation}
Let
$\Delta:=\{y-y': y,y'\in\mathrm{supp}(Z)\}$. Then $\gcd(\Delta)=1$ since $Z$ has lattice span 1, so by Bézout
there exist $\delta_1,\dots,\delta_k\in\Delta$ and $c_1,\dots,c_k\in\mathbb Z$ with
$\sum_{j=1}^k c_j\delta_j=1$. This yields 
\[e^{idt_0}=\prod_{j=1}^k\big(e^{i\,dt_0\delta_j}\big)^{c_j}\stackrel{\eqref{eq:exp-equal-1}}{=}1,\]
hence $dt_0\in 2\pi\mathbb Z$, i.e.\ $t_0\in\tfrac{2\pi}{d}\mathbb Z$.

On the other hand, if $t_0=\tfrac{2\pi m}{d}$ with $m\in\mathbb Z$, then
$e^{it_0dZ}=e^{i\,2\pi mZ}=1$ a.s., so $F(t_0)=\mathbb E(e^{i\,dt_0Z})=1$. This concludes the proof of \eqref{eq:iff-span}.
\end{comment}

 So there exists $\delta_1 \in (0,2\pi/d)$ such that 
 $F(t)\neq 1$ for all $0<t<\delta_1$. By continuity, there exists $\delta_2>0$ such that 
 $|F(t)-1|<\eta$ whenever $t<\delta_2$, with $\eta>0$ small enough to ensure the estimate in \eqref{eq:delta-est2} applies.
 Setting $\delta := \min\{\delta_1,\delta_2\}$, we get that for all $0<t<\delta$,
 \[
 F(t)\neq 1, \qquad \Delta(F(t)) \leq - C(F(t)-1)^2 < 0.
 \]
 Thus $\Delta(F(t))<0$ for $0<|t|<\delta$, so the two roots of \eqref{eq:quadratic-new} are non-real conjugates with nonzero imaginary part.
 \end{proof}

We are now ready to complete the proof of Proposition~\ref{prop:asympt-exp2}.

\begin{proof}[Proof of Proposition~\ref{prop:asympt-exp2}]
 	Recall that $F(t)\in\mathbb R$ and write 
 \begin{equation*}
     F(t)=f(0)+F^{-}(t)+\overline{F^{-}(t)}\stackrel{\eqref{eq:Ghat}}{=}f(0)+(\widehat{F}^{-}(t)-c(t))+\overline{(\widehat{F}^{-}(t)-c(t))}=f(0)-2\re(c(t))+2\re(\widehat{F}^{-}(t)) .
 \end{equation*}
 We set
 \[K(t):=f(0)-2\re(c(t))\stackrel{\eqref{eq:Ghat}}{=}-f(0)+\kappa\,\cos t,\]
 so that $K(t)$ encodes the part of $F(t)$ that is independent of $\widehat{F}^{-}(t)$ (which we recall is our indeterminate in \eqref{eq:quadratic-new}) and
 \begin{equation}\label{eq:real-part-fhat}
     \re(\widehat{F}^{-}(t))=\frac{F(t)-K(t)}{2}.
 \end{equation}
 Hence, thanks to Lemma~\ref{lem:nonreal2},
 \begin{equation*}
     \widehat{F}^{-}(t)\stackrel{\eqref{eq:fhatexpr2}}{=}\alpha(t)+i\beta(t)\stackrel{\eqref{eq:real-part-fhat}}{=}\frac{F(t)-K(t)}{2}+i\beta(t),
 \end{equation*} 
 with $\beta(t)\ne 0$ for $0<t<\delta$.

 Now, taking the imaginary part in \eqref{eq:quadratic-new}, we obtain the equation
 \begin{equation*}
     \beta(F^3-K\,F^2+3F-1)=0,
 \end{equation*}
 and since $\beta(t)\ne 0$ for $0<t<\delta$, we get that for $0<t<\delta$,
 \begin{equation}\label{eq:third-ord2}
     F^3-K\,F^2+3F-1=0.
 \end{equation}
 The Taylor expansion of $K(t)$ around $t=0$ is 
 \begin{equation*}
     K(t)=3-\frac{\kappa}{2}\,t^2+O(t^4).
 \end{equation*}
  Writing $F(t)=1+\eps(t)$ with $\eps(t)\to 0$ as $t\downarrow 0$, we get from \eqref{eq:third-ord2} that
 \begin{equation}\label{eq:balance-eq}
     0=\varepsilon^3+\frac{\kappa}{2}\,t^2(\varepsilon^2+2\varepsilon+1)+O\!\big(t^4\big).
 \end{equation}
Now, since $\kappa> 0$ (and so $\kappa\neq 0$),  we claim that as $t\downarrow 0$,  
\begin{equation}\label{eq:cleim-exp}
    \eps(t)= -\left(\frac{\kappa}{2}\right)^{1/3}\,t^{2/3}+o(t^{2/3}).
\end{equation}
Indeed, set $\varepsilon(t)=t^{2/3}s(t)$. Dividing \eqref{eq:balance-eq} by $t^2$ gives as $t\downarrow 0$,
\begin{equation}\label{eq:s}
0=s^3+\tfrac{\kappa}{2}\bigl(1+2s t^{2/3}+s^2 t^{4/3}\bigr)+O(t^2)=s^3+\tfrac{\kappa}{2}+o(1),
\end{equation}
where in the last equality we used that $\varepsilon(t)=t^{2/3}s(t)\to 0$ as $t\downarrow 0$. Since $\varepsilon(t)$ (and so also $s(t)$) is real, we get $s(t)=-\left(\frac{\kappa}{2}\right)^{1/3}+o(1)$  and so \eqref{eq:cleim-exp} holds.
This completes the proof for the expansion of $F(t)$.
 
 \medskip

 We now determine the expansions of $F^-(t)$ and $\overline{G}(t)$. Solving the quadratic equation \eqref{eq:quadratic-new}, the two possible continuous choices of branch for $\widehat{F}^-(t)$ are
 \begin{equation*}
 	\widehat{F}_{\pm}^-(t)=\frac{1-3F\pm\sqrt{\Delta(F)}}{2F^2},\qquad{\text{with $\Delta(F)=(3F-1)^2-4F^3$,}}
 \end{equation*}
 and recalling \eqref{eq:Ghat}, we obtain that the two possible continuous choices of branch for $F^-(t)$ are
  \begin{equation}\label{expr-for-F-2}
 	F_{\pm}^-(t)=\frac{1-3F\pm\sqrt{\Delta(F)}}{2F^2}-c(t),\qquad{\text{with $\Delta(F)=(3F-1)^2-4F^3.$}}
 \end{equation}
 Moreover, from \eqref{eq:Ghat}, we have that the two possible continuous choices of branch for $\overline{G}(t)$ are
 \begin{equation}\label{eq:Gtexpr2}
 	\overline{G}_{\pm}(t) = -F(t)\,\widehat F^-_{\pm}(t).
 \end{equation}
 Therefore, since $F(t)= 1-\left(\frac{\kappa}{2}\right)^{1/3}\,t^{2/3}+o(t^{2/3})$ as $t\downarrow 0$, we get from \eqref{eq:Ghat} and \eqref{expr-for-F-2} that
 	\begin{equation*}
 		F_{\pm}^-(t)=\left(\frac{1-f(0)}{2}\right)+ \frac{1}{2}(-1 \pm i\sqrt{3})\left(\frac{\kappa}{2}\right)^{1/3}t^{2/3}+o(t^{2/3}),
 	\end{equation*}
 	and from \eqref{eq:Gtexpr2},
 	\begin{equation*}
 		\overline{G}_{\pm}(t)=1+\frac{1}{2}(-1 \mp i\sqrt{3})\left(\frac{\kappa}{2}\right)^{1/3}t^{2/3}+o(t^{2/3}).
 	\end{equation*}
    Noting that $\overline{G}(t)$, introduced in \eqref{eq:G-bar-above}, is the characteristic function of a non-negative integer-valued random variable, and using Proposition~\ref{prop:Tauberian}, we have that the only admissible expression is
    \begin{equation*}
 		\overline{G}(t)=1+\frac{1}{2}(-1 + i\sqrt{3})\left(\frac{\kappa}{2}\right)^{1/3}t^{2/3}+o(t^{2/3}),
 	\end{equation*}
    and so
    \begin{equation*}
 		F^-(t)=\left(\frac{1-f(0)}{2}\right)+ \frac{1}{2}(-1 - i\sqrt{3})\left(\frac{\kappa}{2}\right)^{1/3}t^{2/3}+o(t^{2/3}).
 	\end{equation*}
 This completes the proof for the expansion of $F^{-}(t)$ and $\overline{G}(t)$  as $t\downarrow0$. To recover the expansion of $G(t)$, it is enough to note that $G(t)=\overline{G}(-t)$ and the latter is just the complex conjugate of $\overline{G}(t)$.
 \end{proof}

\subsection{The SDP case}
\label{sec-recursive-SDP}

We now move to the case when $\XDP = \op{SDP}$. 
The proof is very similar to the LDP case. 
We use the same notation as in Section~\ref{sec-recursive-LDP}, but now everything is defined with $\op{SDP}$ instead of $\op{LDP}$.

Thanks to Properties~\ref{item-busemann-pos}~and~\ref{item-busemann-sym} of Theorem~\ref{thm-busemann-property} we know that $f$ is a symmetric distribution on $\BB Z$ and $g$ is a distribution on $\BB Z_{\geq -1}$. 
We introduce the three characteristic functions 
 \begin{equation}\label{eq:FGdefs}
 	F(t):=\sum_{x\in\mathbb Z} f(x)\,e^{itx},\qquad
 	F^{-}(t):=\sum_{x \in\BB Z_{ \leq -1}} f(x)\,e^{itx},\qquad
 	G(s):=\sum_{y\in \BB Z_{\geq -1}} g(y)\,e^{isy}.
 \end{equation}
 By the symmetry of $f$, we note that $F(t)\in\mathbb R$. Moreover, since every characteristic function $\varphi(t)$ satisfies $\varphi(-t)=\overline{\varphi(t)}$, in what follows, we only focus on the case $t\geq 0$.

   \begin{prop}[Asymptotics of the characteristic functions, SDP case]\label{prop:asympt-exp}
 	Setting 
 	\[\kappa:=3-6\,g(-1)^2-f(0),\] 
    it holds that  $\kappa=0$, and, as $t\downarrow 0$,
 		\begin{align*}
 			&F(t)= 1-g(-1)^{2/3}\,t^{4/3}+o(t^{4/3}),\\
 			&F^-(t)=\frac{1}{2}(1-f(0))-2i \, g(-1)^2\, t+\frac{-1 + i\sqrt{3}}{2}  g(-1)^{2/3}\,t^{4/3}+o(t^{4/3}),\\
 			&G(t)=1+\frac{- 1- i\sqrt{3}}{2}  g(-1)^{2/3}\,t^{4/3}+o(t^{4/3}).
 		\end{align*}
        Moreover, $g(-1)> 0$.
 \end{prop}

 Before proving Proposition~\ref{prop:asympt-exp}, we note that it immediately implies Item~\ref{item-busemann-tail-SDP} of Theorem~\ref{thm-busemann-tail}.

\begin{proof}[Proof of Item~\ref{item-busemann-tail-SDP} of Theorem~\ref{thm-busemann-tail} (assuming Proposition~\ref{prop:asympt-exp})]
We first combine Proposition~\ref{prop:asympt-exp} with Proposition~\ref{prop:Tauberian} to prove the claimed tail behaviors. The asymptotic expansions of $F(t)$ and $G(t)$ in Proposition~\ref{prop:asympt-exp} ensures that $f(0)\neq 1$ and $g(-1)\neq 1$. Moreover, $g(-1)\neq 0$ as stated in Proposition~\ref{prop:asympt-exp}.
Recalling the notation introduced in \eqref{eq:den-dist-f-g}, in the case of $f$, we apply Proposition~\ref{prop:Tauberian} to the random variable $X$ with $\BB P[X=x] = \frac{2}{1-f(0)} f(x)$ for all $x\geq1$, so that $X$ is non-negative (recall that $f(x) = f(-x)$). Since the characteristic function of $X$ for $t>0$ is
\begin{align*}
    \varphi_X(t)&=\frac{2}{1-f(0)}\Big(F(t)-F^{-}(t)-f(0)\Big)\\
    &=1+\frac{4i\,g(-1)^2}{1-f(0)}\,t+\frac{(-1-i\sqrt{3})\,g(-1)^{2/3}}{1-f(0)}\,t^{4/3}+o(t^{4/3}),
\end{align*}
we obtain that
\begin{align*}
    \BB P[X>R]
    &=\frac{\sqrt{3}\,g(-1)^{2/3}}{1-f(0)}
    \jb{\frac{1}{3\,\Gamma(1-1/3)}\sec\left(\tfrac{\pi}{6}\right)}
    R^{-4/3}+o(R^{-4/3})\\
    &=\frac{\jb{\sqrt{3}}\,\Gamma(1/3)\,g(-1)^{2/3}}{3\pi(1-f(0))}\,R^{-4/3}+o(R^{-4/3}).
\end{align*}
Hence, recalling that $f(x)=\BB P\left[\mcl X(0)-\mcl X(-1)=x\right]$ for all $x\in\BB Z$,
\begin{equation}\label{eq:conc-asympt-3}
    \BB P\left[|\mcl X(0)-\mcl X(-1)|>R\right]
    =\frac{\sqrt 3\,\Gamma(1/3)\,g(-1)^{2/3}}{3\pi}\,R^{-4/3}+o(R^{-4/3}).
\end{equation}

In the case of $g$, we apply Proposition~\ref{prop:Tauberian} to the random variable $Y$ with $\BB P[Y=y]=\frac{1}{1-g(-1)}g(y)$ for all $y\geq0$, so that $Y$ is non-negative.
Since the characteristic function of $Y$ for $t>0$ is
\begin{equation*}
    \varphi_Y(t)=\frac{1}{1-g(-1)}\left(G(t)-g(-1)e^{-it}\right)
    =1+\frac{i\,g(-1)}{1-g(-1)}\,t
    +\frac{g(-1)^{2/3}(-1-i\sqrt{3})}{2(1-g(-1))}\,t^{4/3}
    +o(t^{4/3}),
\end{equation*}
where we used that $e^{-it}=1-it+O(t^2)$, we obtain that
\begin{align*}
    \BB P[Y>R]
    &=\frac{\sqrt{3}\,g(-1)^{2/3}}{2(1-g(-1))}
    \jb{\frac{1}{3\Gamma(1-1/3)}\sec\left(\tfrac{\pi}{6}\right)}
    R^{-4/3}+o(R^{-4/3})\\
    &=\frac{\jb{\sqrt{3}}\,\Gamma(1/3)\,g(-1)^{2/3}}{6\pi(1-g(-1))}R^{-4/3}+o(R^{-4/3})
\end{align*}
Hence, recalling that $g(y)=\BB P\left[\mcl X(1)-\mcl X(0)=y\right]$ for all $y\in\BB Z_{\geq-1}$,
\begin{equation}\label{eq:conc-asympt-4}
    \BB P\left[\mcl X(1)-\mcl X(0)>R\right]
    =\frac{\jb{\sqrt{3}}\,\Gamma(1/3)\,g(-1)^{2/3}}{6\pi}R^{-4/3}+o(R^{-4/3}).
\end{equation}
Setting $c_2=\frac{\jb{\sqrt{3}}\,\Gamma(1/3)\,g(-1)^{2/3}}{6\pi}$ in \eqref{eq:conc-asympt-3}~and~\eqref{eq:conc-asympt-4} gives the expressions in the statement of the theorem.

The tail expansions~\eqref{eq:conc-asympt-3} and~\eqref{eq:conc-asympt-4}, together with the stationarity property~\ref{item-busemann-stationary} of Theorem~\ref{thm-busemann-property}, imply that $\BB E[|\mcl X(k)-\mcl X(k-1)|]<\infty$ for every $k\in\BB Z$. From the expansions of $F(t)$ and $G(t)$ in Proposition~\ref{prop:asympt-exp}, we get $F'(0)=G'(0)=0$, which implies that $\BB E[\mcl X(k)-\mcl X(k-1)]=0$ for every $k\in\BB Z$.
\end{proof}

\begin{remark}
    Also in this case, we do not know the exact value of $c_2=\frac{\jb{\sqrt{3}}\,\Gamma(1/3)\,g(-1)^{2/3}}{6\pi}$ because we do not know the value of $g(-1)$.
\end{remark}

As in Section~\ref{sec-recursive-LDP}, the first step in the proof of Proposition~\ref{prop:asympt-exp} is to express $\mcl X$ in terms of $\mcl X^1$.

\begin{lem}[Express $\mcl X$ in terms of $\mcl X^1$ in the SDP case]\label{lem:SDP-relating-X1-X}
Let $\XDP = \op{SDP}$. Define the KMSW encoding walk $\mcl Z  :\BB Z\to\BB Z^2$ and the infinite directed triangulations $M_{0,\infty} \eqD M_{1,\infty}$ and their associated Busemann functions $\mcl X$ and $\mcl X^1$ as at the beginning of Section~\ref{sec-recursive}.
    \begin{itemize}
        \item If $\mcl Z(1)-\mcl Z(0)=(1,-1)$, then 
            \eqb \label{eqn-sdp-relating1}
            \Big(\mcl X(-1)\,,\,\mcl X(1)\Big)=\left(\mcl X^1(-2) - \min\{\mcl X^1(-1),1\}\,,\,  - \min\{\mcl X^1(-1),1\}\right).
            \eqe
        \item If $\mcl Z(1)-\mcl Z(0)=(-1,0)$, then 
            \eqb \label{eqn-sdp-relating2}
            \Big(\mcl X(-1)\,,\,\mcl X(1)\Big)=\left( - \mcl X^1(1)-1\,,\,-1\right).
            \eqe
        \item If $\mcl Z(1)-\mcl Z(0)=(0,1)$, then 
            \eqb \label{eqn-sdp-relating3}
            \Big(\mcl X(-1)\,,\,\mcl X(1)\Big)=\left(\mcl X^1(-1)-\min\{0,\mcl X^1(2)+1\} \,,\, \mcl X^1(2)-\min\{0,\mcl X^1(2)+1\} \right).
            \eqe
    \end{itemize}
\end{lem}

\begin{proof}
Recall that $ M_{0,\infty} $ and $ M_{1,\infty} $ are constructed via the KPZ procedure from $\mcl Z|_{[0,\infty)}$ and $\mcl Z|_{[1,\infty)}$, respectively.
We label the boundary vertices of  $M_{0,\infty}$ from left to right by  $\{x^0_k\}_{k\in\BB Z}$ so that $\lambda_0 = (x_0^0,x_1^0)$ and the boundary vertices of  $M_{1,\infty}$ from left to right by  $\{x^1_k\}_{k\in\BB Z}$ so that $\lambda_1 = (x_0^1,x_1^1)$.
    By Theorem~\ref{thm-busemann} (applied to, e.g., a sequence of vertices along the right boundary of $M_{1,\infty}$), there exists $w \in \mcl V(M_{1,\infty})$ such that 
    \eqb \label{eqn-SDP-relating-w}
        \op{SDP}_{M_{j,\infty}}(x^{j}_{k'} , w) - \op{SDP}_{M_{j,\infty}}(x^{j}_k , w) = \mcl X^j(k') - \mcl X^j(k) , \quad \forall j \in \{0,1\}, \quad\forall k,k' \in [-2,2]\cap\BB Z .
    \eqe 
    We analyze how the Busemann function changes when we perform one backward step of the KMSW procedure (Definition~\ref{def-kmsw}), to get $M_{0,\infty}$ from $M_{1,\infty}$. Exactly as in the LDP case~\eqref{eqn-ldp-busemann-renormalize}, we introduce
\eqb  \label{eqn-sdp-busemann-renormalize}
\wt{\mcl X}(k) := \mcl X(k) + \op{SDP}_{M_{0,\infty}}(x^0_0 , w) - \op{SDP}_{M_{1,\infty}}(x^1_0 , w)
=  \op{SDP}_{M_{0,\infty}}(x^0_k , w) - \op{SDP}_{M_{1,\infty}}(x^1_0 , w) .
\eqe 
Then $\mcl X(\cdot) = \wt{\mcl X}(\cdot) - \wt{\mcl X}(0)$. We will express $\wt{\mcl X}(-1) ,\wt{\mcl X}(0)$, and $\wt{\mcl X}(1)$ in terms of $\mcl X^1$, then subtract $\wt{\mcl X}(0)$ from  $\wt{\mcl X}(-1)$ and $\wt{\mcl X}(1)$ to get an expression for $\mcl X(-1)$ and $\mcl X(1)$ in terms of $\mcl X^1$. 

\smallskip
 
If $\mcl Z(1)-\mcl Z(0)=(1,-1)$, then the edges on the boundary of $M_{0,\infty}$ are the same as the edges on the boundary of $M_{1,\infty}$, with the only difference being that the root edge $\lambda_0$ of $M_{0,\infty}$ is the missing edge to the left of the root edge $\lambda_1$ of $M_{1,\infty}$; see the top-right diagram in \cref{fig-boundary-ev-LDP-SDP}. 
In this case, an SDP geodesic in $M_{0,\infty}$ from the vertex $x^0_{-1} = x_{-2}^1$ to $w$ must be an SDP geodesic in $M_{1,\infty}$ from the vertex $x^1_{-2}$ to $w$ because the vertex $x^0_{0}$  has only outgoing edges in $M_{0,\infty}$ thanks to the last claim in Lemma~\ref{lem-uibot-bdy}. Hence, in the notation~\eqref{eqn-sdp-busemann-renormalize},
\eqbn
\wt{\mcl X}(-1) = \mcl X^1(-2) .
\eqen 
An SDP geodesic in $M_{0,\infty}$ from $x_0^0 = x_{-1}^1$ to $w$ must be the shortest between an SDP geodesic in $M_{1,\infty}$ from $x^1_{-1}$ to $w$ and an SDP geodesic in $M_{1,\infty}$ from $x^1_{0}$ to $w$ with the additional edge $\lambda_0$ prepended. Recalling that $\mcl X^1(0) = 0$, this gives
\eqbn
\wt{\mcl X}(0) = \min\left\{\mcl X^1(-1) , 1 \right\} .
\eqen
Since the edge $\lambda_0$ is incoming to $x_1^0 = x_0^0$, we also have $\wt{\mcl X} (1) = \mcl X^1(0)  = 0$. 
Subtracting now gives~\eqref{eqn-sdp-relating1}.

\smallskip
   
If $\mcl Z(1)-\mcl Z(0)=(-1,0)$, then the edges on the boundary of $M_{0,\infty}$ are the same as the edges on the boundary of $M_{1,\infty}$, with the only difference being that the root edge $\lambda_1$ of $M_{1,\infty}$ is replaced in the boundary of $M_{0,\infty}$ by a path of length two with the following properties. The first edge in the path is a missing edge whose initial vertex coincides with the initial vertex of $\lambda_1$. The second edge in the path is the new root edge $\lambda_0$ whose terminal vertex coincides with the terminal vertex of $\lambda_1$; see the middle-right diagram in \cref{fig-boundary-ev-LDP-SDP}. In this case,
\eqbn
\wt{\mcl X}(-1) = \mcl X^1(0) =0 ,\quad
\wt{\mcl X}(0) = \mcl X^1(1) + 1,\quad
\wt{\mcl X}(1)=  \mcl X^1(1),
\eqen 
  where for $\wt{\mcl X}(0)$ we used that an SDP geodesic in $M_{0,\infty}$ from $x_0^0$ to $w$ must be an SDP geodesic in $M_{1,\infty}$ from $x^1_{1}$ to $w$ with the additional edge $\lambda_0$ prepended. Subtracting gives~\eqref{eqn-sdp-relating2}. 

    \smallskip
  
  If $\mcl Z(1)-\mcl Z(0)=(0,1)$, then the edges on the boundary of $M_{0,\infty}$ are the same as the edges on the boundary of $M_{1,\infty}$, with the only difference being that the root edge $\lambda_1$ of $M_{1,\infty}$ and the edge $e$ immediately to the right of it are replaced in the boundary of $M_{0,\infty}$ by the new root edge $\lambda_0$ whose initial vertex coincides with the initial vertex of $\lambda_1$ and whose terminal vertex coincides with the terminal vertex of $e$; see the bottom-right diagram in \cref{fig-boundary-ev-LDP-SDP}. In this case, an SDP geodesic in $M_{0,\infty}$ from the vertex $x^0_{-1} = x_{-1}^1$ to $w$ must be an SDP geodesic in $M_{1,\infty}$ from the vertex $x^1_{-1}$ to $w$ because the vertex $x^0_{0}$ has only outgoing edges in $M_{0,\infty}$ thanks to the last claim in Lemma~\ref{lem-uibot-bdy}. Hence,
  \eqbn
  \wt{\mcl X}(-1) = \mcl X^1(-1) .
  \eqen 
  An SDP geodesic in $M_{0,\infty}$ from $x_0^0 = x_0^1$ to $w$ must be the shortest between an SDP geodesic in $M_{1,\infty}$ from $x^1_{0}$ to $w$ and an SDP geodesic in $M_{1,\infty}$ from $x^1_{2}$ to $w$ with the additional edge $\lambda_0$ prepended. Recalling that $\mcl X^1(0) =0$, this gives
 \eqbn
  \wt{\mcl X}(0) = \min\left\{ 0 , \mcl X^1(2) +1\right\}.
  \eqen
  Since $\lambda_0$ is incoming to $x_1^0 = x_2^1$, we also have $\wt{\mcl X}(1) = \mcl X^1(2)$. Subtracting gives~\eqref{eqn-sdp-relating3}. 
\end{proof}
 
Lemma~\ref{lem:SDP-relating-X1-X} immediately implies the following recursive equation for the distributions $f$ and $g$.

 \begin{prop}[SDP recursive equation]\label{prop:SDP-recursive-equation} 
 For all $x\in\mathbb Z$ and $y\in\mathbb Z_{\ge-1}$.
 \begin{equation}\label{eq:master-eq}
 \;3\,f(x) g(y) 
 = f(x)\Bigg( f(y) + \sum_{j=-1}^{y+1} g(j)g\big(y-j\big) \Bigg)
 + \BB{1}_{\{y=-1\}}\Bigg( \sum_{j\le-2} f(j) f\big(x-j-1\big) + g(-1)^2 f(x+1) + g(x-1) \Bigg).\;
 \end{equation}
 \end{prop}

\begin{proof}
    Fix $x\in\mathbb Z$ and $y\in\mathbb Z_{\ge -1}$. Using the fact that $\mcl X^1 \eqD \mcl X$ (which follows from the fact that $M_{0,\infty} \eqD M_{1,\infty}$), we can write
    \begin{align*}
        \BB P &\left[ \Big(\mcl X^1(-1)\,,\,\mcl X^1(1)\Big)=(x,y)\right]=\BB P \left[ \Big(\mcl X(-1)\,,\,\mcl X(1)\Big)=(x,y)\right].
    \end{align*}
    Thanks to Lemma~\ref{lem:SDP-relating-X1-X} \jb{and using that the map $M_{1,\infty}$ is independent from $\mcl Z(1) - \mcl Z(0)$ by definition of the KMSW procedure (Definition~\ref{def-kmsw}),} we have that 
    \begin{align}\label{eq:splitting-proba-in-3}
        \BB P \left[ \Big(\mcl X(-1)\,,\,\mcl X(1)\Big)=(x,y)\right]&=\frac{1}{3}\BB P \left[ \left(\mcl X^1(-2) - \min\{\mcl X^1(-1),1\}\,,\, - \min\{\mcl X^1(-1),1\}\right]
        =(x,y)\right]\notag\\
        &\,+
        \frac{1}{3}\BB P \left[ \left( - \mcl X^1(1)-1\,,\,-1\right)
        =(x,y)\right]\\
        &\,+
        \frac{1}{3}\BB P \left[ \left(\mcl X^1(-1)-\min\{0,\mcl X^1(2)+1\} \,,\, \mcl X^1(2)-\min\{0,\mcl X^1(2)+1\} \right)
        =(x,y)\right].\notag
    \end{align}
    By stationarity (recall \eqref{eq:stationarity-rel} and \eqref{eq:den-dist-f-g}) and the fact that $\mcl X(0) = \mcl X^1(0) = 0$,
    \eqbn
    f(x) = \BB P\left[   - \mcl X(-1) = x \right] = \BB P\left[  - \mcl X^1(-1) = x \right]
    \quad \text{and} \quad 
    g(y)= \BB P\left[ \mcl X(1)  = y \right]  = \BB P\left[ \mcl X^1(1)  = y \right]  .
    \eqen 
    Furthermore, the increments $\{\mcl X^1(k) - \mcl X^1(k-1)\}_{k\in\BB Z}$ are independent, the increments $\{\mcl X^1(k) - \mcl X^1(k-1)\}_{k \geq 1}$ all have the same distribution, and the increments $\{\mcl X^1(k) - \mcl X^1(k-1)\}_{k \leq 0}$ all have the same distribution (thanks to Items~\ref{item-busemann-ind}~and~\ref{item-busemann-stationary} in Theorem~\ref{thm-busemann-property}). Applying these facts in \eqref{eq:splitting-proba-in-3} gives
    \begin{align*}
        3f(-x)g(y) &=f(-x)f(y)+\BB{1}_{\{y=-1\}}\sum_{j\leq-2}  f(-x-j-1) f(j)\\
        &\,\,\,\,\,+ \BB{1}_{\{y=-1\}}g(-x-1)\\
        &\,\,\,\,\,+f(-x)\sum_{j=-1}^{y+1} g(j)g\big(y-j\big)
        +\BB{1}_{\{y=-1\}}g(-1)^2f(-x+1),
    \end{align*}
    where the first probability on the right-hand side of \eqref{eq:splitting-proba-in-3} involving the minimum was computed by distinguishing the cases when $\mcl X^1(-1)\leq 0$ and $\mcl X^1(-1)\geq 1$, and the second probability involving the minimum was computed by distinguishing the cases when $\mcl X^1(2)\geq -1$ and $\mcl X^1(2)= -2$ (since we must have that $\mcl X^1(2)\geq -2$ because $g$ is supported on $\BB Z_{\geq -1}$). Changing $-x$ to $x$, we obtain the expression in the statement of the lemma.
\end{proof}

 We now proceed with the proof of Proposition~\ref{prop:asympt-exp} by first establishing some preliminary results.

\begin{lem}\label{lem:Y-not-const2}
    Let $Y\sim f$. Then $Y$ is not a.s.\ constant and $\BB P(Y=0)=f(0)>0$. Moreover, $g(-1)> 0$.
 \end{lem}

 \begin{proof}
We first prove that $Y$ is not a.s.\ constant. Since $f(x) = f(-x)$, it is enough to prove that $f(0)<1$. Assume for contradiction that $f(0)=1$.
Setting $x=0$ in \eqref{eq:master-eq} gives, for all $y\ge -1$,
\begin{equation}\label{eq:fwjbwif}
3g(y)
= f(y)+\sum_{j=-1}^{y+1} g(j)g(y-j)
+\BB 1_{\{y=-1\}}\,g(-1).
\end{equation}
Take $y=-1$ in \eqref{eq:fwjbwif}. Noting $f(-1)=0$ (since $f(0) =1$) and that the sum reduces to
$2g(-1)g(0)$, we obtain
\[
3g(-1)=2g(-1)g(0)+g(-1)
\quad\Longrightarrow\quad
g(-1)=g(-1)g(0).
\]
Hence either $g(-1)=0$ or $g(0)=1$.
If $g(0)=1$, taking $y=0$ in \eqref{eq:fwjbwif} yields $3=1+1$, a contradiction;
thus $g(-1)=0$. Taking $y=0$ in \eqref{eq:fwjbwif}:
\[
3g(0)=1+g(0)^2\quad\Longrightarrow\quad g(0) = \frac12(3-\sqrt 5)   \le \tfrac12.
\]
Next take $y=1$ in \eqref{eq:fwjbwif}. The only nonzero terms in the sum are
$j=0,1$, so $3g(1)=2g(0)g(1)$, which forces $g(1)=0$ because $g(0)\leq \tfrac{1}{2}$.
An induction using \eqref{eq:fwjbwif} shows that for every $y\ge2$ the sum again
reduces to $2g(0)g(y)$ and hence $g(y)=0$. Therefore
$g(k)=0$ for all $k\ge1$, and since also $g(-1)=0$ we have
$\sum_{k}g(k)=g(0)\le\tfrac12\neq 1$; a contradiction.

\smallskip

We now prove that $f(0)>0$ and $g(-1)> 0$.
Fixing $y=-1$ in \eqref{eq:master-eq} and summing over $x\in\mathbb Z$, we obtain
\begin{equation}\label{eq:wekfbwebfqo}
3g(-1)= f(-1)+2g(-1)g(0)+\sum_{j\le-2} f(j)+g(-1)^2+1.
\end{equation}
By symmetry of $f$, we have that $f(-1)+\sum_{j\le-2} f(j)=\frac{1-f(0)}{2}$,
so from \eqref{eq:wekfbwebfqo},
\begin{equation}\label{eq:B-lb}
3g(-1)=\frac{3-f(0)}{2}+2g(-1)g(0)+g(-1)^2
\ \ge\ \frac{3-f(0)}{2}.
\end{equation}
In particular, $g(-1)>0$. Now, evaluate \eqref{eq:master-eq} at $(x,y)=(0,-1)$:
\begin{equation*}
3f(0)g(-1)
= f(0)\big(f(-1)+2g(-1)g(0)\big)
   +\sum_{j\le-2} f(j)f(-j-1)+g(-1)^2 f(1)+g(-1)
\ge g(-1).
\end{equation*}
If $f(0)=0$, the last equation gives a contradiction because $g(-1)>0$.
\end{proof}
 
 Now, multiplying  \eqref{eq:master-eq} by $e^{i(tx+sy)}$ and summing over $x\in\mathbb Z$ and $y\in\BB Z_{\ge-1}$, we get the equation
 \begin{align}\label{eq:big-rel2}
     3F(t)G(s) &= F(t)\left[ F(s) - F^{-}(s) + f(-1)e^{-is} + G^{2}(s) - g(-1)^{2}e^{-2is} \right]\notag\\
     &\quad+ e^{-is}\Bigl[ e^{it} F^{-}(t)F(t) - f(-1)F(t) + g(-1)^{2} F(t) e^{-it} + G(t) e^{it} \Bigr].
 \end{align}
 Setting $t=0$ and using that $F(0)=1$,  $F^{-}(0)=(1-f(0))/2$ and $G(0)=1$, we get that
 \begin{equation}\label{eq:first-sub}
     3G(s) = \left[ F(s) - F^{-}(s) +   G^{2}(s) - g(-1)^{2}e^{-2is} \right]+ e^{-is}\Bigl[  (1-f(0))/2  + g(-1)^{2}   + 1  \Bigr].
 \end{equation}
 Similarly, setting $s=0$, we get that
 \begin{equation*}
     3F(t) = F(t)\left[ 2 - (1-f(0))/2   - g(-1)^{2} \right]+\Bigl[ e^{it} F^{-}(t)F(t)  + g(-1)^{2} F(t) e^{-it} + G(t) e^{it} \Bigr].
 \end{equation*}
 This latter equation allows us to write $G(t)$ in terms of $F(t)$ and $F^{-}(t)$. More precisely, we have that 
 \begin{equation}\label{eq:expr-G}
     G(t) = -F(t)\,\widehat F^-(t)
 \end{equation}
 where 
 \begin{equation}\label{eq:fhatmin}
     \widehat{F}^{-}(t):=F^{-}(t)+c(t)
 \end{equation}
 and 
 \begin{equation}\label{eq:defn-ct}
     c(t):=\frac{2g(-1)^2+e^{it}(-3-2g(-1)^2+f(0))}{2e^{2it}}.
 \end{equation}
 Substituting \eqref{eq:expr-G} in \eqref{eq:first-sub}, one obtains the quadratic equation\footnote{Just performing the substitution, one gets a minus sign on the left-hand side of \eqref{eq:quadratic} that we removed for simplicity.} in $\widehat{F}^{-}$:
 \begin{equation}\label{eq:quadratic}
     F^2 (\widehat{F}^{-})^2+(3F-1) \widehat{F}^{-}+F=0.
 \end{equation}
 Note that this is exactly the same equation that we obtained in~\eqref{eq:quadratic-new} in the LDP case (but $c(t)$ is defined differently).

We are now ready to complete the proof of Proposition~\ref{prop:asympt-exp}.

 \begin{proof}[Proof of Proposition~\ref{prop:asympt-exp}]
 The fact that $g(-1)> 0$ follows from Lemma~\ref{lem:Y-not-const2}. We now determine the asymptotic expansions.
 	Recalling that $F(t)\in\mathbb R$ and writing
 \begin{equation*}
     F(t)=f(0)+F^{-}(t)+\overline{F^{-}(t)}\stackrel{\eqref{eq:fhatmin}}{=}f(0)+(\widehat{F}^{-}(t)-c(t))+\overline{(\widehat{F}^{-}(t)-c(t))}=f(0)-2\re(c(t))+2\re(\widehat{F}^{-}(t)),
 \end{equation*}
 we set 
 \[K(t):=f(0)-2\re(c(t))\stackrel{\eqref{eq:defn-ct}}{=}f(0)+(3+2g(-1)^2-f(0))\cos(t)-2g(-1)^2\cos(2t),\]
 so that $K(t)$ encodes the part of $F(t)$ that is independent of $\widehat{F}^{-}(t)$ (which we recall is our indeterminate in \eqref{eq:quadratic}) and 
 \begin{equation}\label{eq:real-part-fhat2}
     \re(\widehat{F}^{-}(t))=\frac{F(t)-K(t)}{2}.
 \end{equation}
 Hence, thanks to Lemma~\ref{lem:nonreal2} (which is valid also in this SDP case thanks to Lemma~\ref{lem:Y-not-const2}),
 we can write
 \begin{equation*}
     \widehat{F}^{-}(t)=\alpha(t)+i\beta(t)\stackrel{\eqref{eq:real-part-fhat2}}{=}\frac{F(t)-K(t)}{2}+i\beta(t),
 \end{equation*} 
 and there exists $\delta>0$ such that $\beta(t)\ne 0$ for $0<t<\delta$.

 Now, taking the imaginary part in \eqref{eq:quadratic}, we obtain the equation
 \begin{equation*}
     \beta(F^3-K\,F^2+3F-1)=0,
 \end{equation*}
 and since $\beta(t)\ne 0$ for $0<t<\delta$, we get that for $0<t<\delta$,
 \begin{equation}\label{eq:third-ord}
     F^3-K\,F^2+3F-1=0,
 \end{equation}
 The Taylor expansion of $K(t)$ around $t=0$ is 
 \begin{equation*}
     K(t)=3-\frac{\kappa}{2}\, t^2+\left(\frac{\kappa}{24}-g(-1)^2\right)\,t^4+O(t^6),
 \end{equation*}
 where we recall that  $\kappa=3-6g(-1)^2-f(0)$. Writing $F(t)=1+\eps(t)$ with $\eps(t)\to 0$ as $t\downarrow 0$, we get from \eqref{eq:third-ord} that
 \begin{equation*}
     0=\eps^3+\frac{\kappa}{2}\,t^2 (\eps^2+2\eps+1)-\left(\frac{\kappa}{24}-g(-1)^2\right)\,t^4 (\eps^2+2\eps+1)+O(\eps^4+t^6+t^4\eps).
 \end{equation*}
Now, to ensure that the right-hand side of the latter equation tends to zero as $t$ decreases to zero:
 \begin{itemize}
     \item If $\kappa=0$, we must have that $\eps^3 (1 + o(1)) = -g(-1)^2 t^4 ( 1+o(1))$ as $t\downarrow 0$, and so $\eps(t)= -g(-1)^{2/3}\,t^{4/3}+o(t^{4/3})$ as $t\downarrow 0$.
     \item If $\kappa\neq 0$, we must have that $\eps^3 (1+o(1))=-\frac{\kappa}{2}\,t^2 (1+o(1))$ as $t\downarrow 0$. Since $F$ is real, $F(0)=1$, and $|F(t)|\leq 1$ for all $t\in \BB R$, it must be that $\kappa>0$ and so $\eps(t)= -(\frac{\kappa}{2})^{1/3}\,t^{2/3}+o(t^{2/3})$ as $t\downarrow 0$.
 \end{itemize}
 The above two items complete the proof for the expansion of $F(t)$.
 
 \medskip

 We now determine the expansions of $F^-(t)$ and $G(t)$. Solving the quadratic equation \eqref{eq:quadratic}, the two possible continuous choices of branch for $\widehat{F}^-(t)$ are
 \begin{equation*}
 	\widehat{F}_{\pm}^-(t)=\frac{1-3F\pm\sqrt{\Delta(F)}}{2F^2},\qquad{\text{with $\Delta(F)=(3F-1)^2-4F^3$,}}
 \end{equation*}
 and recalling \eqref{eq:fhatmin}, we obtain that the two possible continuous choices of branch for $F^-(t)$ are
  \begin{equation}\label{expr-for-F-}
 	F_{\pm}^-(t)=\frac{1-3F\pm\sqrt{\Delta(F)}}{2F^2}-c(t),\qquad{\text{with $\Delta(F)=(3F-1)^2-4F^3.$}}
 \end{equation}
 Moreover, from \eqref{eq:expr-G}, we have that the two possible continuous choices of branch for $G(t)$ are
 \begin{equation}\label{eq:Gtexpr}
 	G_{\pm}(t) = -F(t)\,\widehat F^-_{\pm}(t).
 \end{equation}
 Therefore,
 \begin{itemize}
 	\item If $\kappa=0$, since $F(t)= 1-g(-1)^{2/3}\,t^{4/3}+o(|t|^{4/3})$ as $t\downarrow 0$, we get from \eqref{eq:defn-ct} and \eqref{expr-for-F-} that
 	\begin{equation*}
 		F_{\pm}^-(t)=\left(\frac{1-f(0)}{2}\right)-2 i g(-1)^2 t+\frac{1}{2} (-1 \pm i\sqrt{3}) g(-1)^{2/3}t^{4/3}+o(t^{4/3}),
 	\end{equation*}
 	and from \eqref{eq:Gtexpr},
 	\begin{equation*}
 		G_{\pm}(t)=1+\frac{1}{2}\left(-g(-1)^{2/3} \mp i\sqrt{3} g(-1)^{2/3}\right)t^{4/3}+o(t^{4/3}).
 	\end{equation*}
 	\item If $\kappa>0$, since $F(t)= 1-\left(\frac{\kappa}{2}\right)^{1/3}\,t^{2/3}+o(t^{2/3})$ as $t\downarrow0$, we get from \eqref{eq:defn-ct} and \eqref{expr-for-F-} that
 	\begin{equation}\label{eq:wrong-F-asympt}
 		F_{\pm}^-(t)=\left(\frac{1-f(0)}{2}\right)+\frac{1}{2} (-1 \pm i\sqrt{3})\left(\frac{\kappa}{2}\right)^{1/3}t^{2/3}+o(t^{2/3})
 	\end{equation}
 	and from \eqref{eq:Gtexpr},
 	\begin{equation*}
 		G_{\pm}(t)=1+\frac{1}{2}(-1 \mp i\sqrt{3})\left(\frac{\kappa}{2}\right)^{1/3}t^{2/3}+o(t^{2/3}).
 	\end{equation*} 
 \end{itemize}
 Noting that the conjugate of $F^{-}(t)$ is the characteristic function of a non-negative integer-valued random variable, and using Proposition~\ref{prop:Tauberian}, we have that the only admissible expressions when $\kappa=0$ are the ones in the proposition statement. 

 \bigskip
 
 Hence, it remains to prove that the case $\kappa>0$ is not possible. This does not appear to follow from the recursive equation of Proposition~\ref{prop:SDP-recursive-equation}, and it is the only part of the proof that is not directly parallel to the LDP case. If $\kappa>0$, then the asymptotics in \eqref{eq:wrong-F-asympt} combined with Proposition~\ref{prop:Tauberian} give that as $x\to\infty$,
 \begin{equation}\label{eq:incorrect-tail}
     \BB P\left[ |\mcl X(0) - \mcl X(-1)| \geq x \right] = c \, x^{-2/3} + o(x^{-2/3}),
 \end{equation}
 for some $c>0$.

 Recall from Theorem~\ref{thm-busemann} that $\{x_k\}_{k\in\BB Z} $ is the sequence of boundary vertices of $M_{0,\infty}$, ordered from left to right so that $\lambda_0 = (x_0,x_1)$. 
The idea of the proof is to show that $\mcl X(-k)$ is bounded above by $O(k)$ with high probability, which will violate the heavy-tailed central limit theorem if $\kappa > 0$. We will do this using the triangle inequality and the fact that the rightmost directed path started from $x_{-k}$ typically takes $O(k)$ units of time to coalesce into the rightmost directed path started from $x_1$ (i.e., the right boundary of $M_{0,\infty}$). See Figure~\ref{fig-sdp-identify} for an illustration.

\begin{figure}[ht!]
\begin{center}
\includegraphics[width=0.5\textwidth]{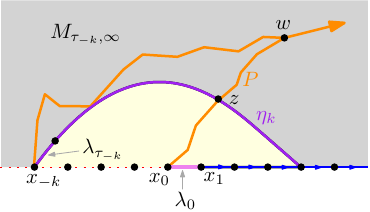}  
\caption{\label{fig-sdp-identify}  
Illustration for the proof that $\kappa>0$ is not admissible. 
}
\end{center}
\end{figure}

For $k\in \BB N$, let $\tau_{-k}$ be the smallest $n\in\BB N$ which $x_{-k}$ is the initial vertex of $\lambda_n$, as in~\eqref{eqn-bdy-vertex-hit}. By~\eqref{eqn-bdy-vertex-kmsw}, $\tau_{-k}$ can equivalently be defined as the smallest $n\in\BB N$ for which $\mcl L(n) = -k$. By the convergence of the re-scaled encoding walk $k^{-1} \mcl Z( \lfloor k^2 \cdot \rfloor )$ to a (correlated) two-dimensional Brownian motion, for each $\delta \in (0,1)$, there exists $A  > 1$ such that for every $k\in\BB N$, it holds with probability at least $1-\delta$ that 
\eqb  \label{eqn-sdp-identify-max}
 \max_{n \in [0,\tau_{-k}]\cap\BB Z} |\mcl R(n)| \leq A  k .
\eqe 

Let $\eta_k$ be the path on the boundary of $M_{\tau_{-k},\infty}$ from $x_{-k}$ to the point where the right boundaries of $M_{\tau_{-k},\infty}$ and $M_{0,\infty}$ meet. Equivalently, $\eta_k$ consists of the edge $\lambda_{\tau_{-k}}$ followed by the upper-right boundary of the submap of $M_{0,\infty}$ obtained by applying the KMSW procedure to $\mcl Z|_{[0,\tau_{-k}]}$.  By Lemma~\ref{lem-uibot-bdy}, $\eta_k$ is a directed path in $M_{\tau_{-k},\infty}$ started from $x_k$.  By~\eqref{eqn-kmsw-bdy} of Lemma~\ref{lem-kmsw-bdy}, if~\eqref{eqn-sdp-identify-max} holds, then the length of $\eta_k$ is at most $2 A k +1$. Hence
\eqb  \label{eqn-sdp-identify-bdy}
\BB P\left[ |\eta_k| \leq 2 A k + 1\right] \geq 1 -\delta . 
\eqe  

By Theorem~\ref{thm-busemann}, there exists a vertex $w = w(k) \in \mcl V( M_{\tau_{-k} ,\infty} ) $ such that the SDP Busemann function satisfies
\eqb  \label{eqn-sdp-identify-busemann} 
  \op{SDP}_{M_{0,\infty}}(x_{-k}, w)   -  \op{SDP}_{M_{0,\infty}}(x_0 , w)   =  \mcl X(-k)  - \mcl X(0) =   \mcl X(-k)   .
\eqe  
Let $P$ be an SDP geodesic from $x_0$ to $w$ in $M_{0,\infty}$. Since $\eta_k$ is a path in $M_{0,\infty}$ from $x_{-k}$ to $x_r$ for some $r\geq 0$, topological considerations imply that $P$ must hit a vertex on $\eta_k$. Let $z$ be such a vertex. 
Then
\eqb  \label{eqn-sdp-identify-tri} 
 \op{SDP}_{M_{0,\infty}}(x_0 , w) = \op{SDP}_{M_{0,\infty}}(x_0 , z) + \op{SDP}_{M_{0,\infty}}(z , w) .
\eqe 
On the other hand, concatenating the segment of $\eta_k$ from $x_{-k}$ to $z$ with the segment of $P$ from $z$ to $w$ gives a directed path from $x_{-k}$ to $w$. Hence
\eqb  \label{eqn-sdp-identify-bound} 
  \op{SDP}_{M_{0,\infty}}(x_{-k}, w)  \leq |\eta_k| +  \op{SDP}_{M_{0,\infty}}(z , w) .
\eqe
We therefore obtain that with probability at least $1-\delta$, 
\begin{align*}
   \mcl X(-k) 
&=   \op{SDP}_{M_{0,\infty}}(x_{-k}, w)   -  \op{SDP}_{M_{0,\infty}}(x_0 , w) \quad &&\text{by~\eqref{eqn-sdp-identify-busemann}} \notag\\ 
&\leq   |\eta_k|  -  \op{SDP}_{M_{0,\infty}}(x_0 , z)  &&\text{by~\eqref{eqn-sdp-identify-tri} and~\eqref{eqn-sdp-identify-bound}} \notag\\
&\leq   |\eta_k|  \leq  2 A k + 1 &&\text{by~\eqref{eqn-sdp-identify-bdy}} . 
\end{align*}
Since $\delta \in (0,1)$ is arbitrary, any possible limit in law of the random variable $ k^{-3/2} \mcl X(-k) $ is a.s.\ non-positive. 
Since $\mcl X(-k) \eqD -\mcl X(-k)$ by Item~\ref{item-busemann-sym} of Theorem~\ref{thm-busemann-property}, this implies that $k^{-3/2} \mcl X(-k)$ converges in probability to zero. By the heavy-tailed central limit theorem, this shows that the increments $\mcl X(-k+1) - \mcl X(-k) \eqD \mcl X(0) - \mcl X(-1)$ for $k\leq 0$ cannot have the tail asymptotic in \eqref{eq:incorrect-tail}. This shows that the case $\kappa>0$ is not admissible, concluding the entire proof.
\end{proof}

\begin{proof}[Proof of Theorem~\ref{thm-busemann-conv}]
This is immediate from Theorems~\ref{thm-busemann-property}~and~\ref{thm-busemann-tail} and the heavy-tailed functional central limit theorem (see, e.g., \cite[Chapter VII]{js-limit-thm}). %see Propostion 10.1 in Curien peeling notes
%\end{proof}
 \end{proof}

%% file: tex/finite.tex
The goal of this section is to prove the finite-volume versions of our results, i.e.\ Theorems~\ref{thm-boltzmann-ldp}--\ref{thm-cell-sdp}.
%Theorems~\ref{thm-cell-ldp}, \ref{thm-cell-sdp}, and Theorem~\ref{thm-boltzmann-ldp}. 
We will first prove the statements for size-$n$ KMSW cells (Theorems~\ref{thm-cell-ldp} and \ref{thm-cell-sdp}), then use these to deduce the statement for Boltzmann bipolar-oriented triangulations (Theorem~\ref{thm-boltzmann-ldp}).

\subsection{LDP in size-$n$ KMSW cells}
\label{sec-cell-ldp}

Recall that $M_{-\infty,\infty}$ denotes the UIBOT, the random walk $\mcl Z =(\mcl L,\mcl R): \BB Z \to\BB Z^2$ is its KMSW encoding walk as in Proposition~\ref{prop-kmsw-uibot}, $\{x_k\}_{k\in\BB Z}$ is the ordered sequence of boundary vertices of $M_{0,\infty}$ (as in Theorem~\ref{thm-busemann}), and $\tau_k$ (introduced in~\eqref{eqn-bdy-vertex-hit}) is the smallest $n\in\BB N$ such that $x_k$ is the initial vertex of the active edge $\lambda_n$. 

For $n\in\BB N$, let $M_{0,n}$ be the submap of $M_{0,\infty}$ obtained by applying the KMSW procedure (Definition~\ref{def-kmsw}) to $\mcl Z|_{[0,n]}$, as in Section~\ref{sec-cell-intro}. In this subsection we will deduce Theorem~\ref{thm-cell-ldp} from our results for the UIBOT. 

We will only be working with longest directed paths in this subsection, so in what follows we assume that $\mcl X$ is the Busemann function in Theorem~\ref{thm-busemann} with $\XDP = \op{LDP}$. 
The following deterministic lemma is the main input in the proof of the lower bound for LDP distances in $M_{0,n}$.

\begin{figure}[ht!]
\begin{center}
\includegraphics[width=0.5\textwidth]{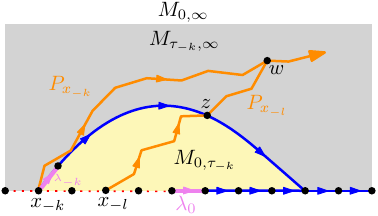}  
\caption{\label{fig-long-path}  
Illustration of the proof of Lemma~\ref{lem-long-path}. We prove a lower bound for the length of the segment of the leftmost LDP geodesic $P_{x_{-\el}}$ before it hits $z$. 
}
\end{center}
\end{figure}

\begin{lem} \label{lem-long-path}
Let $\el ,k \in\BB N$ with $\el < k$ and recall that $\tau_{-k}$ is the time from~\eqref{eqn-bdy-vertex-hit}. There is a directed path contained in $M_{0,\tau_{-k}}$ with length at least 
\eqbn
\max\left\{ \mcl X(-\el) - \mcl X(-k) , 0\right\} .
\eqen
\end{lem}
\begin{proof} 
This is an elementary geometric argument. 
See Figure~\ref{fig-long-path} for an illustration. 
To lighten notation, we abbreviate $\op{LDP} = \op{LDP}_{M_{0,\infty}}$. 
Let $P_{x_{-\el}}$ and $P_{x_{-k}}$ be the leftmost LDP geodesics in $M_{0,\infty}$ from $x_{-\el}$ and $x_{-k}$ to $\infty$, as in Lemma~\ref{lem-geo-infty}. 

Let $T$ be the smallest time for which the edge $P_{x_{-\el}}(T+1)$ does not belong to $  M_{0,\tau_{-k}}$ and let $z$ be the initial vertex of $P_{x_{-\el}}(T+1)$. Then $P_{x_{-\el}}|_{[0,T]}$ is a directed path in $M_{0,\tau_{-k}}$ from $x_{-\el}$ to $z$, of length $T = \op{LDP}(x_{-\el} , z)$. We will prove a lower bound for this length. 

Let $w$ be any vertex with the property that the segments of $P_{x_{-\el}}$ and $P_{x_{-k}}$ after they hit $w$ are identical (such a vertex exists by Lemma~\ref{lem-geo-infty}). By the definition of the Busemann function in Theorem~\ref{thm-busemann}, we can choose $w$ so that
\eqb \label{eqn-long-path-busemann}
  \op{LDP}(x_{-\el} , w )  - \op{LDP}(x_{-k}, w)  = \mcl X(-\el) - \mcl X(-k) .
\eqe  
The vertices $x_{-k}$ and $z$ each lie on the right boundary of $M_{\tau_{-k},\infty}$. The segment of this right boundary between $x_{-k}$ and $z$ is a directed path from $x_{-k}$ to $z$ in $M_{0,\infty}$. Therefore,
\eqb  \label{eqn-long-path-mono}
\op{LDP}(z,w) \leq \op{LDP}(x_{-k} , w) .
\eqe 
Since $z$ and $w$ are both hit by the LDP geodesic $P_{x_{-\el}}$, 
\eqb \label{eqn-long-path-geo}
 \op{LDP}(x_{-\el} , w)  =  \op{LDP}(x_{-\el} ,z )  + \op{LDP}(z,w)  .
\eqe
We have
\allb
 \op{LDP}(x_{-\el} ,z )
&=  \op{LDP}(x_{-\el} , w)  -  \op{LDP}(z,w) \quad &&\text{by~\eqref{eqn-long-path-geo}} \notag\\
&=   \mcl X(-\el) - \mcl X(-k) +   \op{LDP}(x_{-k} , w)  -   \op{LDP}(z,w) \quad &&\text{by~\eqref{eqn-long-path-busemann}} \notag\\
&\geq   \mcl X(-\el) - \mcl X(-k)   \quad &&\text{by~\eqref{eqn-long-path-mono}} .
\alle
As we saw above, the directed path $P_{x_{-\el}}|_{[0,T]}$ is a directed path in $M_{0,\tau_{-k}}$ with length $ \op{LDP}(x_{-\el} ,z )$, so this concludes the proof.
\end{proof}

We can now prove the lower bound in Theorem~\ref{thm-cell-ldp}. We actually obtain a slightly more quantitative statement. 

\begin{lem} \label{lem-cell-ldp-lower}
For each $n\in\BB N$ and $A> 1$, 
\eqb \label{eqn-cell-ldp-lower}
\BB P\left[ \text{$M_{0,n}$ contains a directed path of length at least $A^{-1} n^{3/4}$}  \right] 
\geq 1 -  A^{-2/3 + o(1)} ,
\eqe
where the $o(1)$ goes to zero as $A\to\infty$, at a rate which is uniform in $n$. 
\end{lem}

\begin{proof}
Fix \jb{$\ep\in(0,2/3)$}, which we will eventually send to zero to get the $o(1)$ error in~\eqref{eqn-cell-ldp-lower}.

\jb{If $n\leq A^{4/3}$, then $A^{-1}n^{3/4}\leq1$, so the root edge $\lambda_0$ gives the required path with probability one. Henceforth assume that $n>A^{4/3}$, so that $A^{-1}n^{3/4}>1$.}

By Theorem~\ref{thm-busemann-property}, the increments $\mcl X(-k+1)-\mcl X(-k)$ for $k\geq1$ are i.i.d.
By Theorem~\ref{thm-busemann-tail}, there is a universal constant $c_1>0$ such that for all $k\geq1$
\eqb \label{eqn-use-lower-tail}
\BB P\left[\mcl X(-k+1)-\mcl X(-k)\geq A^{-1}n^{3/4}\right]
\geq c_1A^{2/3}n^{-1/2}.
\eqe
By multiplying the bound~\eqref{eqn-use-lower-tail} over $\lfloor A^{-2/3+\ep}n^{1/2}\rfloor$ values of $k$, we obtain
\allb \label{eqn-lower-tail-prod}
&\BB P\left[
\max\left\{\mcl X(-k+1)-\mcl X(-k):
k\in\jb{[2,A^{-2/3+\ep}n^{1/2}+1]}\cap\BB Z\right\}
\geq A^{-1}n^{3/4}
\right]\notag\\
&\qquad\qquad\qquad\qquad
\geq1-\left(1-c_1A^{2/3}n^{-1/2}\right)^{\lfloor A^{-2/3+\ep}n^{1/2}\rfloor}\notag\\
&\qquad\qquad\qquad\qquad
\geq1-c_2e^{-c_3A^\ep}
\alle
for constants $c_2,c_3>0$ depending only on $\ep$.

The process $\mcl L$ is a random walk with i.i.d.\ increments sampled uniformly from $\{-1,0,1\}$. By~\eqref{eqn-bdy-vertex-kmsw}, the time $\tau_{-k}$ is the first positive time at which $\mcl L$ hits $-k$. By standard estimates for random walk,
\eqb \label{eqn-lower-tail-time}
\BB P\left[
\tau_{\jb{-(\lfloor A^{-2/3+\ep}n^{1/2}\rfloor+1)}}\leq n
\right]\geq1-c_4A^{-2/3+\ep},
\eqe
for a universal constant $c_4>0$\jb{, where we used $n^{-1/2}\leq A^{-2/3}\leq A^{-2/3+\ep}$ to absorb the extra $1$ in the hitting level}.

By Lemma~\ref{lem-long-path}, if the event in~\eqref{eqn-lower-tail-prod} occurs, then the map $M_{0,\tau_{\jb{-(\lfloor A^{-2/3+\ep}n^{1/2}\rfloor+1)}}}$ contains a directed path of length at least $A^{-1}n^{3/4}$. If the event in~\eqref{eqn-lower-tail-time} also occurs, then $M_{0,n}$ also contains a directed path of length at least $A^{-1}n^{3/4}$. These events occur simultaneously with probability at least $1-O(A^{-2/3+\ep})$. Since $\ep>0$ is arbitrary, this gives~\eqref{eqn-cell-ldp-lower}.
\end{proof}

We now turn our attention to the upper bound. The main input is the following elementary estimate for the maximum of the Busemann function.

\begin{lem} \label{lem-ldp-fluctuation}
For each $k\in\BB N$ and $A> 1$,  
\eqb \label{eqn-cell-ldp-fluctuation}
\BB P\left[ \max_{\el \in [-k,k] \cap \BB Z} |\mcl X(\el)| \leq  A k^{3/2}  \right] 
\geq 1 - O(A^{-2/3}),
\eqe
with the implicit constant in the $O(1)$ uniform in $k$. 
\end{lem}

\begin{proof}
By Theorems~\ref{thm-busemann-property} and~\ref{thm-busemann-tail}, the increments $\mcl X(\el)-\mcl X(\el-1)$ are independent and each $|\mcl X(\el)-\mcl X(\el-1)|$ is stochastically dominated by a positive $2/3$-stable random variable \jb{plus a deterministic constant $C\geq1$} (the scale parameter for this stable random variable \jb{and the constant $C$} can be taken to be uniform in $\el$ by the stationarity property in Theorem~\ref{thm-busemann-property}).
Hence, we can couple $\mcl X$ with a collection $\{Y_\el\}_{\el\in\BB Z}$ of i.i.d.\ stable random variables in such a way that $|\mcl X(\el)-\mcl X(\el-1)|\leq\jb{C+}Y_\el$ for each $\el\in\BB Z$. Under this coupling,
\eqbn
\max_{\el\in[-k,k]\cap\BB Z}|\mcl X(\el)|
\leq \jb{2Ck+}\sum_{\el=-k+1}^kY_\el,
\eqen
which has the same law as $\jb{2Ck+}(2k)^{3/2}Y_0$ by the stability property. The bound~\eqref{eqn-cell-ldp-fluctuation} then follows from the standard tail estimate for the stable distribution\jb{, since $2Ck\leq2Ck^{3/2}$; increasing the implicit constant covers bounded values of $A$}.
\end{proof}

\begin{lem} \label{lem-cell-ldp-upper}
For each $n\in\BB N$ and $A >1$, 
\eqb  \label{eqn-cell-ldp-upper}
\BB P\left[ \max\left\{ \op{LDP}_{M_{0,n}}(x,y)  : x,y\in\mcl V(M_{0,n}) \right\} \leq A n^{3/4} \right] \geq 1 - A^{-2/3 + o(1)}
\eqe 
where the $o(1)$ goes to zero as $A\to\infty$, uniformly in $n$, and we set $\op{LDP}_{M_{0,n}}(x,y) = -\infty$ if there is no directed path in $M_{0,n}$ from $x$ to $y$. 
\end{lem}

\begin{figure}[ht!]
\begin{center}
\includegraphics[width=0.99\textwidth]{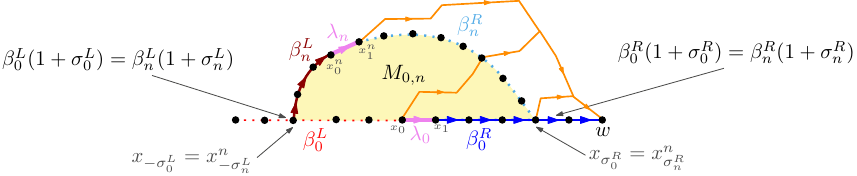}  
\caption{\label{fig-long-path2}  
Illustration of the proofs of Lemmas~\ref{lem-cell-ldp-upper}~and~\ref{lem-cell-sdp-upper}. The edge $\lambda_0$ is part of $\beta_0^R$ and the edge $\lambda_n$ is part of $\beta_n^R$. 
}
\end{center}
\end{figure}

\begin{proof} 
Recall that $M_{n,\infty}$ is the submap of $M_{-\infty,\infty}$ obtained by applying the KMSW procedure to $(\mcl Z(\cdot+n) - \mcl Z(n))|_{[0,\infty)}$. 
Let $\mcl X_n$ be the Busemann function for $M_{n,\infty}$, defined in the same way as $\mcl X$ but with $M_{n,\infty}$ in place of $M_{0,\infty}$. Since $M_{n,\infty} \eqD M_{0,\infty}$, we have $\mcl X \eqD \mcl X_n$. The basic idea of the proof is as follows. By~\eqref{eqn-kmsw-bdy} of Lemma~\ref{lem-kmsw-bdy}, the number of edges on the boundary of $M_{0,n}$ is bounded above by $2\max_{j \in [0,n]\cap\BB Z} |\mcl Z(j)| +1$ (with $|\mcl Z(j)|$ denoting the Euclidean norm), which is typically of order $n^{1/2}$. 
Using Lemma~\ref{lem-ldp-fluctuation}, we can therefore upper-bound the maximal values of each of $|\mcl X|$ and $|\mcl X_n|$ at the times corresponding to vertices on the boundary of $M_{0,n}$ by $O(n^{3/4})$ (see~\eqref{eqn-cell-bdy-sup}). If there is a long directed path between two vertices of $M_{0,n}$, then by concatenating paths we would find that the values of $|\mcl X_n|$ on the upper boundary of $M_{0,n}$ are much larger than the values of $|\mcl X|$ on the lower boundary of $M_{0,n}$. This would violate our bounds for $\mcl X$ and $\mcl X_n$. 
We now proceed with the details. See Figure~\ref{fig-long-path2} for an illustration.
\medskip

\noindent\textit{Step 1: bounds for $\mcl X$ and $\mcl X_n$.} \jbb{Recall that $M_{0,n}$ is obtained by applying the KMSW procedure to $\mcl Z|_{[0,n]}$.}
Let $\beta_0^R$ (resp.\ $\beta_n^R$) be the directed path which traverses the portion of the boundary of $M_{0,\infty}$ (resp.\ $M_{n,\infty}$) lying to the right of the initial vertex of $\lambda_0$ (resp.\ initial vertex of $\lambda_n$). \jbb{Write $LR$ and $UR$ for the lower-right and upper-right boundaries of $M_{0,n}$, respectively, and set $\sigma_0^R=|LR|$ and $\sigma_n^R=1+|UR|$, where $|\cdot|$ counts the number of edges. The first $\sigma_0^R$ edges of $\beta_0^R$ traverse $LR$, while the first $\sigma_n^R$ edges of $\beta_n^R$ traverse $\lambda_n$ followed by the images of $UR$ in $M_{n,\infty}$ (recall Definition~\ref{def-kmsw}). These prefixes have vertices indexed from $0$ to $\sigma_0^R$ and from $0$ to $\sigma_n^R$, respectively, and the edges of $UR$ are missing and unoriented in $M_{0,n}$.} Then
\eqb \label{eqn-cell-merge-time-R}
\beta_0^R(k  + \sigma_0^R) = \beta_n^R(k + \sigma_n^R) ,\quad \forall k \geq 1 . 
\eqe  
In an exactly analogous manner, we also define the left boundary paths $\beta_0^L$ and $\beta_n^L$ and the times $\sigma_0^L$ and $\sigma_n^L$, for which~\eqref{eqn-cell-merge-time-R} holds with $L$ in place of $R$. 

By~\eqref{eqn-kmsw-bdy} of Lemma~\ref{lem-kmsw-bdy}, $\sigma_0^R, \sigma_n^R , \sigma_0^L, \sigma_n^L$ are each bounded above by $2\max_{j\in [0,n]\cap\BB Z} |\mcl Z(j)| +1$. 
By Hoeffding's inequality for sums of i.i.d.\ bounded random variables, for any fixed $\ep \in \jbb{(0,2/3)}$,
\eqb \label{eqn-cell-bdy-tail}
\BB P\left[ \max\left\{ \sigma_0^R, \sigma_n^R , \sigma_0^L, \sigma_n^L \right\} \leq A^\ep n^{1/2} \right] \geq 1  -  c_1 e^{-c_2 A^{2\ep}} 
\eqe
for universal constants $c_1,c_2 > 0$. 

By~\eqref{eqn-cell-bdy-tail} combined with Lemma~\ref{lem-ldp-fluctuation} (applied once for each of $\mcl X$ and $\mcl X_n$ and with $k = A^\ep n^{1/2}$), it holds with probability at least $1 - O(A^{\jbb{-2/3+\ep}}) $ that 
\eqb \label{eqn-cell-bdy-sup}
  \max\left\{ \max_{\el \in [-\sigma_0^L ,\sigma_0^R ] \cap \BB Z} |\mcl X(\el)| ,  \max_{\el \in [-\sigma_n^L ,\sigma_n^R ] \cap \BB Z} |\mcl X_n(\el)| \right\}  \leq  A n^{3/4}    .
\eqe  
\medskip

\noindent\textit{Step 2: relating $\mcl X$ and $\mcl X_n$ to longest directed paths.}
Henceforth assume that the event in~\eqref{eqn-cell-bdy-sup} occurs.  
Let $\{x^n_k\}_{k\in\BB Z}$ (resp.\ $\{x_k\}_{k\in\BB Z}$) be the vertices on the boundary of $M_{n,\infty}$ (resp.\ $M_{0,\infty}$), enumerated in west to east order so that $\lambda_n =  (x_{0}^n , x_1^n) $ (resp.\ $\lambda_0 = (x_{0} , x_1)$). By~\eqref{eqn-cell-merge-time-R} and its analog with $L$ in place of $R$,
\eqb \label{eqn-cell-vertex-merge} 
x_{k + \sigma_0^R} = x_{k + \sigma_n^R}^n \quad \text{and} \quad x_{-k-\sigma_0^L} = x_{-k-\sigma_n^L}^n ,\quad\forall k \geq 0.
\eqe 

By the definition of the Busemann function in Theorem~\ref{thm-busemann}, applied, e.g., to a sequence of points $\{w_j\}_{j\in\BB N}$ on the right boundary of $M_{0,\infty}$, we can find a vertex $w$ of $M_{0,\infty}$ such that
\allb \label{eqn-cell-ldp-infty}
\mcl X(\el) &= \op{LDP}_{M_{0,\infty}}(x_\el , w) - \op{LDP}_{M_{0,\infty}}(x_0 , w)  , \quad\forall \el \in [-\sigma_0^L , \sigma_0^R] \cap\BB Z \quad \text{and} \notag\\
\mcl X_n(\el) &= \op{LDP}_{M_{n,\infty}}(x_\el^n , w) - \op{LDP}_{M_{n,\infty}}(x_0^n , w)  , \quad\forall \el \in [-\sigma_n^L , \sigma_n^R] \cap\BB Z  .
\alle
We claim that every directed path in $M_{0,\infty}$ started from $x_{\sigma_n^R}^n$ must stay in $M_{n,\infty}$. Indeed, such a path cannot hit the boundary of $M_{n,\infty}$ to the west of $x_{\sigma_n^R}^n$ since this would create a directed cycle in $M_{-\infty,\infty}$. Such a path cannot cross the boundary of $M_{n,\infty}$ to the east of $x_{\sigma_n^R}^n$ since by the definition of $\sigma_n^R$, the portion of the boundary of $M_{n,\infty}$ to the right of $x_{\sigma_n^R}^n$ is also part of the boundary of $M_{0,\infty}$. By~\eqref{eqn-cell-vertex-merge}, $x^n_{\sigma_n^R}  =x_{\sigma_0^R}$. Hence 
\eqb \label{eqn-cell-ldp-agree}
 \op{LDP}_{M_{n,\infty}}(x_{\sigma_n^R}^n , w) =  \op{LDP}_{M_{0,\infty}}(x_{\sigma_0^R} , w).
\eqe
We thus obtain that with probability at least $1 - O(A^{\jbb{-2/3+\ep}}) $
\allb \label{eqn-ldp-upper-shift}
&\left| \op{LDP}_{M_{n,\infty}}(x_0^n , w) - \op{LDP}_{M_{0,\infty}}(x_0 , w) \right| \notag\\
&\qquad \leq \left|   \op{LDP}_{M_{n,\infty}}(x_{\sigma_n^R}^n , w) -  \op{LDP}_{M_{0,\infty}}(x_{\sigma_0^R} , w)  \right| 
+  \left|    \mcl X_n(\sigma_n^R) - \mcl X(\sigma_0^R) \right| \quad &&\text{by~\eqref{eqn-cell-ldp-infty}} \notag\\
&\qquad = \left|    \mcl X_n(\sigma_n^R) - \mcl X(\sigma_0^R) \right|  \quad &&\text{by~\eqref{eqn-cell-ldp-agree}} \notag\\
&\qquad \leq 2 A n^{3/4}  \quad &&\text{by~\eqref{eqn-cell-bdy-sup}}  . 
\alle
\medskip

\noindent\textit{Step 3: bound for paths in $M_{0,n}$.}
Now let $x,y \in \mcl V(M_{0,n})$ be vertices which can be joined by a directed path in $M_{0,n}$. There is a vertex on $\bdy M_{0,n} \cap \bdy M_{0,\infty}$, i.e., a vertex $x_\el$ for $\el \in [-\sigma_0^L ,\sigma_0^R]\cap\BB Z$, which can be joined to $x$ by a directed path in $M_{0,n}$. Indeed, it is enough to consider a segment of the upper-left boundary of $M_{0,j}$ (whose edges are included in $M_{0,j}$), where $j \in [0,n]\cap\BB Z$ is chosen so that $x$ is one of the vertices of $\lambda_j$. 
Similarly, there exists $k\in [-\sigma_n^L ,\sigma_n^R]$ such that $y$ can be joined to $x_k^n$ by a directed path in $M_{0,n}$. By concatenating paths from $x_\el$ to $x$ to $y$ to $x_k^n$ to $w$, we get 
\eqbn
\op{LDP}_{M_{0,\infty}}(x_\el , w) 
\geq \op{LDP}_{M_{0,n}}(x,y) + \op{LDP}_{M_{n,\infty}}(x_k^n ,w) .
\eqen
Re-arranging gives that with probability at least $1 - O(A^{\jbb{-2/3+\ep}}) $
\begin{align*}
\op{LDP}_{M_{0,n}}(x,y) 
&\leq \op{LDP}_{M_{0,\infty}}(x_\el , w)  - \op{LDP}_{M_{n,\infty}}(x_k^n ,w) \notag\\ 
&= \op{LDP}_{M_{0,\infty}}(x_0 , w) - \op{LDP}_{M_{n,\infty}}(x_0^n , w) + \mcl X(\el) - \mcl X_n(k) \quad &&\text{by~\eqref{eqn-cell-ldp-infty}} \notag\\
&\leq 4 A n^{3/4}   \quad &&\text{by~\eqref{eqn-cell-bdy-sup} and~\eqref{eqn-ldp-upper-shift}}  .
\end{align*}
Since 
%we know that the event in~\eqref{eqn-ldp-upper-shift} holds with probability at least $1 - O(A^{-2(1-\ep)/3}) $ and 
$\ep > 0$ is arbitrary, this gives~\eqref{eqn-cell-ldp-upper}.
\end{proof}

\begin{proof}[Proof of Theorem~\ref{thm-cell-ldp}]
Combine Lemmas~\ref{lem-cell-ldp-lower} and~\ref{lem-cell-ldp-upper}. 
\end{proof}

\subsection{SDP in size-$n$ KMSW cells}
\label{sec-cell-sdp}

In this section, we prove Theorem~\ref{thm-cell-sdp}.  
Throughout, we assume that $\mcl X$ is the Busemann function of Theorem~\ref{thm-busemann} with $\XDP = \op{SDP}$. 
We also let $\{x_k\}_{k\in\BB Z}$ be the ordered sequence of vertices on the boundary of $M_{0,\infty}$, as in Theorem~\ref{thm-busemann}, so that $x_0$ is the initial vertex of the root edge. 
We begin with a lower bound for shortest directed paths.

\begin{lem} \label{lem-sdp-tri}
Almost surely, for each $\el , k \in \BB Z_{\geq 0}$ with $\el < k$,
\eqb  \label{eqn-sdp-tri}
\op{SDP}_{M_{0,\infty}}(x_\el,x_k) \geq   \mcl X(\el) - \mcl X(k) .
\eqe 
\end{lem}
\begin{proof}
By the definition of the Busemann function (Theorem~\ref{thm-busemann}), we can find $w\in M_{0,\infty}$ (depending on $\el$ and $k$) such that 
\eqb  \label{eqn-sdp-lower-busemann}
\mcl X(k) - \mcl X(\el) = \op{SDP}_{M_{0,\infty}}(x_k , w) - \op{SDP}_{M_{0,\infty}}(x_\el, w) .
\eqe 
Concatenating a directed path from $x_\el$ to $x_k$ with a directed path from $x_k$ to $w$ gives a directed path from $x_\el$ to $w$, so
\eqb  \label{eqn-sdp-conc} 
 \op{SDP}_{M_{0,\infty}}(x_\el, w) \leq  \op{SDP}_{M_{0,\infty}}(x_k , w)  + \op{SDP}(x_\el , x_k) .
\eqe 
(Note that there is at least one directed path from $x_\el$ to $x_k$, namely a segment of the right boundary of $M_{0,\infty}$.)
By re-arranging~\eqref{eqn-sdp-conc}, then applying~\eqref{eqn-sdp-lower-busemann}, we get~\eqref{eqn-sdp-tri}.
\end{proof}

\begin{lem} \label{lem-cell-sdp-lower}
Recall the definition of the lower-right boundary of $M_{0,n}$ from Definition~\ref{def-kmsw}, and recall that $M_{0,n}$ contains all of the edges on its lower-right boundary. 
For each $\delta > 0$, there exists $A = A(\delta) > 1$ 
such that for each $n \in \BB N$, the following is true. With probability at least $1-\delta$, there exists a vertex $y$ on the lower-right boundary of $M_{0,n}$ such that
\eqbn
\op{SDP}_{M_{0,n}}(x_0 , y) \geq A^{-1} n^{3/8}   .
\eqen 
\end{lem}
\begin{proof}
By Theorem~\ref{thm-busemann-conv}, the re-scaled Busemann function $k^{-3/4} \mcl X(\lfloor k \cdot \rfloor)$ restricted to positive times converges in the scaling limit when $k\to\infty$ to a $4/3$-stable Lévy process with only positive jumps. Such a process has probability zero to stay positive during any positive-length time interval. From this, we get that there exists $A = A(\delta) > 1$ and $k_* = k_*(\delta) \in\BB N$ such that for each $k\geq k_*$, (the choice of $A^{-1/2}$ is for later convenience) 
\eqb \label{eqn-cell-sdp-lower-stable}
\BB P\left[ \text{$\exists r \in [0,k]\cap\BB Z$ such that $\mcl X(r) \leq -  A^{-1/2} k^{3/4} $}  \right] \geq 1 - \delta/2 .
\eqe 

The second coordinate $\mcl R$ of the KMSW encoding function is a random walk with i.i.d.\ increments sampled uniformly from $\{-1,0,1\}$. 
By~\eqref{eqn-bdy-vertex-kmsw}, when $k\geq 1$, the time $\tau_k$ (which is the smallest time $j$ that $x_k$ is the initial vertex of $\lambda_j$) is the first positive time at which $\mcl R$ it hits $-k$. By standard estimates for random walks, there exists $A = A(\delta) > 1$ such that~\eqref{eqn-cell-sdp-lower-stable} holds and for each large enough $n\in\BB N$ (how large depends only on $\delta$), 
\eqb \label{eqn-cell-sdp-lower-bdy}
\BB P\left[ \tau_{ \lceil A^{-1/2} n^{1/2}\rceil} \leq n  \right] \geq 1-\delta /2 .
\eqe 

If $\tau_{\lceil A^{-1/2} n^{1/2} \rceil} \leq n$, then the vertices $x_1,\dots, x_{\lceil A^{-1/2} n^{1/2} \rceil}$ all lie on the lower-right boundary of $M_{0,n}$. 
By~\eqref{eqn-cell-sdp-lower-stable} (with $k = \lceil A^{-1/2} n^{1/2} \rceil$) and~\eqref{eqn-cell-sdp-lower-bdy}, for each sufficiently large $n\in\BB N$ (how large depends only on $\delta$), the following is true.
It holds with probability at least $1-\delta$ that there exists $r \in \BB N$ such that $x_r$ lies on the lower-right boundary of $M_{0,n}$ and $\mcl X(r) \leq -  A^{-1} n^{3/8}$. By Lemma~\ref{lem-sdp-tri} (applied with $(0,r)$ in place of $(\el,k)$), this implies that
\eqb \label{eqn-cell-sdp-lower-end}
\op{SDP}_{M_{0,\infty}}(x_0 , x_r )  \geq  A^{-1} n^{3/8}  .
\eqe 
Since $M_{0,n} \subset M_{0,\infty}$, we have $\op{SDP}_{M_{0,n}}(x_0 , x_r ) \geq \op{SDP}_{M_{0,\infty}}(x_0 , x_r)$, which concludes the proof after enlarging the constant $A$ so that the bound in \eqref{eqn-cell-sdp-lower-end} holds for all $n\in \BB N$ (this is always possible because $\op{SDP}_{M_{0,n}}(x_0 , x_1 )\geq 1 $ since $\lambda_0=(x_0,x_1)$ is in $M_{0,n}$ for all $n\in\BB N$). 
\end{proof}

Next we prove an upper bound for the SDP in $M_{0,n}$, which will be done via a similar argument as in the LDP case (Lemma~\ref{lem-cell-ldp-upper}).

\begin{lem} \label{lem-sdp-fluctuation}
For each $k\in\BB N$, it holds with probability tending to 1 as $A\to \infty$ (uniformly in $k$) that 
\eqb \label{eqn-cell-sdp-fluctuation}
  \max_{\el \in [-k,k] \cap \BB Z} |\mcl X(\el)| \leq  A k^{3/4} . 
\eqe 
\end{lem}
\begin{proof}
This is immediate from the fact that $k^{-3/4} X(\lfloor k \cdot \rfloor)$ converges in law to a stable Lévy process (Theorem~\ref{thm-busemann-conv}). 
\end{proof}

\begin{lem} \label{lem-cell-sdp-upper} 
%Recall that the \textbf{lower boundary} and \textbf{upper boundary} of $M_{0,n}$ are sets $M_{0,n} \cap M_{-\infty,0}$ and $M_{0,n} \cap M_{n,\infty}$, respectively. 
For each $n\in\BB N$, it holds with probability tending to 1 as $A\to\infty$ (uniformly in $n$) that the following is true:
\begin{itemize}
\item For every vertex $x$ on the lower boundary of $M_{0,n}$, there exists a vertex $y$ on the upper boundary of $M_{0,n}$ such that $\op{SDP}_{M_{0,n}}(x , y) \leq A n^{3/8}$.
\item For every vertex $y$ on the upper boundary of $M_{0,n}$, there exists a vertex $x$ on the lower boundary of $M_{0,n}$ such that $\op{SDP}_{M_{0,n}}(x , y) \leq A n^{3/8}$. 
\end{itemize}
\end{lem}
\begin{proof} 
The proof is very similar to that of Lemma~\ref{lem-cell-ldp-upper}, except for some minor changes in Step 3. So, we will be brief. See again Figure~\ref{fig-long-path2} for an illustration.
Let $\mcl X_n$ be the Busemann function for $M_{n,\infty}$, defined in the same way as $\mcl X$ but with $M_{n,\infty}$ in place of $M_{0,\infty}$. Since $M_{n,\infty} \eqD M_{0,\infty}$, we have $\mcl X \eqD \mcl X_n$.  
\medskip

\noindent\textit{Step 1: bounds for $\mcl X$ and $\mcl X_n$.}
Let $\beta_0^R$ (resp.\ $\beta_n^R$) be the directed path which traverses the portion of the boundary of $M_{0,\infty}$ (resp.\ $M_{n,\infty}$) lying to the right of the initial vertex of $\lambda_0$ (resp.\ $\lambda_n$). \jbb{Write $LR$ and $UR$ for the lower-right and upper-right boundaries of $M_{0,n}$, respectively, and set $\sigma_0^R=|LR|$ and $\sigma_n^R=1+|UR|$. The first $\sigma_0^R$ edges of $\beta_0^R$ traverse $LR$, while the first $\sigma_n^R$ edges of $\beta_n^R$ traverse $\lambda_n$ followed by the images of $UR$ in $M_{n,\infty}$ (recall Definition~\ref{def-kmsw}). These prefixes have vertices indexed from $0$ to $\sigma_0^R$ and from $0$ to $\sigma_n^R$, respectively, and the edges of $UR$ are missing and unoriented in $M_{0,n}$.} Then
\eqb \label{eqn-cell-merge-time-sdp}
\beta_0^R(k  + \sigma_0^R) = \beta_n^R(k + \sigma_n^R) ,\quad \forall k \geq 1. 
\eqe  
In an exactly analogous manner, we also define the left boundary paths $\beta_0^L$ and $\beta_n^L$ and the times $\sigma_0^L$ and $\sigma_n^L$, for which~\eqref{eqn-cell-merge-time-R} holds with $L$ in place of $R$. 
 
By Lemma~\ref{lem-sdp-fluctuation}, applied exactly as in the proof of~\eqref{eqn-cell-bdy-sup} from Lemma~\ref{lem-cell-ldp-upper}, it holds with probability tending to 1 as $A\to\infty$ (uniformly in $n$) that
\eqb \label{eqn-cell-bdy-sdp}
  \max\left\{ \max_{\el \in [-\sigma_0^L ,\sigma_0^R ] \cap \BB Z} |\mcl X(\el)| ,  \max_{\el \in [-\sigma_n^L ,\sigma_n^R ] \cap \BB Z} |\mcl X_n(\el)| \right\}  \leq  A n^{3/8}    .
\eqe 
\medskip

\noindent\textit{Step 2: relating $\mcl X$ and $\mcl X_n$ to \jbb{shortest} directed paths.}
Henceforth assume that the event in~\eqref{eqn-cell-bdy-sdp} occurs. 
Let $\{x_k\}_{k\in\BB Z}$ be the vertices on the boundary of $M_{0,\infty}$, enumerated in west to east order so that $\lambda_0 = (x_{0} , x_1)$. Similarly, let $\{x^n_k\}_{k\in\BB Z}$ be the vertices on the boundary of $M_{n,\infty}$, enumerated in west to east order so that $\lambda_n = (x_{0}^n , x_1^n)$. By~\eqref{eqn-cell-merge-time-R} and its analog with $L$ in place of $R$,
\eqb \label{eqn-cell-vertex-sdp} 
x_{k + \sigma_0^R} = x_{k + \sigma_n^R}^n \quad \text{and} \quad x_{-k-\sigma_0^L} = x_{-k-\sigma_n^L}^n ,\quad\forall k \geq 0.
\eqe 
\jbb{In this step, let $w$ be any common target for which the following Busemann identities hold. In Step 3 below, we will choose such a target sufficiently far along the selected geodesic:}
\allb \label{eqn-cell-sdp-infty}
\mcl X(\el) &= \op{SDP}_{M_{0,\infty}}(x_\el , w) - \op{SDP}_{M_{0,\infty}}(x_0 , w)  , \quad\forall \el \in [-\sigma_0^L , \sigma_0^R] \cap\BB Z \quad \text{and} \notag\\
\mcl X_n(\el) &= \op{SDP}_{M_{n,\infty}}(x_\el^n , w) - \op{SDP}_{M_{n,\infty}}(x_0^n , w)  , \quad\forall \el \in [-\sigma_n^L , \sigma_n^R] \cap\BB Z  .
\alle
By the same argument leading to~\eqref{eqn-cell-ldp-agree}, 
\eqb \label{eqn-cell-sdp-agree}
 \op{SDP}_{M_{n,\infty}}(x_{\sigma_n^R}^n , w) =  \op{SDP}_{M_{0,\infty}}(x_{\sigma_0^R} , w) .
\eqe
By exactly the same calculation as in~\eqref{eqn-ldp-upper-shift}, we then obtain with probability tending to 1 as $A\to\infty$ (uniformly in $n$)
\allb \label{eqn-sdp-upper-shift}
&\left| \op{SDP}_{M_{n,\infty}}(x_0^n , w) - \op{SDP}_{M_{0,\infty}}(x_0 , w) \right| \notag\\
&\qquad \leq \left|   \op{SDP}_{M_{n,\infty}}(x_{\sigma_n^R}^n , w) -  \op{SDP}_{M_{0,\infty}}(x_{\sigma_0^R} , w)  \right| 
+  \left|    \mcl X_n(\sigma_n^R) - \mcl X(\sigma_0^R) \right| \quad &&\text{by~\eqref{eqn-cell-sdp-infty}} \notag\\
&\qquad= \left|    \mcl X_n(\sigma_n^R) - \mcl X(\sigma_0^R) \right|  \quad &&\text{by~\eqref{eqn-cell-sdp-agree}} \notag\\
&\qquad\leq 2 A n^{3/8}  \quad &&\text{by~\eqref{eqn-cell-bdy-sdp}}  . 
\alle
\medskip

\noindent\textit{Step 3: bound for paths in $M_{0,n}$.}
Let $x  $ be a vertex on the lower boundary $ M_{0,n}$. 
Let $P_x : \BB N \to\mcl E(M_{0,\infty})$ be \jbb{the leftmost} SDP geodesic from $x$ to $\infty$ in $M_{0,\infty}$ (which exists by Lemma~\ref{lem-geo-infty}). 
Let $T$ be the smallest time for which $P_x(T+1) \notin M_{0,n}$ and let $z$ be the initial vertex of $P_x(T+1)$. Then $z$ lies on the upper boundary $M_{0,n}$. 
\jbb{By the KMSW construction, the first edge after exit belongs to $M_{n,\infty}$. The suffix of $P_x$ from $z$ stays in $M_{n,\infty}$: it cannot leave through the missing left boundary, whose vertices have no incoming edges in this map, or cross outward through its rightmost right boundary (Lemma~\ref{lem-uibot-bdy}). This suffix is a leftmost SDP geodesic in $M_{n,\infty}$. By coalescence (Lemma~\ref{lem-geo-infty}, applied in each map), vertices sufficiently far along this suffix are reachable from all of the finitely many boundary vertices appearing in~\eqref{eqn-cell-sdp-infty}. Theorem~\ref{thm-busemann} therefore allows us to choose $w$ on this suffix, beyond $z$, so that both identities in~\eqref{eqn-cell-sdp-infty} hold. The geodesic segments from $x$ to $z$ and from $z$ to $w$ lie in $M_{0,n}$ and $M_{n,\infty}$, respectively. Consequently,}
\eqb   \label{eqn-sdp-upper-geo}
\op{SDP}_{M_{0,n}}(x,z) = \op{SDP}_{M_{0,\infty}}(x,z) = \op{SDP}_{M_{0,\infty}}(x,w) - \op{SDP}_{\jbb{M_{n,\infty}}}(z,w) .
\eqe 
We can write $x = x_\el$ and $z = x_k^n$ for some $\el \in [-\sigma_0^L , \sigma_0^R]\cap\BB Z$ and $k \in [-\sigma_n^L , \sigma_n^R]\cap\BB Z$. We then have with probability tending to 1 as $A\to\infty$ (uniformly in $n$)
\begin{align*}
    \op{SDP}_{M_{0,n}}(x,z) 
&=   \op{SDP}_{M_{0,\infty}}(x_\el ,w) - \op{SDP}_{\jbb{M_{n,\infty}}}(x^n_k ,w) \quad &&\text{by~\eqref{eqn-sdp-upper-geo}} \notag\\
&=  \mcl X(\el) - \mcl X_n(k) + \op{SDP}_{M_{0,\infty}}(x_0, w) - \op{SDP}_{\jbb{M_{n,\infty}}}(x_0^n , w)\quad &&\text{by~\eqref{eqn-cell-sdp-infty}} \notag\\
&\leq 4 A n^{3/8} \quad &&\text{by~\eqref{eqn-cell-bdy-sdp} and~\eqref{eqn-sdp-upper-shift}} .
\end{align*}
\jbb{Absorbing the factor $4$ into $A$}, we get that with probability tending to 1 as \jbb{$A\to\infty$, uniformly in $n$}, each vertex on the lower boundary of $M_{0,n}$ can be joined to the upper boundary by a directed path of length at most $A n^{3/8}$. 

The statement with the upper and lower boundaries revered follows since the law of $M_{0,n}$ is invariant under reversing the orientations of the edges and applying an orientation-reversing homeomorphism $\BB C\to\BB C$. This, in turn, follows since $\mcl Z(n-\cdot) -\mcl Z(n)$ has the same law as $(\mcl R(\cdot), \mcl L(\cdot))$.  
\end{proof}

\begin{proof}[Proof of Theorem~\ref{thm-cell-sdp}]
Combine Lemma~\ref{lem-cell-sdp-upper} and Lemma~\ref{lem-cell-sdp-lower}. 
\end{proof}

\subsection{Boltzmann bipolar-oriented triangulations}
\label{sec-boltzmann}

In this section we will prove Theorem~\ref{thm-boltzmann-ldp}. The key idea is to find a submap $\wh M(r) \subset M_{0,\infty}$ with a marked left boundary vertex $\wh z(r)$ such that $(\wh M(r) , \wh z(r))$ has the same law as the pair $(M(r), z(r))$ from the theorem statement; see Lemma~\ref{lem-uiqbot-boltzmann}. Once we have done this, the proof of Theorem~\ref{thm-boltzmann-ldp} will proceed by comparing the map $\wh M(r)$ to the map $M_{0,n}$ from Theorems~\ref{thm-cell-ldp} and~\ref{thm-cell-sdp}, with $n$ chosen to be comparable to $r^2$. 

To find the map $\wh M(r)$, we first let $\wh M_{0,\infty} \subset M_{0,\infty}$ be as in Proposition~\ref{prop-future-map-law}, so that $\wh M_{0,\infty}$ has the law of the UIQBOT (Definition~\ref{def-uiqbot}). 
Recall that $\mcl Z$ is the KMSW encoding walk for the UIBOT $M_{-\infty,\infty}$ as in Proposition~\ref{prop-kmsw-uibot}. 
Let $\wh{\mcl Z} = (\wh{\mcl L} , \wh{\mcl R})$ be the KMSW encoding walk for $\wh M_{0,\infty}$, so that $\wh{\mcl Z}$ is obtained by conditioning $\mcl Z$ on the event that its first coordinate stays non-negative for all time (Definition~\ref{def-uiqbot}).  

We couple $\wh{\mcl Z}$ and $\mcl Z$ together as in~\eqref{eqn-walk-quadrant}: that is, for $m\in\BB N$ let $S_m \in \BB N_0$ be the smallest number for which
\eqb \label{eqn-boltzmann-uiqbot-time}
m = \# \left\{ \jb{1 \leq} n \leq S_m : \lambda_n \in \wh M_{0,\infty} \right\} .
\eqe
Then we can take $\wh{\mcl Z}(0) = (0,0)$ and
\eqb \label{eqn-boltzmann-uiqbot-coupling}
\wh{\mcl Z}(m) - \wh{\mcl Z}(m-1) = \mcl Z(S_m) - \mcl Z(S_m-1), \quad \forall m \geq 1 .
\eqe
By Lemma~\ref{lem-quad-kmsw}, the walk defined by~\eqref{eqn-boltzmann-uiqbot-coupling} is indeed the KMSW encoding walk for $\wh M_{0,\infty}$. 
  
We now identify a submap of $\wh M_{0,\infty}$ which has the same law as the Boltzmann map of Theorem~\ref{thm-boltzmann-ldp}. See Figure~\ref{fig-boltzmann-subset} for an illustration.

\begin{figure}[ht!]
\begin{center}
\includegraphics[width=0.6\textwidth]{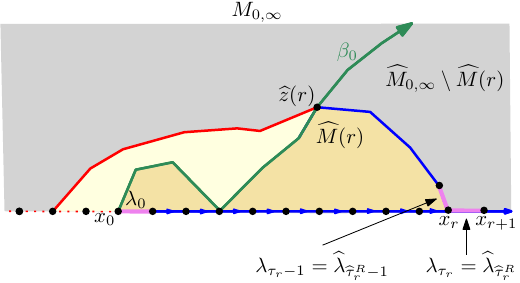}  
\caption{\label{fig-boltzmann-subset}  
The map $\wh M(r) \subset \wh M_{0,\infty}$, which has sink $x_0$, source $x_{r}$, and marked left boundary vertex $\wh z(r)$. We show in Lemmas~\ref{lem-uiqbot-boltzmann} 
%and~\ref{lem-uiqbot-boltzmann-cond} 
that $(\wh M(r) , \wh z(r))$ has the same law as a Boltzmann bipolar-oriented triangulation with right boundary length $r$ and a marked left boundary vertex. 
}
\end{center}
\end{figure}

\begin{lem} \label{lem-uiqbot-boltzmann}
For $r\geq 1$, let
\eqb \label{eqn-uiqbot-boltzmann}
\wh\tau_r^R := \min\left\{ n\in\BB N : \wh{\mcl R}(n) = -r \right\}. 
\eqe
Let $\wh M(r)$ be the submap of $\wh M_{0,\infty}$ obtained by applying the KMSW procedure to $\wh{\mcl Z}|_{[0,\wh\tau_r^R-1]}$. 
Also, let $\wh z(r)$ be the last (in directed order) vertex on the left boundary of $\wh M(r)$ which also lies on the left boundary of $\wh M_{0,\infty}$. 
Then the law of $(\wh M(r) , \wh z(r))$ is the Boltzmann distribution on bipolar-oriented triangulations with right boundary length $r$ and a marked left boundary vertex (Definition~\ref{def-boltzmann-right}). In particular, $\wh M(r)$ is a bipolar-oriented map with no missing edges. 
Moreover,  with probability tending to one as $A\to\infty$, uniformly in the choice of $r$,
\begin{equation}\label{eq:edges-in-Mr}
    A^{-1}r^{2} \leq \#\mcl E(\wh M(r)) \leq Ar^2.
\end{equation}
\end{lem}

\begin{proof}
Recalling the comment immediately after Definition~\ref{def-boltzmann-right}, it is enough to show the following:
\begin{itemize}
\item  $\wh M(r)$ is sampled from the Boltzmann distribution on bipolar-oriented triangulations with right boundary length $r$ (Definition~\ref{def-boltzmann-right}), weighted by (a constant times) $L+1$, where $L$ is the left boundary length.  
\item Conditional on $\wh M(r)$,  $\wh z(r)$ is uniformly sampled from the set of left boundary vertices of $\wh M(r)$. 
\end{itemize}

We start by determining the marginal law of  $\wh M(r)$. For $r\geq 1$, as in~\eqref{eqn-bdy-vertex-hit} and~\eqref{eqn-bdy-vertex-kmsw}, let $\tau_r$ be the smallest $n \in \BB N$ for which the unconditioned walk $\mcl Z$ satisfies $\mcl R(n) = -r$. 
By the definition of $\wh{\mcl Z}$ (recall~\eqref{eqn-cond-walk-rn}), the law of $\wh{\mcl Z}|_{[0,\wh\tau_r^R]}$ is the same as the law of $\mcl Z|_{[0,\tau_r ]}$ weighted by the Radon-Nikodym derivative 
\eqbn
(\mcl L(\tau_r) + 1) \BB 1\left\{ \mcl L(j) \geq 0 ,\: \forall j \in [0,\tau_r]\cap\BB Z\right\} .
\eqen
Since the event in the indicator function is the same as the conditioning event in Lemma~\ref{lem-kmsw-boltzmann}, this lemma implies that the marginal law of $\wh M(r)$ is as described in the first item above. 

We now show that the conditional law of $\wh z(r)$ given $\wh M(r)$ is uniform on the set of left boundary vertices of $\wh M(r)$.  The map $\wh M(r)$ has $\wh{\mcl L}(\wh \tau_r^R)+1$ vertices on its left boundary.
Let $\wh\el_r \in \BB N_0$ be chosen so that $\wh z(r)$ is the $\wh\el_r$th vertex in order along the left boundary of $\wh M(r)$. \jbb{Here the boundary vertices are indexed starting from $0$. Using Assertion~\ref{item-lower-left} of Lemma~\ref{lem-kmsw-bdy} starting from time $\wh\tau_r^R$ determines the last boundary vertex that remains in $\wh M_{0,\infty}$ and yields}  
\eqb  \label{eqn-left-bdy-time}
\jbb{\wh\el_r = \min\left\{ \wh{\mcl L}(j) : j \geq \wh\tau_r^R \right\}} .
\eqe 
By the strong Markov property, the conditional law of $\wh{\mcl Z}(\cdot+\wh\tau_r^R)|_{[0 ,\infty)}$ given $\wh{\mcl Z}|_{[0,\wh\tau_r^R]}$ is that of a random walk with the same increment distribution as $\mcl Z$ started at $(\wh{\mcl L}(\wh\tau_r^R) , -r)$ and conditioned so that its first coordinate stays non-negative. 
By Lemma~\ref{lem-walk-pos-prob} and \eqref{eqn-left-bdy-time}, for each $s\in \jbb{[0,\wh{\mcl L}(\wh\tau_r^R)-1]} \cap\BB Z$, 
\begin{align*}
    \BB P\left[ \wh\el_r  \leq s \,\middle|\,  \wh{\mcl Z}|_{[0,\wh\tau_r^R]} \right]&=1-\BB P\left[ \wh\el_r  > s \,\middle|\,  \wh{\mcl Z}|_{[0,\wh\tau_r^R]} \right]\\
&=1-\BB P\left[ \wh{\mcl L}(j)  \geq \wh{\mcl L}(\wh\tau_r^R)- (\wh{\mcl L}(\wh\tau_r^R)-s-1), \forall j\geq \jbb{\wh\tau_r^R} \,\middle|\,  \wh{\mcl Z}|_{[0,\wh\tau_r^R]} \right]\\
&= 1-\frac{\wh{\mcl L}(\wh\tau_r^R)-s}{\wh{\mcl L}(\wh\tau_r^R) + 1}= \frac{s+1}{\wh{\mcl L}(\wh\tau_r^R) + 1}  .
\end{align*}
\jbb{The same formula holds for $s=\wh{\mcl L}(\wh\tau_r^R)$, since the minimum in~\eqref{eqn-left-bdy-time} is at most its starting value.}
Hence, the conditional law of $\wh\el_r$ given $\wh{\mcl Z}|_{[0,\wh\tau_r^R]}$ is uniform on $[0,\wh{\mcl L}(\wh\tau_r^R)] \cap\BB Z$, which implies that the conditional law of $\wh z(r)$ given $\wh M(r)$ is uniform on the set of left boundary vertices of $\jbb{\wh M(r)}$.

It remains to show~\eqref{eq:edges-in-Mr}. By the convergence of $n^{-1/2} \wh{\mcl Z}(\lfloor n \cdot \rfloor)$ to a correlated two-dimensional Brownian motion started at $(0,0)$ and conditioned to stay in the upper half-plane (Lemma~\ref{lem-cond-walk-bm}), with probability tending to one as $A\to\infty$, uniformly in the choice of $r$, 
\eqbn  
A^{-1}  r^2 \leq \wh\tau_r^R\leq A r^2.
\eqen 
Hence, with probability tending to one as $A\to\infty$, uniformly in the choice of $r$, the map $\wh M(r)$ (which is obtained by applying the KMSW procedure to $\wh{\mcl Z}|_{[0,\wh\tau_r^R-1]}$) satisfies~\eqref{eq:edges-in-Mr}.
\end{proof}

Let $(\wh M(r), \wh z(r))$ be as in Lemma~\ref{lem-uiqbot-boltzmann}. Since $(\wh M(r), \wh z(r))$ has the same law of $( M(r),  z(r))$ from Theorem~\ref{thm-boltzmann-ldp}, it will suffice to prove Theorem~\ref{thm-boltzmann-ldp} with $(\wh M(r), \wh z(r))$ in place of $( M(r),  z(r))$.

The vertices along the right boundary of the UIQBOT $\wh M_{0,\infty}$, starting from the source vertex, are the same as the vertices along the right boundary of $M_{0,\infty}$, starting from the initial vertex of $\lambda_0$. As in Theorem~\ref{thm-busemann}, we call these vertices $x_0,x_1,\dots$. 
By Assertion~\ref{item-lower-right} of Lemma~\ref{lem-kmsw-bdy}, the time $\wh\tau_r^R$ of~\eqref{eqn-uiqbot-boltzmann} is the smallest $m \in\BB N_0$ for which the initial vertex of $\wh\lambda_m$ is $x_r$ in the map $\wh M(r)$. 
In particular, the source and sink vertices of $\wh M(r)$ are the vertices $x_0$ and $x_{r}$, respectively.
 
Let $\tau_r$ be as in~\eqref{eqn-bdy-vertex-hit}, so that $\tau_r$ is the smallest $n\in\BB N_0$ for which the initial vertex of $\lambda_n$ is $x_r$. By the coupling~\eqref{eqn-boltzmann-uiqbot-coupling},
\eqb \label{eqn-hit-coupling}
\tau_r = S_{\wh\tau_r^R} .
\eqe 

We will now prove upper and lower bounds for LDPs and SDPs in $\wh M(r)$.

\begin{lem} \label{lem-boltzmann-ldp-upper} 
With probability tending to one as $A\to\infty$, uniformly in the choice of $r$, 
\eqb  \label{eqn-boltzmann-ldp-upper} 
 \op{LDP}_{\wh M(r)} \left(  x_0 , x_{r} \right)  \leq A  r^{3/2} .
\eqe 
\end{lem}
\begin{proof} 
By~\eqref{eqn-bdy-vertex-kmsw}, the time $\tau_r$ is the smallest $n\in\BB N$ such that $\mcl R(n) \leq -r$. 
Since $\mcl R$ is a lazy random walk, it holds with probability tending to one as $A\to\infty$, uniformly in the choice of $r$, that $\tau_r \leq A^{2/3}  r^2$ (the reason for the choice of $A^{2/3}$ will be clearer later).
By~\eqref{eqn-hit-coupling}, if this is the case, then 
\eqb  \label{eqn-boltzmann-ldp-contain}
\wh M(r) \subset M_{0,\lfloor A^{2/3} r^2 \rfloor},   
\eqe 
where here we recall that $M_{0,n} $ is the submap of $M_{0,\infty}$ obtained by applying the KMSW procedure to $\mcl Z|_{[0,n]}$. 
By Theorem~\ref{thm-cell-ldp} (with $A^{1/2}$ in place of $A$ and  $\lfloor A^{2/3}r^2 \rfloor$ in place of $n$), it holds with probability tending to one as $A\to\infty$, uniformly in the choice of $r$, that every directed path in $M_{0,\lfloor A^{2/3} r^2 \rfloor}$ has length at most $A r^{3/2}$. Since also~\eqref{eqn-boltzmann-ldp-contain} holds with probability tending to one as $A\to\infty$, uniformly in the choice of $r$, we obtain the lemma statement.
\end{proof}

\begin{lem} \label{lem-boltzmann-ldp-lower} 
With probability tending to one as $A\to\infty$, uniformly in the choice of $r$, 
\eqb  \label{eqn-boltzmann-ldp-lower} 
 \op{LDP}_{\wh M(r)} \left(  x_0 , x_{r} \right)  \geq A^{-1}  r^{3/2} .
\eqe 
\end{lem}

\begin{proof}
Fix $\delta \in (0,1)$. We will show that if $A$ is chosen to be large enough (depending only on $\delta$), then~\eqref{eqn-boltzmann-ldp-lower} holds with probability at least $1-\delta$. 

Recall that $\wh\tau_r^R$ is the first time that $\wh{\mcl R}$ reaches $-r$. 
The idea of the proof is to argue that for a suitably chosen $C > 1$, depending on $\delta$, the law of the walk segment
\eqb  \label{eqn-boltzmann-ldp-lower-shift} 
\left( \wh{\mcl Z}(\cdot  + \wh\tau_{\lfloor r / 2 \rfloor}^R) - \wh{\mcl Z}( \wh\tau_{\lfloor r / 2 \rfloor}^R)    \right) |_{[0, \lfloor r^2 / C \rfloor ] }
\eqe 
is absolutely continuous with respect to the law of $\mcl Z|_{[0,\lfloor r^2/C \rfloor]}$, with a Radon-Nikodym derivative which can be controlled independently of $r$. We will then apply the lower bound in~Theorem~\ref{thm-cell-ldp} to get that, if $A$ and $C$ are chosen appropriately, then with high probability, the submap of $\wh M_{0,\infty}$ obtained by applying the KMSW procedure to~\eqref{eqn-boltzmann-ldp-lower-shift} contains a directed path of length at least $A^{-1} r^{3/2}$ and is contained in $\wh M(r)$. This will give us a directed path of length at least $A^{-1} r^{3/2}$ contained in $\wh M(r)$, which is sufficient since any vertex of $\wh M(r)$ can be joined by directed paths from the source and to the sink. 
\medskip

\noindent\textit{Step 1: reducing to a lower bound for an event depending on a segment of $\wh{\mcl Z}$.}
Let $\ep \in (0,1)$ be a small number, which we will eventually choose to be small enough depending only on $\delta$. 
For $C > 1$, \jbb{write $N=\lfloor r^2/C\rfloor$, $t=\wh\tau_{\lfloor r/2\rfloor}^R$ and $T=\wh\tau_r^R$. Let} $G_r = G_r(C)$ be the event that $\tau_{\lceil r/2\rceil } \jbb{>N}$ and there is a directed path in $M_{0, \lfloor r^2 /C \rfloor}$ with length at least $A^{-1} r^{3/2}$. By an elementary estimate for the lazy random walk $\mcl R$, followed by Theorem~\ref{thm-cell-ldp}, we can find $C = C(\ep)  > 1$ and $A = A(\ep) > 1$ such that 
\eqb \label{eqn-boltzmann-ldp-lower-cell}
\BB P[G_r] \geq 1-\ep ,\quad \forall r \in \BB N .
\eqe
 
Let $\wh G_r$ be defined in the same manner as the event $G_r$ of~\eqref{eqn-boltzmann-ldp-lower-cell} but with the walk~\eqref{eqn-boltzmann-ldp-lower-shift} in place of $\mcl Z|_{\jbb{[0,N]}}$. 
That is, $\wh G_r$ is the event that $\jbb{T-t>N}$ and the map obtained by applying the KMSW procedure to the walk~\eqref{eqn-boltzmann-ldp-lower-shift} contains a directed path of length at least $A^{-1} r^{3/2}$. 

If $\wh G_r$ occurs, then $\jbb{T-t>N}$, \jbb{so that $t+N\leq T-1$}, which implies that the map in the definition of $\wh G_r$ is a submap of $\wh M(r)$. Hence on $\wh G_r$, the map $\wh M(r)$ contains a directed path of length at least $A^{-1} r^{3/2}$. The longest directed path in a bipolar-oriented map is between the source and the sink, so if $\wh G_r$ occurs, then~\eqref{eqn-boltzmann-ldp-lower} holds. 

The left side of~\eqref{eqn-boltzmann-ldp-lower} is always at least one, so it suffices to prove the lemma statement only for sufficiently large values of $r$ (finitely many small values of $r$ can be dealt with by increasing $A$). 
Hence it suffices to prove that for each sufficiently large $r\in\BB N$ (how large can depend on $\delta$), 
\eqb \label{eqn-boltzmann-ldp-lower-show}
 \BB P[\wh G_r] \geq 1- \delta .
\eqe

\smallskip

\noindent\textit{Step 2: conditional law given $\wh{\mcl Z}|_{[0,\jbb{t}]}$.}
For $\el\in\BB N$, by the strong Markov property, the conditional law of $\wh{\mcl Z}(\cdot  + \wh\tau_{\lfloor r / 2 \rfloor}^R)|_{[0,\infty)}$ given $\wh{\mcl Z}|_{[0,\wh\tau_{\lfloor r / 2 \rfloor}^R]}$ on the event $\{\wh{\mcl L}(\wh\tau_{\lfloor r / 2 \rfloor}^R)  = \el\}$ is that of a random walk with the same increment distribution as $\mcl Z$ started from $(\el , -\lfloor r/2 \rfloor)$ and conditioned so that its first coordinate is non-negative for all time. 
 
By combining this with~\eqref{eqn-cond-walk-rn}, we find that the conditional law of $\wh{\mcl Z}(\cdot  + \wh\tau_{\lfloor r / 2 \rfloor}^R)|_{[0, \lfloor r^2 /C \rfloor ] }$ given $\wh{\mcl Z}|_{[0,\wh\tau_{\lfloor r / 2 \rfloor}^R]}$ on the event $\{\wh{\mcl L}(\wh\tau_{\lfloor r / 2 \rfloor}^R)  = \el\}$ is described as follows. Let $\mcl Z' = (\mcl L',\mcl R')$ be a random walk with the same increment distribution as $\mcl Z$ but started from $(\el , -\lfloor r/2\rfloor)$. Consider the law of $\mcl Z'|_{[0,r^2/C]}$ weighted by 
\eqbn
(\el+1)^{-1} (\mcl L'(\lfloor r^2/C \rfloor)+1) \BB 1\left\{ \mcl L'(j) \geq 0,\: \forall j \in [0,r^2/C]\cap\BB Z\right\} .
\eqen

By shifting space by $-(\el , -\lfloor r/2\rfloor)$, we get that the conditional law of the walk~\eqref{eqn-boltzmann-ldp-lower-shift} given $\wh{\mcl Z}|_{[0,\wh\tau_{\lfloor r / 2 \rfloor}^R]}$ on the event $\{\wh{\mcl L}(\wh\tau_{\lfloor r / 2 \rfloor}^R)  = \el\}$ is the same as the law of the walk $\mcl Z|_{[0, \lfloor r^2 /C \rfloor ]}$ weighted by $(\el+1)^{-1}(\jb{\el+}\mcl L(\jbb{N})+1)\BB 1_{F_{\el,r}}$, where 
\eqb \label{eqn-boltzmann-ldp-lower-cond} 
F_{\el,r} := \left\{ \mcl L(j) \geq - \el , \: \forall j \in [0,   r^2 /C  ]\cap\BB Z\right\}  .
\eqe  
\smallskip

\noindent\textit{Step 3: bounding the Radon-Nikodym derivative.}
Recall the events $G_r$ and $\wh G_r$ defined immediately before and after~\eqref{eqn-boltzmann-ldp-lower-cell}.
By the previous paragraph, on the event $\{\wh{\mcl L}(\wh\tau_{\lfloor r / 2 \rfloor}^R) = \el\}$,
\eqb \label{eqn-boltzmann-ldp-lower-rn}
\BB P\left[ \wh G_r^c \,\middle|\, \wh{\mcl Z}|_{[0,\wh\tau_{\lfloor r / 2 \rfloor}^R]} \right]
= (\el+1)^{-1} \BB E\left[ (\jb{\el+}\mcl L(\lfloor r^2/C\rfloor) +1) \BB 1_{F_{\el,r}} \BB 1_{G_r^c} \right] .
\eqe
Since $\mcl L$ is a lazy random walk and $C > 1$,
\eqb \label{eqn-second-moment}
\BB E\left[ (\jb{\el+}\mcl L(\lfloor r^2/C\rfloor) +1)^2 \right]
\leq \jb{(\el+1)^2+}2r^2 .
\eqe
We apply the Cauchy-Schwarz inequality to the right-hand side of~\eqref{eqn-boltzmann-ldp-lower-rn}, followed by~\eqref{eqn-boltzmann-ldp-lower-cell} and~\eqref{eqn-second-moment}.
We get that, on $\{\wh{\mcl L}(\wh\tau_{\lfloor r / 2 \rfloor}^R) = \el\}$,
\allb \label{eqn-boltzmann-ldp-lower-cs}
\BB P\left[ \wh G_r^c \,\middle|\, \wh{\mcl Z}|_{[0,\wh\tau_{\lfloor r / 2 \rfloor}^R]} \right]
&\leq (\el+1)^{-1} \BB E\left[ (\jb{\el+}\mcl L(\lfloor r^2/C\rfloor) +1)^2 \BB 1_{F_{\el,r}} \right]^{1/2} \BB P\left[ G_r^c \right]^{1/2} \notag \\
&\leq \jb{\left(1+\frac{2r^2}{(\el+1)^2}\right)^{1/2}} \ep^{1/2} .
\alle

By the convergence of $n^{-1/2} \wh{\mcl Z}(\lfloor n \cdot \rfloor)$ to conditioned Brownian motion (Lemma~\ref{lem-cond-walk-bm}), there exists $c_\delta > 0$ such that for each sufficiently large $r\in\BB N$ (how large depends only on $\delta$),
\eqb \label{eqn-boltzmann-ldp-lower-half}
\BB P\left[ \wh{\mcl L}(\wh\tau_{\lfloor r / 2 \rfloor}^R) \geq c_\delta r \right] \geq 1 - \delta/2 .
\eqe
By~\eqref{eqn-boltzmann-ldp-lower-cs} and~\eqref{eqn-boltzmann-ldp-lower-half}, for each such $r$,
\allb \label{eqn-boltzmann-ldp-lower-end}
\BB P\left[ \wh G_r^c \right]
&\leq \BB P\left[ \wh G_r^c \,\middle|\, \wh{\mcl L}(\wh\tau_{\lfloor r / 2 \rfloor}^R) \geq c_\delta r \right]
+ \BB P\left[ \wh{\mcl L}(\wh\tau_{\lfloor r / 2 \rfloor}^R) < c_\delta r \right] \notag\\
&\leq \jb{\left(1+\frac{2r^2}{(c_\delta r+1)^2}\right)^{1/2}} \ep^{1/2} + \delta/2 .
\alle
If we choose $\ep$ to be small enough (and hence $A$ and $C$ to be large enough), depending on $\delta$, \jb{so that $\ep \leq \delta^2/[4(1+2/c_\delta^2)]$,} then the right side of~\eqref{eqn-boltzmann-ldp-lower-end} is at most $\delta$ for every $r \in \BB N$. This implies~\eqref{eqn-boltzmann-ldp-lower-show} for all sufficiently large $r\in\BB N$, as required.
\end{proof}

\begin{lem} \label{lem-boltzmann-sdp-lower}
With probability tending to one as $A\to\infty$, uniformly in the choice of $r$, there exists a vertex $y$ on the right boundary of $\wh M(r)$ such that 
\eqb  \label{eqn-boltzmann-sdp-lower} 
 \op{SDP}_{\wh M(r)} \left(  x_0 ,  y\right)  \geq A^{-1}  r^{3/4} .
\eqe 
\end{lem}
\begin{proof} 
Since $\mcl R$ is a lazy random walk and $\tau_r$ is the first time it hits $-r$, it holds with probability tending to one as $A\to\infty$, uniformly in the choice of $r$, that $\tau_r \jbb{>} A^{-1/2}  r^2$. \jbb{Indeed, the non-strict bound with $A/2$ in place of $A$ implies this strict bound.}
By~\eqref{eqn-hit-coupling}, if this is the case, then 
\eqb  \label{eqn-boltzmann-sdp-contain}
 M_{0,\lfloor A^{-1/2} r^2 \rfloor} \cap \wh M_{0,\infty}  \subset \wh M(r) .    
\eqe  
Since the lower-right boundary of $ M_{0,\lfloor A^{-1/2} r^2 \rfloor}$ is a subset of the right boundary of $M_{0,\infty}$, which equals the right boundary of $\wh M_{0,\infty}$, if~\eqref{eqn-boltzmann-sdp-contain} holds, then the right boundary of $ M_{0,\lfloor A^{-1/2} r^2 \rfloor} $ is a subset of the right boundary of $\wh M(r)$. 
 
By the proof of Lemma~\ref{lem-cell-sdp-lower} (with $A^{1/2}$ in place of $A$), see in particular~\eqref{eqn-cell-sdp-lower-end}, it holds with probability tending to one as $A\to\infty$, uniformly in the choice of $r$, that there exists a vertex $y$ on the right boundary of $ M_{0,\lfloor A^{-1/2} r^2 \rfloor} $ such that 
\eqbn   
 \op{SDP}_{M_{0, \infty}} \left(  x_0 ,  y\right)  \geq A^{-1}  r^{3/4} . 
\eqen 
Since $ \op{SDP}_{\wh M(r)} \left(  x_0 ,  y\right) \geq  \op{SDP}_{M_{0, \infty}} \left(  x_0 ,  y\right)$, 
combining this with the previous paragraph concludes the proof. 
\end{proof}

\begin{lem} \label{lem-boltzmann-sdp-upper}
With probability tending to one as $A\to\infty$, uniformly in the choice of $r$, 
\eqb  \label{eqn-boltzmann-sdp-upper} 
 \op{SDP}_{\wh M(r)} \left(  x_0 , x_r \right)  \leq A  r^{3/4} .
\eqe 
\end{lem}
\begin{proof}
Let $I_r$ (resp.\ $J_r$) be the set of vertices and edges along the left boundary of $\wh M(r)$ which lie between $x_0$ and $\wh z(r)$ (resp.\ $\wh z(r)$ and $x_{r}$), including the endpoints. 
\medskip

\noindent
\textit{Step 1: path from $x_0$ to $J_r$.} 
We first argue that with probability tending to one as $A\to\infty$, uniformly in the choice of $r$,  
\eqb \label{eqn-boltzmann-sdp-upper-half} 
\text{$\exists y\in J_r$ such that $\op{SDP}_{\wh M(r)} \left(  x_0 ,  y \right)  \leq A  r^{3/4}$} . 
\eqe 
We will eventually combine~\eqref{eqn-boltzmann-sdp-upper-half} with a forward/reverse symmetry argument for $\wh M(r)$ to obtain~\eqref{eqn-boltzmann-sdp-upper}. 

As in the proof of Lemma~\ref{lem-boltzmann-ldp-upper}, see \eqref{eqn-boltzmann-ldp-contain}, it holds with probability tending to one as $A\to\infty$, uniformly in the choice of $r$, that 
\eqb  \label{eqn-boltzmann-sdp-contain-upper}
\wh M(r) \subset M_{0,\lfloor A^{4/3} r^2 \rfloor}    . 
\eqe  
By the first item of Theorem~\ref{thm-cell-sdp} (applied with $A^{1/2}$ in place of $A$ and with $\lfloor A^{4/3}r^2 \rfloor$ in place of $n$) it holds with probability tending to one as $A\to\infty$, uniformly in the choice of $r$, that there exists a directed path $P$ from $x_0$ to the upper boundary of $M_{0,\lfloor A^{4/3} r^2 \rfloor} $ with length at most $A r^{3/4}$. 
Now assume that this is the case and~\eqref{eqn-boltzmann-sdp-contain-upper} holds. We will show that~\eqref{eqn-boltzmann-sdp-upper-half} holds. 

The path $P$ is a directed path in $M_{0,\infty}$ started from $x_0$, so by the definition of $\wh M_{0,\infty}$ (third paragraph of Section~\ref{sec-busemann-setup}), $P$ is contained in $\wh M_{0,\infty}$. 
By the definition of $\wh z(r)$ in Lemma~\ref{lem-uiqbot-boltzmann}, the vertices in $J_r$ form a (directed) path in $\wh M_{0,\infty}$ from the left boundary of $\wh M_{0,\infty}$ to $x_{r}$. By planarity and~\eqref{eqn-boltzmann-sdp-contain-upper}, the path $P$ must intersect $J_r$. 
Thus~\eqref{eqn-boltzmann-sdp-upper-half} holds. 
\medskip

\noindent
\textit{Step 2: path from $x_0$ to $x_{r}$ via symmetry.}
Recall from Lemma~\ref{lem-uiqbot-boltzmann} that the law of $(\wh M(r) , \wh z(r))$ is the Boltzmann distribution on bipolar-oriented triangulations with right boundary length $r$ and a marked left boundary vertex.
By this and Definition~\ref{def-boltzmann-right}, reversing all of the orientations of the edges of $\wh M(r)$, then applying an orientation-reversing homeomorphism $\BB C\to\BB C$, gives us another decorated map with the same law as $(\wh M(r) , \wh z(r))$. 
This procedure has the effect of switching the source and the sink and also switching $I_r$ with $J_r$. Hence the main result of Step 1 (see~\eqref{eqn-boltzmann-sdp-upper-half}) implies that with probability tending to one as $A\to\infty$, uniformly in the choice of $r$,  
\eqb \label{eqn-boltzmann-sdp-upper-reverse}
\text{$\exists y' \in I_r $ such that $ \op{SDP}_{\wh M(r)} \left( y' , x_r \right)  \leq A  r^{3/4} $} .
\eqe 

By planarity, any path from $x_0$ to a vertex of $J_r$ and any path from a vertex of $I_r$ to $x_r$ must have a vertex in common. By concatenating a segment of the first path with a segment of the second path, we obtain from~\eqref{eqn-boltzmann-sdp-upper-half} and \eqref{eqn-boltzmann-sdp-upper-reverse} that with probability tending to one as $A\to\infty$, uniformly in the choice of $r$,  
 \eqb \label{eqn-boltzmann-sdp-upper0}
  \op{SDP}_{\wh M(r)} \left(  x_0 ,  x_r \right)  \leq 2 A  r^{3/4}  .
\eqe 
Applying this with $A/2$ in place of $A$ concludes the proof.
\end{proof}

\begin{proof}[Proof of Theorem~\ref{thm-boltzmann-ldp}]
Recall from Lemma~\ref{lem-uiqbot-boltzmann} that the map $\wh M(r)$ has the same law as the map $M(r)$ in the theorem statement. The theorem therefore follows from Lemmas~\ref{lem-boltzmann-ldp-upper}, \ref{lem-boltzmann-ldp-lower}, \ref{lem-boltzmann-sdp-lower}, and~\ref{lem-boltzmann-sdp-upper}. 
\end{proof}

%% file: tex/open-problems.tex
We discuss some additional open problems besides the ones already mentioned in Section~\ref{sec-related}. 
As discussed in Section~\ref{sec-lqg}, uniform bipolar-oriented triangulations can be viewed as discrete analogs of a hypothetical directed version of the $\sqrt{4/3}$-LQG metric. 

\begin{prob} \label{prob-limit}
Construct two directed metrics on $\BB C$ which describe the scaling limits of longest and shortest directed path distances in the UIBOT. Obtain a natural coupling of these directed metrics with a $\sqrt{4/3}$-LQG cone together with a pair of coupled SLE$_{12}$ curves. 
\end{prob}

A potential approach to constructing a full scaling limit for directed distances on the UIBOT is to understand the joint scaling limit of the Busemann functions at different times. 
For $n\in\BB Z$, let $\mcl X_n$ be the $\XDP$ Busemann function on the boundary of the map $M_{n,\infty} $ obtained by applying the KMSW procedure (Definition~\ref{def-kmsw}) to $(\mcl Z(\cdot+n) - \mcl Z(n))|_{[0,\infty)}$. We have $M_{n,\infty}\eqD M_{0,\infty}$, so $\mcl X_n \eqD \mcl X_0$. Hence, Theorem~\ref{thm-busemann-conv} implies that each $\mcl X_n$ converges in law to a stable Lévy process of index $2/3$ (in the LDP case) or $4/3$ (in the SDP case).  

\begin{prob} \label{prob-busemann-joint}
Describe the scaling limit of the joint law of the collection of Busemann functions $\{\mcl X_n\}_{n\in\BB Z}$ as a collection of coupled stable Lévy processes. Even a non-trivial joint scaling limit of two of these Busemann functions would be interesting. 
\end{prob}

It is possible to show that a variety of different random planar maps in the $\gamma$-LQG universality have the same exponent for undirected graph distances. This can be done using a strong coupling of their encoding walks with Brownian motion~\cite{kmt,zaitsev-kmt}; see~\cite{ghs-map-dist,dg-lqg-dim}. Moreover, this exponent is known to be the reciprocal of the dimension of the $\gamma$-LQG metric space~\cite[Corollary 1.7]{gp-kpz}. It would be interesting to prove analogous results for directed distances (the argument of~\cite{ghs-map-dist} does not carry over in an obvious way). 

\begin{prob} \label{prob-universal}
Prove that the exponents for LDP and SDP distances obtained in this paper are equal to analogous exponents for other models of bipolar-oriented planar maps in the $\sqrt{4/3}$-LQG universality class (e.g., quadrangulations or maps with unconstrained face degree). Also relate the exponents in this paper to exponents for continuum models. E.g., the longest increasing subsequence for a size-$n$ permutation sampled from the skew Brownian permuton of parameters $(\rho,q) = (-1/2,1/2)$~\cite{borga-skew-permuton} should grow like $n^{3/4}$.
\end{prob}
 
There is a substantial literature concerning bounds for Hausdorff dimensions of special sets associated with LQG or with the directed landscape, see, e.g., the surveys~\cite{ddg-metric-survey, gm-dl-survey}. One can ask analogous questions for uniform bipolar-oriented triangulations.

\begin{prob} \label{prob-sets}
Compute the scaling exponents for the sizes of various special sets associated with longest and shortest directed paths on uniform bipolar-oriented triangulations, e.g. the following. 
\begin{itemize}
\item The intersection of an $\XDP$ geodesic with a rightmost (or leftmost) directed path.  
\item The set of \textbf{$k$-star points} for $k\geq 2$, defined as vertices which are the starting points of $k$ macroscopically distinct $\XDP$ geodesics.  
\item For two specified vertices $x$ and $y$, the set of vertices $z$ for which the $\XDP$s from $x$ to $z$ and from $y$ to $z$ are equal. 
\end{itemize}
\end{prob}

Theorem~\ref{thm-busemann-conv} gives an exact scaling limit result for directed distances in the UIBOT, but in the finite-volume setting we only have up-to-constants bounds (Theorems~\ref{thm-boltzmann-ldp}, \ref{thm-cell-ldp}, and~\ref{thm-cell-sdp}).

\begin{prob} \label{prob-finite-lim}
Prove a scaling limit result for some natural observable related to directed distances in finite uniform bipolar-oriented triangulations. 
\end{prob}

The most important combinatorial tool for analyzing distance on uniform planar maps (or triangulations, quadrangulations, etc.) is the Schaeffer bijection~\cite{cv-bijection,schaeffer-bijection}. This bijection encodes planar maps via labeled trees, where the labels exactly correspond to graph distances to a special root vertex. 
It is the key tool in the proof that uniform random planar maps converge in the scaling limit to the Brownian map~\cite{legall-uniqueness,miermont-brownian-map}.

\begin{prob} \label{prob-schaffer}
Is there an analog of the Schaeffer bijection for uniform bipolar-oriented triangulations (or bipolar-oriented maps with other face degree distributions), where the labels correspond to the lengths of longest (or shortest) directed paths? 
\end{prob}

%% file: tex/tauberian.tex
We prove the Tauberian result in Proposition~\ref{prop:Tauberian}.
\jb{We work throughout with non-negative integer-valued random variables, as in Proposition~\ref{prop:Tauberian}.}

Throughout this section, let $X$ be a non-negative \jb{integer-valued} random variable. For $z \in \BB C$ with $|z| \leq 1$, define
\eqb \label{eqn-gen-func-def}
\phi(z) = \BB E[z^{X}]
\jb{= \sum_{k=0}^{\infty} \BB P[X=k]z^k}
\eqe
so that the characteristic function of $X$ (as in the statement of Proposition~\ref{prop:Tauberian}) is $\varphi(t) = \phi(e^{it})$.
\jb{The power series defining $\phi$ converges uniformly on the closed unit disk, so $\phi$ is continuous there and holomorphic on the open unit disk.}

\subsection{Proofs for the tail exponent asymptotics}

In this section, after proving the preliminary result of Lemma~\ref{lem-holo-asymp}, we prove in Lemmas~\ref{lem-tauberian}~and~\ref{lem-tauberian'} the two main claims of Proposition~\ref{prop:Tauberian}, i.e.\ the ones in \eqref{eqn-prob-asymp2} and \eqref{eqn-prob-asymp'2}, postponing the justification of the claim about the positivity of the imaginary part of the constant $c$ in \eqref{eqn-prob-asymp2} to Section~\ref{sect:neg-img}.

We start with the following lemma. 
 
\begin{lem} \label{lem-holo-asymp}
Let $\BB D = \{z\in\BB C : |z| < 1\}$. Let $\psi : \ol{\BB D} \to \BB C$ be holomorphic on $\BB D$ and continuous on $\ol{\BB D}$.
Assume that there exists $\nu \in (0,1)$ and $c \in \BB C \setminus i \BB R$ such that
\eqb \label{eqn-holo-asymp-bdy1}
\psi(e^{i t}) = 1 + (c \BB 1_{(t > 0)} + \ol c \BB 1_{(t < 0)} ) |t|^\nu  + o(t^\nu)  ,\quad \text{as $t \to 0$}. 
\eqe 
Then setting $b  = \re(c)\sec\Big(\tfrac{\pi\nu}{2}\Big)  \in \BB R\setminus \{0\}$,
\eqb \label{eqn-holo-asymp-bdy2}
\psi(1-r) = 1  + b r^\nu + o(r^\nu)  ,\quad \text{as $r \to 0$}. 
\eqe 
\end{lem}
\begin{proof} 
It is convenient to change coordinate to the upper half-plane so that we can use a scaling argument. 
So, we define
\eqb
\wt \psi(z) := \psi\left(  \frac{i - z}{i + z}  \right) .
\eqe
The function $\wt \psi$ is bounded, holomorphic on the upper-half plane $\BB H := \{z \in \BB C : \im z > 0\}$, and continuous on $\ol{\BB H}$. The fractional linear transformation with which we pre-composed $\psi$ maps 0 to 1, maps the real line to the unit circle, and is smooth in a neighborhood of the origin. From this, we see that~\eqref{eqn-holo-asymp-bdy1} is equivalent to
\eqb \label{eqn-tauberian-H}
\wt \psi(x) = 1 + \chi(x)  |x|^\nu  + o(x^\nu)  ,\quad \text{as $\BB R\ni x \to 0$} ,\quad \text{where} \quad \chi(x) :=   (c 2^{\nu} \BB 1_{(x > 0)} + \ol c 2^{\nu}  \BB 1_{(x < 0)} ) 
\eqe  
since if $t \in (-\pi , \pi)$ and $e^{it}=\frac{i-x}{i+x}$ then $t=2\arctan(x)\sim2x$; 
and~\eqref{eqn-holo-asymp-bdy2} is equivalent to
\eqb \label{eqn-laplace-H}
\wt \psi(i y) = 1  + b 2^{\nu}  y^\nu  + o(y^\nu)  ,\quad \text{as $\BB R_+ \ni y \to 0$}  
\eqe
since if $1-r=\frac{i-iy}{i+iy}$ then $r=\frac{2y}{1+y}\sim2y$. 
%FullSimplify[E^(2 I ArcTan[x]) - (I - x)/(I + x), Assumptions -> {-1 < x < 1}]

Since $\wt \psi$ is holomorphic, its real and imaginary parts are harmonic. We may therefore use the formula for the Poisson kernel on the half-plane to write
\eqb \label{eqn-poisson-kernel} 
\wt \psi(i y) = \frac{1}{\pi} \int_{-\infty}^\infty \frac{  y }{ x^2  + y^2   } \wt \psi(x) \,dx  .
\eqe
Let 
\eqb  \label{eqn-weighted-kernel}
h(y) := \frac{1}{\pi} \int_{-\infty}^\infty \frac{ \chi(x) |x|^\nu y }{ x^2  + y^2   }   \,dx  
\eqe 
with $\chi(x)$ as in~\eqref{eqn-tauberian-H}. 
By making the change of variables \ewain{$x = y u$, $dx = y\,du$}, we see that $h(y) = y^\nu h(1)$.  
We also note that $h(y)$ is real (the imaginary parts of $c':= 2^\nu c$ and $\ol c'$ cancel by symmetry) and non-zero (since $\re c'$ is non-zero).

We are going to compare $\wt \psi(i y)$ to $1 + h(y)$.
Given $\ep > 0$, we use~\eqref{eqn-tauberian-H} to choose
$\delta >0$ small enough so that
\eqb\label{eq:as-bnd-sim}
|\wt \psi(x) - 1 - \chi(x) \jbb{|x|^\nu}|
\jbb{\leq} \ep |x|^\nu,
\quad \forall x \in [-\delta,\delta].
\eqe
By~\eqref{eqn-poisson-kernel}, writing $\|\wt \psi\|_\infty$
for the (finite) $L^\infty$ norm of $\wt\psi$, we have
\allb
\label{eqn-poisson-decomp}
\left| \wt \psi(i y)
- \frac{1}{\pi}\int_{-\delta}^\delta
\frac{(1 + \chi(x)|x|^\nu)y}{x^2 + y^2}\,dx \right|
&=
\left|
\frac{1}{\pi}
\int_{\BB R\setminus[-\delta,\delta]}
\frac{y}{x^2+y^2}\wt\psi(x)\,dx
-
\frac{1}{\pi}\int_{-\delta}^\delta
\frac{(1+\chi(x)|x|^\nu-\wt\psi(x))y}{x^2+y^2}\,dx
\right| \notag\\
&\leq
\frac{\|\wt\psi\|_\infty}{\pi}
\int_{\BB R\setminus[-\delta,\delta]}
\frac{y}{x^2+y^2}\,dx
+
\frac{\ep}{\pi}\int_{-\delta}^\delta
\frac{|x|^\nu y}{x^2+y^2}\,dx,
\alle
where the last estimate follows from \eqref{eq:as-bnd-sim}.
For a fixed value of $\delta>0$,
\eqb \label{eqn-complement-int}
\int_{\BB R\setminus[-\delta,\delta]}
\frac{(1+|x|^\nu)y}{x^2+y^2}\,dx = O(y).
\eqe
By~\eqref{eqn-complement-int} and the fact that
$\frac{1}{\pi}\int_{-\infty}^\infty
\frac{y}{x^2+y^2}\,dx=1$,
the integral inside the absolute value on the left side
of~\eqref{eqn-poisson-decomp} satisfies
\allb \label{eqn-inside-int}
\frac{1}{\pi}\int_{-\delta}^\delta
\frac{(1+\chi(x)|x|^\nu)y}{x^2+y^2}\,dx
&=
\frac{1}{\pi}\int_{-\infty}^\infty
\frac{(1+\chi(x)|x|^\nu)y}{x^2+y^2}\,dx
-
\frac{1}{\pi}
\int_{\BB R\setminus[-\delta,\delta]}
\frac{(1+\chi(x)|x|^\nu)y}{x^2+y^2}\,dx
\notag\\
&=1+h(y)+O(y).
\alle
By~\eqref{eqn-complement-int}, the first term on the right
side of~\eqref{eqn-poisson-decomp} is at most $O(y)$.
\jbb{As for the second term on the right side, positivity
of the kernel gives}
\eqbn
\jbb{
\frac{\ep}{\pi}\int_{-\delta}^\delta
\frac{|x|^\nu y}{x^2+y^2}\,dx
\leq
\frac{\ep}{\pi}\int_{\BB R}
\frac{|x|^\nu y}{x^2+y^2}\,dx
=
\ep\sec\left(\frac{\pi\nu}{2}\right)y^\nu.
}
\eqen
Plugging \eqref{eqn-inside-int} into the left-hand side
of~\eqref{eqn-poisson-decomp} and using these estimates
for the right-hand side of~\eqref{eqn-poisson-decomp} gives
\eqb\label{eq:final-bnd}
\left|\wt\psi(i y)-1-h(y)\right|
\leq
\jbb{\ep\sec\left(\frac{\pi\nu}{2}\right)y^\nu+C_\delta y,
\qquad 0<y\leq1}.
\eqe
\jbb{Here $C_\delta\geq0$ is a finite constant depending
on $\delta$ and the fixed data, but not on $y$.
For fixed $\ep$ and the corresponding $\delta$,
dividing by $y^\nu$ and letting $y\downarrow0$ makes
the remainder $C_\delta y^{1-\nu}$ tend to zero.}
Since $\ep$ is arbitrary and $h(y)=y^\nu h(1)$,
this gives~\eqref{eqn-laplace-H} with
$b2^\nu=h(1)
\stackrel{\eqref{eqn-weighted-kernel}}{=}
\frac{1}{\pi}\int_{-\infty}^{\infty}
\frac{\chi(x)|x|^\nu}{1+x^2}\,dx$.
Using the symmetry of $\chi(x)$, this becomes
\[
b
=
\frac{1}{\pi}\big(c+\overline{c}\big)
\int_0^\infty\frac{x^\nu}{1+x^2}\,dx
=
\frac{2\re(c)}{\pi}
\int_0^\infty\frac{x^\nu}{1+x^2}\,dx.
\]
The remaining integral is classical:
\[
\int_0^\infty\frac{x^\nu}{1+x^2}\,dx
=
\frac{\pi}{2}\csc\Big(\tfrac{\pi(\nu+1)}{2}\Big)
=
\frac{\pi}{2}\sec\Big(\tfrac{\pi\nu}{2}\Big),
\qquad 0<\nu<1.
\]
Therefore
\[
b=\re(c)\sec\Big(\tfrac{\pi\nu}{2}\Big),
\]
which is real and nonzero since $c\notin i\mathbb{R}$ and
$\cos(\tfrac{\pi\nu}{2})\neq0$ for $0<\nu<1$.
\end{proof}

We now prove the main claim in \eqref{eqn-prob-asymp2}.

\begin{lem} \label{lem-tauberian} 
Let $\phi$ be as in~\eqref{eqn-gen-func-def}. Assume that for some $c \in \BB C\setminus i \BB R$ and $\nu \in (0,1)$,  
\eqb \label{eqn-tauberian}
\phi(e^{i t}) = 1 + (c \BB 1_{(t > 0)} + \ol c \BB 1_{(t < 0)} ) |t|^\nu  + o(t^\nu)  ,\quad \text{as $t \to 0$}. 
\eqe 
Then $\re(c)<0$ and setting $a=-\frac{\re(c)}{\Gamma(1-\nu)}\sec\Big(\tfrac{\pi\nu}{2}\Big)>0$,
\eqb \label{eqn-prob-asymp}
\BB P[X > R] = a R^{-\nu}  +  o(R^{-\nu}) ,\quad \text{as $R\to\infty$}. 
\eqe 
\end{lem}
\begin{proof}%[Proof of Lemma~\ref{lem-tauberian}]
By a standard Tauberian theorem for the Laplace transform (see, e.g.,~\cite[Example (c), Page 447]{feller-book} or~\cite[Corollary 8.1.7]{reg-var-book}), for any constant $a > 0$, the asymptotics 
\eqb \label{eqn-laplace-asymp}
\phi(1-r) = 1 - \Gamma(1-\nu) a r^\nu  + o(r^\nu) ,\quad (0,1) \ni r \to 0  
\eqe 
implies~\eqref{eqn-prob-asymp}. So, we need to show that~\eqref{eqn-tauberian} implies~\eqref{eqn-laplace-asymp} with $a=-\frac{\re(c)}{\Gamma(1-\nu)}\sec\Big(\tfrac{\pi\nu}{2}\Big)>0$. This, in turn, follows from Lemma~\ref{lem-holo-asymp} applied to the holomorphic function $\phi$ (note that, since $0<\nu<1$, $a$ is positive only if $\re(c)<0$).  
\end{proof}

We now prove the main claim in \eqref{eqn-prob-asymp'2}.

\begin{lem} \label{lem-tauberian'} 
Let $\phi$ be as in~\eqref{eqn-gen-func-def}.
Assume that for some $c_1 > 0$, $c_2 \in \BB C \setminus  \BB R$, and $\nu \in (0,1)$,  
\eqb \label{eqn-tauberian'}
\phi(e^{i t}) = 1 + i c_1 t +     (c_2 \BB 1_{(t > 0)} + \ol c_2 \BB 1_{(t < 0)} ) |t|^{1+\nu}       + o(t^{1+\nu})  ,\quad \text{as $t \to 0$}. 
\eqe 
Then $\im(c_2)<0$ and setting $a=-\frac{\nu}{\Gamma(1-\nu)}\sec\left(\tfrac{\pi\nu}{2}\right)\im(c_2)>0$, 
\eqb \label{eqn-prob-asymp'}
\BB P[X > R] = a R^{-\nu-1}  +  o(R^{-\nu-1}) ,\quad \text{as $R\to\infty$}. 
\eqe 
\end{lem} 
\begin{proof}
By a standard Tauberian theorem for the Laplace transform (see, e.g.,~\cite[Theorem 8.1.6]{reg-var-book} with $n=1$), for any $b \in\BB R \setminus \{0\}$, the asymptotics 
\eqb \label{eqn-laplace-asymp'}
\phi(1-r) = 1 -  c_1  r  + b  r^{1+\nu}  + o(r^{1+\nu} ) ,\quad (0,1) \ni r \to 0  
\eqe 
implies~\eqref{eqn-prob-asymp'} with $a=\frac{b\nu}{\Gamma(1-\nu)}$, which is necessarily positive, since $b\not=0$ and probabilities are non-negative. 
So, we need to show that~\eqref{eqn-tauberian'} implies~\eqref{eqn-laplace-asymp'} with the constant $b=-\im(c_2)\sec\left(\tfrac{\pi\nu}{2}\right)$. Then the fact that $\im (c_2)<0$ immediately follows from the fact that $a>0$. 
 
The asymptotics~\eqref{eqn-tauberian'} implies that $\frac{d}{dt} \phi(e^{i t})$ exists and equals $i c_1$. 
By the main theorem of~\cite{pitman-cf-derivative}, this implies that $\lim_{T \to \infty} \BB E[X \BB 1_{\{|X| \leq T\}}] = c_1$. Since $X$ is non-negative, this implies that $\lim_{T \to \infty} \BB E[X \BB 1_{(X \leq T)}] = c_1 $. By the monotone convergence theorem, this implies that $\BB E[X] = c_1$. This, in turn, implies that the complex derivative
\eqbn
\phi'(z) = \lim_{\ol{\BB D} \ni z} \frac{\phi(z) -1}{z-1} 
\eqen
exists and equals $c_1$.
%We have $\frac{d}{dz}  z^X   = X z^{X-1}$, which is bouned above in absolute value by $X$ for $z\in\ol{\BB D}$. By the mean value inequality, $|z^X-1| \leq   X |z-1| $ if $z \in \ol{\BB D}$. Since $\BB E[X] < \infty$, dominated convergence gives that $\BB E[ (z^X -1)/(z-1) ] $ converges to $\BB E[X]$. Alternatively, when $X$ is integer valued, $\left| \frac{z^n-1}{z-1} \right|   =   \left| \sum_{k=0}^{n-1} z^k   \right| \leq n$.

Let
\eqbn
G(z) := \frac{\phi(z) -1}{c_1 ( z-1 )} .
\eqen
The previous paragraph implies that $G$ is continuous on $\ol{\BB D}$ and holomorphic on $\BB D$. 
By~\eqref{eqn-tauberian'} and since $e^{i t } = 1 + i t + O(t^2)$ as $t\to 0$,
\eqbn
G(e^{i t}) = 1 +     (c_3 \BB 1_{(t > 0)} + \ol c_3 \BB 1_{(t < 0)} ) |t|^{ \nu}       + o(t^{ \nu})  ,\quad \text{as $t \to 0$}  
\eqen
where $c_3 = c_2 /(i c_1) \in \BB C\setminus i\BB R$. By Lemma~\ref{lem-holo-asymp} applied to $G$,  
\eqbn
G(1-r) = 1 +\left(\re(c_3)\sec\left(\tfrac{\pi\nu}{2}\right)\right) r^\nu  + o(r^\nu) ,\quad \text{as $r\to 0$} .
\eqen
Re-arranging this implies~\eqref{eqn-laplace-asymp'} with $b = - c_1\re(c_2 /(ic_1))\sec\left(\tfrac{\pi\nu}{2}\right)=-\im(c_2)\sec\left(\tfrac{\pi\nu}{2}\right)$, where we used that $c_1$ is real.
\end{proof}

\begin{proof}[Proof of Proposition~\ref{prop:Tauberian}]
In the case of Assumption~\ref{ass:1}, the asymptotics~\eqref{eqn-prob-asymp2} and the fact that $\re(c) < 0$ follow from Lemma~\ref{lem-tauberian}. The fact that $\im(c) > 0$ is justified in Lemma~\ref{lem:bound-im-part} below.
The case of Assumption~\ref{ass:2} follows directly from Lemma~\ref{lem-tauberian'}. 
\end{proof}

\subsection{Proof that the imaginary part is positive for small $t$ when $0<\nu<1$}\label{sect:neg-img}

\jb{Let $f$ be the probability mass function of $X$. We} prove the claims about the positivity of the imaginary part of the constant $c$ appearing in \eqref{eqn-tauberian2}. Since 
\[
\im(\phi(e^{it}))=\sum_{k\geq 0}f(k)\im(e^{itk})=\sum_{k\geq 0}f(k)\sin(tk),
\]
and Assumption~\ref{ass:1} guarantees that for some $c \in \BB C \setminus i \BB R$ and $\nu \in (0,1)$,   
\eqb 
\varphi(t) = \phi(e^{i t}) = 1 + (c \BB 1_{(t > 0)} + \ol c \BB 1_{(t < 0)} ) |t|^\nu  + o(t^\nu)  ,\quad \text{as $t \to 0$},
\eqe
it is sufficient to establish the following result.

\begin{lem}\label{lem:bound-im-part}
Let $f:\mathbb{N}_0\to[0,\infty)$ and set $S(k):=\sum_{y\ge k} f(y)$.
Assume that for some $c'>0$ and $0<\nu<1$,
\begin{equation}\label{eq:tail}
S(k)=c'\,k^{-\nu}+o(k^{-\nu}),\qquad \text{as }k\to\infty.
\end{equation}
Then there exists $t_0>0$ such that
\[
\sum_{k\ge 0} f(k)\,\sin(tk)>0\qquad\text{for all }t\in(0,t_0].
\]
\end{lem}

We first prove two preliminary lemmas.

\begin{lem}\label{lem:abel}
For every $t>0$,
\begin{equation}\label{eq:abel}
\sum_{k\ge0} f(k)\sin(tk)
= \sum_{k\ge0} S(k+1)\big(\sin((k+1)t)-\sin(kt)\big).
\end{equation}
Equivalently,
\begin{equation}\label{eq:int-repr}
\sum_{k\ge0} f(k)\sin(tk)
= \int_{0}^{\infty} S\big(\lfloor u/t\rfloor+1\big)\cos u\,du .
\end{equation}
\end{lem}

\begin{proof}
For $N\ge1$, by telescoping,
\[
\sum_{k=0}^{N} f(k)\sin(tk)=\sum_{k=0}^{N-1} S(k+1)\big(\sin((k+1)t)-\sin(kt)\big)-S(N+1)\sin(Nt),
\]
because $f(k)=S(k)-S(k+1)$. Letting $N\to\infty$, since $S(N+1)=O(N^{-\nu})$ and $\nu>0$, we get that the boundary term vanishes and \eqref{eq:abel} follows. Using
$\sin((k+1)t)-\sin(kt)=\int_{kt}^{(k+1)t}\cos u\,du$ yields \eqref{eq:int-repr}.
\end{proof}

\begin{lem}\label{lem:riemann}
Let
\[
I_\nu(t):=\int_{0}^{\infty}\big(\lfloor u/t\rfloor+1\big)^{-\nu}\cos u\,du.
\]
If $0<\nu<1$, then as $t\downarrow0$,
\begin{equation}\label{eq:Ialpha<1}
I_\nu(t)=t^{\nu}\int_{0}^{\infty}u^{-\nu}\cos u\,du+O(t)
= \Gamma(1-\nu)\sin\frac{\pi\nu}{2}\,t^{\nu}+o(t^{\nu}).
\end{equation}
\end{lem}

\begin{proof}
For each $\nu>0$,
\[
I_\nu(t)-t^{\nu}\int_{0}^{\infty}u^{-\nu}\cos u\,du
=\int_{0}^{\infty}\Big[\big(\lfloor u/t\rfloor+1\big)^{-\nu}-(u/t)^{-\nu}\Big]\cos u\,du .
\]
Split the integral at $u=t$:
\[
\Big|I_\nu(t)-t^{\nu}\int_{0}^{\infty}u^{-\nu}\cos u\,du\Big|
\le t^{\nu}\int_{0}^{t}u^{-\nu}\,du
+\int_{t}^{\infty}\Big|\big(\lfloor u/t\rfloor+1\big)^{-\nu}-(u/t)^{-\nu}\Big|\,du .
\]
The first term is
\[
t^{\nu}\int_{0}^{t}u^{-\nu}\,du=\frac{t}{1-\nu}.
\]
For the second term, set $x=u/t\ge1$. By the mean value theorem for $x\mapsto x^{-\nu}$, there is a constant $C=C(\nu)$ such that for $x\ge1$,
\[
\big|(\lfloor x\rfloor+1)^{-\nu}-x^{-\nu}\big|\le C\,x^{-\nu-1}.
\]
Hence
\[
\int_{t}^{\infty}\Big|\big(\lfloor u/t\rfloor+1\big)^{-\nu}-(u/t)^{-\nu}\Big|\,du
\le C\int_{t}^{\infty}(u/t)^{-\nu-1}\,du
= C t .
\]
Combining the two bounds gives
\[
I_\nu(t)=t^{\nu}\int_{0}^{\infty}u^{-\nu}\cos u\,du+O(t).
\]
Since $0<\nu<1$, we have $t=o(t^{\nu})$, so the remainder is $o(t^{\nu})$. Finally, using the classical identity (valid for $0<\nu<1$)
\[
\int_{0}^{\infty}u^{-\nu}\cos u\,du=\Gamma(1-\nu)\sin\frac{\pi\nu}{2},
\]
we obtain \eqref{eq:Ialpha<1}.
\end{proof}

\begin{proof}[Proof of Lemma~\ref{lem:bound-im-part}]
Write $F(t):=\sum_{k\ge0} f(k)\sin(tk)$. By Lemma~\ref{lem:abel},
\[
F(t)=\int_{0}^{\infty} S\big(\lfloor u/t\rfloor+1\big)\cos u\,du .
\]
Decompose $S(n)=c'n^{-\nu}+r(n)$ with $r(n)=o(n^{-\nu})$, and set
\[
I_\nu(t):=\int_{0}^{\infty}\big(\lfloor u/t\rfloor+1\big)^{-\nu}\cos u\,du,
\qquad
R(t):=\int_{0}^{\infty} r\big(\lfloor u/t\rfloor+1\big)\cos u\,du .
\]
Then $F(t)=c'I_\nu(t)+R(t)$. By Lemma~\ref{lem:riemann},
\[
c'I_\nu(t)=c'\,\Gamma(1-\nu)\sin\Big(\frac{\pi\nu}{2}\Big)\,t^{\nu}+o(t^{\nu}).
\]
Since $\Gamma(1-\nu)\sin\Big(\frac{\pi\nu}{2}\Big) > 0$, it remains to show $R(t)=o(t^{\nu})$.
Fix $\varepsilon>0$ and choose $K_0$ large enough so that $|r(n)|\le\varepsilon n^{-\nu}$ for all $n\ge K_0$.
Write
\[
R(t)=\sum_{k\ge0} r(k+1)\int_{kt}^{(k+1)t}\cos u\,du
=:\,R_{<K}(t)+R_{\ge K}(t),
\]
where $K:=\max\{K_0,\lceil 1/t\rceil\}$, and the sums defining $R_{<K}$ and $R_{\ge K}$ run over $k<K$ and $k\ge K$, respectively.

Since $\Big|\int_{kt}^{(k+1)t}\cos u\,du\Big|\le t$, we have
\[
|R_{<K}(t)|\le t\sum_{k<K}|r(k+1)|.
\]
Given $\delta>0$, pick $N$ such that $|r(n)|\le\delta n^{-\nu}$ for $n\ge N$. Split the sum at $N$:
\[
t\sum_{k<K}|r(k+1)|\le t\sum_{k<N}|r(k+1)|+ \delta\,t\sum_{N\le k<K}(k+1)^{-\nu}
=O( t+\delta\,t\,K^{1-\nu}).
\]
Since $K \asymp 1/t$,  this yields $|R_{<K}(t)|=O( t+\delta\,t^\nu)=o(t^{\nu})$ because $t=o(t^{\nu})$ when $0<\nu<1$, and $\delta>0$ is arbitrary.

Now we look at $R_{\geq K}(t)$. Let $b_k(t):=\int_{kt}^{(k+1)t}\cos u\jb{\,du}=\sin((k+1)t)-\sin(kt)$. Then $|b_k(t)|\le 2|\sin(t/2)|\le t$.

\jb{Assume $\varepsilon\in(0,1)$, set $J:=\lceil K/\varepsilon\rceil$, and split the sum defining $R_{\ge K}(t)$ at $J$.}
For $k\ge K\ge K_0$, we have $|r(k+1)|\le \varepsilon (k+1)^{-\nu}$, hence
\jb{
\[
\left|\sum_{k=K}^{J-1}r(k+1)b_k(t)\right|
\le \varepsilon t\sum_{k=K}^{J-1}(k+1)^{-\nu}
=O(\varepsilon tJ^{1-\nu})
=O(\varepsilon^\nu tK^{1-\nu}).
\]
}

\jb{For the remaining tail, recall that
$r(k+1)=S(k+1)-c'(k+1)^{-\nu}$.
Both $S(k+1)$ and $(k+1)^{-\nu}$ are nonnegative, nonincreasing, and tend to zero. Moreover,
\[
\left|\sum_{k=J}^{N}b_k(t)\right|
=|\sin((N+1)t)-\sin(Jt)|\le 2.
\]
Summation by parts therefore gives
\[
\left|\sum_{k\ge J}r(k+1)b_k(t)\right|
\le 2S(J+1)+2c'(J+1)^{-\nu}
\le 2(2c'+\varepsilon)(J+1)^{-\nu}
=O(\varepsilon^\nu K^{-\nu}),
\]
with implicit constants independent of $\varepsilon$.}

Therefore $|R_{\ge K}(t)|=\jb{O\left(\varepsilon^\nu
\bigl(tK^{1-\nu}+K^{-\nu}\bigr)\right)}$.
With $K\asymp 1/t$ this gives
$|R_{\ge K}(t)|=\jb{O(\varepsilon^\nu t^\nu)}$.

Combining the two bounds and letting $\varepsilon\to0$ shows $R(t)=o(t^{\nu})$
which completes the proof.
\end{proof}